\documentclass[11pt,reqno]{amsart}
\usepackage{amssymb,mathrsfs,color}
\usepackage{tikz}
\usepackage{pgfplots}
\usetikzlibrary{snakes}
\usetikzlibrary{intersections, pgfplots.fillbetween}
\usepackage[margin=1in, letterpaper]{geometry}
\usepackage{enumerate}
\usepackage{graphicx}
\usepackage{xcolor} % A package to add color.
\usepackage{geometry}
\usepackage{amsfonts,amssymb,amsmath}
\usepackage{slashed}
\usepackage{ mathrsfs }
\usepackage[titletoc,title]{appendix}
\usepackage{wrapfig}
\usepackage{cancel}

\usepackage{cite}

\usepackage{nth}

\usepackage{hyperref}

\usepackage{marginnote}

\usepackage{pgfplots}

\usepackage{dsfont}

\definecolor{green}{rgb}{0,0.8,0} % Redefines the color green.

\definecolor{deepgreen}{cmyk}{100,0,100,0}
\newcommand{\Blue}[1]{{\color{blue} #1}}

\newcommand{\pmat}[1]{\begin{pmatrix} #1 \end{pmatrix}}

\newcommand{\Del}[1]{}

\numberwithin{equation}{section}

\newtheorem{theorem}{Theorem}[section]
\newtheorem{corollary}[theorem]{Corollary}%[section]
\newtheorem{lemma}[theorem]{Lemma}%[section]
\newtheorem{proposition}[theorem]{Proposition}%[section]
\newtheorem{remark}[theorem]{Remark}%[section]
\newtheorem{definition}[theorem]{Definition}%[section]
\newcommand{\mwhere}{{\ \ \text{where} \ \ }}
\newcommand{\mand}{{\ \ \text{and} \ \  }}

\newcommand{\abs}[1]{\vert#1\vert}

\newcommand{\set}[1]{\{#1\}}

\newcommand{\angles}[2]{\langle #1,#2\rangle}

\newcommand{\jap}[1]{\langle #1\rangle}

\renewcommand{\div}{\mathrm{div}\,}

\DeclareMathOperator{\sgn}{sgn}

\newcommand{\bfh}{{\bf h}}

\newcommand{\bfn}{{\bf n}}

\newcommand{\bfp}{{\bf p}}
\newcommand{\bfq}{{\bf q}}

\newcommand{\bfs}{{\bf s}}

\newcommand{\bfH}{{\bf H}}

\newcommand{\bfP}{{\bf P}}

\newcommand{\bfT}{{\bf T}}

\newcommand{\bfeta}{\boldsymbol{\eta}}

\newcommand{\bfOmega}{\boldsymbol{\Omega}}

\newcommand{\abar}{{\underline a}}

\newcommand{\ebar}{{\underline e}}

\renewcommand{\hbar}{{\underline h}}

\newcommand{\ellbar}{{\underline \ell}}
\newcommand{\mbar}{{\underline m}}
\newcommand{\nbar}{{\underline n}}

\newcommand{\rbar}{{\underline r}}

\newcommand{\Hbar}{{\underline H}}

\newcommand{\Lbar}{{\underline L}}

\newcommand{\Rbar}{{\underline R}}

\newcommand{\Xbar}{{\underline X}}

\newcommand{\Zbar}{{\underline Z}}

\newcommand{\gammabar}{\underline{\gamma}}

\newcommand{\thetabar}{{\underline{\theta}}}

\newcommand{\xibar}{\underline{\xi}}

\newcommand{\pibar}{\underline{\uppi}}

\newcommand{\varphibar}{\underline{\varphi}}

\newcommand{\Omegabar}{\underline{\Omega}}

\newcommand{\bbP}{\mathbb P}

\newcommand{\bbR}{\mathbb R}
\newcommand{\bbS}{\mathbb S}

\newcommand{\bbZ}{\mathbb Z}

\newcommand{\calA}{\mathcal A}
\newcommand{\calB}{\mathcal B}
\newcommand{\calC}{\mathcal C}
\newcommand{\calD}{\mathcal D}
\newcommand{\calE}{\mathcal E}
\newcommand{\calF}{\mathcal F}

\newcommand{\calH}{\mathcal H}

\newcommand{\calL}{\mathcal L}
\newcommand{\calM}{\mathcal M}
\newcommand{\calN}{\mathcal N}
\newcommand{\calO}{\mathcal O}
\newcommand{\calP}{\mathcal P}
\newcommand{\calQ}{\mathcal Q}
\newcommand{\calR}{\mathcal R}
\newcommand{\calS}{\mathcal S}
\newcommand{\calT}{\mathcal T}

\newcommand{\calX}{\mathcal X}
\newcommand{\calY}{\mathcal Y}
\newcommand{\calZ}{\mathcal Z}

\newcommand{\frakf}{\mathfrak f}

\newcommand{\frakm}{\mathfrak m}
\newcommand{\frakn}{\mathfrak n}

\newcommand{\frakz}{\mathfrak z}

\newcommand{\tilb}{{\tilde{b}}}

\newcommand{\tilc}{{\tilde{c}}}

\newcommand{\tile}{{\tilde{e}}}

\newcommand{\tilf}{{\tilde{f}}}

\newcommand{\tilg}{{\tilde{g}}}

\newcommand{\tilh}{{\tilde{h}}}

\newcommand{\tilk}{{\tilde{k}}}

\newcommand{\tiln}{{\tilde{n}}}

\newcommand{\tilq}{{\tilde{q}}}

\newcommand{\tilr}{{\tilde{r}}}

\newcommand{\tilt}{{\tilde{t}}}

\newcommand{\tily}{{\tilde{y}}}

\newcommand{\tilA}{{\tilde{A}}}

\newcommand{\tilB}{{\tilde{B}}}

\newcommand{\tilF}{{\tilde{F}}}

\newcommand{\wtilH}{{\widetilde{H}}}
\newcommand{\tilI}{{\tilde{I}}}

\newcommand{\tilK}{{\tilde{K}}}

\newcommand{\tilL}{{\tilde{L}}}

\newcommand{\tilN}{{\tilde{N}}}
\newcommand{\wtilN}{{\widetilde{N}}}

\newcommand{\tilR}{{\tilde{R}}}
\newcommand{\wtilR}{{\widetilde{R}}}
\newcommand{\tilS}{{\tilde{S}}}
\newcommand{\wtilS}{{\widetilde{S}}}

\newcommand{\scB}{{\mathscr{B}}}
\newcommand{\scC}{{\mathscr{C}}}

\newcommand{\scE}{{\mathscr{E}}}

\newcommand{\scR}{{\mathscr{R}}}

\newcommand{\ringa}{{\mathring a}}
\newcommand{\ringb}{{\mathring b}}
\newcommand{\ringc}{{\mathring c}}

\newcommand{\ringm}{{\mathring m}}

\newcommand{\ringp}{{\mathring p}}

\newcommand{\ringN}{{\mathring N}}

\newcommand{\ringpi}{\mathring{\uppi}}

\newcommand{\ringSigma}{\mathring{\Sigma}}

\newcommand{\ringPhi}{\mathring{\Phi}}

\newcommand{\ringpsi}{\mathring{\psi}}
\newcommand{\ringepsilon}{\mathring{\epsilon}}

\newcommand{\vecf}{{\vec f}}

\newcommand{\vecu}{{\vec u}}
\newcommand{\vecv}{{\vec v}}

\newcommand{\vecF}{{\vec F}}
\newcommand{\vecG}{{\vec G}}
\newcommand{\vecH}{{\vec H}}

\newcommand{\vecK}{{\vec K}}

\newcommand{\vecN}{{\vec N}}

\newcommand{\vecZ}{{\vec Z}}

\newcommand{\vecphi}{\vec{\phi}}
\newcommand{\vecvarphi}{\vec{\varphi}}
\newcommand{\vecfy}{\vec{\varphi}}

\newcommand{\vecpsi}{\vec{\psi}}

\newcommand{\vecomega}{\vec{\omega}}

\newcommand{\barcalC}{{\underline{\calC}}}
\newcommand{\barcalH}{{\underline{\calH}}}

\newcommand{\tilbfP}{{\tilde{\bfP}}}

\newcommand{\ringsg}{\mathring{\slashed{g}}}
\newcommand{\slambda}{\slashed{\lambda}}

\newcommand{\ringsDelta}{{\mathring{\slashed{\Delta}}}}
\newcommand{\sDelta}{{\slashed{\Delta}}}
\newcommand{\snabla}{{\slashed{\nabla}}}
\newcommand{\spartial}{{\slashed{\partial}}}
\newcommand{\ringsnabla}{{\mathring{\slashed{\nabla}}}}

\newcommand{\ud}{\mathrm{d}}
\newcommand{\udt}{\,\mathrm{d} t}
\newcommand{\udx}{\,\mathrm{d} x}

\newcommand{\dotxi}{{\dot{\xi}}}
\newcommand{\ddotxi}{{\ddot{\xi}}}
\newcommand{\doteta}{{\dot{\eta}}}

\newcommand{\tilpsi}{{\tilde{\psi}}}
\newcommand{\fy}{\varphi}

\newcommand{\tilrho}{{\tilde{\rho}}}
\newcommand{\tilvarphi}{{\tilde{\varphi}}}
\newcommand{\tilfy}{{\tilde{\fy}}}

\newcommand{\tiltau}{{\tilde{\tau}}}
\newcommand{\tilomega}{{\tilde{\omega}}}
\newcommand{\tilOmega}{{\tilde{\Omega}}}
\newcommand{\tilLbar}{{\tilde{\Lbar}}}

\newcommand{\dotell}{{\dot{\ell}}}
\newcommand{\dotphi}{{\dot{\phi}}}
\newcommand{\dotpsi}{{\dot{\psi}}}
\newcommand{\dotfy}{{\dot{\fy}}}

\newcommand{\ddotell}{\ddot{\ell}}

\newcommand{\dotwp}{{\dot{\wp}}}

\newcommand{\ringpartial}{{\mathring{\partial}}}

\newcommand{\far}{{\mathrm{far}}}
\newcommand{\near}{{\mathrm{near}}}

\newcommand{\pert}{{\mathrm{pert}}}
\newcommand{\Int}{{\mathrm{int}}}
\newcommand{\Ext}{{\mathrm{ext}}}
\newcommand{\tilchi}{{\tilde{\chi}}}
\newcommand{\callP}{\calP}

\newcommand{\dotF}{{\dot{F}}}

\newcommand{\temp}{{\mathrm{temp}}}

\newcommand{\tilcalP}{{\widetilde{\calP}}}
\newcommand{\tilcalR}{{\widetilde{\calR}}}
\newcommand{\tilDelta}{{\tilde{\Delta}}}

\newcommand{\tilbfOmegabar}{{\tilde{\underline{\bfOmega}}}}
\newcommand{\bfOmegabar}{{\underline{\bfOmega}}}
\newcommand{\tilhbar}{{\underline{\tilh}}}
\newcommand{\tilnbar}{{\underline{\tiln}}}

\renewcommand{\max}{{\mathrm{max}}}
\renewcommand{\min}{{\mathrm{min}}}
\newcommand{\hyp}{{\mathrm{hyp}}}
\newcommand{\flatt}{{\mathrm{flat}}}
\newcommand{\tilgamma}{{\tilde{\gamma}}}
\newcommand{\Err}{{\mathrm{Err}}}
\newcommand{\ext}{{\mathrm{ext}}}
\newcommand{\supp}{{\mathrm{supp}}}
\newcommand{\Ell}{{\mathrm{ell}}}
\newcommand{\tilpartial}{{\tilde{\partial}}}
\newcommand{\tiluptau}{{\tilde{\uptau}}}
\newcommand{\tiluprho}{{\tilde{\uprho}}}
\newcommand{\tiluptheta}{{\tilde{\uptheta}}}
\newcommand{\tilupsigma}{{\tilde{\upsigma}}}

\newcommand{\ringcalP}{{\mathring{\calP}}}

\newcommand{\pt}{\partial_t}
\newcommand{\pj}{\partial_j}
\newcommand{\pk}{\partial_k}
\newcommand{\pmu}{\partial_\mu}
\newcommand{\pnu}{\partial_\nu} 

\newcommand{\vecwp}{{\wp}}
\newcommand{\dotvecwp}{{\dot{\wp}}}
\newcommand{\veccalN}{{\vec{\calN}}}
\newcommand{\dota}{{\dot{a}}}
\newcommand{\vecbfOmega}{{\vec{\bfOmega}}}
\newcommand{\vecUpomega}{{\vec{\Upomega}}}

\newcommand{\RbfT}{{\bfT}}
\newcommand{\tilupsi}{{\tilde{\uppsi}}}
\newcommand{\euc}{{\mathrm{euc}}}
\newcommand{\cat}{{\mathrm{cat}}}
\newcommand{\tilkappa}{{\tilde{\kappa}}}
\newcommand{\bfUpomega}{{\boldsymbol{\Upomega}}}
\newcommand{\bfupsigma}{{\boldsymbol{\upsigma}}}
\newcommand{\rrho}{\frakz}
\newcommand{\oomega}{\omega}
\newcommand{\secondff}{ {\mathrm{I\!I}} }
\newcommand{\trap}{{\mathrm{trap}}}
\newcommand{\Rone}{{R}}
\newcommand{\tilcalbfP}{{\tilde{\boldsymbol{\calP}}}}

\newcommand{\Reigenfunctioncutoffscale}{{R_1}}
\newcommand{\callO}{\calO}
\newcommand{\bsUpsigma}{{\boldsymbol{\Upsigma}}}
\newcommand{\barbsUpsigma}{{\underline{\bsUpsigma}}}
\newcommand{\tilbsUpsigma}{{\tilde{\boldsymbol{\Upsigma}}}}
\newcommand{\bartilbsUpsigma}{{\underline{\tilbsUpsigma}}}
\newcommand{\HHbar}{\Hbar}
\newcommand{\graph}{{\mathrm{graph}}}
\usepackage{esvect}
\usepackage{upgreek}
\usepackage{enumitem}
\title[Stability of the catenoid for the HVMC equation]{Stability of the catenoid for the hyperbolic vanishing mean curvature equation outside symmetry}

\begin{document}

\begin{abstract} 
We study the problem of stability of the catenoid, which is an asymptotically flat rotationally symmetric minimal surface in Euclidean space, viewed as a stationary solution to the hyperbolic vanishing mean curvature equation in Minkowski space. The latter is a quasilinear wave equation that constitutes the hyperbolic counterpart of the minimal surface equation in Euclidean space. Our main result is the nonlinear asymptotic stability, modulo suitable translation and boost (i.e., modulation), of the $n$-dimensional catenoid with respect to a codimension one set of initial data perturbations without any symmetry assumptions, for $n \geq 5$. The modulation and  the codimension one restriction on the data are necessary and optimal in view of the kernel and the unique simple eigenvalue, respectively, of the stability operator of the catenoid. 

In a broader context, this paper fits in the long tradition of studies of soliton stability problems. From this viewpoint, our aim here is to tackle some new issues that arise due to the quasilinear nature of the underlying hyperbolic equation. Ideas introduced in this paper include a new profile construction  and modulation analysis to track the evolution of the translation and boost parameters of the stationary solution, a new scheme for proving integrated local energy decay for the perturbation in the quasilinear and modulation-theoretic context, and an adaptation of the vectorfield method in the presence of dynamic translations and boosts of the stationary solution.
\end{abstract} 

\thanks{J. L\"uhrmann was supported by NSF grant DMS-1954707. S.-J.~Oh was supported by the Samsung Science and Technology Foundation under Project Number SSTF-BA1702-02, a Sloan Research Fellowship and a NSF CAREER Grant DMS-1945615. S. Shahshahani was supported by the Simons Foundation grant 639284. The authors are grateful to the referees for valuable comments and suggestions.}

\author{Jonas L\"uhrmann}
\author{Sung-Jin Oh}
\author{Sohrab Shahshahani}
\maketitle 

\tableofcontents

%%%%%%%%%%%%%%%%%%%%%%
%%%%%%%%%%%%%%%%%%%%%%
%%%%%%%%%%%%%%%%%%%%%%
\section{Introduction}\label{sec:intro}
%%%%%%%%%%%%%%%%%%%%%%
%%%%%%%%%%%%%%%%%%%%%%
%%%%%%%%%%%%%%%%%%%%%%

\emph{Catenoids} are among the simplest examples of a non-flat minimal hypersurface in Euclidean space. With respect to the Lorentzian generalization of the minimal hypersurface equation, which is a quasilinear wave equation that will be referred to as the \emph{hyperbolic vanishing mean curvature equation} (HVMC equation) in this paper, these minimal hypersurfaces furnish examples of nontrivial asymptotically flat time-independent solutions. From this point of view, a fundamental question is that of the \emph{nonlinear asymptotic stability} of catenoids as solutions to the HVMC equation -- this will be the subject of the present paper.

Our main result is the nonlinear asymptotic stability, modulo suitable translation and boost, of the $n$-dimensional catenoid as a solution to the HVMC equation with respect to a ``codimension-$1$'' set of initial data perturbations without any symmetry assumptions, for $n \geq 5$ (see Theorem~\ref{thm:main-0} below). The codimension-$1$ condition is necessary and sharp, in view of the fact that the linearized HVMC equation around the catenoid admits a one-parameter family of growing solutions corresponding to the negative eigenvalue of the stability operator (second variation of area). The necessity for an adjustment of the translation and boost parameters (i.e., \emph{modulation}) stems from the kernel of the linearized equation arising from Lorentz invariance. Our result extends the pioneering work of Donninger--Krieger--Szeftel--Wong \cite{DKSW}, which considers the same problem in radial symmetry for $n \geq 2$, to the non-symmetric context albeit for $n \geq 5$.

Beyond the intrinsic interest in the asymptotic stability problem for the catenoid, our motivation for this work is to take on specific challenges for soliton stability problems brought on by the quasilinear nature of the corresponding hyperbolic evolution equations.
We are hopeful for applications of our approach in the study of well-known topological solitons arising in quasilinear wave equations such as the Skyrmion for the Skyrme model.

\subsection{Stability Problems for the Hyperbolic Vanishing Mean Curvature Equation}
We begin by giving a precise formulation of the hyperbolic vanishing mean curvature equation. 
Let $(\bbR^{1+(n+1)},\bfeta)$ be the $1+(n+1)$ dimensional Minkowski space with the standard metric
%%%%%%%
%%%%%%%
\begin{align*}
\begin{split}
\bfeta=-\ud X^0\otimes \ud X^0+\ud X^1\otimes \ud X^1+\dots + \ud X^{n+1}\otimes \ud X^{n+1},
\end{split}
\end{align*}
and let $\calM$ be an $n+1$ dimensional connected orientable manifold without boundary. We consider embeddings $\Phi:\calM\to \bbR^{1+{(n+1)}}$ such that the pull-back metric $\Phi^\ast\bfeta$ is Lorentzian (i.e., $\Phi(\calM)$ is timelike), and which satisfy 
%%%%%%%
%%%%%%%
\begin{align}\label{eq:HVMC1}
\begin{split}
\Box_{\Phi^\ast \bfeta}\Phi=0.
\end{split}
\end{align}
The vector $\Box_{\Phi^\ast \bfeta}\Phi$  is the \emph{mean curvature vector} of $\Phi(\calM)$ as a hypersurface in $\bbR^{1+(n+1)}$, and equation~\eqref{eq:HVMC1} is the requirement that this hypersurface have vanishing mean curvature (VMC).  Embeddings satisfying these requirements are called \emph{(timelike) maximal} and equation~\eqref{eq:HVMC1} is referred to as the hyperbolic vanishing mean curvature equation (HVMC equation). When there is no risk of confusion, by a slight abuse of notation, we will often identify $\calM$ with its image $\Phi(\calM)$ and simply refer to $\calM$ as a hypersurface of $\bbR^{1+(n+1)}$. The HVMC equation is the hyperbolic analogue of the elliptic minimal surface equation (or the parabolic mean curvature flow), and arises variationally as the Euler-Lagrange equations of the area functional 
%%%%%%%
%%%%%%%
\begin{align}\label{eq:area1}
\begin{split}
\calA(\Phi)=\int_{\calM}\sqrt{|\det \Phi^\ast\bfeta|}.
\end{split}
\end{align}
Maximal hypersurfaces are also called \emph{membranes} when $n=2$, and \emph{strings} when $n=1$. 

As \eqref{eq:HVMC1} is a system of wave equations, it is natural to consider the associated Cauchy problem which can be described as follows. Given a coordinate patch $U\subseteq \calM$ with coordinates $s=(s^0,\dots,s^n)$, let $U_0:=\{s\in U~\vert~ s^0=0\}\subset U$, and consider two functions $\Phi_0,\Phi_1\colon U_0\to \bbR^{1+(n+1)}$. We assume that $\Phi_0$ is an embedding, that $\Phi_0^\ast \bfeta$ is Riemannian, and that the metric
%%%%%%%
%%%%%%%
\begin{align*}
g_{\mu\nu}:=\begin{cases}
\bfeta(\partial_{\mu}\Phi_0,\partial_\nu\Phi_0),\quad &\mu,\nu=1,\dots,n\\
\bfeta(\Phi_1,\partial_\nu\Phi_0),\quad &\mu=0,\nu=1,\dots n\\
\bfeta(\partial_\mu\Phi_0,\Phi_1),\quad &\mu=1,\dots n, \nu=0\\
\bfeta(\Phi_1,\Phi_1),\quad &\mu=\nu=0
\end{cases}
\end{align*}
satisfies $\sup_{U_0}\det g<0$. We then ask if there is a neighborhood $V\subseteq U$ of $U_0$ such that there is a timelike embedding $\Phi\colon V\to \bbR^{1+(n+1)}$ satisfying \eqref{eq:HVMC1}, as well as $\Phi\vert_{U_0}=\Phi_0$ and $\partial_{0}\Phi\vert_{U_0}=\Phi_1$. Due to the diffeomorphism invariance of the problem, the solution cannot be unique. But, it is shown in \cite{AC2} that this problem admits a solution $\Phi$ and that any two solutions $\Phi$ and $\Psi$ are related by a diffeomorphism. In the present work we are interested in manifolds $\calM$ that can be written as direct products $\calM=\bbR\times M$ (in fact, we will soon restrict attention to the case where $M$ is a catenoid). In this case we use $(t,x)$ to denote points in $\bbR\times M$. Given $\Phi_0:M\to \{0\}\times \bbR^{n+1}\subseteq \bbR^{1+(n+1)}$ and a family of future directed timelike vectors $\Phi_1:M\to \bbR^{1+(n+1)}$, by finite speed of propagation and standard patching arguments, the result of \cite{AC2} implies the existence of an interval $I\ni 0$, and a unique solution $\Phi\colon I\times M\to\bbR^{1+(n+1)}$ to \eqref{eq:HVMC1} such that $\Phi(t,M)\subseteq \{t\}\times\bbR^{n+1}$, $\Phi(0)=\Phi_0$, and $\partial_t\Phi(0)=\Phi_1$.

Having a satisfactory local theory, one can consider the question of global (in time) dynamics of solutions to \eqref{eq:HVMC1}. For instance, in the context of the Cauchy problem formulated on $\bbR\times M$, one can ask if the local solution extends from $I\times M$ to all of $\bbR\times M$, and if so, how it behaves as $t\to\pm\infty$. A special class of maximal hypersurfaces for which the global dynamics are easily described are the products of  Riemannian VMC surfaces in $\bbR^{n+1}$ with $\bbR$. More precisely, if $\Phi_0\colon M\to\bbR^{n+1}$ is a Riemannian embedding with vanishing mean curvature, then $\Phi\colon \bbR\times M\to\bbR^{1+(n+1)}$ given by $\Phi(t,x)=(t,\Phi_0(x))$ satisfies \eqref{eq:HVMC1} with $\Phi(0)=\Phi_0$ and $\partial_t\Phi(0)=(1,0)$.  We refer to such product solutions as \emph{stationary} solutions. 

A natural question regarding the long time dynamics of solutions of \eqref{eq:HVMC1} is the stability of stationary solutions. The simplest case is when $\Phi_0$ is a linear embedding of a hyperplane in $\bbR^{n+1}$. When $n\geq 3$, it was proved in \cite{B1} that small perturbations of a hyperplane solution lead to global solutions which decay back to a hyperplane. A similar result when $n=2$ was later proved in \cite{Lin1}. Analytically, the results in \cite{B1,Lin1} amount to proving global existence and decay estimates for solutions to a system of quasilinear wave equations with small initial data on Minkowski space. From this point of view, hyperplanes can be thought of as the zero solution to~\eqref{eq:HVMC1}.

The first stability result for a non-flat stationary solution of \eqref{eq:HVMC1} is contained in \cite{DonKrie1,DKSW} for the Lorentzian catenoid. The Riemannian catenoid is a VMC surface of revolution in $\bbR^{n+1}$ (see Section~\ref{subsec:Riemcat} for a more detailed description), and the Lorentzian catenoid is the corresponding stationary solution. The authors in \cite{DKSW} consider radial perturbations of the $(1+2)$ dimensional Lorentzian catenoid that satisfy an additional discrete symmetry\footnote{This is an important technical assumption, which avoids the resonances of the linearized operator in dimension two.}. Their main result asserts that if the initial data belong to a codimension one subset in an appropriate topology, then the corresponding solution can be extended globally and converges to a Lorentzian catenoid as $t\to\infty$. The codimension one restriction on the initial data is necessary and sharp (see the comments following Theorem~\ref{thm:main-0}). A similar result for radial perturbations of the Lorentzian helicoid was subsequently obtained in \cite{Marchali1}.  From a PDE point of view, \cite{DKSW,Marchali1} establish the codimension one (asymptotic) stability of a time-independent solution to a quasilinear wave equation on Minkowski space under radial symmetry. 

In this work we prove the codimension one stability of the $(1+n)$ dimensional Lorentzian catenoid in dimensions $n\geq 5$ without any symmetry restrictions on the perturbations. In the previous paragraphs we mentioned the results on the HVMC equation which are most directly related to our work. A more complete account is given in Section~\ref{subsubsec:morehistory} below. Before providing a simplified statement of our main theorem, we describe the catenoid solution in more detail in the next subsection.

%%%%%%%%%%%%%%%%%
%%%%%%%%%%%%%%%%%
\subsection{The Catenoid Solution}\label{subsec:Riemcat}
%%%%%%%%%%%%%%%%%
%%%%%%%%%%%%%%%%%

We recall some basic geometric properties of the catenoid.
In $\bbR^{n+1}$ let $\Xbar=(X',X^{n+1})$ where $X'=(X^1,\dots,X^n)$ denote the first $n$ coordinates and $X^{n+1}$ the last coordinate.  Suppose $I$ is an interval in $\bbR$ (possibly all of $\bbR$), which we identify with the $X^{n+1}$ axis, and let
%%%%%%%
%%%%%%%
\begin{align*}
\begin{split}
\frakf: I \to [1,\infty)
\end{split}
\end{align*}
be a given even function with\footnote{One could also consider $\frakf(0)=\frakf_0>0$, corresponding to a different radius for the neck of the catenoid. But this can be reduced to the case considered here by a rescaling $\Xbar\mapsto\lambda\Xbar$.} $\frakf(0)=1$, $\frakf'(0)=0$. Consider the surface of revolution obtained by rotating the graph of
%%%%%%%
%%%%%%%
\begin{align}\label{eq:catpar}
\begin{split}
X^{n+1} \mapsto X' =(\frakf(X^{n+1}),0,\dots,0)
\end{split}
\end{align}
about the $X^{n+1}$-axis. It can be parameterized as
%%%%%%%
%%%%%%%
\begin{align}\label{eq:Fdef}
\begin{split}
F:I\times \bbS^{n-1}\to \bbR^{n+1},\qquad F(\rrho,\oomega)=(\frakf(\rrho)\Theta(\oomega),\rrho),
\end{split}
\end{align}
where $\Theta:\bbS^{n-1}\to \bbR^{n}$ is the standard embedding of the unit sphere. As a level set our surface is
%%%%%%%
%%%%%%%
\begin{align*}
\begin{split}
\{G(\Xbar):=|X'|^2-\frakf^2(X^{n+1})=0\},
\end{split}
\end{align*}
and the unit outward normal is
%%%%%%%
%%%%%%%
\begin{align}\label{eq:normal}
\begin{split}
\frakn= \frac{\nabla G}{|\nabla G|}= \frac{1}{(1+(\frakf'(\rrho))^2)^{\frac{1}{2}}}(\Theta(\oomega),-\frakf'(\rrho)).
\end{split}
\end{align}
Below we identify $\partial_\rrho$ and $\partial_{a}$, $a\in \{\oomega^1,\dots,\oomega^{n-1}\}$, with their images in $\bbR^{n+1}$ under $dF$, that is,
%%%%%%%
%%%%%%%
\begin{align*}
\begin{split}
\partial_\rrho = (\frakf'\Theta,1),\qquad \partial_{a}= (\frakf \partial_{a}\Theta,0).
\end{split}
\end{align*}
Differentiating \eqref{eq:normal} with respect to the ambient covariant derivative $\nabla$ we find that
%%%%%%%
%%%%%%%
\begin{align*}
\begin{split}
\nabla_{\partial_\rrho}\frakn = -\frac{\frakf''}{(1+(\frakf')^2)^{\frac{3}{2}}}\partial_\rrho,\qquad \nabla_{\partial_a}\frakn = \frac{1}{\frakf(1+(\frakf)^2)^{\frac{1}{2}}}\partial_a.
\end{split}
\end{align*}
From this we see that the second fundamental form $\secondff$ of the surface, as a matrix with components $\angles{\nabla_{\partial_j} \bfn}{\partial_i}$, $i,j\in\{\rrho,\oomega^1,\dots,\oomega^{n-1}\}$, is (here $\ringsg$ denotes the round metric on $\bbS^{n-1}$)
%%%%%%%
%%%%%%%
\begin{align*}
\begin{split}
\secondff=\pmat{\lambda_1(1+(\frakf')^2)&0\\0&\slambda \frakf^2\ringsg},\qquad \lambda_1:=-\frac{\frakf''}{(1+(\frakf')^2)^{\frac{3}{2}}}, \quad \slambda:=\frac{1}{\frakf(1+(\frakf')^2)^{\frac{1}{2}}}.
\end{split}
\end{align*}
It follows that the principal curvatures of the surface are $\{\lambda_1, \dots,\lambda_n\}$ where $\lambda_1$ is as above and $\lambda_j=\slambda$ for $j\geq 2$. The mean curvature of the embedding is then
%%%%%%%
%%%%%%%
\begin{align*}
\begin{split}
\bfH=\frac{1}{n}\Bigg(-\frac{\frakf''}{(1+(\frakf')^2)^{\frac{3}{2}}}+\frac{n-1}{\frakf(1+(\frakf')^2)^{\frac{1}{2}}}\Bigg).
\end{split}
\end{align*}
Therefore, for the surface to have zero mean curvature, $\frakf$ must satisfy the ODE
%%%%%%%%%%%
%%%%%%%%%%%
\begin{align}\label{eq:psiODE}
\frac{\frakf''}{(1+(\frakf')^2)^{\frac{3}{2}}}-\frac{n-1}{\frakf(1+(\frakf')^2)^{\frac{1}{2}}}=0,\qquad \frakf(0)=1,\quad \frakf'(0)=0.
\end{align}
%%%%%%%%%%%
%%%%%%%%%%%
\begin{definition}
The Riemannian Catenoid is the surface of revolution $\barcalC$ in $\bbR^{n+1}$ defined by~\eqref{eq:catpar}, where $\frakf$ satisfies \eqref{eq:psiODE}. The Lorentzian Catenoid is the surface $\calC:=\bbR\times \barcalC$ in the Minkowski space $\bbR^{1+(n+1)}$.
\end{definition}
%%%%%%%%%%%
%%%%%%%%%%%
There is a qualitative difference between the shape of the catenoid in dimension $n=2$ and dimensions $n\geq3$. Indeed, \eqref{eq:psiODE} implies the following ODE for $\rrho$:
%%%%%%%
%%%%%%%
\begin{align*}
\begin{split}
\frac{\ud^2 \rrho}{\ud \frakf^2}+\frac{n-1}{\frakf}\frac{\ud\rrho}{\ud \frakf}+\frac{n-1}{\frakf}\big(\frac{\ud \rrho}{\ud \frakf}\big)^3=0.
\end{split}
\end{align*}
From this, one can derive that $I=\bbR$ when $n=2$ and $I=(-S,S)$ when $n\geq 3$, where (see for instance \cite{Tam-Zhou} for more details on these calculations)
%%%%%%%
%%%%%%%
\begin{align}\label{eq:Sendpointdef1}
\begin{split}
S:=\int_{1}^{\infty}\frac{\ud \frakf}{\sqrt{\frakf^{2(n-1)}-1}}<\infty.
\end{split}
\end{align}
As we will see more explicitly below, one significance of this difference for the analysis is that in dimension $n=2$, the zero modes for the linearized operator, which correspond to the symmetries of the ambient space, are not eigenfunctions (that is, they do not belong to $L^2$), but rather resonances. In this work we will consider only  the high dimensional case $n\geq 5$, where the geometry of the catenoid approaches the flat geometry at a fast rate. To make these statements precise, we compute an expression for the induced metric on $\barcalC$. Using polar coordinates $X'=\Rbar \Omegabar$ for the first $n$ coordinates in the ambient $\bbR^{n+1}$ and denoting the $X^{n+1}$ coordinate by $\Zbar$, the ambient Euclidean metric becomes
%%%%%%%
%%%%%%%
\begin{align*}
\begin{split}
\ud \Zbar^2+ \ud \Rbar^2 + \Rbar^2\ud \Omegabar^2,
\end{split}
\end{align*}
where $\ud \Omegabar^2$ denotes the standard metric on $\bbS^{n-1}$. On $\barcalC\cap\{\Zbar>0\}$ we introduce radial coordinates $(\rbar,\thetabar)\in[1,\infty)\times \bbS^{n-1}$ by 
%%%%%%%
%%%%%%%
\begin{align}\label{eq:dZdr}
\begin{split}
(\rbar,\thetabar)\mapsto (\Zbar= \Zbar(\rbar), \Rbar = \rbar,\Omegabar = \Theta(\thetabar)),\qquad \frac{\ud \Zbar}{\ud \rbar} = \frac{1}{\sqrt{\rbar^{2(n-1)}-1}}.
\end{split}
\end{align} 
The induced Riemannian metric on $\barcalC$ in these coordinates becomes
%%%%%%%
%%%%%%%
\begin{align*}
\begin{split}
\Big(1+\frac{1}{\rbar^{2(n-1)}-1}\Big)\ud \rbar\otimes \ud \rbar+\rbar^2  \ringsg_{ab}\ud\thetabar^a\otimes\ud\thetabar^b.
\end{split}
\end{align*}
Instead of the geometric radial coordinate function $\rbar$, we use $\rho \in(-\infty,\infty)$ with
%%%%%%%
%%%%%%%
\begin{align*}
\begin{split}
\rbar(\rho)=\jap{\rho}:=\sqrt{1+\rho^2}.
\end{split}
\end{align*}
The coordinates $(\rho,\omega)\in \bbR\times \bbS^{n-1}$, where $\omega=\thetabar$, now describe all of $\barcalC$ (not just half) with $Z=Z(\rbar(\rho))$ if $\rho\geq0$ and $Z= - Z(\rbar(\rho))$ if $\rho<0$. The metric on $\barcalC$ in these coordinates becomes
%%%%%%%
%%%%%%%
\begin{align}\label{eq:RiemmCatmetric1}
\begin{split}
g_{\barcalC}:=\frac{\rho^2\jap{\rho}^{2(n-2)}}{\jap{\rho}^{2(n-1)}-1}\ud \rho\otimes\ud \rho+\jap{\rho}^2\ringsg_{ab} \ud \omega^a\otimes\ud\omega^b.
\end{split}
\end{align}
Using $t$ to denote the variable in $\bbR$ in $\calC=\bbR\times \barcalC$, the Lorentzian metric on $\calC$ becomes
%%%%%%%
%%%%%%%
\begin{align*}
\begin{split}
-\ud t\otimes\ud t+\frac{\rho^2\jap{\rho}^{2(n-2)}}{\jap{\rho}^{2(n-1)}-1}\ud \rho\otimes\ud \rho+\jap{\rho}^2 \ringsg_{ab} \ud \omega^a\otimes\ud\omega^b.
\end{split}
\end{align*}
From the second variation of the area functional one can see that the stability, or  linearized, operator for the catenoid as a minimal surface is  $\HHbar:=\Delta_\barcalC+ |\secondff|^2$ (respectively, $\Box_\calC+|\secondff|^2$ in the Lorentzian case), where $\Delta_\barcalC$ denotes the Laplacian on $\barcalC$ (respectively, $\Box_\calC=-\partial_t^2+\Delta_\barcalC$ denotes the d'Alembertian on $\calC$). See for instance \cite{Tam-Zhou,F-CS}. As mentioned earlier, it is shown in \cite{F-CS,Tam-Zhou} that $\HHbar$ admits a unique positive eigenvalue, indicating the instability of the catenoid as a minimal surface:
%%%%%%%
%%%%%%%
\begin{align*}
\begin{split}
\HHbar \varphibar_\mu=\mu^2\varphibar_\mu.
\end{split}
\end{align*}
Heuristically, this instability corresponds to the shrinking of the neck of the catenoid. See for instance~\cite{KL1,DKSW} for more discussion on this. On the other hand, since every translation of $\barcalC$ in the ambient $\bbR^{n+1}$ is another minimal surface (another catenoid), by differentiating in the translation parameter one obtains $n+1$ zero modes of $\HHbar$. Explicitly, in the $(\rho,\omega)$ coordinates above, these are given by
%%%%%%%
%%%%%%%
\begin{align*}
\begin{split}
\ebar_j=\frac{\Theta^j(\omega)}{\jap{\rho}^{n-1}},\quad 1\leq j\leq n,\qquad \ebar_{n+1}=\mathrm{sgn}(\rho)\frac{\sqrt{\jap{\rho}^{2(n-1)}-1}}{\jap{\rho}^{n-1}},
\end{split}
\end{align*}
corresponding to translations in the direction $\frac{\partial}{\partial X^j}$ respectively. In general, the zero mode $\ebar_{n+1}$, corresponding to translation in the direction of the axis of symmetry, does not belong to $L^2$ for any $n\geq 2$. For the other directions, $\ebar_j$ are in $L^2$ when $n\geq3$, while they logarithmically  fail to be in $L^2$ when $n=2$. In the Lorentzian case, the Lorentz boosts of the ambient $\bbR^{1+(n+1)}$ give the additional zero modes $t\ebar_j$ of $\Box_\calC+|\secondff|^2$, which, for $1\leq j\leq n$ and $n\geq 3$, are referred to as the \emph{generalized eigenfunctions} of the linear operator.\footnote{Ambient rotations about the axis of symmetry map $\barcalC$ to itself, so differentiation along the rotation parameter yields the trivial zero mode $\ebar=0$ for $\HHbar$. Similarly for translations along the time axis in the Lorentzian case. When $n\geq 3$, scaling changes the value of $S$ in \eqref{eq:Sendpointdef1} and differentiation in the scaling parameter yields a zero solution which is neither an eigenfunction nor a resonance.} See Section~\ref{subsec:eigenfunctions} for the discussion of eigenfunctions and generalized eigenfunctions in the context of the first order formulation.

We end this subsection by giving a more explicit description of the boosted and translated catenoid, which will be needed for the statement of our theorem. For any $\ell_0\in \bbR^n\backslash\{0\}$ let $P_{\ell_0}$ denote the orthogonal projection in the direction of $\ell_0$ ($P_{\ell_0} v=(|\ell_0|^{-2}\ell_0\cdot v)\ell_0$), and $P_{\ell_0}^\perp$ the orthogonal projection to the complement. Here and below, by a slight abuse of notation, we  view $\bbR^n$ as a subset of $\bbR^{n+1}$ using the embedding  $X'\mapsto (X',0)^\intercal$. The corresponding ambient Lorentz boost $\Lambda_{\ell_0}$, with $0<|\ell_0|<1$, is defined by
%%%%%%%
%%%%%%%
\begin{align}\label{eq:Lambdadef1}
\begin{split}
\Lambda_{\ell_0}=\pmat{\gamma&-\gamma\ell_0^\intercal\\-\gamma \ell_0&A_{\ell_0}},\qquad A_{\ell_0}= \gamma P_{\ell_0}+P_{\ell_0}^\perp,\quad \gamma=\frac{1}{\sqrt{1-|\ell_0|^2}},
\end{split}
\end{align}
and the inverses of $\Lambda_{\ell_0}$ and $A_{\ell_0}$ are
%%%%%%%
%%%%%%%
\begin{align*}
\begin{split}
\Lambda_{\ell_0}^{-1}=\Lambda_{-\ell_0},\qquad A_{\ell_0}^{-1}=\gamma^{-1}P_{\ell_0}+P_{\ell_0}^\perp.
\end{split}
\end{align*}
In what follows, Lorentz boosts are always with respect to a direction vector of length strictly less than one.
The Lorentzian catenoid boosted by $\ell_0\in \bbR^n$ and translated by $a_0\in \bbR^n$ is  the following HVMC submanifold of $\bbR^{1+(n+1)}$,
%%%%%%%
%%%%%%%
\begin{align*}
\begin{split}
\calC_{a_0,\ell_0}:=\{\Lambda_{-\ell_0} X~\vert~ X\in \calC\}+\pmat{0\\a_0}.
\end{split}
\end{align*}

\subsection{First Statement of the Main Result}\label{sec:introfirststatement}
We are now ready to give a first formulation of the main result of this paper. For a manifold $\calX$, we will use the notation $\calT_p\calX$ to denote the tangent space at $p\in \calX$. We also use the parameterization $F\colon I\times \bbS^{n-1}\to \bbR^{n+1}$ introduced in \eqref{eq:Fdef}. By a slight abuse of notation we often identify $\barcalC$ with $I\times \bbS^{n-1}$ and view functions on $I\times \bbS^{n-1}$ as functions on $\barcalC$.

\begin{theorem}\label{thm:main-0} Let $\Phi_0\colon I\times \bbS^{n-1} \to\{0\}\times\bbR^{n+1}$, $n\geq 5$, be an embedding and~$\Phi_1\colon I\times \bbS^{n-1}\to \bbR^{1+(n+1)}$ be a family of future directed timelike vectors such that $\Phi_0=F$ and $\Phi_1=(1,0)$ outside of a compact set. Suppose $\Phi_0$ and $\Phi_1$ belong to an appropriate codimension-1 subset in a suitable topology, and are sufficiently close to $F$ and $(1,0)$, respectively, in this topology. Then there is a unique complete timelike VMC hypersurface $\calM$ in $\bbR^{1+(n+1)}$ such that $\calM\cap \{X^0=0\}=\Phi_0(\barcalC)$ and $\calT_{\Phi_0(p)}\calM$ is spanned by $\ud_p \Phi_0(\calT_p\barcalC)$ and $\Phi_1(p)$ for any $p\in I \times \bbS^{n-1}$. Moreover, there exist $a_0,\ell_0\in \bbR^n$ such that the ambient Euclidean distance between $\calM\cap\{X^0=t\}$ and $\calC_{a_0,\ell_0}\cap \{X^0=t\}$ tends to zero as $t\to\infty$.
\end{theorem}
%%%%%%%%%%%%
%%%%%%%%%%%%
A more precise version of the theorem will be stated as Theorem~\ref{thm:main} below. We pause to make a few comments.
\begin{itemize}[leftmargin=*]
  \item[{\bf{1.}}] In view of the growing mode of the stability operator  (see Section~\ref{subsec:Riemcat}), the codimension-1 restriction on the data in Theorem~\ref{thm:main} is optimal. See for instance \cite{KL1}. However, we do not pursue the question of uniqueness or regularity of the codimension-1 set in the initial data topology. See Item (3) in Section~\ref{subsubsec:introoverallscheme} for more on this point.
  \item [{\bf{2.}}] The fact that the solution approaches a boosted and translated catenoid is related to the presence of a non-trivial kernel and generalized kernel for the linearized operator. As we saw in Section~\ref{subsec:Riemcat}, the kernel and generalized kernel are generated by  the translation and boost symmetries. Therefore, to obtain decay for the perturbative part of the solution we need to choose the translation and boost parameters dynamically (modulation) to stay away from the kernel and generalized kernel. To give a more precise description of how we track these parameters, we need to decompose the solution into a \emph{profile} and a perturbation, and set up a first order formulation of the problem. These aspects are summarized in Sections~\ref{sec:profileintro} and~\ref{subsec:introfirstorderandcodim}. In Remark~\ref{rem:parametersintro1} in Section~\ref{subsec:introsecondthm} we give more precise references for how the parameters are tracked. Translation and boost symmetries are common features of quasilinear soliton stability problems that arise from Lorentz invariant theories. Developing a robust modulation approach for translation invariant quasilinear wave equations is one of the main achievements of this work. In this direction,  our novel profile construction plays a central role. We hope that our methods will find applications in other quasilinear soliton stability problems.
  \item [{\bf{3.}}] The assumption $\Phi_0=F$ and $\Phi_1=(1,0)$ outside a compact set can be replaced by sufficient decay at spatial infinity. Indeed, outside an ambient cone $\scC$ with vertex at $(-R,0)$, with $R$ sufficiently large, the problem reduces to a quasilinear wave equation on Minkowski space. By finite speed of propagation, this problem can be analyzed separately in this region, for instance using the vectorfield $(t-r)^{p} \partial_{t}$. This will lead to suitably decaying and small data on the cone $\scC$ which can be taken as the starting point of the analysis in this paper. Note that in this region the distance between $\calC_{a_j,\ell_j}\cap\{X^0=t\}$ for any $a_j,\ell_j$, $j=1,2$, with $|\ell_j|<1$ decays to zero as $t\to\infty$ in view of the strong asymptotic decay of the catenoid metric to the flat metric.
\item [{\bf{4.}}] We consider dimensions $n\geq5$ in this work, because this range is a more accessible analytic setting to approach some of the structural challenges in quasilinear soliton stability problems. Specifically, this restriction has the following two advantages: (i) The faster spatial decay rate  of the difference between the catenoid and flat metrics (faster decay of the tail of the soliton) amounts to weaker interactions between the profile and the radiation. (ii) The faster time decay of waves in higher dimensions allows us to directly obtain twice integrability of the time derivatives of the boost and translation parameters. This strong decay enters in proving integrated local energy decay for the perturbative part of the solution. See Section~\ref{subsec:ideas-iled}. The case $n=2$ (outside of radial symmetry) poses the additional challenge that the zero modes corresponding to translations and boosts are no longer eigenfunctions, but rather resonances (see Section~\ref{subsec:Riemcat}).  We expect that the general scheme in this paper is applicable to dimensions $n=3,4$ and hope to address these cases in future work.

\item[{\bf{5.}}] The minimal surface equation is widely studied in Riemannian geometry and calculus of variations. In particular, the spectral properties of the stability operator for the catenoid are well-understood. See for instance \cite{F-CS,Tam-Zhou}. This makes our problem a natural starting point for the study of asymptotic stability of solitons in quasilinear wave equations. 
\end{itemize}
%%%%%%%%%%%%
%%%%%%%%%%%%

\subsection{Overall Scheme and Main Difficulties} \label{subsec:ideas-outline}
%%%%%%%%%%%%%%%%%%%%%%
%%%%%%%%%%%%%%%%%%%%%%
In Sections~\ref{sec:profileintro}--\ref{subsec:introoutlinerp} below, we will describe the main ideas for the proof of Theorem~\ref{thm:main-0}. Before we do so, let us begin with an executive summary of the overall scheme, as well as a discussion of the main difficulties in the proof.

\subsubsection{Overall scheme}\label{subsubsec:introoverallscheme} The overall scheme of our proof of Theorem~\ref{thm:main-0} is as follows:
\begin{enumerate}
\item {\it Decomposition of solution.} The basic idea is to make the (formal) decomposition
\begin{equation} \label{eq:basic-decomp}
	\hbox{(Solution)} = \underbrace{\calQ}_{\hbox{profile}} \hbox{``}+\hbox{''} \underbrace{\psi}_{\hbox{perturbation}}
\end{equation}
This decomposition will be made precise in Section~\ref{sec:profileintro}.

A key goal is to show that the perturbation $\psi$ decays to zero as $t \to \infty$ in a suitable sense (see Item~(4) and Section~\ref{subsec:introoutlinerp}). In the absence of any obstructions, the profile $\calQ$ would be the object that we wish to prove the asymptotic stability of -- the catenoid in our case. However, as discussed earlier, the linearized HVMC equation around the catenoid, $(-\partial_{t}^{2} + \underline{H}) \psi = 0$, admits non-decaying solutions, namely (i)~a $1$-dimensional family of exponentially growing solutions, which arises from the simple positive eigenvalue of $\underline{H}$, and (ii)~a $2n$-dimensional family of solutions growing at most linearly in $t$, which arises from the $n$-dimensional kernel of $\underline{H}$ generated by translational symmetries. To avoid these obstructions, we employ the ideas of \emph{modulation} and \emph{shooting}, which we turn to now.

\smallskip

\item {\it Modulation.} To ensure transversality to the $2n$-dimensional family of non-decaying solutions in (ii), we impose $2n$ orthogonality conditions on $\psi$ at each time. To compensate for such a restriction, we allow the profile $\calQ$ to depend on $2n$ time-dependent parameters. Since the family in (ii) arises from translational symmetries, it is natural to introduce an $n$-dimensional \emph{position} vector $\xi(\sigma) = (\xi^{1}, \ldots, \xi^{n})(\sigma)$, an $n$-dimensional \emph{velocity} (or \emph{boost}) vector $\ell(\sigma) = (\ell^{1}, \ldots, \ell^{n})(\sigma)$, a foliation $\sigma$ (whose leaves will represent an appropriate notion of time for our problem) and an \emph{approximate solution} (or \emph{profile}) $\calQ = \calQ(\xi(\cdot), \ell(\cdot))$ to HVMC that represents ``a moving catenoid at position $\xi(\sigma)$ with velocity $\ell(\sigma)$ at each time $\sigma$''. Appropriate choices of the profile $\calQ$, the foliation $\sigma$ and the $2n$ orthogonality conditions would lead, upon combination with the HVMC equation, to $2n$ equations that dictate the evolution of $(\xi(\sigma), \ell(\sigma))$ in terms of $\psi$.

\smallskip

\item {\it Shooting argument.} To avoid the exponential growth stemming from obstruction (i), we further decompose the perturbation $\psi$ as follows:
\begin{equation*}
	\psi = a_{+}(\sigma) Z_{+} + a_{-}(\sigma) Z_{-} + \phi,
\end{equation*}
where $Z_{+}$ and $Z_{-}$ are uniformly bounded functions, $a_{+}(\sigma)$ and $a_{-}(\sigma)$ obey ODEs in $\sigma$ with growing and damping linear parts, respectively, and $\phi$ obeys $2n+2$ orthogonality conditions so as to be transversal (to a sufficient extent) to all possible linear obstructions to decay at each time\footnote{Strictly speaking, the term $a_{-}(\sigma) Z_{-}$ decays forward-in-time, so the reader may wonder why we also took it out of $\phi$ by imposing $2n+2$ orthogonality conditions (as opposed to $2n+1$, the dimension of the space of forward-in-time non-decaying solutions to $(-\partial_{t}^{2} + \underline{H}) \psi = 0$). The reason is that we want $\phi$ to exhibit only dispersive behaviors like $\bbP_{c} \psi$ in the simple example $(-\partial_{t}^{2} + \underline{H}) \psi = 0$.}. If $\calQ$ were simply the Lorentzian catenoid, then we may choose $a_{+} Z_{+} + a_{-} Z_{-}$ and $\phi$ to be the $L^{2}$-orthogonal projections of $\psi$ to the negative eigenvalue and the absolutely continuous spectrum of $\underline{H}$, respectively. 

By analyzing the ODE for $a_{-}$, the modulation equations for $(\xi, \ell)$, and the wave equation for $\phi$ (see (4) below), we will show that $\dot{\xi} - \ell$, $\dot{\ell}$ and $\psi$ decay as long as the unstable mode $a_{+}(\sigma)$ satisfies the so-called \emph{trapping assumption}, which roughly says that $a_{+}(\sigma)$ decays in time. We then employ a topological \emph{shooting argument} to select a family of initial data -- which is codimension $1$ in the sense described below in Theorem~\ref{thm:main} and \eqref{eq:shootingbdata1}, \eqref{eq:codim1} -- such that $a_{+}$ indeed continues to satisfy the trapping assumption for all times $\sigma \geq 0$.

\smallskip

\item {\it Integrated local energy decay and vectorfield method.} Finally, we study the quasilinear wave equation satisfied by $\phi$, which also satisfies $2n+2$ orthogonality conditions. Under suitable bootstrap assumptions (to handle nonlinear terms) and the trapping assumption for $a_{+}$  (see (3)), we prove the pointwise decay of $\phi$ via the following steps:
\begin{align*}
&\hbox{(transversality to linear obstructions)}  \\
&\Rightarrow \hbox{(uniform boundedness of energy and integrated local energy decay (ILED))} \\
& \Rightarrow 
\hbox{(pointwise decay)}
\end{align*}
Here, integrated local energy decay estimates (ILED; also known as Morawetz estimates) refer to, roughly speaking, bounds on integrals of the energy density on spacetime cylinders for finite energy solutions. They are a weak form of dispersive decay. These have the advantage of being $L^{2}$-based and hence being amenable to a wide range of techniques, such as the vectorfield method, Fourier transform, spectral theory etc. A powerful philosophy, that has recently arisen in works \cite{DR1, Tat2, MTT, OlSt} related to the problem of black hole stability, is to view integrated local energy decay as a key intermediate step for obtaining stronger pointwise decay (see also \cite{Tat1, MeTa} for papers in the related context of global Strichartz estimates). Specifically, in our proof we adapt the $r^{p}$-method of Dafermos--Rodnianski \cite{DR1}, extended by Schlue \cite{Schlue1} and Moschidis \cite{Moschidis1}. See Section~\ref{subsec:introoutlinerp} for further discussions.

\end{enumerate}

\subsubsection{Main difficulties}\label{subsubsec:intromaindiff} We face several significant challenges in implementing the above scheme for our problem. A summary of the main difficulties is as follows:
\begin{itemize}
\item {\it (Quasilinearity)} First and foremost, the hyperbolic vanishing mean curvature equation is \emph{quasilinear}. While the stability property of the linearized equation $(-\partial_{t}^{2} + \underline{H}) \phi = f$ around the Lorentzian catenoid $\calC$ is well-understood, upgrading it to the nonlinear asymptotic stability result Theorem~\ref{thm:main-0} involves a number of serious difficulties; specifically, see {\it (Proof of integrated local energy decay)} below. Furthermore, since the highest order term is nonlinear, at various places in the proof we need to be careful to avoid any derivative losses.

\item {\it (Gauge choice)} Another basic point about HVMC is that it is an equation for a geometric object, namely a hypersurface in $\bbR^{1+(n+1)}$. Hence, we need to fix a way of describing the hypersurface by a function to perform any analysis -- this is the well-known problem of \emph{gauge choice}. 

\item {\it (Profile and foliation construction)} In order for the above scheme to work, it is crucial for the profile $\calQ$, representing a moving catenoid at position $\xi(\sigma)$ with velocity $\ell(\sigma)$ at each time $\sigma$, to solve HVMC up to an adequately small error. Unfortunately, the obvious construction based on the standard $t = X^{0}$-foliation would lead to an inappropriately large error. The key issue is the inaccuracy of the construction in the far-away region (i.e., as $\rho \to \pm \infty$), which is fatal due to the slow spatial decay of the catenoid (i.e., mere polynomial decay towards the flat hyperplane as $\rho \to \pm \infty$). As we will see, we are led to consider a different foliation $\sigma$ consisting of \emph{moving asymptotically null leaves}; see Section~\ref{sec:profileintro}.

\item {\it (Proof of integrated local energy decay)} Existing methods \cite{MMT1, MST}, combined with the detailed knowledge of the spectral properties of the stability operator $\underline{H}$ for $\underline{\calC}$ \cite{F-CS, Tam-Zhou}, establish integrated local energy decay for the linearized equation $(-\partial_{t}^{2} + \underline{H}) \phi = f$ around the Lorentzian catenoid $\calC$ when $\phi = \bbP_{c} \phi$ ($L^{2}$-projection to the absolutely continuous spectrum).  See Section~\ref{subsec:iled-catenoid}. Transferring this estimate to the solution $\phi$ to the quasilinear wave equation satisfying our orthogonality conditions, however, is met with several difficulties, such as (i) quasilinearity, (ii) existence of a trapped null geodesic (traveling around the collar $\set{\rho = 0}$ in case of $\calC$), (iii) existence of zero and negative eigenvalues of $\underline{H}$ (what we referred to as linear obstructions to decay) and (iv) nonstationarity of the profile $\calQ$.

\item {\it (Modulation theory and vectorfield method)} Standard modulation theory \cite{Stuart1, Weinstein1} is based on the standard $t = X^{0}$-foliation, whose leaves are flat spacelike hypersurfaces; the method needs to be adapted to the foliation $\sigma$ used in our profile construction. Similarly, the standard $r^{p}$-method utilizes a foliation consisting of \emph{non}-moving asymptotically null leaves \cite{DR1, Schlue1, Moschidis1}, which needs to be adapted to our foliation $\sigma$ of moving asymptotically null leaves. The presence of linear obstructions to decay (i.e., zero and negative eigenvalues for $\underline{H}$) also needs to be incorporated.
\end{itemize}

In Sections~\ref{sec:profileintro}--\ref{subsec:introoutlinerp}, we describe the main ideas in this paper for resolving the above difficulties.

%%%%%%%%%%%%%%%%%%%%%%
%%%%%%%%%%%%%%%%%%%%%%
\subsection{Profile, Foliation and Gauge Construction}\label{sec:profileintro}
%%%%%%%%%%%%%%%%%%%%%%
%%%%%%%%%%%%%%%%%%%%%%
Here we describe our profile $\calQ$ and foliation $\tau$, as well as the gauge we use to express our solution as a scalar function $\psi$ on $\calQ$; these constructions make the basic decomposition \eqref{eq:basic-decomp} precise. Since this decomposition is needed for the discussion of other parts of our proof, we shall give the full construction here.

Let $\xi = \xi(\sigma)$ and $\ell = \ell(\sigma)$ be two functions on an interval  in $\bbR$ with values in $\bbR^n$ and $|\ell|,|\dotxi|<1$. 
%%%%%%%%%%%%%
%%%%%%%%%%%%%
By a slight abuse of notation, we will always view $\xi$ and $\ell$ as vectors both in $\bbR^n$ and in $\bbR^{n+1}$, using the embedding of $\bbR^{n}$ in $\bbR^{n+1}$ given by $X'\mapsto (X',0)^\intercal$. 
Recall that $\barcalC$ denotes the Riemannian catenoid in $\bbR^{n+1}$, and  $\calC=\bbR\times \barcalC$ the product Lorentzian catenoid in $\bbR^{1+(n+1)}$.
%%%%%%%%
%%%%%%%%
Given two functions $\xi(\cdot)$ and $\ell(\cdot)$ as above, let
%%%%%%%
%%%%%%%
\begin{align}\label{eq:calCsigmadefintro1}
\begin{split}
\calC_{\sigma}:=\{\Lambda_{-\ell(\sigma)} X~\vert~ X\in \calC\}+\pmat{0\\\xi(\sigma)-\sigma\ell(\sigma)}.
\end{split}
\end{align}
Note that if $\dotell\equiv0$ and $\xi(\sigma)\equiv \sigma\ell+a_0$ for a constant vector $a_0\in \bbR^n\subseteq\bbR^{n+1}$, then $\calC_\sigma$ is the Lorentzian catenoid obtained by boosting $\calC$ by $\Lambda_{-\ell}$ and then translating the result by $a_0$.

We will assume that $|\ell|,|\dotxi|<1$ and that $|\ell(0)|,|\xi(0)|$, and $|\dotell|$ are sufficiently small\footnote{The smallness assumptions are not essential and are a consequence of how we have set things up. The smallness requirement on $|\dotell|$ is to guarantee that the curve $(\sigma,\xi-\gamma R\ell)$ is timelike. The smallness conditions on $|\ell(0)|$ and $|\xi(0)|$ are so that $\tilbsUpsigma_0$ contains $\scC_{-R}$ in the interior of its future. If we remove these assumptions we simply need to take $R$ larger and replace $R-2$ by $R-C$ for a larger constant $C$ in the definition of $\scC_{-R}$. In our applications the smallness of the initial data and the bootstrap assumptions imply all the smallness conditions required here.}. We will first define a foliation of the interior of the ambient cone (here and below we use the notation $\Xbar=(X^1,\dots,X^n)$)
%%%%%%%
%%%%%%%
\begin{align*}
\begin{split}
\scC_{-R} =\{X\in\bbR^{1+(n+1)}~\vert~  X^0+R-2=|\Xbar|\}
\end{split}
\end{align*}
as $\cup_\sigma \bsUpsigma_\sigma$, and then define the profile of our solution on the leaf $\bsUpsigma_\sigma$ to be $\calC_\sigma\cap\bsUpsigma_\sigma$.  Note that we will restrict attention to compactly supported perturbations\footnote{This simplifying assumption is not essential for the proof. See the third comment following the statement of Theorem~\ref{thm:main-0}.} of $\calC$ (in a suitable sense to be described below), so by finite speed of propagation we already know the form of our solution in the exterior of $\scC_{-R}$. The leaves $\bsUpsigma_\sigma$ will be chosen to be asymptotically null, more precisely hyperboloidal, away from the moving center $\xi(\sigma)$. To define this foliation, we first fix reference hyperboloids defined as (here $\gamma$ is evaluated at~$\sigma$)
%%%%%%%
%%%%%%%
\begin{align*}
\begin{split}
\calH_\sigma = \{y=(y^0,y',y^{n+1})~\vert~y^0-\gamma^{-1}\sigma+R=\sqrt{|y'|^2+1}\}.
\end{split}
\end{align*}
In general, we will denote the restriction to $\{X^{n+1}=S\}$ by an underline, so (a similar construction can be carried out with respect $\{X^{n+1}=-S\}$)
%%%%%%%
%%%%%%%
\begin{align*}
\begin{split}
\barcalH_\sigma = \calH_\sigma\cap\{y^{n+1}=S\}.
\end{split}
\end{align*}
The boosted and translated hyperboloids are denoted by $\tilbsUpsigma_\sigma$, that is (here $\ell$ and $\xi$ are evaluated at~$\sigma$),
%%%%%%%
%%%%%%%
\begin{align}\label{eq:tilSigmadefintro1}
\begin{split}
\tilbsUpsigma_\sigma=\big(\Lambda_{-\ell}\calH_\sigma+(0,\xi-\sigma\ell)^\intercal\big)=\{X~\vert~X^0-\sigma+\gamma R=\sqrt{|X'-\xi+\gamma  R\ell|^2+1}\}.
\end{split}
\end{align}
The restriction of $\tilbsUpsigma_\sigma$ to $\{X^{n+1}=S\}$ is denoted by (here $x=(x^0,x')\in \bbR\times \bbR^{n}$ denote the rectangular coordinates on $\{X^{n+1}=S\}$)
%%%%%%%
%%%%%%%
\begin{align*}
\begin{split}
\bartilbsUpsigma_\sigma:=\tilbsUpsigma_\sigma \cap \{X^{n+1}=S\}=\{x~\vert~x^0-\sigma+\gamma R=\sqrt{|x'-\xi+\gamma  R\ell|^2+1}\}.
\end{split}
\end{align*}
%%%%%%%%%%%%
%%%%%%%%%%%%
\begin{remark}
Let ${\underline{\scC}}_{-R}=\{(x^0,x',S)~\vert~x^0+R-2=|x'|\}$. The fact that $(\sigma,\eta(\sigma))$, with $\eta=\xi-\gamma R\ell$, is a timelike curve (that is, $|\doteta|<1$) implies that $\cup_{\sigma\geq0}\bartilbsUpsigma_\sigma$ gives a foliation of the region $\scR:=\{(x^0,x')\in\{X^{n+1}=S\}~\vert~x^0+\gamma(0)R\geq \sqrt{|x'-\eta(0)|^2+1}\}$, which contains ${\underline{\scC}}_{-R}$ because we have assumed that  $|\xi(0)|$ and $|\ell(0)|$ are small ($\scR$ is also contained in a slightly larger cone ${\underline{\scC}}_{-(R+3)}$). Indeed, if $(x^0,x')\in\scR$ belongs to $\bartilbsUpsigma_{\sigma_2}\cap\bartilbsUpsigma_{\sigma_1}$, $\sigma_2>\sigma_1$, then 
%%%%%%%
%%%%%%%
\begin{align*}
\begin{split}
\sigma_2-\sigma_1&=\sqrt{|x'-\eta(\sigma_1)|^2+1}-\sqrt{|x'-\eta(\sigma_2)|^2+1}\leq \big||x'-\eta(\sigma_1)|-|x'-\eta(\sigma_2)|\big|\leq |\eta(\sigma_2)-\eta(\sigma_1)|\\
&<\sigma_2-\sigma_1,
\end{split}
\end{align*}
which is impossible. Here to pass to the last inequality we have used $|\doteta|<1$. It follows that the map $U:(\sigma,y)=(\sigma,y^0,y')\mapsto \Lambda_{-\ell}y+(0,\xi-\sigma\ell)$ from $\cup_{\sigma\geq0}\barcalH_\sigma$ is a diffeomorphism to its image. To see that it covers all of $\scR$, given $x=(x^0,x')\in\scR$ choose $\sigma_0>x^0+R+5$ and note that $x$ lies between $\bartilbsUpsigma_0=U(\barcalH_0)$ and $\bartilbsUpsigma_{\sigma_0}=U(\barcalH_\sigma)$, and since $U$ is a diffeomorphism onto its image, $x$ must lie on $\bartilbsUpsigma_\sigma$ for some $\sigma\in[0,\sigma_0)$. It follows from this that $\tilbsUpsigma_\sigma$ foliate a region containing $\scC_{-R}$.
\end{remark}
%%%%%%%%%%%%
%%%%%%%%%%%%
We introduce a smoothed out version of the minimum function $\frakm:\bbR\times\bbR\to \bbR$ in the following way. We start by fixing a small number $\delta_{1} > 0$ and a smooth function $\tilde{\frakm} : \bbR \to \bbR$, which equals $0$ for $x < - \delta_{1}$, equals $x$ for $x > \delta_{1}$, and satisfies $\tilde{\frakm}' \in [0, 1]$. We then set $\frakm(x, y) := x + \tilde{\frakm}(y-x)$. By construction, all derivatives of $\frakm$ are bounded and
%%%%%%%
%%%%%%%
\begin{align*}
\begin{split}
\frakm(x,y)= \min(x,y)\qquad \mathrm{when~}|x-y|>\delta_1.
\end{split}
\end{align*}
Define $\sigma,\sigma_\temp: \{-S\leq X^{n+1}\leq S\} \to \bbR$ by 
%%%%%%%
%%%%%%%
\begin{align}\label{eq:sigmadefintro1}
\begin{split}
&\sigma_\temp(X)=\sigma'\qquad \mathrm{if~}X\in \tilbsUpsigma_{\sigma'},\\
&\sigma(X)= \frakm(X^0,\sigma_\temp(X)).
\end{split}
\end{align}
Finally let 
%%%%%%%
%%%%%%%
\begin{align}\label{eq:Sigmadefintro1}
\begin{split}
\bsUpsigma_{\sigma'}=\{X~\vert~\sigma(X)=\sigma'\},\qquad \barbsUpsigma_{\sigma'} = \bsUpsigma_{\sigma'}\cap \{X^{n+1}=S\}.
\end{split}
\end{align}
The transition from $\tilbsUpsigma_\sigma$ to $\bsUpsigma_\sigma$ corresponds to considering an asymptotically null (more precisely, hyperboloidal) foliation only starting at a large radius from the center $\xi(\sigma)$. Note that $\frakm$ has been chosen so that $\cup_\sigma\bsUpsigma_\sigma$ gives a smooth foliation of a region containing $\scC_{-R}$. Indeed, as we have seen, the level sets of $\sigma_\temp$ and $X^0$ provide such foliations, and $d\sigma_\temp$ and $dX^0$ are full rank. Since
%%%%%%%
%%%%%%%
\begin{align*}
\begin{split}
d\sigma = (\partial_{x}\frakm) dX^0+ (\partial_y\frakm) d\sigma_\temp,
\end{split}
\end{align*}
and we have $\partial_{x} \frakm, \partial_{y} \frakm \geq 0$ (since $\tilde{\frakm}' \in [0, 1]$) and $\partial_{x} \frakm + \partial_{y} \frakm = 1$ by construction, we see that $d \sigma$ is also full rank. The hyperboloidal (where $X^0\geq \sigma_\temp+\delta_1$) and flat (where $\sigma_\temp\geq X^0+\delta_1$) parts of these surfaces are denoted by $\bsUpsigma_\sigma^\hyp, \barbsUpsigma_\sigma^\hyp$ and $\bsUpsigma_\sigma^\flatt, \barbsUpsigma_\sigma^\flatt$ respectively. We have thus constructed the \emph{foliation} $\sigma$ adapted to $\xi, \ell$. We define our \emph{profile} as the hypersurface
\begin{align*}
\begin{split}
\calQ := \cup_{\sigma} \Sigma_{\sigma}, \quad \hbox{ where } \Sigma_\sigma:=\calC_\sigma\cap \bsUpsigma_\sigma.
\end{split}
\end{align*}

Next, we \emph{fix a gauge}, that is, describe a parameterization (or embedding) of the VMC hypersurface $\calM$ and a way to measure its deviation from the the profile $\calQ = \cup_\sigma \Sigma_\sigma$.
We will do this by fixing a vector $$N:\cup_{\sigma}\Sigma_\sigma\to\bbR^{1+(n+1)},$$ and defining \emph{the perturbation} $\psi:\cup_{\sigma}\Sigma_\sigma\to \bbR$
by the requirement that (later we will further decompose $\psi$ as in Item (3) of Section~\ref{subsubsec:introoverallscheme}; see also Section~\ref{subsec:introfirstorderandcodim})
%%%%%%%
%%%%%%%
\begin{align}\label{eq:psidefintro1}
\begin{split}
p+\psi(p)N(p)\in \calM,\qquad \forall~ p\in \cup_{\sigma}\Sigma_\sigma.
\end{split}
\end{align}
Under suitable smallness assumptions on the perturbation, this condition determines $\psi$ uniquely (see Lemma~\ref{rem:normalneighborhood}).
If $\dotell\equiv0$ and $\dotxi\equiv \ell$, from the second variation of the area we would expect  that the relevant perturbations are those which are in the direction  of the (spacetime) normal to~$\Sigma_\sigma$. In general, since $\xi$ and $\ell$ do not necessarily obey $\dotell=0$ and $\dotxi=\ell$, there will be additional errors. Nevertheless, the most natural geometric choice for $N$ still seems to be the normal to $\cup_{\sigma}\Sigma_\sigma$, or perhaps $\Lambda_{-\ell}n$, where $n$ is the normal to the straight Lorentzian catenoid~$\calC$. However, for reasons that will be discussed below we choose to work with a less geometric $N$ which is defined as follows. First, if $\Sigma_\sigma\ni p=\Lambda_{-\ell}q+(0,\xi(\sigma)-\sigma\ell(\sigma))^\intercal$ for some $q$ in $\calC$, let $n_\wp(p)=\Lambda_{-\ell(\sigma)}n(q)$, where $n(q)$ is the normal to $\calC$ at $q$. 
Let 
%%%%%%%
%%%%%%%
\begin{align}\label{eq:tilNdefintro1}
\begin{split}
\tilN=\chi \tilN_{\mathrm{int}}+(1-\chi)\frac{\partial}{\partial X^{n+1}}
\end{split}
\end{align}
where $\tilN_{\mathrm{int}}$ is the normal to $\Sigma_{\sigma}$ viewed as a subspace of $\bsUpsigma_{\sigma}$, and where $\chi$ is some cutoff function which is equal to one in $\calC_\sigma\cap\bsUpsigma_\sigma^\flatt$ and equal to zero in $\calC_\sigma\cap\bsUpsigma_\sigma^\hyp$. We then define $N$ to be parallel to $\tilN$ and such that $\bfeta(n_\wp,N)=1$:
%%%%%%%
%%%%%%%
\begin{align}\label{eq:Ndefintro1}
\begin{split}
N\parallel \tilN,\qquad \bfeta(n_\wp,N)=1.
\end{split}
\end{align}

A few remarks are in order about this gauge choice. 
\begin{itemize}[leftmargin=*]
  \item[{\bf{1.}}] In the exterior hyperboloidal region, $N$ is parallel to $\frac{\partial}{\partial X^{n+1}}$. This choice is motivated by the fact that in this region the catenoid looks almost like a hyperplane, so we are in fact parameterizing the VMC hypersurface $\calM$ as a graph over a hyperplane. The advantage is that this simplifies the derivation of the equations and the form of the nonlinear terms. As will become clear in the course of the proof of our main theorem, the precise structure of the nonlinearity is important only in this exterior region where we will be able to treat the difference between a hyperplane and a catenoid perturbatively. We will come back to the normalization of the length of $N$. 
  
  \item[{\bf{2.}}] In the interior our choice of $N$ is crude, but since $\ell$ and $\xi-\sigma\ell$ are small, it is still close to the geometric normal $n_{\wp}$. Our choice in this region is consistent with our general philosophy that besides some spectral information on the linearized operator and appropriate modulation equations for the parameters (which will be a consequence of our first order formulation and orthogonality conditions), precise structures are not so important in the interior region. 
  
  \item[{\bf{3.}}] Finally, the reason for the length normalization of $N$ is that we want the linear part of the equation satisfied by $\psi$ to be (except for errors coming from $\dotell$ and $\dotxi-\ell$ not vanishing) 
%%%%%%%
%%%%%%%
\begin{align*}
\begin{split}
\Box \psi +|\secondff|^2\psi,
\end{split}
\end{align*}
where $\Box$ and $\secondff$ denote, respectively, the wave operator and second fundamental form of $\calC_\sigma$, in the case where $\dotell\equiv0$ and $\dotxi\equiv\ell$. This is important because $\Box+|\secondff|^2$ is precisely the operator $-\partial_{t}^{2} + \underline{H}$ conjugated by the Lorentz transform with parameter $\ell$ (after a suitable translation).
\end{itemize}

To summarize, our profile is defined as $\calQ = \cup_\sigma \Sigma_\sigma$ with $\Sigma_\sigma=\calC_\sigma\cap\bsUpsigma_\sigma$, and $\calC_\sigma$ and $\bsUpsigma_\sigma$ as defined in \eqref{eq:calCsigmadefintro1}, \eqref{eq:tilSigmadefintro1}, \eqref{eq:sigmadefintro1}, \eqref{eq:Sigmadefintro1}, and the perturbation is described by a scalar function $\psi:\cup_\sigma \Sigma_\sigma\to \bbR$ defined by \eqref{eq:psidefintro1}, \eqref{eq:tilNdefintro1}, \eqref{eq:Ndefintro1}. We will denote the hyperboloidal and flat parts of the profile by $\calC_{\hyp}:=\{X^0\geq \sigma_\temp(X)+\delta_1\}$ and $\calC_{\flatt}:=\{X^0\leq \sigma_\temp(X)-\delta_1\}$ respectively. 
%%%%%%%%%%%%
%%%%%%%%%%%%
\begin{remark}
The following simplified picture is helpful when thinking about the foliation and the definition of the profile. Imagine the scenario where we want to decompose a solution $u$ of a semilinear equation $\Box u = F(u)$ in terms of a soliton $Q$ and a remainder $\psi$. Suppose the equation is translation and Lorentz invariant, and let $Q_{\xi,\ell}$ denote the translated and boosted soliton. We foliate the domain by leaves which are flat up to a radius of order $R$ about $\xi(\tau)$, and then become asymptotically null and approach the cone through $(-R,0)^\intercal$ translated by $\xi(\tau)$ and boosted by $\ell(\tau)$, as in the following figure
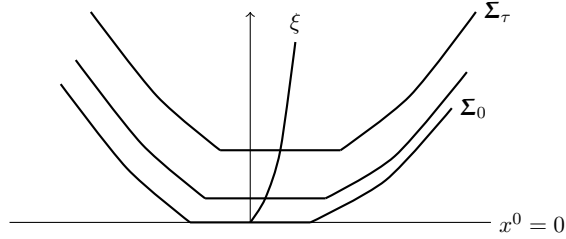
\begin{center}
\begin{tikzpicture}[scale=0.8,transform shape]
  \draw (-4,0) -- (4,0);
  \draw[->] (0,0) -- (0,3.5) ;
   \coordinate  (A) at (0,0);
\coordinate  (B) at (0.25,0.4);
\coordinate  (C) at (0.5,1.2);
\coordinate  (D) at (0.75,3);
\draw[thick] plot [smooth] coordinates { (A) (B) (C) (D)};
\draw[thick] plot [smooth] coordinates { (-1,0) (A) (1,0)};
\draw[thick] plot [smooth] coordinates { (1,0) (2.25,0.7) (3.35,1.9)};
\draw[thick] plot [smooth] coordinates { (-1,0) (-2.05,0.9) (-3.15,2.3)};
\draw[thick] plot [smooth] coordinates { (-0.75,0.4) (B) (1.25,0.4)};
\draw[thick] plot [smooth] coordinates { (1.25,0.4) (2.4,1.1) (3.6,2.5)};
\draw[thick] plot [smooth] coordinates { (-0.75,0.4) (-1.8,1.3) (-2.9,2.7)};
\draw[thick] plot [smooth] coordinates { (-0.5,1.2) (C) (1.5,1.2)};
\draw[thick] plot [smooth] coordinates { (1.5,1.2) (2.65,2.1) (3.75,3.5)};
\draw[thick] plot [smooth] coordinates { (-0.5,1.2) (-1.55,2.1) (-2.65,3.5)};
\node[above] at (D) {$\xi$};
\node[right] at (3.75,3.5) {$\bsUpsigma_\tau$};
\node[right] at (3.35,1.9) {$\bsUpsigma_0$};
\node[right] at (4,0) {$x^0=0$};
\end{tikzpicture}
\end{center}
Our profile construction corresponds to decomposing the solution as $u=Q_{\xi(\tau),\ell(\tau)}+\psi$ on the leaf~$\bsUpsigma_\tau$.
\end{remark}
%%%%%%%%%%%%
%%%%%%%%%%%%

%%%%%%%%%%%%
%%%%%%%%%%%%

\subsection{First-Order Formulation, Modulation Equations and Selection of a Codimension One Set of Initial Data} \label{subsec:introfirstorderandcodim}

The role of the first-order formulation is to derive the evolution equations for the modulation parameters $\xi$ and $\ell$. The modulation parameters are fixed by imposing the matching number of ``orthogonality conditions'' on the perturbation.  The orthogonality conditions also guarantee that the perturbation stays away from the kernel of the linearized operator. Our approach is based on that of \cite{Stuart1}, which is in turn motivated by \cite{Weinstein1}. 

The first order formulation is closely related to a Hamiltonian formulation of the original Euler-Lagrange equations. To arrive at an adequate first order formulation we need to fix a time function. In our case we have already discussed the foliation and we simply take the time function to be $\sigma$ in \eqref{eq:sigmadefintro1}. This is a degenerate choice because the level sets of $\sigma$ are asymptotically null.  To deal with this, we make use of the observation that the orthogonality conditions may be localized to a large compact set (see for instance \cite{GNT1}), and we impose conditions that involve the perturbation only on the flat part of $\Sigma_{\sigma}$. An implication of localizing the orthogonality conditions is that the perturbation enters linearly in the parameter ODEs. Since the derivatives of parameters also enter linearly in the equation for the perturbation (see Section~\ref{subsec:introoutlinerp}), some care is needed to avoid circularity in the estimates. The key here is that the linear contributions of the perturbation stem from the localization of the eigenfunctions to the complement of a large compact set. Hence, the spatial decay of the eigenfunctions furnishes extra smallness.

Two more technical issues deserve further explanation. In view of the gauge invariance of the problem, the choice of momentum variable for the first order formulation is not obvious.  The proper choice must be such that the orthogonality conditions result in non-degenerate first order ODEs for $\ell$ and $\xi$. We motivate our choice in Section~\ref{susec:momentumvariable}. The derivation of the equations in first order form is rather technical and occupies most of Section~\ref{sec:interior}.

Additionally, due to the quasilinear nature of the HMVC equation, sometimes more derivatives of the modulation parameters arise than we can a priori control in our bootstrap. In principle, it may be possible to use the hyperbolic structure of the equation to solve for the highest order time derivatives in terms of spatial derivatives of the perturbation, and to use integration by parts to avoid the loss of regularity (see for instance \cite[Section~4.1.3]{DKSW}). However, this approach would have to carefully exploit the structure of the equation, which becomes especially difficult in view of the complex form of the equations in the first order formulation. Instead, we modify the orthogonality conditions to obtain \emph{smoothing of the modulation parameters}. This is a robust approach that does not rely on the algebraic structures of the equations. The details are carried out in Section~\ref{subsec:modulationeqs}. Conceptually, we exploit the freedom that while the final values of the parameters are determined by the initial conditions for the HVMC equation, their trajectories are not. Technically, this is achieved by choosing the orthogonality conditions so that $\xi$ and $\ell$ satisfy ODEs of the forms $\dotell=S\calF_{\ell}$ and $\dotxi-\ell=S\calF_\xi$, where $\calF_\ell$ and $\calF_\xi$ depend on the perturbation and its derivatives, and $S$ is a smoothing operator in the time variable. We choose the integral kernel $k(\sigma, \sigma')$ of $S$ to be compactly supported in the range $\sigma' < \sigma$ to preserve the causality of the smoothed-out modulation equations (i.e., $\xi(\sigma), \ell(\sigma)$ are independent of the solution at future times $\sigma' > \sigma$). 

Finally we say a few words about the shooting argument discussed in Item (3) in Section~\ref{subsubsec:introoverallscheme}. The decomposition $\psi=a_{+}Z_{+}+a_{-}Z_{-}+\phi$ is derived in Section~\ref{subsec:unstableint}. The ODEs satisfied by $a_{\pm}$ are given in equation~\eqref{eq:wpoutline10}, and again involve a smoothing operator in the time variable. The trapping assumption is stated in equation~\eqref{eq:a+trap}. Note that this is at the level of the derivative of $a_{+}$. Finally, the standard topological argument (see for instance \cite{MRR1}) is described in Step 2a of the proof of Theorem~\ref{thm:main} in Section~\ref{subsec:proofofmaintheorem}. 
For more background on the construction of center-stable manifolds we refer to \cite{Schlag1,NakSchbook1}.

\subsection{Uniform Boundedness of Energy and Integrated Local Energy Decay} \label{subsec:ideas-iled}
We now discuss the ideas behind our proof of the uniform boundedness and integrated local energy decay estimates for $\phi$. In the case of the linearized equation $(-\partial_{t}^{2} + \underline{H}) \psi = f$ around the Lorentzian catenoid, both bounds follow from existing methods \cite{Tat1, MeTa, MMT1, MST}; see Section~\ref{subsec:iled-catenoid} below. The challenge is to extend these estimates to $\phi$ in our decomposition of the solution. Here, $\phi$ solves the equation
\begin{equation*}
	\calP \phi = f,
\end{equation*}
where $\calP$ is the linearized HMVC operator around $\calQ$ modulo terms involving $\dot{\ell}$ and $\dot{\xi}-\ell$, which are regarded as nonlinearities. The right-hand side $f$ consists of the profile error (i.e., failure of $\calQ$ to solve HVMC; this includes terms linear in $\dot{\ell}$ and $\dot{\xi}-\ell$) and quasilinear nonlinearity. We refer to \eqref{eq:phi1}, \eqref{eq:LEDcalPint1} (interior) and \eqref{eq:abstractexteq1} (exterior) for the precise expressions. We work under a trapping assumption for $a_{+}$ and suitable bootstrap assumptions on $\xi$, $\ell$, $a_{-}$ and $\phi$; see Section~\ref{sec:bootstrap}.

The proof of uniform boundedness of energy for $\phi$ follows by using the global almost stationary vectorfield $\bfT$ (see Section~\ref{sec:GNGC} for the definition) as a vectorfield multiplier, and using the orthogonality conditions to obtain coercivity of the spatial part of the operator $\calP$. To control higher $\bfT$-derivatives, we use $\bfT$ as a commuting vectorfield. Using the equation and elliptic regularity estimates, we may also control higher spatial derivatives. We refer to Proposition~\ref{prop:energyestimate1} for the precise statements and proofs.

The proof of integrated local energy decay for $\phi$ is significantly more difficult due to the reasons discussed in Section~\ref{subsec:ideas-outline}, including 
\begin{enumerate}
\item [(i)] quasilinearity (i.e., occurrence of nonlinear second-order terms),
\item [(ii)] trapping (i.e., existence of an unstable trapped null geodesics along $\set{\rho = 0}$),
\item [(iii)] eigenvalues of the stability operator (i.e., zero and negative eigenvalues of $\underline{H}$),
\item [(iv)] nonstationarity of $\calQ$.
\end{enumerate}
We seek to \emph{divide-and-conquer} these difficulties.

Our first main tool is a \emph{vectorfield multiplier argument} that resembles the proof for the base case $-\partial_{t}^{2} + \underline{H}$ on the Lorentzian catenoid $\calC$ in Section~\ref{subsec:iled-catenoid}, but adapted to our profile $\calQ$. This argument gives the desired integrated local energy decay estimate with an additional lower order term $\| \phi \|_{L^{2}(K)}$ on the right-hand side, where $K$ is a spacetime cylinder around the trajectory $\xi$.  For details, see the proof of Lemma~\ref{lem:phifar1} below. In particular, this argument takes care of issues (i) (quasilinearity) and (ii) (trapping), which are ``high time frequency'' issues. 

To handle the remaining issues we introduce our next key tool, the \emph{time smoothing operator} $P_{\leq N}$, where $N^{-1}$ is the smoothing scale (equivalently, $N$ is the time frequency localization scale). Our initial observation is that $\| \phi - P_{\leq N} \phi \|_{L^{2}(K)}$ is small compared to the left-hand side of integrated local energy decay if $N$ is sufficiently large (see Lemma~\ref{lem:LEDhighfreq1} below), so we only need to control $P_{\leq N} \phi$. We have thus reduced the problem to the consideration of only ``low time frequencies''!

The key benefit of $P_{\leq N}$ is that, by elliptic regularity (for the part of $\calP$ that does not involve $\bfT$), any potentially dangerous second order term $\partial^{2} P_{\leq N} \phi$ may be bounded in terms of $P_{\leq N} \phi$ and $\calP (P_{\leq N} \phi)$. Hence, \emph{the equation $\calP P_{\leq N} \phi = P_{\leq N} f + [\calP, P_{\leq N}] \phi$ may be thought of as }$\calP_{0} P_{\leq N} \phi = \hbox{(perturbative error)}$, where $\calP_{0}$ on the left-hand side is the operator obtained by conjugating $(-\partial_{t}^{2} + \underline{H})$ with the Lorentz transformation with parameter $\underline{\ell}$. In the context of the bootstrap argument, $\underline{\ell}$ would be the final velocity parameter.  This summarizes how issue~(iv) (nonstationarity of $\calQ$) shows up and gets resolved in our proof.

It remains to establish an integrated local energy decay estimate for $P_{\leq N} \phi$, which would in particular control $\| P_{\leq N} \phi \|_{L^{2}(K)}$. As discussed earlier, to obtain such a bound from the properties of $-\partial_{t}^{2} + \underline{H}$, we need $P_{\leq N} \phi$ to satisfy $2n+2$ orthogonality conditions on suitable time slices (in this case, they are boosts of $\{ X^{0} = const \}$ by $\underline{\ell}$). This is issue~(iii) (eigenvalues of the stability operator). Our idea is to use a suitable multiplier argument to \emph{transfer} our orthogonality conditions on $\Sigma_{\sigma}$ to the needed ones; see the proof of Proposition~\ref{prop:LED1}. We remark that, at this point, we need doubly integrable decay rates of $\dot{\xi} - \ell$ and $\dot{\ell}$ to control the error. 
This procedure also requires the right-hand side of the equation to be localized to the flat portion of $\Sigma_{\sigma}$. For this reason, we enact (in fact, twice for technical purposes) the so-called \emph{near-far decomposition} in our proof; see \eqref{eq:phifardef1}, \eqref{eq:calPphinear1}, \eqref{eq:psi-nearfar} and \eqref{eq:psi-nearnear}.

We end this part with a remark on the time smoothing operator $P_{\leq N}$. We define this operator as a smooth cutoff in time frequencies, where the Fourier transform is defined in suitable coordinates. However, unlike the Fourier transform in time, whose definition usually requires taking the Laplace transform first and then considering its analytic continuation, the time smoothing operator is easier to make sense of as an integral operator on physical space.

\subsection{Vectorfield Method for Moving Profile}\label{subsec:introoutlinerp} The final part of the scheme from Section~\ref{subsubsec:introoverallscheme} is proving pointwise decay for the perturbation. For this purpose we use the $r^p$-vectorfield method introduced in \cite{DR1}. This method combines an ILED estimate in a bounded region with vectorfield estimates outside a compact set to obtain pointwise decay. We refer to \cite[Sections 1.1--1.4]{Moschidis1} for a review of the history of the vectorfield method. The $r^p$ method applies to wave equations on asymptotically flat spacetimes. In its simplest form in \cite{DR1} it yields the (interior) pointwise decay rate $t^{-1}$ on $\bbR^{1+n}$. In \cite{Schlue1} the method was extended to give the decay rate $t^{-\frac{3}{2}+\delta}$. A further extension was obtained in \cite{Moschidis1} giving the rate $t^{-\frac{n}{2}}$. We refer to \cite{DR1,Moschidis1} for a general review of the method, and to \cite[Section 9.4]{Moschidis1} for an explanation of the scheme for the improved $t^{-\frac{n}{2}}$ decay. In this work we adapt the method from \cite{Moschidis1}. Our setup differs from the one in \cite{Moschidis1} in a few important respects which we now describe. 

The first new aspect is that our foliation is centered at the trajectory $\xi$ (see Section~\ref{sec:profileintro}). To deal with this, we introduce a null frame that is adapted to the dynamically constructed foliation. We then define our weighted multiplier and commutator vectorfields with respect to this null frame, with spatial weights that are measured from the moving center $\xi$. The remarkable fact is that the wave equation written in the moving null frame has the right structure for the application of the $r^p$-vectorfield method. In particular, because in general $\dotell\neq0$ and $\dotxi\neq \ell$, there will be new error terms with time decay in the wave equation itself, and in the multiplier and commutator identities. The important point is that these errors do not grow as $r\to\infty$, where $r$ denotes the distance to the moving center on a fixed leaf of the foliation (corresponding to \emph{null infinity}), so they can be estimated in our bootstrap argument.  Related to this issue, is the failure of the profile to be an exact solution of the HVMC equation. This implies that the radiation satisfies a wave equation with a source term with time decay. One of the main advantages of our foliation, and the adapted definition of the commutators, is that there is no spatial growth when the commutators fall on the source term.

Another difference of our setup with that of \cite{Moschidis1} is that our linearized operator has an order zero potential. Moreover, the elliptic part of the operator has a nontrivial kernel. These differences become relevant when using the improved decay of higher time derivatives of the perturbation to get improved decay for arbitrary derivatives. In \cite{Moschidis1} this is achieved by viewing the wave equation as an elliptic equation with the time derivatives as source terms, and applying global elliptic estimates. In our context we need a separate argument to deal with the zero order potential and the kernel. These arguments are presented in Lemma~\ref{lem:ellipticestimate1} and Corollary~\ref{cor:higherdL2decay}. The orthogonality conditions from the first order formulation are used here to guarantee that the projection of the perturbation on the kernel has sufficient decay. 

Our modified scheme in dimension $n=5$ gives the pointwise decay $t^{-\frac{9}{4}+\frac{\kappa}{2}}$, $\kappa>0$ arbitrary, for the perturbation. This is different from the rate $t^{-\frac{5}{2}}$ in \cite{Moschidis1}, and we now explain the reason for the discrepancy. Let $\phi$ denote the perturbation. An intermediate step in deriving pointwise decay is proving that $\|\partial^3\phi\|_{L^2}\lesssim t^{-3}$, when at least one of the derivatives is with respect to time. Global elliptic estimates are then applied in \cite{Moschidis1} to conclude that $\|\partial^3\phi\|_{L^2}\lesssim t^{-3}$ for arbitrary derivatives. A similar argument gives $\|\partial^2\phi\|_{L^2}\lesssim t^{-2}$. The $t^{-\frac{5}{2}}$ pointwise decay in \cite{Moschidis1} then follows from the Gagliardo-Nirenberg estimate $\|\phi\|_{L^\infty}\lesssim \|\partial^2\phi\|_{L^2}^{\frac{1}{2}}\|\partial^3\phi\|_{L^2}^{\frac{1}{2}}$. In our case, the equation for $\phi$ contains a source term that depends linearly on the parameter derivatives, which we denote by $\dotwp$. Spatial derivatives do not improve the time decay of this term, so we cannot hope to improve the decay of  $\|\partial^k\phi\|_{L^2}$ beyond the decay of $\dotwp$. On the other hand, the ODEs for the parameters can be used to bound $|\dotwp|$ by a small multiple of $\|\jap{r}^{-c_1}\phi\|_{L^2}$, $c_1<\frac{7}{2}$ (here $r$ denotes the distance to the center $\xi$). Using the elliptic estimates discussed in the previous paragraph (see Lemma~\ref{lem:ellipticestimate1}) we can estimate $\|\jap{r}^{-c_2}\phi\|_{L^2}$, $c_2<\frac{5}{2}$, by~$|\dotwp|$, where the restriction on $c_2$ comes from the order zero potential in the linear operator. Taking $c_1=c_2=\frac{5}{2}-\kappa$ we get the estimate $\|\jap{r}^{-\frac{5}{2}+\kappa}\phi\|_{L^2}\lesssim t^{-\frac{5}{2}+\kappa}$. This sharp estimate can then be used to obtain the non-sharp estimate $\|\partial^3\phi\|_{L^2}\lesssim t^{-\frac{5}{2}+\kappa}$. Combined with Gagliardo-Nirenberg we obtain the pointwise decay $t^{-\frac{9}{4}+\frac{\kappa}{2}}$. Note that if we used elliptic estimates with fractional derivatives (instead of weights) and a fractional Gagliardo-Nirenberg estimate, we could hope to obtain the decay rate $t^{-\frac{5}{2}+\kappa}$. Since the rate $t^{-\frac{9}{4}+\frac{\kappa}{2}}$ is already sufficient to close our bootstrap, we did not further complicate the argument by introducing fractional derivatives.
%%%%%%%%%%%%%%%%%

%%%%%%%%%%%%%
%%%%%%%%%%%%%%%%%%%%%%
%%%%%%%%%%%%%%%%%%%%%%
\subsection{Second Statement of the Main Result}\label{subsec:introsecondthm}
%%%%%%%%%%%%%%%%%%%%%%
%%%%%%%%%%%%%%%%%%%%%%
To state our main result more precisely, we first describe the initial data. Consider two functions
%%%%%%%
%%%%%%%
\begin{align}\label{eq:initialdata1}
\begin{split}
&\psi_0,\psi_1\in C^\infty_0(\barcalC),\qquad \supp~ \psi_0,\psi_1\subseteq \barcalC\cap\{|\Xbar|<R/2\},\\
&\sum_{j=0}^M(\|\jap{\rho}^{1+j}\partial_\Sigma^{1+j}\psi_0\circ F\|_{L^2(\sqrt{|g_\barcalC|}\ud\omega\ud\rho)}+\|\jap{\rho}^{1+j}\partial_\Sigma^{j}\psi_1\circ F\|_{L^2(\sqrt{|g_\barcalC|}\ud\omega\ud\rho)})\leq 1,
\end{split}
\end{align} 
where $M$ is a fixed large constant, $\partial_\Sigma$ denotes $\partial_\rho$ or $\jap{\rho}^{-1}\partial_\omega$, $\partial_\omega$ is a unit size derivative on $\bbS^{n-1}$, and $\ud \omega$ is the standard spherical measure on $\bbS^{n-1}$.
Using the notation introduced in Sections~\ref{subsec:Riemcat} and~\ref{sec:profileintro}, see \eqref{eq:Fdef} and \eqref{eq:tilNdefintro1}, consider
%%%%%%%
%%%%%%%
\begin{align}\label{eq:initialdata2}
\begin{split}
\Phi_0[\psi_0]= F+(\psi_0\circ F) N\circ F,\qquad \Phi_1[\psi_1]=(1,0)+(\psi_1\circ F)N\circ F.
\end{split}
\end{align}
We also let $\tilvarphi_\mu:=\chi \varphibar_\mu$ where the cutoff function $\chi$ is equal to one on $\barcalC\cap\{|\Xbar|<R/3\}$ and supported on $\barcalC\cap\{|\Xbar|<R/2\}$. As discussed earlier, our stability theorem holds under a codimension one condition on the initial data. This condition is given by the vanishing of a certain functional on $(\psi_0,\psi_1)$. But, as the exact form of the vanishing condition is a bit complicated to state at this point, we defer this until Section~\ref{sec:bootstrap}, and simply refer to \eqref{eq:codim1} in the statement of our main theorem. To be precise, the condition \eqref{eq:codim1} singles out a codimension one submanifold in the topology given by the norm in \eqref{eq:initialdata1}. The data set in Theorem~\ref{thm:main} are then parameterized as a graph over this submanifold through the function $b_0$ in the statement of Theorem~\ref{thm:main}. See also Remark~\ref{rem:codim1}.
%%%%%%%%%%%%
%%%%%%%%%%%%
\begin{theorem}\label{thm:main}
 Let $n\geq 5$, and consider $\Phi_0$, $\Phi_1$ as in \eqref{eq:initialdata1}, \eqref{eq:initialdata2}, and assume that $(\epsilon\psi_0,\epsilon\psi_1)$ satisfy \eqref{eq:codim1}. If $\epsilon\geq0$ is sufficiently small, then there exist $b_0\in\bbR$ with $|b_0|\lesssim 1$ and  $\Phi:\bbR\times I \times \bbS^{n-1}\to \bbR^{1+(n+1)}$ satisfying \eqref{eq:HVMC1}, such that $\Phi\vert_{\{t=0\}}=\Phi_0[\epsilon(\psi_0+b_0\tilvarphi_\mu)]$ and $\partial_t\Phi\vert_{\{t=0\}}=\Phi_1[\epsilon(\psi_1-\mu b_0\tilvarphi_\mu)]$. Moreover, there exist $\ell,\xi:\bbR\to \bbR^{n}$ satisfying  $|\ell|,|\dotxi|\lesssim \epsilon$ and
%%%%%%%
%%%%%%%
\begin{align*}
\begin{split}
|\dotell(\sigma)|, |\dotxi(\sigma)-\ell(\sigma)|\to0\qquad \mathrm{as}~\sigma\to\infty,
\end{split}
\end{align*}
such that the image of $\Phi$ can be parameterized as
%%%%%%%
%%%%%%%
\begin{align*}
\begin{split}
\cup_\sigma\Sigma_\sigma\ni p\mapsto p+\psi(p)N(p),
\end{split}
\end{align*}
with $\|\psi\|_{L^\infty(\Sigma_\sigma)}\to0$ as $\sigma\to\infty$. More precisely, there exists a positive $\kappa\ll1$ such that
%%%%%%%
%%%%%%%
\begin{align*}
\begin{split}
|\dotell(\sigma)|,|\dotxi(\sigma)-\ell(\sigma)|\lesssim \epsilon \sigma^{-\frac{5}{2}+\kappa},\qquad \|\psi\|_{L^\infty(\Sigma_\sigma)}\lesssim \epsilon \sigma^{-\frac{9}{4}+\kappa},\qquad \mathrm{as}~\sigma\to\infty.
\end{split}
\end{align*}
\end{theorem}
%%%%%%%%%%%%
%%%%%%%%%%%%
More precise decay estimates on $\psi$ and the parameters can be found in Propositions~\ref{prop:bootstrappar1} and~\ref{prop:bootstrapphi1}. We now make a few remarks about Theorem~\ref{thm:main}.
%%%%%%%%%%%%
%%%%%%%%%%%%
\begin{remark}\label{rem:parametersintro1}
It follows from the decay rate of $\dotell$ and $\dotxi-\ell$ that there exist $\abar,\ellbar\in\bbR^n$ such that $\ell(t)\to\ellbar$ and $\xi(t)\to \abar+ \ellbar t$ as $t\to\infty$. In this sense our theorem implies that the solution approaches a fixed, boosted and translated Lorentzian catenoid. The differential equations governing the evolution of the parameters are derived in Section~\ref{subsec:modulationeqs}.
\end{remark}
%%%%%%%%%%%%
%%%%%%%%%%%%
%%%%%%%%%%%%
%%%%%%%%%%%%
%%%%%%%%%%%%
%%%%%%%%%%%%
\begin{remark}\label{rem:codim1}
As mentioned earlier the codimension one condition of the data in Theorem~\ref{thm:main} is optimal. Since the functional $\calZ$ in \eqref{eq:codim1} depends differentiably on $(\psi_0,\psi_1)$ with respect to the norm in \eqref{eq:initialdata1}, condition \eqref{eq:codim1} determines a codimensional one submanifold in the ball of radius $\epsilon$ in this topology.  However, in this work, we do not pursue the question of uniqueness and continuous dependence of $b_0$ on the initial data $\psi_0$ and $\psi_1$.  As a result, we cannot infer that the set of initial data, considered in Theorem~\ref{thm:main} form a codimension one \emph{submanifold} in any topology. If $b_0$ is also a $C^1$ function, then our data set can be viewed as belonging to a codimensional one subset in the commonly used sense.
\end{remark}
%%%%%%%%%%%%
%%%%%%%%%%%%

\subsection{Further Discussions}\label{subsec:discussion}
%%%%%%%%%%%%%%%%%%%%%%
%%%%%%%%%%%%%%%%%%%%%%
Further discussion of related works and subjects are in order.

\subsubsection{Other prior works on the hyperbolic vanishing mean curvature equation}\label{subsubsec:morehistory}

Beyond the previously mentioned result \cite{AC2} on local existence for the HVMC equation for sufficiently smooth initial data, we point out the low regularity local well-posedness results \cite{Ettinger, AIT21}. 
We refer to \cite{Wong1} for the study of local well-posedness in relation to the action principle formulation.
The global nonlinear stability of hyperplanes under the HVMC evolution was considered in \cite{B1, Lin1, Stefanov11, Wong17}. Under symmetric perturbations the nonlinear stability of the Lorentzian catenoid was studied in \cite{KL1, DKSW} and that of the Lorentzian helicoid in \cite{Marchali1}. Simple planar travelling wave solutions to the HVMC equation were proven to be globally nonlinearly stable in \cite{AW20}.
Singularity formation has been analyzed in \cite{NT13, JNO15, Wong18, BMP1}.
For a discussion of the relevance of the HVMC equation in physics, we refer the reader to \cite{AC1, AC2, Hoppe13}.
The Lorentzian constant positive mean-curvature flow has been considered in \cite{Wong2}.

\subsubsection{Comparison with black hole stability}
The present paper concerns nonlinear asymptotic stability of a stationary solution to a multi-dimensional quasilinear wave equation without any symmetry assumptions. Despite obvious differences in the inherent complexities of the underlying PDEs, our main result may be formally compared with the recent colossal works \cite{DHRT1,KlSz1,KlSz2} on the nonlinear asymptotic stability of Kerr and Schwarzschild black holes as stationary solutions to the vacuum Einstein equation, which is a $(3+1)$-dimensional quasilinear wave equation, without any symmetry assumptions.

Our problem and the black hole stability problem share some important features. The nontrivial kernel of the linearized operator around the stationary solution necessitates modulation of some parameters and a suitable choice of gauge (i.e., a way to represent the solution among many equivalent descriptions). In the case of the Schwarzschild black hole, a codimension condition on the initial data naturally appears as in our problem \cite{DHRT1, KlSz1}. At the level of proofs, this paper and the above works share the same basic strategy for proving the pointwise decay of the perturbation, namely, to first prove an integrated local energy decay (or Morawetz) estimate and the uniform boundedness of energy, then to establish pointwise decay by some version of the vectorfield method. Indeed, this powerful strategy was mostly developed in works with the black hole stability problem in mind -- see \cite{DR1}, and also \cite{Tat2, MTT, OlSt}.

Needless to say, our problem is simpler compared to the black hole stability problem in a number of aspects, such as the spatial dimension, the gauge choice (compare our choice described in Section~\ref{sec:profileintro} with \cite{DHRT1, KlSz1, KlSz2}), and the analysis of the linearized problem (compare the discussion in Section~\ref{subsec:Riemcat} with \cite{HHV1,DHR1,ABBM1,HKW1}). Nevertheless, in this paper we satisfactorily resolve a key issue that is shared by many soliton stability problems, but not with the black hole stability problem -- this is the issue of \emph{modulation of the translation and boost parameters}. In our problem, as well as in many soliton stability problems, the stationary solution (the catenoid or the soliton) is defined on a natural ambient spacetime, and it is of interest to track the evolution of the translation and boost parameters in relation to the ambient spacetime. In contrast, in general relativity there is no notion of an ambient spacetime, and the analogous issue is subsumed in the choice of the gauge in the black hole stability problem. As discussed earlier, this issue is resolved in our work by a new construction of a dynamic profile representing a ``moving catenoid,'' the use of localized orthogonality conditions that enables us to utilize a suitable first-order formulation of the equation to derive the evolution equations for the parameters, a robust scheme for establishing integrated local energy decay for perturbations of the dynamic profile from the case of the stationary solution, as well as an adaptation of the $r^{p}$-method for the dynamic profile. In view of the pervasiveness of the same issue in soliton stability problems, we are hopeful that our ideas might be useful elsewhere as well.

\subsubsection{Soliton stability problem for semilinear dispersive equations}
There is a vast literature on the problem of stability of solitons for \emph{semilinear} dispersive equations; for those who are interested, we recommend the excellent survey articles of Kowalczyk--Martel--Mu\~{n}oz \cite{KMM17} and Tao \cite{Tao09} as a good starting point. In relation to this rich and beautiful subject, our aim in this paper is to specifically tackle those challenges that arise from the \emph{quasilinearity} of the equation. Our aim, in turn, is motivated by the conjectured asymptotic stability of some well-known topological solitons solving quasilinear wave equations, such as the Skyrmion for the Skyrme model \cite{MantonSutcliffe}.

%%%%%%%%%%%%%%%%%%%%%%
%%%%%%%%%%%%%%%%%%%%%%
\subsection{Outline of the Paper}\label{subsec:introremainingoutline}
%%%%%%%%%%%%%%%%%%%%%%
%%%%%%%%%%%%%%%%%%%%%%
The remainder of this paper is organized as follows. Section~\ref{sec:prelim} contains the notation and some preliminary results. In Section~\ref{sec:interior} we derive a first order formulation of our problem in terms of a vector unknown $\vecpsi$, for a given set of parameters $\xi$ and $\ell$. We also state the corresponding orthogonality conditions and carry out a further decomposition of $\vecpsi=\vecphi+a_{+}\vecZ_{+}+a_{-}\vecZ_{-}$, by separating the contribution of growing mode of the linearized operator. For this decomposition and our choice of orthogonality conditions, we then derive the modulation equations satisfied by $\xi$, $\ell$, and $a_{\pm}$.

In Section~\ref{sec:coordinates}, we give a more detailed description of the foliation, various coordinates, and vectorfields, again for a given choice of parameters $\xi$ and $\ell$. We also derive expressions for the relevant operators in terms of the described coordinates and vectorfields.

The bootstrap assumptions are stated in Section~\ref{sec:bootstrap}, where, in Propositions~\ref{prop:bootstrappar1} and~\ref{prop:bootstrapphi1} we also give more precise decay estimates than the ones given in Theorem~\ref{thm:main}. The proof of Propositions~\ref{prop:bootstrappar1} and~\ref{prop:bootstrapphi1} will occupy the remaining sections of the paper, and in Section~\ref{sec:bootstrap} we further show that Theorem~\ref{thm:main} follows from the bootstrap propositions.

Section~\ref{sec:parametercontrol} contains the proof of Proposition~\ref{prop:bootstrappar1} which closes the bootstrap assumptions for all parameters except the growing mode $a_{+}$. For the latter, a separate shooting argument is needed, which is carried out in the proof of Theorem~\ref{thm:main} in the last part of Section~\ref{sec:bootstrap}.

The proof of Proposition~\ref{prop:bootstrapphi1} is contained in Sections~\ref{sec:LED} and~\ref{sec:exterior}. Section~\ref{sec:LED} contains a general local energy decay at the linear level. In view of the calculations in Section~\ref{sec:coordinates} and the bootstrap assumptions in Section~\ref{sec:bootstrap}, the assumptions on the linear operator in this estimate are satisfied for us. In Section~\ref{sec:exterior} we use the linear result of Section~\ref{sec:LED}  to prove nonlinear energy and local energy decay estimates. Using these, we also prove $r^p$ weighted energy estimates, which in turn are used to prove decay estimates and Proposition~\ref{prop:bootstrapphi1}.
%%%%%%%%%%%%%%%%%%%%%%
%%%%%%%%%%%%%%%%%%%%%%
%%%%%%%%%%%%%%%%%%%%%%
\section{Preliminaries}\label{sec:prelim}
%%%%%%%%%%%%%%%%%%%%%%
%%%%%%%%%%%%%%%%%%%%%%
%%%%%%%%%%%%%%%%%%%%%%
%%%%%%%%%%%%%%%%%%%%%%
%%%%%%%%%%%%%%%%%%%%%%
\subsection{Notation and Conventions}
%%%%%%%%%%%%%%%%%%%%%%
%%%%%%%%%%%%%%%%%%%%%%
Here we collect some of the notation and conventions that are used repeatedly in this work. This is meant as a reference for the reader, and some of the precise definitions will appear only later in the paper. Some of the notation and conventions which are used more locally in various parts of the paper do not appear in this list.
 \subsubsection{\underline{The profile and the main variables}} $\barcalC$ denotes the Riemannian catenoid with its standard embedding in $\bbR^{n+1}$, and $\calC = \bbR \times \underline{\calC}$ the product Lorentzian catenoid in $\bbR^{1+(n+1)}$. The boost and translation\footnote{To be precise, to leading order $\xi\approx t\ell+a$ where $a$ is a fixed translation parameter.} parameters are denoted by $\ell$ and $\xi$ respectively, where $|\ell|,|\dotxi|<1$. In our applications we will always have $|\ell|,|\dotxi|\ll1$. Here, and below, the dot over a parameter denotes the time derivative. We will also sometimes use a prime $'$ to denote the derivative of a function of a single variable (such as time). Given $\ell$ and $\xi$ as above, the boosted catenoid $\calC_\sigma$ and $\bsUpsigma_\sigma$ are defined as in \eqref{eq:calCsigmadefintro1} and \eqref{eq:Sigmadefintro1}, and the profile is $\cup_\sigma\Sigma_\sigma$, $\Sigma_\sigma=\calC_\sigma\cap\bsUpsigma_\sigma$. The almost normal vector to the profile is denoted by $N:\cup_\sigma\Sigma_\sigma\to \bbR^{1+(n+1)}$, and the perturbation, defined in \eqref{eq:psidefintro1}, by $\psi:\cup_\sigma\Sigma_\sigma\to \bbR$. In the first order formulation, $\vecpsi=(\psi,\dotpsi)$ denotes\footnote{When there is no risk of confusion, we identify row and column vectors in this work. So, for instance, we use both $(\psi,\dotpsi)$ and $(\psi,\dotpsi)^\intercal$ for $\vecpsi$.} the vector form of the perturbation, where $\dotpsi$ is the momentum variable and roughly corresponds to the time derivative of $\psi$. Corresponding to the negative eigenvalue of the linearized operator $-\Delta_\barcalC-|\secondff|^2$ (with eigenfunction $\fy_\mu$, see Section~\ref{subsec:Riemcat}) there are two projection coefficients in the first order formulation: $a_{+}$ denotes the unstable (growing mode) coefficient corresponding to the eigenfunction, and $a_{-}$ the stable (decaying mode) coefficient. The remainder, after subtracting the contribution of the corresponding eigenfunction from $\vecpsi$, is denoted by $\vecphi$ at the vector level (in the first order formulation) and by $\phi$ at the scalar level (see Section~\ref{subsec:unstableint}). We will denote the flat and hyperboloidal parts of the profile by $\calC_{\flatt}:=\{X^0\geq \sigma_\temp(X)+\delta_1\}$ and $\calC_{\hyp}:=\{X^0\leq \sigma_\temp(X)-\delta_1\}$ respectively. We often refer to the region inside a large compact set contained in $\calC_\flatt$ as the interior, and the complement of this region as the exterior.
 \subsubsection{\underline{Parameter derivatives}} We will use $\dotwp$ to denote the parameter derivatives $\dotell$ and $\dotxi-\ell$. When used as a vector, $\dotwp=(\dotell,\dotxi-\ell)^\intercal$ in that order. When used schematically, for instance in estimates or to denote dependence on parameter derivatives, the order will not be important, so for example $O(\dotwp)$ denotes terms that are bounded by $|\dotell|$ or $|\dotxi-\ell|$. The distinction will be clear from the context. More generally, $\dotwp^{(k)}$ denotes a total of $k$ derivatives of the parameters, so for instance $\dotwp^{(2)}$ could be any of $\ddotell$, $|\dotxi-\ell|^2$, $\ddotxi-\dotell$, etc. $\dot{\wp}^{(\leq k)}$ denotes a total of up to $k$, but at least one, parameter derivatives. $\wp^{(\leq k)}$ denotes a total of up to $k$ parameter derivatives, but possibly also an undifferentiated $\ell$. We sometimes also use the notation $\wp$ for $\ell$. Note that $\xi$ itself cannot be written as $\wp$ ($\xi$ is expected to grow linearly in time), but $\dotxi$ can be written as $\dotxi=\dotxi-\ell+\ell$, which is a sum of terms of the form $\dotwp$ and $\wp$. A similar notation is used for $\dota_{\pm}^{(k)}$, $a_{\pm}^{(\leq k)}$, etc. Note that here even the undifferentiated $a_{\pm}$ are expected to have time decay (for the growing mode $a_{+}$ only after appropriately modifying the initial data; see Theorem~\ref{thm:main}). 
 \subsubsection{\underline{Constants}} $\epsilon$ is the smallness parameter for the size of the initial perturbation.  $\kappa$ is a small positive absolute constant which arises in the decay rates in the bootstrap argument; see Section~\ref{sec:bootstrap}. In our bootstrap argument, the energy of the perturbation enters to linear order in estimating the parameter derivatives, and the parameter derivatives enter linearly in the energy estimates. What breaks the circularity is that the linear appearance of the energy of the perturbation in the estimates for the parameter derivatives is always accompanied by a small (but not decaying) constant. This small constant is denoted by $\delta_\wp$ in the bootstrap assumptions of Section~\ref{sec:bootstrap}. The final time of the bootstrap interval is denoted by $\tau_f$. There are also a few large radii that appear in our arguments. $R$ is  a large constant such that the initial data are supported in $\barcalC\cap\{|\Xbar|<R/2\}$; see \eqref{eq:initialdata1}. Also, the transition region from the flat  to hyperboloidal parts of the foliation happens in the region $|\Xbar|\simeq R$; see Section~\ref{sec:profileintro}. The constant $R_1\gg1$ is such that $R_1\ll R$  and such that the support of the test functions\footnote{These are the truncated eigenfunctions of the linearized operator in the first order formulation.} $\vecZ_i$, $i\in\{\pm,1,\dots,2n\}$, in Section~\ref{sec:interior} is contained in $|\rho|\leq R_1$ (see \ref{subsec:modulationeqs}). We will use $\Reigenfunctioncutoffscale$ as an absolute constant and gain smallness in inverse powers of $\Reigenfunctioncutoffscale$. The size of the data, $\epsilon$, is considered small relative to any inverse power of $\Reigenfunctioncutoffscale$. In particular, since in view of the bootstrap assumptions in Section~\ref{sec:bootstrap} we have $|\ell|\lesssim \epsilon$, quantities such as $(\Reigenfunctioncutoffscale)^m|\ell|$ are considered small, for any power $0<m\leq m_{\mathrm{lrg}}$, where $m_{\mathrm{lrg}}$ is fixed a large integer. The smallness of the constant $\delta_\wp$ above is in terms of $\ell$ and inverse powers of $\Reigenfunctioncutoffscale$ (the reason the energy of the perturbation enters linearly in the equation for the parameter derivatives is that the eigenfunctions are truncated at scale $\Reigenfunctioncutoffscale$, so one should expect the error to get smaller for larger $\Reigenfunctioncutoffscale$).
 
 \subsubsection{\underline{Coordinates, derivatives, and vectorfields}}\label{subsec:prelimvfs} We will mainly work with two sets of coordinates: $(t,\rho,\omega)$ in the interior (flat) part of the foliation and $(\tau,r,\theta)$ in the exterior (hyperboloidal) part. The precise definitions are given in Section~\ref{sec:interior} and~\ref{sec:profile2} for the interior, and in Section~\ref{sec:profile2} for the exterior. In addition to these, in a few occasions we will use the global non-geometric coordinates $(\uptau,\uprho,\uptheta)$, see Section~\ref{sec:GNGC}, and the global geometric coordinates $(\tiluptau,\tiluprho,\tiluptheta)$, see Section~\ref{sec:GGC}. $\RbfT$ denotes the global almost stationary vectorfield, which in terms of the global non-geometric coordinates introduced in Section~\ref{sec:GNGC} is given by $\partial_{\uptau}$. In general $\partial$ denotes arbitrary derivatives that have size of order one, and $\partial_\Sigma$ the subset of these derivatives that are tangential to the leaves of the foliations. In the exterior region, $\tilpartial_\Sigma$ denotes derivatives which can be written as a linear combination of $\partial_\Sigma$ and $r^{-1}\RbfT$, with coefficients of size of order one. In general we denote the number of derivatives by a superscript. For instance $\partial_\Sigma^{\leq k}$ means up to $k$ tangential derivatives. There are also a few commutator and multiplier vectorfields which are used in the exterior in Section~\ref{sec:exterior} in the context of proving decay estimates for the perturbation. The precise definitions are given in Section~\ref{sec:profile2}, but we give a brief description here: $L$ and $\Lbar$ are the outgoing and incoming almost null vectorfields. $\Omega$ is the rotation vectorfield. $T$, which is comparable and almost colinear with $\RbfT$, is defined by $T=\frac{1}{2}(L+\Lbar)$. In the exterior region where these vectorfields are defined we use $\tilde{r}L$, $\Omega$, and $T$ as commutators, and use  $X^k$ (when $k=1$ we simply write $X$) to denote an arbitrary string of $k$ such vectorfields. Here $\tilr$ is a geometric radial variable introduced in Section~\ref{sec:profile2}. 
 \subsubsection{\underline{Volume forms}} In general we use $\ud V$ to denote the induced volume form from the ambient space $\bbR^{1+(n+1)}$. If there is any risk of confusion we use a subscript to denote the subset on which the volume form is induced (for instance $\ud V_S$ for the subset $S$). When working in a fixed set of coordinates we sometimes write out the volume form explicitly. In the exterior region, it is sometimes more convenient to use the coordinate volume form for the Minkowski metric rather than the geometric induced one. It will be clear from the bootstrap assumptions that these two volume forms are comparable and therefore various norms defined with respect to them are equivalent. The volume form on the standard unit sphere will be denoted by $\ud\theta$ or $\ud S$ interchangeably (or $\ud \omega$, $\ud \uptheta$, etc, depending on the coordinate system we are using).
\subsubsection{\underline{Cutoffs}} We use the notation $\chi$ for smooth cutoff functions defined on $\cup_\sigma\Sigma_\sigma$ and taking values in $[0,1]$. We may denote the set on which $\chi$ is equal to one by a subscript. For instance $\chi_S$ is equal to one on $S$ and equal to zero outside of a neighborhood of $S$ (we will make the support more precise when needed). For a positive number $c$, $\chi_{\geq c}$ denotes a cutoff which is one in the region $\{|\uprho|\geq c\}$ and equal to zero outside of $\{|\uprho|\geq \frac{c}{2}\}$. Here $\uprho$ is the radial coordinate from the global non-geometric coordinates in Section~\ref{sec:GNGC}. We also define $\chi_{<c}:=1-\chi_{\geq c}$.
\subsubsection{\underline{Dimension}}The main result of this work is valid for dimensions $n\geq 5$. For concreteness we have set $n=5$ in many places (for instance for the decay rates in the bootstrap assumptions) and kept the notation $n$ in other places (for instance in some multiplier identities) where we thought this would add to the clarity of exposition. The reader can set $n=5$ everywhere, and the modifications for higher dimensional cases are minimal.
\subsubsection{\underline{The normal and decay of eigenfunctions}}\label{subsubsec:normal} For the standard Riemannian catenoid as described in Section~\ref{subsec:Riemcat}, and with the notation used there, the normal vector is given by 
%%%%%%%
%%%%%%%
\begin{align*}
\begin{split}
\nu\equiv \nu(z,\omega)=(\frac{\Theta(\omega)}{\jap{z}^{n-1}},\sqrt{1-\jap{z}^{2-2n}}).
\end{split}
\end{align*}
As mentioned in Section~\ref{subsec:Riemcat}, the first $n$ components, $\nu^i=\frac{\Theta^i}{\jap{z}^{n-1}}$, $i=1,\dots,n$,
appear as eigenfunctions of the main linearized operator \Hbar. It is useful to keep in mind that these have decay $\jap{z}^{1-n}$ and satisfy (here $\sqrt{|h|}$ denotes the volume form associated to the metric \eqref{eq:RiemmCatmetric1})
%%%%%%%
%%%%%%%
\begin{align*}
\begin{split}
\int \nu^i\nu^j \sqrt{|h|}\ud \omega\ud z = C\delta^{ij},
\end{split}
\end{align*}
where $\delta^{ij}$ is the Kronecker delta and $C$ is a constant of order one. We also remark that since the metric is asymptotically flat, the eigenfunction $\varphibar_\mu$ from Section~\ref{subsec:Riemcat} is exponentially decaying.
\subsubsection{\underline{Two asymptotic ends}} Many of the estimates and identities in this work are derived only near one of the asymptotic ends of the solution. In all cases, the other asymptotic end can be treated in exactly the same way, possibly with a change of overall sign. This remark applies in particular to many of the vectorfied identities and estimates, for instance in sections~\ref{sec:LED} and~\ref{sec:exterior}. The two ends asymptote to the limiting hyperplanes $X^{n+1}=\pm S$ in the ambient space (see \eqref{eq:Sendpointdef1}).
\subsubsection{\underline{Notation for the second fundamental form}} As discussed in the introduction, the stability (or linearized) operator for the Riemannian catenoid is $-\Delta_\barcalC-|\secondff|^2$, where $\secondff$ denotes the second fundamental form. We will sometimes use the notation $V=|\secondff|^2$ when working with this linearized operator. When proving more general linear estimates (such as local energy decay) we still use $V$ for the potential, and impose conditions on the linearized operator which are satisfied by $-\Delta_\barcalC-|\secondff|^2$. This distinction between the different uses of $V$ will be clear from the context.
\subsubsection{\underline{Exterior parametrization over a hyperplane}}\label{subsubsec:calOnotation} Outside a large compact set, we can parameterize each asymptotic end of the solution as a graph over a hyperplane (for instance the hyperplanes $\{X^{n+1}=\pm S\}$). The function giving this parameterization for the Riemannian catenoid is denoted by $Q$. We use $Q_\wp$ to denote the corresponding function when taking into account the boost and translation parameters, although when there is no risk of confusion we drop $\wp$ from the notation and simply write $Q$. See Section~\ref{sec:coordinates}.
\subsubsection{\underline{The $O$ and $\calO$ notation}} The notation $f=O(g)$ is used as usual to mean $|f|\leq C|g|$ for some constant $C$. The notation $f=o_{\alpha}(g)$ is also used in the usual way to mean that $|f|/|g|$ goes to zero as the parameter $\alpha$ approaches a limiting value which will be clear from the context (usually zero or infinity). We will also use the notation $\calO$ whose meaning we now explain. In order to prove decay estimates for $\phi$ we will commute the equation it satisfies with $\RbfT$ (see Subsection~\ref{subsec:prelimvfs} for the notation). To obtain the desired decay in time, it is important that $k$ applications of $\RbfT$ improve the decay of $\phi$ by $\uptau^{-k}$ for $k\leq 2$ (the upper bound $2$ comes from setting $n=5$). Similarly, we will need improved decay estimates on the time derivatives of the parameters up to two commutations of $\RbfT$. These improved decay rates are reflected in the bootstrap assumptions in Section~\ref{sec:bootstrap} and the estimates in Section~\ref{sec:parametercontrol} (see for instance Proposition~\ref{lem:OmegaiTkphi}). For this, it is important that the various error terms that appear in our estimates have improved time decay up to two orders of differentiation in $\RbfT$. In this process, we also need to commute the equation satisfied by $\phi$ with the weighted derivatives $\tilr L$ and $\Omega$ (see Subsection~\ref{subsec:prelimvfs}), which have size of order $\tilr$, near the asymptotically flat ends. Again, it is important that certain error terms have faster $\tilr$ decay in the exterior region with every application of $L$ and $\frac{1}{\tilr}\Omega$, up to the order of commutation. To capture these improved decay properties we use the notation $\calO$. That is, an error term of the form $\calO(f)$ is still bounded by $|f|$ after any number of differentiations by $\tilr L$ or $\Omega$ in the exterior, and by $\sum_{j\leq k}|\dotwp^{(j)}\RbfT^{k-j}f|$ after $k$ differentiations by $\RbfT^k$ globally. For instance, an error that is denoted by $\calO(\dotwp)$ will still be bounded by $\calO(\dotwp)$ after applications of $\tilr L$ and $\Omega$ in the exterior, and by $\calO(\dotwp^{(k+1)})$ after $k$ applications of $\RbfT$ globally. Also note that more than two differentiations in $\RbfT$ does not change the decay rate, so for instance a term of the form $\calO(\dotwp)$ is still bounded by $\calO(\dotwp^{(3)})$ after $k$ applications of $\RbfT$, $k\geq 3$. That we can bound higher derivatives of the parameters by their lower derivatives is a consequence of the parameter smoothing, which is carried out in Sections~\ref{subsec:modulationeqs} and~\ref{subsec:unstableint} to prevent loss of regularity (see also Section~\ref{sec:parametercontrol} for the corresponding estimates). Even though we start using this notation already in Section~\ref{sec:interior}, the corresponding properties of these error terms follow only after the bootstrap estimates are stated in Section~\ref{sec:bootstrap}. It is worth mentioning that the error terms in sections~\ref{sec:interior} and~\ref{sec:parametercontrol} are always estimated after integrating against a compactly supported function. Therefore, the spatial decay of these terms is not relevant, and are not specified when using the $\calO$ notation there.
%%%%%%%%%%%%%%%%%%%%%%
%%%%%%%%%%%%%%%%%%%%%%
%%%%%%%%%%%%%%%%%%%%%%
\subsection{Local Existence}\label{subsec:lwp}
%%%%%%%%%%%%%%%%%%%%%%
%%%%%%%%%%%%%%%%%%%%%%
%%%%%%%%%%%%%%%%%%%%%%
As mentioned in the introduction, a systematic treatment of local existence for the HVMC equation is contained in \cite{AC1,AC2}. For our purposes it is convenient to also have a formulation with respect to an almost null foliation of the ambient space.  The results of  \cite{AC1,AC2} can be adapted to this setting using the arguments of \cite{Luk1} (see also \cite{Rendall1}) which proves local existence for a class of nonlinear wave equations with characteristic initial data. Without reproducing the details of these arguments, we record the desired corollary of these works for our future reference. Before doing so, we need to introduce some more notation. Recall the definition of the profile $\cup_\sigma(\calC_\sigma\cap \bsUpsigma_\sigma)$ from Section~\ref{sec:profileintro}. Given fixed $\ell_0, \xi_0\in \bbR^n$, let $\mathring{\calC}_\sigma(\ell_0,\xi_0)$ and $\mathring{\bsUpsigma}_\sigma(\ell_0,\xi_0)$ denote the submanifolds of $\bbR^{1+(n+1)}$ corresponding to the choices $\ell\equiv\ell_0$ and $\xi(\tau)\equiv\xi_0+\ell_0\tau$. The corresponding choice of transversal vector $N$ is denoted by $\ringN_{\ell_0,\xi_0}$. Then, for each $\tau>0$ let 
%%%%%%%
%%%%%%%
\begin{align*}
\begin{split}
\calD_0^\tau(\ell_0,\xi_0):=\cup_{\sigma\in[0,\tau)}  \mathring{\Sigma}_\sigma(\ell_0,\xi_0):=\cup_{\sigma\in[0,\tau)} (\mathring{\calC}_\sigma(\ell_0,\xi_0)\cap \mathring{\bsUpsigma}_\sigma(\ell_0,\xi_0)).
\end{split}
\end{align*}
We equip each leaf $\mathring{\Sigma}_\sigma(\ell_0,\xi_0)=\mathring{\calC}_\sigma(\ell_0,\xi_0)\cap \mathring{\bsUpsigma}_\sigma(\ell_0,\xi_0)$ with the (Riemannian) metric induced from the ambient space, and denote the tangential derivatives of size one by $\ringpartial_\Sigma$. The restriction of $\frac{\partial}{\partial X^0}+\ell_0$ to $\calD_0^\tau(\ell_0,\xi_0)$ is denoted by $T$ (note that $\frac{\partial}{\partial X^0}+\ell_0$ is tangent to $\calD_0^\tau(\ell_0,\xi_0)$). We use $\rho$ to denote the distance along $\mathring{\Sigma}_\sigma(\ell_0,\xi_0)$ to $\Xbar=\xi_0+\sigma\ell_0$, with respect to the induced metric.
%%%%%%%%%%%%
%%%%%%%%%%%%
\begin{proposition}\label{prop:LWP}
Let $n\geq 5$, and consider $\ringPhi_0(p)=p+\ringepsilon\ringpsi_0 \ringN_{\ell_0,\xi_0}$, $\ringPhi_1= (1,\ell_0)+\ringepsilon\ringpsi_1 \ringN_{\ell_0,\xi_0}$, where $\ringpsi_j$ are smooth functions on $\ringSigma_0(\ell_0,\xi_0)$ with $\|\ringpartial_\Sigma^k\ringpsi_0\|_{L^2(\Sigma_0(\ell_0,\xi_0))}$ and $\|\jap{\rho}^{-1}\ringpartial_\Sigma^{k-1}\ringpsi_1\|_{L^2(\Sigma_0(\ell_0,\xi_0))}$ finite for $k\leq M$, $M$ sufficiently large. If $|\ell_0|$ and $\ringepsilon>0$ are sufficiently small, then there exists $\tau\gtrsim1$ and a unique smooth function $\ringpsi:\calD_0^{\tau}(\ell_0,\xi_0)\to\bbR$, such that $\ringPhi:\calD_0^{\tau}(\ell_,\xi_0)\to\bbR^{1+(n+1)}$ defined by
%%%%%%%
%%%%%%%
\begin{align}\label{eq:lwppar}
\begin{split}
\ringPhi(p)= p+\ringpsi(p)\ringN_{\ell_0,\xi_0}(p)
\end{split}
\end{align} 
satisfies \eqref{eq:HVMC1}, and $\ringpsi(0)=\ringepsilon\ringpsi_0$, $T\ringpsi(0)=\ringepsilon\ringpsi_1$.
\end{proposition}
%%%%%%%%%%%%
%%%%%%%%%%%% 
We also want to be able to parameterize the solution given by Proposition~\ref{prop:LWP} as in Section~\ref{sec:profileintro} for other choices of $\ell(\tau)$ and $\xi(\tau)$, with $|\ell|,|\dotxi|\ll1$.  This is possible according to the following normal neighborhood lemma.
%%%%%%%%%%%%
%%%%%%%%%%%%
\begin{lemma}\label{rem:normalneighborhood}
Consider $\ell$ and $\xi$ with $|\ell|$, $|\dotxi|$, and $|\xi(0)-\xi_0|$ sufficiently small, and the solution $\calM$ from Proposition~\ref{prop:LWP} parameterized by \eqref{eq:lwppar} for $\tau\in[0,\tau_0]$. Then condition \eqref{eq:psidefintro1} for $\sigma\in [0,\tau_0]$ determines $\psi$ uniquely and $p\mapsto p+\psi(p) N(p)$ as defined in \eqref{eq:psidefintro1} gives another parameterization of $\calM$.
\end{lemma}
%%%%%%%%%%%%
%%%%%%%%%%%%
\begin{proof}
To see that $\psi$ is uniquely determined, we need to show that for each $p\in\cup\Sigma_\sigma$, the line $p+sN(p)$ intersects $\calM$ only once. Let $P$ be a point of intersection. By construction, there is a unique point $\ringp\in \cup\ringSigma_\sigma(\ell_0,\xi_0)$ such that $P$ is on the line through $\ringp$ in the direction of $\ringN_{(\ell_0,\xi_0)}(\ringp)$. Moreover, since $\ringN_{(\ell_0,\xi_0)}$ is almost normal to $\calM$, this line satisfies $$\mathrm{dist}(P+s\ringN_{(\ell_0,\xi_0)}(\ringp),\calM)\gtrsim |s|,$$ where $\mathrm{dist}(A,\calP)$ denotes the Euclidean distance from $A$ to $\calM$. But then, since $|N(p)-\ringN_{(\ell_0,\xi_0)}(\ringp)|\ll1$ (which follows from the smallness of $\ell$ and $\ell_0$), 
%%%%%%%
%%%%%%%
\begin{align*}
\begin{split}
\mathrm{dist}(P+sN(p),\calM)\geq \mathrm{dist}(P+s\ringN_{(\ell_0,\xi_0)}(\ringp),\calM)-|s||N(p)-\ringN_{(\ell_0,\xi_0)}(\ringp)|\gtrsim |s|,
\end{split}
\end{align*}
which shows that the line $P+sN(p)$ does not intersect $\calM$ again. To see that we have a parameterization, suppose $p+\psi(p)N(p) = q+\psi(q)N(q)$ for some $p,q\in\cup_\sigma\Sigma_\sigma$. Then by derivative bounds on $\psi$ and $N$,
%%%%%%%
%%%%%%%
\begin{align*}
\begin{split}
|p-q|\leq |\psi(p)N(p)-\psi(q)N(q)|\ll |p-q|,
\end{split}
\end{align*}
which can happen only if $p=q$.
\end{proof}
%%%%%%%%%%%%%%%%%%%%%%
%%%%%%%%%%%%%%%%%%%%%%
%%%%%%%%%%%%%%%%%%%%%%
\subsection{Local Energy Decay for the Product Lorentzian Catenoid} \label{subsec:iled-catenoid}
%%%%%%%%%%%%%%%%%%%%%%
%%%%%%%%%%%%%%%%%%%%%%
%%%%%%%%%%%%%%%%%%%%%%
The notation used in this section is independent of the rest of the paper. We prove a local energy decay (LED) estimate for $\Box+V$, where $\Box$ denotes the wave operator of the product Lorentzian catenoid with metric
%%%%%%%
%%%%%%%
\begin{align*}
\begin{split}
g= -\ud t\otimes\ud t+\frac{\rho^2\jap{\rho}^{2(n-2)}}{\jap{\rho}^{2(n-1)}-1}\ud \rho\otimes\ud\rho+\jap{\rho}^2 \ud \omega\otimes \ud \omega,
\end{split}
\end{align*}
and $V$ is a smooth, time independent, potential satisfying $|V|\lesssim \jap{\rho}^{-6}$. This abstract estimate will be used during the proof of LED in Section~\ref{sec:LED} and the choice of multipliers here motivate the ones made there. To start, let $\psi$ satisfy
%%%%%%%
%%%%%%%
\begin{align}\label{eq:LEDproduct1}
\begin{split}
(\Box+V)\psi=G.
\end{split}
\end{align}
In coordinates, we have
%%%%%%%
%%%%%%%
\begin{align*}
\begin{split}
\Box:=-\partial_t^2+\Delta = -\partial_t^2+g^{\rho\rho}\partial_\rho^2+\frac{\partial_\rho(|g|^{1/2}g^{\rho\rho})}{|g|^{1/2}}\partial_\rho+\jap{\rho}^{-2}\ringsDelta,
\end{split}
\end{align*}
where $\ringsDelta$ denotes the Laplacian on the round sphere of radius one. We will also use the notation $\sDelta=\jap{\rho}^{-2}\ringsDelta$, and similarly for $\ringsnabla$ and $\snabla$.  We use $L^p_x$ to denote the $L^p$ norm on constant $t$ hypersurfaces with respect to the volume form induced by $g$. We will also use the notation $L^p_tL^q_x[t_1,t_2]$ to indicate that the $L^p_t$ norm is calculated over the time interval $[t_1,t_2]$. In this section we use the notation
%%%%%%%
%%%%%%%
\begin{align*}
\begin{split}
&\|\phi\|_{LE[t_1,t_2]}^2:=\|\rho\jap{\rho}^{-\frac{1}{2}(3+\alpha)}\partial_t\phi\|_{L^2_{t,x}[t_1,t_2]}^2+\|\rho\jap{\rho}^{-\frac{1}{2}(5+\alpha)}\phi\|_{L^2_{t,x}[t_1,t_2]}^2+\|\rho\jap{\rho}^{-\frac{3}{2}}\snabla\phi\|_{L^2_{t,x}[t_1,t_2]}^2\\
&\phantom{\|\phi\|_{LE[t_1,t_2]}^2:=}+\|\jap{\rho}^{-\frac{1+\alpha}{2}}\partial_\rho\phi\|_{L^2_{t,x}[t_1,t_2]}^2,\\
 &\|f\|_{LE^\ast[t_1,t_2]}^2:=\|\jap{\rho}^{\frac{1+\alpha}{2}}f\|_{L^2_{t,x}[t_1,t_2]}^2,
\end{split}
\end{align*}
where $0<\alpha\ll1$ is a fixed small constant. Note the degeneracy at $\rho=0$ for the non-radial derivatives in the definition of the $LE$ norm. This degeneracy appears due to the presence of the trapped sphere at $\rho=0$. The energy norm is defined as
%%%%%%%
%%%%%%%
\begin{align*}
\begin{split}
\|\phi\|_E^2=\|\partial_t\phi\|_{L^2_x}^2+\|\partial_x\phi\|_{L^2_x}^2,
\end{split}
\end{align*}
where by definition $\|\partial_x\phi\|_{L^2_x}^2=\int (g^{-1})^{ij}(\partial_i\phi)(\partial_j\phi)\sqrt{|g|}\ud\omega\ud\rho$ (here and in the remainder of this section $i,j$ run over $1,\dots,n$). The $L^2_x$ pairing with respect to $\sqrt{|g|}$ will be denoted by~$\angles{\cdot}{\cdot}$. We also use the following notation to denote the $L^2_{t,x}$ norm over the region $[t_1,t_2]\times\{|\rho|\leq1\}$:
%%%%%%%
%%%%%%%
\begin{align*}
\begin{split}
\|\phi\|_{L^2_{t,x}([t_1,t_2]\times\{|\rho|\leq1\})}.
\end{split}
\end{align*}

We assume  that $\Delta+V$ has spectrum consisting of the absolutely continuous part $(-\infty,0]$ and possibly a finite number of eigenvalues at zero and in $(0,\infty)$. In particular, in dimensions $n\geq 5$ there cannot be any threshold resonances\footnote{By a threshold resonance, we mean a function $\varphi$ belonging to $\calL \calE_{\sharp}^{1}$ (defined in the proof of Proposition~\ref{prop:LEDproduct}) but not to $L^2_x$, such that $(\Delta+V)\varphi=0$. Threshold resonances do not exist in our applications, because the strong spatial decay of $V$ and the difference between the coefficients of $\Delta$ and the Euclidean Laplacian imply that a threshold resonance $\varphi$ must decay at the same rate as the Newtonian potential near each asymptotically flat end (i.e., $\abs{\varphi}  \lesssim |\rho|^{-(n-2)}$ as $\rho \to \pm \infty$). Indeed, this can be seen by writing $\varphi \chi_{>1}(\pm \rho) = \Delta_{\bbR^{n}}^{-1} ( (\Delta - \Delta_{\bbR^{n}} + V) (\varphi \chi_{>1}(\pm \rho)) + [\chi_{>1}(\pm \rho), \Delta] \varphi )$, and noting that the expression inside $\Delta_{\bbR^{n}}^{-1}$ is decaying sufficiently fast thanks to the assumptions on $g$ and $V$, as well as $\varphi \in \calL \calE_{\sharp}^{1}$ (here $\chi_{>1}(s)$ is a smooth function that equals $1$ for $s > 1$ and $0$ for $s < \frac{1}{2}$). In dimensions $n\geq 5$ this implies that $\varphi\in L^2_x$, and gives a contradiction.}. The eigenfunctions are assumed to satisfy the decay rate $\jap{\rho}^{-n+1}$ or faster. We use $\bbP_c$ to denote the projection onto the continuous spectrum of $\Delta+V$ (with respect to the volume form induced by $g$). These conditions are easily verified when $V=|\secondff|^2$ is the squared norm of the second fundamental form of the standard embedding of $\barcalC$ in $\bbR^{n+1}$; see Section~\ref{subsec:Riemcat} and \cite{F-CS,Tam-Zhou}. In the following proposition we use the notations $\|\jap{\partial_t}h\|_X=\|h\|_X+\|\partial_th\|_X$ and $\|h\|_{X+Y}=\inf_{h=h_1+h_2}(\|h_1\|_X+\|h_2\|_Y)$.
%%%%%%%%%%%%
%%%%%%%%%%%
\begin{proposition}\label{prop:LEDproduct}
Suppose $\psi$ satisfies \eqref{eq:LEDproduct1} on a time interval $[t_1,t_2]$. Then for any small constant $\delta>0$ the following estimates hold
%%%%%%
%%%%%%
\begin{align*}
\begin{split}
&\sup_{[t_1,t_2]}\|\bbP_c\psi(t)\|_E\lesssim \|\bbP_c\psi(t_1)\|_{E}+\|\bbP_c G \|_{L^1_tL^2_x[t_1,t_2]},\\
&\sup_{[t_1,t_2]}\|\bbP_c\psi(t)\|_E\lesssim \|\bbP_c\psi(t_1)\|_{E}+C_\delta\|\bbP_cG\|_{LE^\ast[t_1,t_2]}+\delta(\|\bbP_c\psi\|_{LE[t_1,t_2]}+\|\partial_t\bbP_c\psi\|_{L^2_{t,x}([t_1,t_2]\times\{\rho\leq1\})}),\\
&\sup_{[t_1,t_2]}\|\bbP_c\psi(t)\|_E\lesssim \|\bbP_c\psi(t_1)\|_{E}+C_\delta\|\jap{\partial_t}\bbP_cG\|_{LE^\ast([t_1,t_2])}+\|\bbP_cG\|_{L^\infty_t L^2_x[t_1,t_2]}+\delta\|\bbP_c\psi\|_{LE[t_1,t_2]},\\
&\| \bbP_{c} \psi \|_{LE[t_1,t_2]}\lesssim \|\bbP_{c} \psi(t_{1})\|_E + \| \bbP_{c} \psi(t_{2})\|_E+\|\bbP_{c} G\|_{LE^\ast[t_1,t_2]+L^1_tL^2_x[t_1,t_2]}.
\end{split}
\end{align*}
\end{proposition}
%%%%%%%%%%%
%%%%%%%%%%%
%%%%%%%%%%%%
%%%%%%%%%%%%
\begin{remark}\label{rem:LEDproduct1}
Combining the first and fourth estimates in the proposition we get
 %%%%%%%
%%%%%%%
\begin{align*}
\begin{split}
\sup_{[t_1,t_2]}\|\bbP_c\psi(t)\|_E+\|\bbP_c\psi\|_{LE[t_1,t_2]}\lesssim \|\bbP_c\psi(t_1)\|_{E}+\|\bbP_cG\|_{L^1_tL^2_x[t_1,t_2]}.
\end{split}
\end{align*}
This is sufficient for most applications, but in a few instances we need to estimate $\bbP_c G$ in the $LE^\ast$ norm. For this we need to use the second and third energy estimates in the statement of the proposition. Note that the last term on the right-hand side of the second estimate cannot be absorbed by the $LE$ norm due to the degeneracy of the $LE$ norm at $\rho=0$.
\end{remark}
%%%%%%%%%%%%
%%%%%%%%%%%%
\begin{proof}[Proof of Proposition~\ref{prop:LEDproduct}]
%%%%%%%%%%%
%%%%%%%%%%%
Let $\phi=\bbP_c\psi$ and $f=\bbP_cG$. The first two estimates are standard energy estimates which follow from multiplying the equation by $\partial_t\phi$. The fact that $\phi$ is orthogonal to the eigenfunctions of $\Delta+V$ guarantees that the flux on $\{t=\mathrm{constant}\}$ bounds the energy. For the third estimate we use the same arguments as in the previous estimates but with the following modifications in how we treat the contribution of $f\partial_t\phi$. First we write $f\partial_t\phi=\chi f\partial_t\phi+(1-\chi) f \partial_t\phi$, with $\chi\equiv\chi(\rho)$ a cutoff to the region $\{|\rho|\leq 1\}$ where the $LE$ norm is degenerate. The contribution of $(1-\chi) f \partial_t\phi$ is treated as in the first estimate. The integral of $\chi f\partial_t\phi$ is treated by integration by parts in $t$ and observing that using Hardy and Poincar\'e type inequalities (see for instance the proof of Lemma~\ref{lem:LEDhighfreq1} in Section~\ref{sec:LED}),
%%%%%%%
%%%%%%%
\begin{align*}
\begin{split}
\int_{t_1}^{t_2}\int_{\{|\rho|\leq1\}}\chi(\partial_tf)\phi\udx \udt\leq \delta\|\phi\|_{LE[t_1,t_2]}^2+C_\delta\|\partial_tf\|_{LE^\ast[t_1,t_2]}^2.
\end{split}
\end{align*}
Note that in view of the localization $\chi$, the boundary terms from the integration by parts in $t$ can again be estimated using a Hardy estimate.

Finally, we prove the last estimate in the statement of the proposition. First, observe that this estimate is a trivial consequence of the first estimate in the proposition if $t_{2} - t_{1} \leq 1$; hence it suffices to only consider $t_{2} - t_{1} \geq 1$. Next, we claim that it suffices to establish the following:
\begin{align}\label{eq:LEDproducttemp-sch}
\| \bbP_{c} \psi \|_{LE^{1}_{\sharp}(\bbR)} \lesssim \| \bbP_{c} (\Box + V) \psi \|_{LE^{\ast}_{\sharp}(\bbR)} \qquad \hbox{ for all } \psi \in \calS_{t, x},
\end{align}
where
\begin{align*}
	\| \phi \|_{LE^{1}_{\sharp}[t_{1}, t_{2}]} &:= \| \rho \partial_{t, x} \phi \|_{L^{2}_{t, x}([t_{1}, t_{2}] \times \{ | \rho | < 1 \})} 
		+ \| \partial_{\rho} \phi \|_{L^{2}_{t, x}([t_{1}, t_{2}]  \times \{ | \rho | < 1 \})} \\
	&\phantom{:=} + \sup_{j \geq 1} \left(\| |\rho|^{-\frac{1}{2}} \partial_{t,x} \phi \|_{L^{2}_{t, x}([t_{1}, t_{2}]  \times \{ 2^{j-1} \leq | \rho | < 2^{j} \})} + \| |\rho|^{-\frac{3}{2}} \phi \|_{L^{2}_{t, x}([t_{1}, t_{2}]  \times \{ 2^{j-1} \leq | \rho | < 2^{j} \})} \right), \\
	\| f \|_{LE^{\ast}_{\sharp}[t_{1}, t_{2}]} &:= \| f \|_{L^{2}_{t, x}([t_{1}, t_{2}]  \times \{ | \rho | < 1 \})} 
		+ \sum_{j \geq 1} \| |\rho|^{\frac{1}{2}} f \|_{L^{2}_{t, x}([t_{1}, t_{2}]  \times \{ 2^{j-1} \leq | \rho | < 2^{j} \})},
\end{align*}
and $\calS_{t, x}$ is the class of smooth functions $\psi$ on the product Lorentzian catenoid such that $\partial^{\alpha} \psi$ for every multi-index $\alpha$ is decaying faster than $(1 + t^{2} + \rho^{2})^{-\frac{N}{2}}$ for any $N \geq 0$. We remark that $LE^{1}_{\sharp} [t_{1}, t_{2}] \hookrightarrow LE[t_{1}, t_{2}]$ and $LE^{\ast}[t_{1}, t_{2}] \hookrightarrow LE^{\ast}_{\sharp}[t_{1}, t_{2}]$. The refined spaces $LE^{1}_{\sharp} [t_{1}, t_{2}]$ and $LE^{\ast}_{\sharp}[t_{1}, t_{2}]$ will be used (only) within the present proof in order to borrow sharp estimates proved in \cite{MST, MMT1}.

The key step in the proof of the claim is the construction of a function $\tilde{\phi}$ on $[t_{1}, t_{2}] \times \barcalC$ such that 
\begin{equation*}
	\| \tilde{\phi} \|_{LE^{1}_{\sharp}[t_{1}, t_{2}]} + \| (\Box + V) \tilde{\phi} - f \|_{LE^{\ast}_{\sharp}[t_{1}, t_{2}]} \lesssim \| \phi(t_{1}) \|_{E} + \| \phi(t_{2}) \|_{E} + \| f \|_{LE_{\sharp}^{\ast}[t_{1}, t_{2}]+L^{1}_{t} L^{2}_{x}[t_{1}, t_{2}]},
\end{equation*}
and $(\tilde{\phi}, \partial_{t} \tilde{\phi})(t_{i}) = (\phi, \partial_{t} \phi)(t_{i})$ for $i = 1, 2$. Then, extending $\phi - \bbP_{c} \tilde{\phi}$ by zero outside of $[t_{1}, t_{2}]$ and applying \eqref{eq:LEDproducttemp-sch} (after a straightforward approximation procedure), we would arrive at the last estimate in the statement of the proposition.

The desired $\tilde{\phi}$ can be constructed by following the extension procedure in \cite[Section~7.1]{MST}. Since our setting is slightly different from \cite{MST} (especially since $\phi$ is not compactly supported in space in general due to the presence of $\bbP_{c}$), we sketch the main steps of the construction of $\tilde{\phi}$. First, we introduce a large radius $R_{\mathrm{out}}$ (to be fixed below) and a smooth cutoff function $\chi_{\mathrm{out}}(\rho)$ that equals $1$ when $\abs{\rho} > 100 (t_{2} - t_{1}) + 2 R_{\mathrm{out}} $ and equals $0$ when $\abs{\rho} < 50 (t_{2} - t_{1}) + R_{\mathrm{out}}$, and define $\tilde{\phi}_{\mathrm{out}}$ to be the solution to the Cauchy problem
\begin{equation*}
	(\Box + V) \tilde{\phi}_{\mathrm{out}} = \chi_{\mathrm{out}}(\rho) f, \quad (\tilde{\phi}_{\mathrm{out}}, \partial_{t} \tilde{\phi}_{\mathrm{out}})(t_{1}) = \chi_{out}(\rho) (\phi, \partial_{t} \phi)(t_{1}).
\end{equation*}
Choosing $R_{\mathrm{out}}$ large enough so that $\Box + V$ has propagation speed $\leq 2$ in $\set{\abs{\rho} > R_{\mathrm{out}}}$, we have $\tilde{\phi}_{\mathrm{out}} = 0$ in, say, $[t_{1}, t_{2}] \times \set{\abs{\rho} < R_{\mathrm{out}} + 10(t_{2} - t_{1})}$, and $\tilde{\phi}_{\mathrm{out}} = \phi$ in, say, $[t_{1}, t_{2}] \times \set{\abs{\rho} > 2 R_{\mathrm{out}} + 200 (t_{2} - t_{1})}$ (recall that $(\Box+V) \phi = f$ in $[t_{1}, t_{2}] \times \barcalC$). Taking $R_{\mathrm{out}}$ even larger, we may harmlessly replace $\Box + V$ on $\tilde{\phi}_{\mathrm{out}}$ by a small perturbation of the Minkowski d'Alembertian as $\rho \to +\infty$ and $-\infty$, for which the sharp integrated local energy estimate in \cite[Theorem~1.7]{MST} applies. Thus, we have
\begin{equation*}
	\sup_{t \in [t_{1}, t_{2}]}\| \tilde{\phi}_{\mathrm{out}}(t) \|_{E} + \| \tilde{\phi}_{\mathrm{out}} \|_{LE^{1}_{\sharp}[t_{1}, t_{2}]}
	\lesssim \| \phi(t_{1}) \|_{E} + \| f \|_{LE_{\sharp}^{\ast}[t_{1}, t_{2}]+L^{1}_{t} L^{2}_{x}[t_{1}, t_{2}]}.
\end{equation*}

Next, we introduce $\chi_{\mathrm{in}}(\rho) := 1 - \chi_{\mathrm{out}}(\rho)$ and apply the procedure in \cite[Section~7.1, Eq.~(7.3) and below]{MST} to $\phi - \tilde{\phi}_{\mathrm{out}}$ and $(\Box + V)(\phi - \tilde{\phi}_{\mathrm{out}}) = \chi_{\mathrm{in}} f$. Observe that both $\phi -\tilde{\phi}_{\mathrm{out}}$ and $\chi_{\mathrm{in}} f$ are supported in $[t_{1}, t_{2}] \times \set{\abs{\rho} \lesssim_{R_{\mathrm{out}}} t_{2} - t_{1}}$ as is required in \cite{MST}, since $\phi = \tilde{\phi}_{\mathrm{out}}$ in $[t_{1}, t_{2}] \times \set{\abs{\rho} > 2 R_{\mathrm{out}} + 200 (t_{2} - t_{1})}$, $\chi_{\mathrm{in}}$ is supported in $[t_{1}, t_{2}] \times \set{\abs{\rho} \leq 2 R_{\mathrm{out}} + 100 (t_{2} - t_{1})}$, and $t_{2} - t_{1} \geq 1$. This produces a function $\tilde{\phi}_{\mathrm{in}}$ on $[t_{1}, t_{2}] \times \barcalC$ that satisfies
\begin{align*}
	& \| \tilde{\phi}_{\mathrm{in}} \|_{LE^{1}_{\sharp}[t_{1}, t_{2}]} + \| (\Box + V) \tilde{\phi}_{\mathrm{in}} - \chi_{\mathrm{in}}(\rho) f \|_{LE^{\ast}_{\sharp}[t_{1}, t_{2}]} \\
	&\lesssim \| \phi - \tilde{\phi}_{\mathrm{out}}(t_{1}) \|_{E} + \| \phi - \tilde{\phi}_{\mathrm{out}}(t_{2}) \|_{E} + \| \chi_{\mathrm{in}}(\rho) f \|_{LE_{\sharp}^{\ast}[t_{1}, t_{2}]+L^{1}_{t} L^{2}_{x}[t_{1}, t_{2}]},
\end{align*}
and $(\tilde{\phi}_{\mathrm{in}}, \partial_{t} \tilde{\phi}_{\mathrm{in}})(t_{i}) = (\phi - \tilde{\phi}_{\mathrm{out}}, \partial_{t} (\phi - \tilde{\phi}_{\mathrm{out}}))(t_{i})$ for $i = 1, 2$. Finally, defining $\tilde{\phi} := \tilde{\phi}_{\mathrm{out}} + \tilde{\phi}_{\mathrm{in}}$, it may be readily checked that it satisfies the desired properties stated above.

To establish \eqref{eq:LEDproducttemp-sch}, we start by proving
%%%%%%
%%%%%%
\begin{align}\label{eq:LEDproducttemp1}
\begin{split}
\|\phi\|_{LE_{\sharp}^{1}(\bbR)}\lesssim \|(\Box + V) \phi\|_{LE^\ast_{\sharp}(\bbR)}+\|\phi\|_{L^2_{t,x}(\bbR \times B)} \qquad \hbox{ for all } \phi \in \calS_{t, x}
\end{split}
\end{align}
where $B$ denotes a large compact spatial region. By \cite[Proposition~3.2]{MST} applied to each asymptotically flat end (i.e., $\rho \to +\infty$ and $-\infty$), we obtain
\begin{align} \label{eq:LEDproducttemp1-extr}
	\| \phi \|_{LE_{\sharp, > R}^{1} [t_{1}, t_{2}]} \lesssim \sup_{i=1,2} \| \phi(t_{i}) \|_{E_{>R}} + \| (\Box + V) \phi \|_{LE^{\ast}_{\sharp, >R}[t_{1}, t_{2}]} + R^{-\frac{3}{2}} \| \phi \|_{L^{2}_{t, x}([t_{1}, t_{2}] \times \{ \frac{R}{4} < \abs{\rho} < 2R \} )},
\end{align}
for a sufficiently large $R \in 2^{\bbZ}$, where
\begin{align*}
	\| \phi \|_{LE^{1}_{\sharp, > R}[t_{1}, t_{2}]} &:= \sup_{j \geq \log_{2} R} \left(\| |\rho|^{-\frac{1}{2}} \partial_{t,x} \phi \|_{L^{2}_{t, x}([t_{1}, t_{2}]  \times \{ 2^{j-1} \leq | \rho | < 2^{j} \})} + \| |\rho|^{-\frac{3}{2}} \phi \|_{L^{2}_{t, x}([t_{1}, t_{2}]  \times \{ 2^{j-1} \leq | \rho | < 2^{j} \})} \right), \\
	\| f \|_{LE^{\ast}_{\sharp, > R}[t_{1}, t_{2}]} &:=\sum_{j \geq \log_{2} R} \| |\rho|^{\frac{1}{2}} f \|_{L^{2}_{t, x}([t_{1}, t_{2}]  \times \{ 2^{j-1} \leq | \rho | < 2^{j} \})}, \\
	\| \phi(t) \|_{E_{>R}} &:= \| \partial_{t} \phi(t) \|_{L^{2}(\{ |\rho| > R\})} + \| \partial_{x} \phi(t) \|_{L^{2}(\{ |\rho| > R\})}.
\end{align*}
We supplement this bound with the following bound with non-sharp weights at spatial infinity, with a sharp control of $\phi$ in a bounded spatial region: 
\begin{align} \label{eq:LEDproducttemp1-intr}
	\| \phi \|_{LE [t_{1}, t_{2}]} \lesssim \sup_{i=1,2} \| \phi(t_{i}) \|_{E} + \| (\Box + V) \phi \|_{LE^{\ast}[t_{1}, t_{2}]} + \| \phi \|_{L^{2}_{t, x}([t_{1}, t_{2}] \times B)}.
\end{align}
Observe that the combination of \eqref{eq:LEDproducttemp1-extr} and \eqref{eq:LEDproducttemp1-intr} immediately gives \eqref{eq:LEDproducttemp1}.

To complete the proof of \eqref{eq:LEDproducttemp1}, it only remains to prove \eqref{eq:LEDproducttemp1-intr}. With $\partial_\rho^\ast$ denoting the formal adjoint of $\partial_\rho$ for the  pairing $\angles{\cdot}{\cdot}$, let
\begin{align*}
\begin{split}
Q:=\beta\partial_\rho-\partial_\rho^\ast\beta,\qquad \partial_\rho^\ast=-\partial_\rho-\frac{1}{2}(\partial_\rho\log|g|),
\end{split}
\end{align*}
where $\beta\equiv\beta(\rho)$ is to be chosen. The main positive commutator identity is
%%%%%%%
%%%%%%%
\begin{align}\label{eq:comm1}
\begin{split}
\angles{(\Box+V) \phi}{Q\phi}= -\partial_t\angles{\partial_t\phi}{Q\phi}+\overbrace{ \angles{\Delta \phi}{Q\phi} }^{-\frac{1}{2}\angles{[Q,\Delta]\phi}{\phi} }-\frac{1}{2}\overbrace{\angles{[Q,V]\phi}{\phi}}^{2\angles{(\beta\partial_\rho V)\phi}{\phi}}.
\end{split}
\end{align}
Note that
\[
 Q \phi = (\beta \partial_\rho - \partial_\rho^\ast \beta) \phi = 2 \beta (\partial_\rho \phi) - (\partial_\rho^\ast \beta) \phi.
\]
Then we have 
\begin{equation} \label{equ:DeltaQ}
 \begin{aligned}
  \langle - \Delta \phi, Q \phi \rangle &= \langle \partial_\rho^\ast g^{\rho \rho} \partial_\rho \phi, Q \phi \rangle - \langle \jap{\rho}^{-2} \ringsDelta \phi, Q \phi \rangle \\
  &= \langle \partial_\rho^\ast g^{\rho \rho} \partial_\rho \phi, 2 \beta (\partial_\rho \phi) - (\partial_\rho^\ast \beta) \phi \rangle - \langle \jap{\rho}^{-2} \ringsDelta \phi, 2 \beta (\partial_\rho \phi) - (\partial_\rho^\ast \beta) \phi \rangle.
 \end{aligned}
\end{equation}
For the first term on the right-hand side we integrate by parts repeatedly and rearrange to find that
\begin{align*}
  \langle \partial_\rho^\ast g^{\rho \rho} \partial_\rho \phi, 2 \beta (\partial_\rho \phi) - (\partial_\rho^\ast \beta) \phi \rangle &= \langle g^{\rho \rho} \partial_\rho \phi, (2 (\partial_\rho \beta) - (\partial_\rho^\ast \beta)) \partial_\rho \phi \rangle + \langle g^{\rho \rho} \partial_\rho \phi, 2 \beta (\partial_\rho^2 \phi) \rangle \\
  &\quad- \langle g^{\rho \rho} \partial_\rho \phi, (\partial_\rho \partial_\rho^\ast \beta) \phi \rangle \\
  &= \langle \partial_\rho \phi, ( 2 g^{\rho \rho} (\partial_\rho \beta) - g^{\rho \rho} (\partial_\rho^\ast \beta) + \partial_\rho^\ast ( g^{\rho \rho} \beta ) ) \partial_\rho \phi \rangle \\
  &\quad- \frac{1}{2} \langle g^{\rho \rho} (\partial_\rho \partial_\rho^\ast \beta), \partial_\rho (\phi^2) \rangle \\
  &= \langle ( 2 g^{\rho \rho} (\partial_\rho \beta) - (\partial_\rho g^{\rho \rho}) \beta ) \partial_\rho \phi, \partial_\rho \phi \rangle - \frac{1}{2} \langle (\partial_\rho^\ast g^{\rho \rho} \partial_\rho \partial_\rho^\ast \beta) \phi, \phi \rangle \\
  &= \langle (2 g^{\rho \rho} (\partial_\rho \beta) - (\partial_\rho g^{\rho \rho}) \beta) \partial_\rho \phi, \partial_\rho \phi \rangle + \frac{1}{2} \langle (\Delta \partial_\rho^\ast \beta) \phi, \phi \rangle. 
\end{align*}
Note that $\partial_\rho g^{\rho\rho}=-2\rho^{-3}\jap{\rho}^{-(2n-1)}(\jap{\rho}^{2(n-1)}-\jap{\rho}^2-(n-2)\rho^2)$ has the opposite sign from $\rho$, so if $\beta$ is odd and increasing, then the coefficient of $(\partial_\rho\phi)^2$ is positive. For the second term on the right-hand side of~\eqref{equ:DeltaQ} we obtain that
\begin{align*}
 - \langle \jap{\rho}^{-2} \ringsDelta \phi, 2 \beta (\partial_\rho \phi) - (\partial_\rho^\ast \beta) \phi \rangle &= \langle \jap{\rho}^{-2} \ringsnabla \phi, 2 \beta \ringsnabla \partial_\rho \phi \rangle - \langle \jap{\rho}^{-2} (\partial_\rho^\ast \beta) \ringsnabla \phi, \ringsnabla \phi \rangle \\
 &= \langle \partial_\rho^\ast ( \jap{\rho}^{-2} \beta ) \ringsnabla \phi, \ringsnabla \phi \rangle - \langle \jap{\rho}^{-2} (\partial_\rho^\ast \beta) \ringsnabla \phi, \ringsnabla \phi \rangle \\
 &= 2 \langle {\textstyle \frac{\rho}{\jap{\rho}^4} } \beta \ringsnabla \phi, \ringsnabla \phi \rangle \\
 &= 2 \langle {\textstyle \frac{\rho}{\jap{\rho}^2} } \beta \snabla \phi, \snabla \phi \rangle.
\end{align*}
Putting the above identities together, we arrive at 
\begin{align}\label{eq:comm-hard}
 \langle - \Delta \phi, Q \phi \rangle = \langle (2 g^{\rho \rho} (\partial_\rho \beta) - (\partial_\rho g^{\rho \rho}) \beta) \partial_\rho \phi, \partial_\rho \phi \rangle + 2 \langle {\textstyle \frac{\rho}{\jap{\rho}^2} } \beta \snabla \phi, \snabla \phi \rangle + \frac{1}{2} \langle (\Delta \partial_\rho^\ast \beta) \phi, \phi \rangle.
\end{align}
In view of \eqref{eq:comm-hard}, for \eqref{eq:comm1} we can get  control of the spatial part of the $LE$ norm in \eqref{eq:LEDproducttemp1} with the choice
%%%%%%%
%%%%%%%
\begin{align*}
\begin{split}
\beta:=\chi \rho + K(1-\chi) \beta_E,
\end{split}
\end{align*}
where $K$ is a large constant, $\beta_E=\frac{\rho}{\jap{\rho}}-\alpha\frac{\rho}{\jap{\rho}^{1+\alpha}}$ (an odd increasing function; note that the multiplier $\beta_{E}(r) \partial_{r} - \partial_{r}^{\ast} \beta_{E}(r)$ leads to the proof of LED on the flat (Euclidean) space), and $\chi$ is a radial cut-off to the region $\rho\in[-R,R]$, for some large R, which decays to zero monotonically outside of $[-R,R]$ and is zero on $[-2R,2R]$. The point is that in this way $\rho \beta$ is positive and if $K$ is sufficiently large
%%%%%%%
%%%%%%%
\begin{align*}
\begin{split}
\partial_\rho\beta= -(\partial_\rho\chi)(K \beta_E-\rho)+K(1-\chi)\partial_\rho\beta_E+\chi
\end{split}
\end{align*}
is also positive in the gluing region. Note that for the angular derivative we get a degeneracy of order two at $\rho=0$. In order to control $\partial_t \phi$ in the LED estimate we also use the multiplier identity
%%%%%%%
%%%%%%%
\begin{align}\label{eq:comm2}
\begin{split}
\angles{\Box\phi}{\gamma\phi}=\angles{\gamma\partial_t\phi}{\partial_t\phi}-\angles{\gamma\nabla_g\phi}{\nabla_g\phi}-\partial_t\angles{\gamma\partial_t\phi}{\phi}+\frac{1}{2}\angles{(\Delta\gamma)\phi}{\phi}.
\end{split}
\end{align}
For \eqref{eq:comm2}, the choice $\gamma(\rho)=\frac{\rho^2}{\jap{\rho}^{3+\alpha}}$ gives
%%%%%%%
%%%%%%%
\begin{align*}
\begin{split}
\angles{\frac{\rho^2}{\jap{\rho}^{3+\alpha}}\partial_t\phi}{\partial_t\phi}=&\angles{f-V\phi}{\frac{\rho^2}{\jap{\rho}^{3+\alpha}}\phi}+\angles{\frac{\rho^2}{\jap{\rho}^{3+\alpha}}\nabla_g\phi}{\nabla_g\phi}-\frac{1}{2}\angles{(\Delta\frac{\rho^2}{\jap{\rho}^{3+\alpha}})\phi}{\phi}+\partial_t\angles{\frac{\rho^2}{\jap{\rho}^{3+\alpha}}\partial_t\phi}{\phi}.
\end{split}
\end{align*}
Adding this identity to a suitable multiple of \eqref{eq:comm-hard}, we arrive at \eqref{eq:LEDproducttemp1-intr}. Note that here for the contribution of the last term on the right-hand side of \eqref{eq:comm1} we have used the spatial decay of $\beta\partial_\rho V$. This allows us to bound the error by the last term on the right-hand side of~\eqref{eq:LEDproducttemp1-intr} plus a small term that can be absorbed on the left-hand side of~\eqref{eq:LEDproducttemp1-intr}.

The proof of \eqref{eq:LEDproducttemp-sch} is concluded by removing the $L^{2}_{t, x}$ error in the spatial region $B$ in \eqref{eq:LEDproducttemp1} by a contradiction argument akin to \cite[Proof of Theorem~1.19]{MMT1}. In fact, our situation is a bit simper compared to \cite{MMT1} thanks to the fact that we may directly take the Fourier transform in time for $\phi \in  \bbP_{c} \calS_{t, x}$. More precisely, by the Plancherel theorem in the variable $t$, \eqref{eq:LEDproducttemp-sch} would follow once we establish
\begin{equation} \label{eq:LEDproducttemp-sch-FT}
	\| \bbP_{c} v \|_{\calL \calE^{1}_{\sharp (\tau)}} \lesssim \| (\Delta + V + \tau^{2}) \bbP_{c} v \|_{\calL \calE_{\sharp}^{\ast}} \qquad \hbox{ for all } v \in \calS_{x}, \, \tau \in \bbR,
\end{equation}
where
\begin{align*}
	\| v \|_{\calL \calE^{1}_{\sharp (\tau)}} &:= \| \rho (\tau v, \partial_{x} v) \|_{L^{2}(\{ | \rho | < 1 \})} 
		+ \| \partial_{\rho} v \|_{L^{2}_{x}(\{ | \rho | < 1 \})} \\
	&\phantom{:=} + \sup_{j \geq 1} \left(\| |\rho|^{-\frac{1}{2}} (\tau v, \partial_{x} v)  \|_{L^{2}_{t, x}(\{ 2^{j-1} \leq | \rho | < 2^{j} \})} + \| |\rho|^{-\frac{3}{2}} v\|_{L^{2}_{t, x}(\{ 2^{j-1} \leq | \rho | < 2^{j} \})} \right), \\
	\| f \|_{\calL \calE^{\ast}_{\sharp}} &:= \| f \|_{L^{2}_{x}(\{ | \rho | < 1 \})} 
		+ \sum_{j \geq 1} \| |\rho|^{\frac{1}{2}} f \|_{L^{2}_{x}(\{ 2^{j-1} \leq | \rho | < 2^{j} \})},
\end{align*}
and $\calS_{x}$ is the class of smooth functions $v$ on $\barcalC$ such that $\partial_{x}^{\alpha} \psi$ for every multi-index $\alpha$ is decaying faster than $(1 + \rho^{2})^{-\frac{N}{2}}$ for any $N \geq 0$. Furthermore, fixing $\chi \in C^{\infty}_{c}(\bbR)$, applying \eqref{eq:LEDproducttemp1} to $\phi = a^{\frac{1}{2}} \chi(a t) e^{i \tau t} \bbP_{c} v$ with arbitrarily small $a > 0$ (where we choose $t_{1}$, $t_{2}$ to be outside the support of $\chi(a t)$ so that no boundary terms appear), we obtain
\begin{equation} \label{eq:LEDproducttemp1-FT}
	\| \bbP_{c} v \|_{\calL \calE^{1}_{\sharp (\tau)}} \lesssim \| (\Delta + V + \tau^{2}) \bbP_{c} v \|_{\calL \calE_{\sharp}^{\ast}} + \| \bbP_{c} v \|_{L^{2}_{x}(B)} \qquad \hbox{ for all } v \in \calS_{x}, \, \tau \in \bbR.
\end{equation}

To summarize, our task now is to establish \eqref{eq:LEDproducttemp-sch-FT} using \eqref{eq:LEDproducttemp1-FT}. First, let us show that there exists $\tau_{\ast} > 0$ such that \eqref{eq:LEDproducttemp1-FT} implies \eqref{eq:LEDproducttemp-sch-FT} for $\abs{\tau} \geq \tau_{\ast}$. While this implication does not immediately follow due to the presence of the (degenerate) weight $\rho$ in the region $\{ | \rho | <  1\}$ in the definition of $\calL \calE^{1}_{\sharp (\tau)}$, it may be remedied as follows.  We claim that
\begin{align*}
	\| \bbP_{c} v \|_{L^{2}(B)} &\leq \| \bbP_{c} v \|_{L^{2}(\{ |\rho| < 1 \}}  + C | \tau |^{-1} \| \bbP_{c} v \|_{\calL \calE^{1}_{\sharp(\tau)}}, \\
	\| \bbP_{c} v \|_{L^{2}(\{ |\rho| < 1 \}} &\leq C | \tau |^{-\frac{1}{2}} \| \bbP_{c} v \|_{\calL \calE^{1}_{\sharp(\tau)}}.
\end{align*}
Combining these two inequalities, we may absorb $\| \bbP_{c} v \|_{L^{2}(B)}$ in \eqref{eq:LEDproducttemp1-FT} into the LHS for sufficiently large $| \tau |$, and therefore \eqref{eq:LEDproducttemp-sch-FT} would follow. The first inequality is an obvious consequence of the definition of $\calL \calE^{1}_{\sharp(\tau)}$. To prove the second inequality, we begin with the following calculus inequality (here, $\chi \in C^{\infty}_{c}(\bbR)$):
\begin{equation*}
	\int |\bbP_{c} v|^{2} \chi^{2}(\rho) \, \ud \rho \leq 2\left(\int |\partial_{\rho} \bbP_{c} v|^{2} \chi^{2}(\rho) \, \ud \rho\right)^{\frac{1}{2}} \left(\int |\rho \bbP_{c} v|^{2} \chi^{2}(\rho) \, \ud \rho\right)^{\frac{1}{2}} + 2 \int \rho |\bbP_{c} v|^{2} |\chi (\rho) \partial_{\rho} \chi (\rho)| \, \ud \rho
\end{equation*}
proved using the identity $1 = \partial_{\rho} \rho$ and integration by parts. Let us choose $\chi$ to be a nonnegative cutoff function such that $\chi = 1$ on $\{ |\rho| < 1 \}$ and $\chi = 0$ outside $\{ |\rho| < 2 \}$. Integrating also over the angular variables, the LHS controls $\| \bbP_{c} v \|_{L^{2}(\{|\rho| < 1 \}}$, whereas the RHS is bounded by
\begin{equation*}
	C (\| \partial_{\rho} \bbP_{c} v \|_{L^{2}(\{ |\rho| < 2 \})} + \| \bbP_{c} v \|_{L^{2}(\{ 1 < |\rho| < 2 \})} ) \| \rho \bbP_{c} v \chi \|_{L^{2}}
	\leq C |\tau|^{-1} \| \bbP_{c} v \|_{\calL \calE^{1}_{\sharp (\tau)}}^{2},
\end{equation*}
which proves the desired inequality.

To proceed, assume, for the purpose of contradiction, that \eqref{eq:LEDproducttemp-sch-FT} fails for some $\tau \in (-\tau_{\ast}, \tau_{\ast})$. Then there exist sequences $\{ v_{n} \} \subseteq \calS_{x}$ and $\{\tau_{n} \} \subseteq (-\tau_{\ast}, \tau_{\ast})$, as well as a sequence $\{ \epsilon_{n} \}$ of positive numbers converging to zero, such that
\begin{equation*}
	\| (\Delta + V + \tau_{n}^{2}) \bbP_{c} v_{n} \|_{\calL \calE_{\sharp}^{\ast}} \leq \epsilon_{n} \| \bbP_{c} v_{n} \|_{\calL \calE^{1}_{\sharp (\tau)}}.
\end{equation*}
Since $\{ \tau_{n} \}$ is bounded, by passing to a subsequence, we may assume that $\tau_{n} \to \tau \in \bbR$. Moreover, normalizing $v_{n}$ so that $\| \bbP_{c} v_{n} \|_{\calL \calE^{1}_{\sharp (\tau)}} = 1$ and passing to a subsequence again, we may assume that $\bbP_{c} v_{n} \rightharpoonup v$ with respect to the weak-$\ast$ topology in $\calL \calE^{1}_{\sharp}$. Note that $\bbP_{c} v = v$ since $\bbP_{c}$ is a bounded linear map on $\calL \calE^{1}_{\sharp}$. Moreover, since $\| (\Delta + V + \tau_{n}^{2}) \bbP_{c} v_{n} \|_{\calL \calE_{\sharp}^{\ast}}$ then vanishes as $n \to \infty$, using elliptic regularity and passing to another subsequence if necessary, we may assume that $\bbP_{c} v_{n} \to v$ strongly in $L^{2}_{loc}$. Hence we have produced a function $v$ satisfying
\begin{equation*}
	v \in \calL \calE^{1}_{\sharp}, \quad \bbP_{c} v = v, \quad (\Delta + V + \tau^{2}) v = 0, \quad \| v \|_{L^{2}(B)} \gtrsim 1,
\end{equation*}
where the last lower bound follows from \eqref{eq:LEDproducttemp1-FT}.

At this point, we see that $\tau = 0$ is impossible due to the nonexistence of a resonance at $0$ and $\bbP_{c} v = v$. To rule out $\tau \neq 0$, we follow the argument in \cite[Section~4.1, Step~10]{MMT1} to first establish the following uniform bound for the sequence $v_{n}$ and $j$ sufficiently large (independent of $n$): 
\begin{align*}
	&\| |\rho|^{-\frac{1}{2}} (i \partial_{\rho} \pm \tau_{n}) \bbP_{c} v_{n} \|_{L^{2}(D_{j})}^{2}\\
	&\lesssim \sum_{k=0}^{\infty} 2^{\delta \min\{k-j, 0\}}  \Big[ \| \langle \rho \rangle^{\frac{1}{2}}(\Delta + V - \tau_{n}^{2}) v_{n} \|_{L^{2}(D_{k})} \| \langle \rho \rangle^{-\frac{1}{2}} (\bbP_{c} v_{n}, \partial_{x} \bbP_{c}v_{n}) \|_{L^{2}(D_{k})} \\
	&\phantom{\lesssim \sum_{k=0}^{\infty} 2^{\delta \min\{k-j, 0\}}  \Big[} + \kappa_{k} \| \langle \rho \rangle^{-\frac{1}{2}} (\bbP_{c}v_{n}, \partial_{x} \bbP_{c}v_{n}) \|_{L^{2}(D_{k})}^{2} \Big]
\end{align*}
for some $\delta > 0$ and an absolutely summable sequence $\kappa_{k}$ (derived from asymptotic flatness of $\Delta + V$). Here, $D_{k} = \{ 2^{k-1} \leq |\rho| < 2^{k} \}$ for $k \geq 1$ and $D_{0} = \{ |\rho| < 1 \}$. Observe that, unlike \cite{MMT1}, the spectral parameter $\tau^{2}$ is purely real so the argument applies for both signs $\pm$ (more specifically, using the notation in \cite[Section~4.1, Step~10]{MMT1}, we may take $\epsilon = 0$ and apply the positive commutator argument in each asymptotically flat end to bound $(D_{r} \pm \tau^{\frac{1}{2}}) u$). Taking the limit $n \to \infty$, we obtain
\begin{align*}
\| |\rho|^{-\frac{1}{2}} (i \partial_{\rho} \pm \tau) v \|_{L^{2}(D_{j})}^{2} \lesssim \sum_{k=0}^{\infty} 2^{\delta \min\{k-j, 0\}} \kappa_{k} \to 0 \quad \hbox{ as } j \to \infty.
\end{align*}
Since this bound holds for both signs and $\tau \neq 0$, we obtain $\lim_{j \to \infty} \| |\rho|^{-\frac{1}{2}} v \|_{L^{2}(D_{j})}= 0$. By elliptic regularity and equation $(\Delta + V +\tau^{2}) v = 0$, it furthermore follows that $\lim_{j \to \infty} \| |\rho|^{-\frac{1}{2}} \partial_{x} v \|_{L^{2}(D_{j})}= 0$ (cf.~\cite[Equation (4.9)]{MMT1}). At this point, we can carry out the argument in \cite[Section~4.1, Step~8]{MMT1} to conclude that $v \in L^{2}(\{ \rho > 1\}) \cap L^{2}(\{ \rho < -1\}) $. But then, by Kato's theorem on the absence of embedded eigenvalues for the asymptotically flat self-adjoint operator $\Delta + V$ (or, more directly, a unique continuation argument for $\Delta + V$ at each asymptotically flat end as in \cite{KoTa}), it follows that $v = 0$, which is also impossible. This completes the proof by contradiction of \eqref{eq:LEDproducttemp-sch-FT}. \qedhere
\end{proof}
%%%%%%%%%%%
%%%%%%%%%%%

\begin{remark}
The fourth inequality in Proposition~\ref{prop:LEDproduct} is a version of \emph{two point (integrated) local energy decay} estimate as in  \cite[Eq.~(1.6)]{MST} in the presence of an eigenvalue at zero. As is remarked in \cite{MST}, this inequality holds regardless of the existence of eigenvalues of $\Delta + V$ outside of $(-\infty, 0]$. This is evident from the last part of our proof, where the element $v$ constructed by the contradiction argument automatically solves $(\Delta  + V + \tau^{2}) v = 0$ for some $\tau \in \bbR$.
\end{remark}

%%%%%%%%%%%%%%%%%%%%%%
%%%%%%%%%%%%%%%%%%%%%%
%%%%%%%%%%%%%%%%%%%%%%
\section{Interior}\label{sec:interior}
%%%%%%%%%%%%%%%%%%%%%%
%%%%%%%%%%%%%%%%%%%%%%
%%%%%%%%%%%%%%%%%%%%%%
%%%%%%%%%%%%%%%%%%%%%%

In this section we first define the momentum variable $\dotpsi$ and we derive, in vector form, the first-order equation for $\vecpsi$. Then we introduce suitable orthogonality conditions for the modulation parameters $\ell(t)$ and $\xi(t)$, and derive the equations satisfied by $\dot{\vecwp}=(\dotell,\dotxi-\ell)^\intercal$.
Finally, we enact a further decomposition of the perturbation $\vec{\psi}$ to take into account the unstable mode of the linearized operator.
Throughout this section, all error estimates are to be understood to hold under the bootstrap assumptions~\eqref{eq:a+trap}--\eqref{eq:Tjphienergyb2}. Moreover, we refer to Subsection~\ref{subsubsec:calOnotation} for the definition of the $\calO$ notation.

%%%%%%%%%%%%%%%%%%%
%%%%%%%%%%%%%%%%%%%
%%%%%%%%%%%%%%%%%%%
\subsection{Setup}\label{subsec:setupint1}
%%%%%%%%%%%%%%%%%%%
%%%%%%%%%%%%%%%%%%%
%%%%%%%%%%%%%%%%%%%
Our parametrization for the profile in the flat region $\calC_{\mathrm{flat}} := \{  \sigma_{\text{temp}}(X)\geq  X^0+ \delta_1\}$ is
\begin{equation} \label{equ:definition_profile_flat_region}
 \Psi_\wp(t, \rho, \omega) = \bigl( t, \xi + \gamma^{-1} P_\ell F(\rho, \omega) + P_\ell^\perp F(\rho, \omega) \bigr), 
\end{equation}
where $\xi = \xi(t)$ and $\ell = \ell(t)$ are our time-dependent modulation parameters.

We denote the normal to $\Sigma_t$, in the case where the parameters are treated as fixed, by
\begin{equation*}
 n_\wp := \Lambda_{-\ell} \begin{pmatrix} 0 \\ \nu \end{pmatrix} = \begin{pmatrix} \gamma \ell \cdot \nu \\ A_{-\ell} \nu \end{pmatrix},
\end{equation*}
where $\nu$ is the geometric normal to the Riemannian catenoid $\underline{\calC}$.
Then $\widetilde{N}_{int} = (0, A_{-\ell} \nu)$ is the normal to $\Sigma_t$ viewed as a subspace of $\bsUpsigma_t$. In the interior we define $N$ to be parallel to $\wtilN_{int}$ and such that $\bfeta(n_\wp, N) = 1$, that is,
\begin{equation*}
 N := \begin{pmatrix} 0 \\ |A_{-\ell} \nu|^{-2} A_{-\ell} \nu \end{pmatrix}.
\end{equation*}
Moreover, we write
\begin{equation*}
 W := N - n_\wp.
\end{equation*}
We introduce the scalar perturbation $\psi$ in the interior via the decomposition 
\begin{equation*}
 \Phi = \Psi_\wp + \psi N.
\end{equation*}

Next, we introduce the metric components
\begin{equation*}
 g_{\mu \nu} := \bfeta \bigl( \partial_\mu \Phi, \partial_\nu \Phi \bigr), \quad \quad k_{\mu \nu} := \bfeta \bigl( \partial_\mu \Psi_\wp, \partial_\nu \Psi_\wp \bigr), \quad \quad 0 \leq \mu, \nu \leq n,
\end{equation*}
and the Lagrangian density 
\begin{equation*}
 \calL := \sqrt{|g|} = \sqrt{-\det(g)}.
\end{equation*}
We also introduce 
\begin{align*}
 \Psi_{0;\wp} &= \begin{pmatrix} 1 \\ \ell \end{pmatrix}, \quad \quad \Psi_{j;\wp} = \begin{pmatrix} 0 \\ \gamma^{-1} P_\ell \pj F + P_\ell^\perp \pj F \end{pmatrix}, \quad \quad 1 \leq j \leq n,\\
 h_{\mu \nu} &:= \bfeta \bigl( \Psi_{\mu; \wp}, \Psi_{\nu; \wp} \bigr), \quad \quad 0 \leq \mu, \nu \leq n
\end{align*}
Sometimes we will use the notation 
\begin{align*}
 \partial_\mu \Psi_\wp \Big|_{\substack{\dot{\ell} = 0 \\\dot{\xi} = \ell}} := \Psi_{\mu; \wp}, \quad \quad 0 \leq \mu \leq n.
\end{align*}
{\bf In what follows, we still view $\ell$ as time-dependent in the expressions for $\Psi_{\mu; \wp}$ and $h_{\mu\nu}$.}

\begin{remark} \label{rem:HVMCj_for_Psinu}
 A parametrization of the Lorentzian catenoid boosted by a fixed $\ell \in \bbR^{n+1}$, $|\ell| < 1$, and translated by $a\in \bbR^{n+1}$ (with $\xi$ corresponding to $a+t\ell$) is given by
 \begin{equation*}
  \Gamma_{a,\ell}(t, \rho, \omega) := \bigl( t, a+t\ell + \gamma^{-1} P_\ell F(\rho, \omega) + P_\ell^\perp F(\rho, \omega) \bigr).
 \end{equation*}
 Then $\Gamma_{a,\ell}$ satisfies the HVMC equation
 \begin{equation*}
  \frac{1}{\sqrt{|\kappa|}} \partial_\mu \bigl( \sqrt{|\kappa|} \kappa^{\mu \nu} \pnu \Gamma_{a,\ell} \bigr) = 0
 \end{equation*}
 with $\kappa_{\mu \nu} = \bfeta( \partial_\mu \Gamma_{a,\ell}, \partial_\nu \Gamma_{a,\ell} )$.
 Since the metric coefficients $\kappa_{\mu \nu}$ are time-independent and since by direct computation $\pt \pnu \Gamma_{a,\ell} = 0$ for fixed $\ell$, we in fact have 
 \begin{equation} \label{equ:remark_HVMCj_pj_kappa}
  \pj \bigl( \sqrt{|\kappa|} \kappa^{j \nu} \pnu \Gamma_{a,\ell} \bigr) = 0.
 \end{equation}
 Formally, the expressions for $h_{\mu \nu}$ and $\Psi_{\nu; \wp}$ are the same as those for $\kappa_{\mu \nu}$ and $\pnu \Gamma_{a,\ell}$, only that the parameter $\ell$ is considered time-dependent in $h_{\mu \nu}$ and $\Psi_{\nu; \wp}$, while $\ell$ is considered time-independent in $\kappa_{\mu \nu}$ and $\pnu \Gamma_{a,\ell}$. Since in the preceding equation~\eqref{equ:remark_HVMCj_pj_kappa} only spatial derivatives act from outside, we can conclude that
 \begin{equation*}
 \begin{aligned}
  \pj \bigl( \sqrt{|h|} h^{j\nu} \Psi_{\nu;\wp} \bigr) = 0.
 \end{aligned}
\end{equation*}
\end{remark}

\begin{remark} \label{rem:hinv0j_smallness}
 We point out that the coefficients $(h^{-1})^{0j} = \calO(\ell)$, $1 \leq j \leq n$ are small of order~$\calO(\ell)$. This follows from observing that $h_{00}= -1 + |\ell|^2$, $h_{0j} = \calO(|\ell|)$, and $h_{ij} = \partial_i F \cdot \partial_j F + \calO(|\ell|^2)$. For this reason terms with $(h^{-1})^{0j} $ will not contribute to leading order to the equations.
\end{remark}

\subsection{Definition of the Momentum Variable} \label{susec:momentumvariable}

The HVMC equation $\Box_g \Phi = 0$ for $\Phi = \Psi_\wp + \psi N$ gives rise to a second-order quasilinear wave equation for the scalar $\psi$. In order to be able to formulate first-order modulation equations later on, our first goal is to arrive at a suitable formulation of an associated system of linearized first-order equations for 
\begin{equation}
 \vec{\psi} := \begin{pmatrix} \psi \\ \dot{\psi} \end{pmatrix}
\end{equation}
for a suitably defined momentum variable $\dot{\psi}$. 
We begin by motivating our definition of $\dot{\psi}$. A good starting point is to examine the Euler-Lagrange equation for $\psi$ given by (here and below the notation $\frac{\delta}{\delta\psi}$, etc, simply mean the partial derivative of $\calL$ with respect to the corresponding variable, and not the functional derivative)
\begin{equation} \label{equ:EL_for_phi}
 \partial_t \biggl( \frac{\delta \calL}{\delta (\partial_t \psi)} \biggr) + \partial_j \biggl( \frac{\delta \calL}{\delta (\partial_j \psi)} \biggr) = \frac{\delta \calL}{\delta \psi}.
\end{equation}
It suggests that the quantity $\dot{\psi}$ should be part of 
\begin{equation*}
 \frac{\delta \calL}{\delta (\partial_t \psi)} = \frac12 \sqrt{|g|} (g^{-1})^{\mu \nu} \frac{\delta g_{\mu \nu}}{\delta (\partial_t \psi)}. 
\end{equation*}
Since $g_{\mu \nu} = \bfeta(\partial_\mu \Phi, \partial_\nu \Phi)$ and $\pmu \Phi = \pmu \Psi_\wp + (\pmu \psi) N + \psi \pmu N$, we have 
\begin{equation*}
 \frac{\delta g_{\mu \nu}}{\delta (\partial_t \psi)} = \delta_{\nu 0} \bfeta(\partial_\mu \Phi, N) + \delta_{\mu 0} \bfeta(N, \partial_\nu \Phi),
\end{equation*}
and therefore
\begin{equation*}
 \frac{\delta \calL}{\delta (\partial_t \psi)} = \bfeta( \sqrt{|g|} (g^{-1})^{0 \nu} \partial_\nu \Phi, N).
\end{equation*}
Next, we determine the precise expression for $\frac{\delta \calL}{\delta (\partial_t \psi)}$ up to quadratic (and higher-order) terms in the perturbation $\psi$ and its derivatives $\pmu \psi$. To this end we first record the following expansions
\begin{equation} \label{equ:expansion_ginversemunu}
\begin{aligned}
 (g^{-1})^{\mu \nu} &= (k^{-1})^{\mu \nu} - (k^{-1})^{\mu \alpha} \bfeta( \partial_\alpha \Psi_\wp, (\partial_\beta \psi) N + \psi \partial_\beta N ) (k^{-1})^{\beta \nu} \\
 &\quad \quad - (k^{-1})^{\mu \alpha} \bfeta( (\partial_\alpha \psi) N + \psi (\partial_\alpha N), \partial_\beta \Psi_\wp) \bigr) (k^{-1})^{\beta \nu} + \calO\bigl( (\psi, \pmu \psi)^2 \bigr) 
\end{aligned}
\end{equation}
as well as 
\begin{equation} \label{equ:expansion_sqrtmodg}
\begin{aligned}
 \sqrt{|g|} &= \sqrt{|k|} + \sqrt{|k|} (k^{-1})^{\alpha \beta} \bfeta( \partial_\alpha \Psi_\wp, (\partial_\beta \psi) N + \psi \partial_\beta N ) + \calO\bigl( (\psi, \pmu \psi)^2 \bigr),
\end{aligned}
\end{equation}
Expanding up to terms that are at least quadratic in the perturbation $\psi$ and its derivatives $\pmu \psi$, we have
\begin{equation} \label{equ:expand_delta_L_delta_ptphi}
 \begin{aligned}
  \frac{\delta \calL}{\delta (\partial_t \psi)}  &= \bfeta( \sqrt{|g|} (g^{-1})^{0 \nu} \partial_\nu \Phi, N) \\
  &= \sqrt{|k|} (k^{-1})^{0 \nu} \bfeta( \partial_\nu \Psi_\wp, N) + \bfeta \biggl( \frac{\delta(\sqrt{|g|} (g^{-1})^{0 \nu} \partial_\nu \Phi)}{\delta (\pmu \psi)} \bigg|_{\psi=0} \pmu \psi, N \biggr) \\
  &\quad + \bfeta \biggl( \frac{\delta(\sqrt{|g|} (g^{-1})^{0 \nu} \partial_\nu \Phi)}{\delta \psi} \bigg|_{\psi=0} \psi, N \biggr) + \calE_1,
 \end{aligned}
\end{equation}
where the remainder term $\calE_1$ satisfies $\calE_1 = \calO\bigl( (\psi, \pmu \psi)^2 \bigr)$. 
Here and below, $\vert_{\psi=0}$ means that both $\psi$ and its derivatives are set to zero.
By direct computation, using the expansions \eqref{equ:expansion_ginversemunu} and \eqref{equ:expansion_sqrtmodg}, we obtain
\begin{equation} \label{equ:deltaLdeltaphi_munu_1storder}
\begin{aligned}
 \frac{\delta (\sqrt{|g|} (g^{-1})^{\mu \nu} \pnu \Phi)}{\delta \psi} \bigg|_{\psi=0} &= \sqrt{|k|} (k^{-1})^{\mu \nu} \pnu N + \sqrt{|k|} (k^{-1})^{\mu \nu} (k^{-1})^{\alpha \beta} \bfeta(\partial_\alpha \Psi_\wp, \partial_\beta N) \pnu \Psi_\wp \\
 &\quad - \sqrt{|k|} (k^{-1})^{\mu \alpha} (k^{-1})^{\beta \nu} \bfeta(\partial_\alpha \Psi_\wp, \partial_\beta N) \pnu \Psi_\wp \\
 &\quad - \sqrt{|k|} (k^{-1})^{\mu \alpha} (k^{-1})^{\beta \nu} \bfeta(\partial_\alpha N, \partial_\beta \Psi_\wp) \pnu \Psi_\wp
\end{aligned}
\end{equation}
and 
\begin{equation}  \label{equ:deltaLdeltaphi_0nu_1storder}
\begin{aligned}
 \frac{\delta (\sqrt{|g|} (g^{-1})^{0\nu} \pnu \Phi)}{\delta(\pmu \psi)} \bigg|_{\psi=0} &= \sqrt{|k|} (k^{-1})^{0\mu} N + \sqrt{|k|} (k^{-1})^{\alpha \mu} (k^{-1})^{0\nu} \bfeta(\partial_\alpha \Psi_\wp, N) \pnu \Psi_\wp \\
 &\quad - \sqrt{|k|} (k^{-1})^{0\alpha} (k^{-1})^{\mu\nu} \bfeta(\partial_\alpha \Psi_\wp, N) \pnu \Psi_\wp \\
 &\quad - \sqrt{|k|} (k^{-1})^{0\mu} (k^{-1})^{\beta \nu} \bfeta(N, \partial_\beta \Psi_\wp) \pnu \Psi_\wp.
\end{aligned}
\end{equation}
For later use we also record that
\begin{equation}  \label{equ:deltaLdeltaphi_jnu_1storder}
\begin{aligned}
 \frac{\delta (\sqrt{|g|} (g^{-1})^{j\nu} \pnu \Phi)}{\delta (\pmu \psi)} \bigg|_{\psi=0} &= \sqrt{|k|} (k^{-1})^{j\mu} N + \sqrt{|k|} (k^{-1})^{\alpha \mu} (k^{-1})^{j\nu} \bfeta(\partial_\alpha \Psi_\wp, N) \pnu \Psi_\wp \\
 &\quad - \sqrt{|k|} (k^{-1})^{j\alpha} (k^{-1})^{\mu \nu} \bfeta(\partial_\alpha \Psi_\wp, N) \pnu \Psi_\wp \\
 &\quad - \sqrt{|k|} (k^{-1})^{j\mu} (k^{-1})^{\beta \nu} \bfeta(N, \partial_\beta \Psi_\wp) \pnu \Psi_\wp.
\end{aligned}
\end{equation}

We proceed with a further examination of the terms on the right-hand side of~\eqref{equ:expand_delta_L_delta_ptphi}.
Using that $\bfeta(\pj \Psi_\wp, N) = 0$, we can rewrite the first term on the RHS of \eqref{equ:expand_delta_L_delta_ptphi} as 
\begin{equation}
 \begin{aligned}
  \sqrt{|k|} (k^{-1})^{0 \nu} \bfeta( \partial_\nu \Psi_\wp, N) &= \sqrt{|k|} (k^{-1})^{00} \bfeta( \partial_t \Psi_\wp, N) \\
  &= \sqrt{|h|} (h^{-1})^{00} \bfeta \bigl( (1,\ell), N \bigr) \\
  &\quad + \sqrt{|h|} (h^{-1})^{00} \bfeta \bigl( \pt \Psi_\wp - (1,\ell), N \bigr) \\
  &\quad + \bfeta \Bigl( \bigl( \sqrt{|k|} (k^{-1})^{00} - \sqrt{|h|} (h^{-1})^{00} \bigr) \partial_t \Psi_\wp, N \Bigr) \\
  &= \sqrt{|h|} (h^{-1})^{00} \bfeta \bigl( (1,\ell), N \bigr) \\
  &\quad + \sqrt{|h|} (h^{-1})^{00} \bfeta \Bigl( \bigl( \dot{\ell} \cdot \nabla_\ell + (\dot{\xi}-\ell) \cdot \nabla_\xi \bigr) \Psi_\wp, N \Bigr) + \calE_2,
 \end{aligned}
\end{equation}
where the remainder term satisfies
\begin{equation*}
 \begin{aligned}
  \calE_2 := \bfeta \Bigl( \bigl( \sqrt{|k|} (k^{-1})^{00} - \sqrt{|h|} (h^{-1})^{00} \bigr) \partial_t \Psi_\wp, N \Bigr) = \calO\bigl(\dotwp^2, \dotwp \ell\bigr).
 \end{aligned}
\end{equation*}
Using~\eqref{equ:deltaLdeltaphi_0nu_1storder} we obtain that the second term on the right-hand side of~\eqref{equ:expand_delta_L_delta_ptphi} is explicitly given by
\begin{equation} \label{equ:expand_delta_L_delta_ptphi_2nd_term}
 \begin{aligned}
  &\bfeta \biggl( \frac{\delta(\sqrt{|g|} (g^{-1})^{0 \nu} \partial_\nu \Phi)}{\delta (\pmu \psi)} \bigg|_{\psi=0} \pmu \psi, N \biggr) \\
  &\quad = \sqrt{|k|} (k^{-1})^{0\mu} (\partial_\mu \psi) \bfeta\bigl(N, N - (k^{-1})^{\alpha \beta} \bfeta( N, \partial_\alpha \Psi_\wp ) \partial_\beta \Psi_\wp \bigr).
 \end{aligned}
\end{equation}
Now recall that if $\Psi_\wp$ were a genuine maximal embedding, then $\{ n_\wp, \pnu \Psi_\wp\}$ would form a basis of the ambient space with $(k^{-1})^{\alpha \beta} \bfeta( N, \partial_\alpha \Psi_\wp ) \partial_\beta \Psi_\wp$ denoting the tangential part of $N$. Since $\bfeta(N, n_\wp) = 1$ by construction, the right-hand side of the preceding identity~\eqref{equ:expand_delta_L_delta_ptphi_2nd_term} would then just read $\sqrt{|k|} (k^{-1})^{0\mu} \pmu \psi$. To quantify the difference, we write 
\begin{equation*}
 \begin{aligned}
  (k^{-1})^{\alpha \beta} \bfeta( N, \partial_\alpha \Psi_\wp ) \partial_\beta \Psi_\wp = (h^{-1})^{\alpha \beta} \bfeta( N, \Psi_{\alpha; \wp} ) \Psi_{\beta; \wp} + \calE_{3,1} + \calE_{3,2}
 \end{aligned}
\end{equation*}
with remainder terms
\begin{align*}
 \calE_{3,1} &:= \bigl( (k^{-1})^{\alpha \beta} - (h^{-1})^{\alpha \beta} \bigr) \bfeta( N, \partial_\alpha \Psi_\wp ) \partial_\beta \Psi_\wp, \\
 \calE_{3,2} &:= (h^{-1})^{\alpha \beta} \bfeta\bigl( N, \partial_\alpha \Psi_\wp - \Psi_{\alpha;\wp} \bigr)  \partial_\beta \Psi_\wp + (h^{-1})^{\alpha \beta} \bfeta( N, \Psi_{\alpha;\wp} ) \bigl( \partial_\beta \Psi_\wp - \Psi_{\beta; \wp} \bigr)
\end{align*}
of the form $\calO(\dotwp)$.
Correspondingly, we have
\begin{equation*}
 \begin{aligned}
  \bfeta \biggl( \frac{\delta(\sqrt{|g|} (g^{-1})^{0 \nu} \partial_\nu \Phi)}{\delta (\pmu \psi)} \bigg|_{\psi=0} \pmu \psi, N \biggr) &= \sqrt{|k|} (k^{-1})^{0\mu} \partial_\mu \psi + \calE_3
 \end{aligned}
\end{equation*}
with remainder term
\begin{equation*}
 \begin{aligned}
  \calE_3 := \sqrt{|k|} (k^{-1})^{0\mu} (\partial_\mu \psi) \bfeta\bigl(N, \calE_{3,1} + \calE_{3,2} \bigr) = \calO\bigl( (\partial \psi) \dotwp\bigr).
 \end{aligned}
\end{equation*}

In order to obtain a more favorable structure of the linearized equation for $\vec{\psi}$, it is preferable to remove the linear terms involving $\psi$ from the above candidate for $\dot{\psi}$, that is, in \eqref{equ:expand_delta_L_delta_ptphi}. 
But in order to make sure no second order derivatives of the parameters appear when we calculate $\partial_t\dotpsi$, we replace $k$ by $h$ when subtracting off the linear contributions of $\psi$ in \eqref{equ:expand_delta_L_delta_ptphi}. More specifically, to avoid appearances of $\dotell$, we also only subtract off those expressions where we think of $\ell$ as being time-independent (so for instance $\partial_t N$ should be thought of as zero and we use $\Psi_{\nu; \wp}$ instead of $\partial_\nu \Psi_\wp$). We introduce the following more succinct notation for these terms
\begin{equation} \label{equ:definition_B}
\begin{aligned}
 B \psi &:= \bfeta \biggl( \frac{\delta (\sqrt{|g|} (g^{-1})^{0 \nu} \partial_\nu \Phi)}{\delta \psi} \bigg|_{\substack{\psi = 0 \\ k = h}} \psi, N \biggr) \\
 &= \sqrt{|h|} (h^{-1})^{0j}  \bfeta( \partial_j N, N ) \psi \\
 &\quad - \sqrt{|h|} (h^{-1})^{0 \alpha} (h^{-1})^{j \nu} \bfeta( \Psi_{\alpha; \wp}, \partial_j N ) \bfeta( \Psi_{\nu; \wp}, N ) \psi \\
 &\quad - \sqrt{|h|} (h^{-1})^{0 j} (h^{-1})^{\beta \nu} \bfeta( \partial_j N, \Psi_{\beta; \wp} ) \bfeta( \Psi_{\nu; \wp}, N ) \psi \\ 
 &\quad + \sqrt{|h|} (h^{-1})^{\alpha j} (h^{-1})^{0\nu} \bfeta( \Psi_{\alpha; \wp}, \partial_j N )  \bfeta( \Psi_{\nu; \wp}, N) \psi.
\end{aligned}
\end{equation}
We denote the resulting difference by
\begin{equation}
 \calE_4 := \bfeta \biggl( \frac{\delta (\sqrt{|g|} (g^{-1})^{0 \nu} \partial_\nu \Phi)}{\delta \psi} \bigg|_{\substack{\psi = 0}} \psi, N \biggr) - \bfeta \biggl( \frac{\delta (\sqrt{|g|} (g^{-1})^{0 \nu} \partial_\nu \Phi)}{\delta \psi} \bigg|_{\substack{\psi = 0 \\ k = h}} \psi, N \biggr),
\end{equation}
which satisfies $\calE_4 = \calO( \dotwp \psi )$.

\begin{remark}
The notation $\frac{\delta (\cdot)}{\delta \psi} \big|_{\substack{\psi = 0 \\ k = h}}$ shall indicate that we compute $\frac{\delta}{\delta \psi}$, while the $t$-dependence of $\xi$ and $\ell$ is frozen, i.e., we replace $\dot{\xi}$ by $\ell$ as well as $\dot{\ell}$ by $0$ and we use $\Psi_{\nu; \wp}$ instead of $\partial_\nu \Psi_\wp$ This means that $\partial_t (B \psi)$ will involve at most one $t$-derivative of $\xi$ or $\ell$. Moreover, in those terms $\dot{\xi}$ or $\dot{\ell}$ are always multiplied by $\psi$.
\end{remark}

We arrive at the following definition
\begin{equation}
 \begin{aligned}
  \dot{\psi} := \bfeta \bigl( \sqrt{|g|} (g^{-1})^{0 \nu} \partial_\nu \Phi, N \bigr) - \bfeta \bigl( \sqrt{|h|} (h^{-1})^{00} (1, \ell), N \bigr) - B \psi. 
 \end{aligned}
\end{equation}

\subsection{Relation between $\dot{\psi}$ and $\pt \psi$}

Next, we record the relation between $\dot{\psi}$ and $\pt \psi$.
From the preceding we obtain
\begin{equation*}
 \begin{aligned}
  \dot{\psi} = \sqrt{|h|} (h^{-1})^{00} \bfeta \Bigl( \bigl( \dot{\ell} \cdot \nabla_\ell + (\dot{\xi}-\ell) \cdot \nabla_\xi \bigr) \Psi_\wp, N \Bigr) + \sqrt{|k|} (k^{-1})^{0\mu} \pmu \psi + \calE_1 + \ldots + \calE_4.
 \end{aligned}
\end{equation*}
Solving for $\pt \psi$ yields 
\begin{equation*}
 \begin{aligned}
  \pt \psi &= \frac{1}{\sqrt{|k|} (k^{-1})^{00}} \dot{\psi} - \frac{\sqrt{|k|} (k^{-1})^{0j}}{\sqrt{|k|} (k^{-1})^{00}} \pj \psi \\
  &\quad - \frac{\sqrt{|h|} (h^{-1})^{00}}{\sqrt{|k|} (k^{-1})^{00}} \bfeta \Bigl( \bigl( \dot{\ell} \cdot \nabla_\ell + (\dot{\xi}-\ell) \cdot \nabla_\xi \bigr) \Psi_\wp, N \Bigr) - \frac{1}{\sqrt{|k|} (k^{-1})^{00}} \bigl( \calE_1 + \ldots + \calE_4 \bigr).
 \end{aligned}
\end{equation*}
Upon rewriting the prefactors in terms of $h$, which accrues further errors, we arrive at the relation 
\begin{equation} \label{equ:relation_ptphi_dotphi}
 \begin{aligned}
  \pt \psi &= \frac{1}{\sqrt{|h|} (h^{-1})^{00}} \dot{\psi} - \frac{(h^{-1})^{0j}}{(h^{-1})^{00}} \pj \psi - \bfeta \Bigl( \bigl( \dot{\ell} \cdot \nabla_\ell + (\dot{\xi}-\ell) \cdot \nabla_\xi \bigr) \Psi_\wp, N \Bigr) + f, 
 \end{aligned}
\end{equation}
where 
\begin{equation*}
 \begin{aligned}
  f &:= - \frac{1}{\sqrt{|k|} (k^{-1})^{00}} \bigl( \calE_1 + \ldots + \calE_4 \bigr) + \biggl( \frac{1}{\sqrt{|k|} (k^{-1})^{00}} - \frac{1}{\sqrt{|h|} (h^{-1})^{00}} \biggr) \dot{\psi} \\
  &\quad - \biggl( \frac{\sqrt{|k|} (k^{-1})^{0j}}{\sqrt{|k|} (k^{-1})^{00}} - \frac{\sqrt{|h|} (h^{-1})^{0j}}{\sqrt{|h|} (h^{-1})^{00}} \biggr) \pj \psi \\
  &\quad + \biggl(- \frac{\sqrt{|h|} (h^{-1})^{00}}{\sqrt{|k|} (k^{-1})^{00}} + 1\biggr) \bfeta \Bigl( \bigl( \dot{\ell} \cdot \nabla_\ell + (\dot{\xi}-\ell) \cdot \nabla_\xi \bigr) \Psi_\wp, N \Bigr).
 \end{aligned}
\end{equation*}

\begin{remark}
 Note that the term $f$ still contains $\partial_t \psi$ terms, but those come with additional smallness. Correspondingly, under suitable smallness assumptions we can use the implicit function theorem to solve for $\pt \psi$, as we will do further below.
\end{remark}

\subsection{Computation of $\pt \dot{\psi}$}

Next, we compute the time derivative of $\dot{\psi}$,
\begin{equation} \label{equ:pt_dotphi}
\begin{aligned}
 \pt \dot{\psi} &= \bfeta \Bigl( \pt \bigl( \sqrt{|g|} (g^{-1})^{0 \nu} \partial_\nu \Phi \bigr), N \Bigr) - \bfeta \Bigl( \pt \bigl( \sqrt{|h|} (h^{-1})^{00} (1, \ell) \bigr), N \Bigr) - \pt \bigl( B \psi \bigr) \\
 &\quad \quad + \bfeta \Bigl( \sqrt{|g|} (g^{-1})^{0 \nu} \partial_\nu \Phi - \sqrt{|h|} (h^{-1})^{00} (1, \ell), \pt N \Bigr).
\end{aligned}
\end{equation}
We rewrite the first term on the right-hand side of~\eqref{equ:pt_dotphi} as 
\begin{equation} \label{equ:pt_dotphi_1st_term}
 \begin{aligned}
  \bfeta \Bigl( \pt \bigl( \sqrt{|g|} (g^{-1})^{0 \nu} \partial_\nu \Phi \bigr), N \Bigr) &= \bfeta \Bigl( \pmu \bigl( \sqrt{|g|} (g^{-1})^{\mu \nu} \partial_\nu \Phi \bigr), N \Bigr) - \bfeta \Bigl( \pj \bigl( \sqrt{|g|} (g^{-1})^{j \nu} \partial_\nu \Phi \bigr), N \Bigr) \\
  &= - \bfeta \Bigl( \pj \bigl( \sqrt{|g|} (g^{-1})^{j \nu} \partial_\nu \Phi \bigr) - \pj \bigl( \sqrt{|k|} (k^{-1})^{j \nu} \partial_\nu \Psi_\wp \bigr) , N \Bigr) \\
  &\quad - \bfeta \Bigl( \pj \bigl( \sqrt{|k|} (k^{-1})^{j \nu} \partial_\nu \Psi_\wp \bigr) - \pj \bigl( \sqrt{|h|} (h^{-1})^{j \nu} \Psi_{\nu; \wp} \bigr), N \Bigr), 
 \end{aligned}
\end{equation}
where we used the HVMC equation $\pmu \bigl( \sqrt{|g|} (g^{-1})^{\mu \nu} \partial_\nu \Phi \bigr) = 0$ and that $\pj \bigl( \sqrt{|h|} (h^{-1})^{j \nu} \Psi_{\nu; \wp} \bigr) = 0$, see Remark~\ref{rem:HVMCj_for_Psinu}. Then to leading order the first term on the right-hand side of~\eqref{equ:pt_dotphi_1st_term} is given by
\begin{equation*}
 \begin{aligned}
  &- \bfeta \Bigl( \pj \bigl( \sqrt{|g|} (g^{-1})^{j \nu} \partial_\nu \Phi \bigr) - \pj \bigl( \sqrt{|k|} (k^{-1})^{j \nu} \partial_\nu \Psi_\wp \bigr) , N \Bigr) \\
  &= - \bfeta \biggl( \pj \Bigl( \frac{ \delta (\sqrt{|g|} (g^{-1})^{j \nu} \pnu \Phi) }{\delta (\pmu \psi)} \bigg|_{\substack{\psi = 0, \\ k = h}} \pmu \psi \Bigr), N \biggr) - \bfeta \biggl( \pj \Bigl( \frac{ \delta (\sqrt{|g|} (g^{-1})^{j \nu} \pnu \Phi) }{\delta  \psi} \bigg|_{\substack{\psi = 0, \\ k = h}} \psi \Bigr), N \biggr) + \calE_5 + \calE_6
 \end{aligned}
\end{equation*}
with remainder terms 
\begin{align*}
 \calE_5 &:= \bfeta \biggl( \pj \Bigl( \frac{ \delta (\sqrt{|g|} (g^{-1})^{j \nu} \pnu \Phi) }{\delta (\pmu \psi)} \bigg|_{\substack{\psi = 0, \\ k = h}} \pmu \psi \Bigr), N \biggr) + \bfeta \biggl( \pj \Bigl( \frac{ \delta (\sqrt{|g|} (g^{-1})^{j \nu} \pnu \Phi) }{\delta  \psi} \bigg|_{\substack{\psi = 0, \\ k = h}} \psi \Bigr), N \biggr) \\
 &\quad - \bfeta \biggl( \pj \Bigl( \frac{ \delta (\sqrt{|g|} (g^{-1})^{j \nu} \pnu \Phi) }{\delta (\pmu \psi)} \bigg|_{\psi = 0} \pmu \psi \Bigr), N \biggr) - \bfeta \biggl( \pj \Bigl( \frac{ \delta (\sqrt{|g|} (g^{-1})^{j \nu} \pnu \Phi) }{\delta  \psi} \bigg|_{\psi = 0} \psi \Bigr), N \biggr)
\end{align*}
and 
\begin{align*}
 \calE_6 &:= \bfeta \biggl( \pj \Bigl( \frac{ \delta (\sqrt{|g|} (g^{-1})^{j \nu} \pnu \Phi) }{\delta (\pmu \psi)} \bigg|_{\psi = 0} \pmu \psi \Bigr), N \biggr) + \bfeta \biggl( \pj \Bigl( \frac{ \delta (\sqrt{|g|} (g^{-1})^{j \nu} \pnu \Phi) }{\delta  \psi} \bigg|_{\psi = 0} \psi \Bigr), N \biggr) \\
 &\quad - \bfeta \Bigl( \pj \bigl( \sqrt{|g|} (g^{-1})^{j \nu} \partial_\nu \Phi \bigr) - \pj \bigl( \sqrt{|k|} (k^{-1})^{j \nu} \partial_\nu \Psi_\wp \bigr) , N \Bigr).
\end{align*}
We have 
\begin{equation*}
 \calE_5 = \calO\bigl( \dotwp (\partial \psi), \dotwp (\partial^2 \psi) \bigr) \quad \text{and} \quad \calE_6 = \calO\bigl( (\psi, \partial \psi, \partial^2 \psi)^2 \bigr).
\end{equation*}

The second term on the right-hand side of~\eqref{equ:pt_dotphi_1st_term} is an error term with
\begin{equation*}
 \calE_7 := - \bfeta \Bigl( \pj \bigl( \sqrt{|k|} (k^{-1})^{j \nu} \partial_\nu \Psi_\wp \bigr) - \pj \bigl( \sqrt{|h|} (h^{-1})^{j \nu} \Psi_{\nu; \wp} \bigr), N \Bigr) = \calO\bigl( \dotwp^2, \dotwp \ell \bigr).
\end{equation*}
To see that $\calE_7$ is a quadratic error, we use that $\sqrt{|k|}-\sqrt{|h|} = \calO(|\ell|^2)$ and $(k^{-1})^{ij}-(h^{-1})^{ij} = \calO(|\ell|^2)$. The latter observations follow from Remark~\ref{rem:hinv0j_smallness} and a Taylor expansion.
For the second term on the right-hand side of~\eqref{equ:pt_dotphi} we have 
\begin{equation*}
 \begin{aligned}
  - \bfeta \Bigl( \pt \bigl( \sqrt{|h|} (h^{-1})^{00} (1, \ell) \bigr), N \Bigr) &= - \bfeta \Bigl( \sqrt{|h|} (h^{-1})^{00} (0, \dot{\ell}), N \Bigr) + \calE_8
 \end{aligned}
\end{equation*}
with remainder term 
\begin{equation*}
 \begin{aligned}
  \calE_8 := - \pt \Bigl( \sqrt{|h|} (h^{-1})^{00} \Bigr) \bfeta\bigl( (1,\ell), N \bigr) = \calO\bigl( \dotwp \ell \bigr),
 \end{aligned}
\end{equation*}
and for the third term on the right-hand side of~\eqref{equ:pt_dotphi}, we compute
\begin{equation*}
 \begin{aligned}
  -\pt (B \psi) &= -\bfeta \biggl( \pt \Bigl( \frac{\delta (\sqrt{|g|} (g^{-1})^{0 \nu} \partial_\nu \Phi)}{\delta \psi} \bigg|_{\substack{\psi = 0 \\ k = h}} \psi \Bigr), N \biggr) + \calE_9
 \end{aligned}
\end{equation*}
with
\begin{equation*}
 \begin{aligned}
  \calE_9 := - \bfeta \biggl( \frac{\delta (\sqrt{|g|} (g^{-1})^{0 \nu} \partial_\nu \Phi)}{\delta \psi} \bigg|_{\substack{\psi = 0 \\ k = h}} \psi , \pt N \biggr) = \calO\bigl( \psi \dot{\ell} \bigr).
 \end{aligned}
\end{equation*}
Finally, the fourth term on the right-hand side of~\eqref{equ:pt_dotphi} is again an error term of the form 
\begin{equation*}
 \begin{aligned}
  \calE_{10} := \bfeta \Bigl( \sqrt{|g|} (g^{-1})^{0 \nu} \partial_\nu \Phi - \sqrt{|h|} (h^{-1})^{00} (1, \ell), \pt N \Bigr) = \calO\bigl( (\psi, \partial \psi) \dotwp \bigr) + \calO\bigl( \dotwp^2 \bigr).
 \end{aligned}
\end{equation*}

Combining the preceding expressions, we find that
\begin{equation} \label{equ:pt_dotphi_comp1} 
 \begin{aligned}
  \pt \dot{\psi} &= - \bfeta \biggl( \pj \Bigl( \frac{ \delta (\sqrt{|g|} (g^{-1})^{j \nu} \pnu \Phi) }{\delta (\pmu \psi)} \bigg|_{\substack{\psi = 0, \\ k = h}} \pmu \psi \Bigr), N \biggr) - \bfeta \biggl( \pmu \Bigl( \frac{ \delta (\sqrt{|g|} (g^{-1})^{\mu \nu} \pnu \Phi) }{\delta  \psi} \bigg|_{\substack{\psi = 0, \\ k = h}} \psi \Bigr), N \biggr) \\
  &\quad - \bfeta \Bigl( \sqrt{|h|} (h^{-1})^{00} (0, \dot{\ell}), N \Bigr) + \calE_6 + \ldots + \calE_{10} \\
  &= - \pj \biggl( \bfeta \Bigl( \frac{ \delta (\sqrt{|g|} (g^{-1})^{j \nu} \pnu \Phi) }{\delta (\pmu \psi)} \bigg|_{\substack{\psi = 0, \\ k = h}} \pmu \psi , N \Bigr) \biggr) + \bfeta \Bigl( \frac{ \delta (\sqrt{|g|} (g^{-1})^{j \nu} \pnu \Phi) }{\delta (\pmu \psi)} \bigg|_{\substack{\psi = 0, \\ k = h}} \pmu \psi , \pj N \Bigr) \\
  &\quad - \bfeta \biggl( \frac{ \delta (\sqrt{|g|} (g^{-1})^{\mu \nu} \pnu \Phi) }{\delta  \psi} \bigg|_{\substack{\psi = 0, \\ k = h}} \pmu \psi, N \biggr) - \bfeta \biggl( \pmu \Bigl( \frac{ \delta (\sqrt{|g|} (g^{-1})^{\mu \nu} \pnu \Phi) }{\delta  \psi} \bigg|_{\substack{\psi = 0, \\ k = h}} \Bigr) \psi, N \biggr) \\
  &\quad - \bfeta \Bigl( \sqrt{|h|} (h^{-1})^{00} (0, \dot{\ell}), N \Bigr) + \calE_5 + \ldots + \calE_{10}.
 \end{aligned}
\end{equation}

From \eqref{equ:deltaLdeltaphi_jnu_1storder} we obtain for the first term on the right-hand side of~\eqref{equ:pt_dotphi_comp1} that
\begin{align*}
 &- \pj \biggl( \bfeta \Bigl( \frac{ \delta (\sqrt{|g|} (g^{-1})^{j \nu} \pnu \Phi) }{\delta (\pmu \psi)} \bigg|_{\substack{\psi = 0, \\ k = h}} \pmu \psi , N \Bigr) \biggr) \\
 &= - \pj \Bigl( \sqrt{|h|} (h^{-1})^{j\mu} \bigl( \bfeta(N, N) - (h^{-1})^{\beta \nu} \bfeta(N_\ell, \Psi_{\beta; \wp}) \bfeta(\Psi_{\nu; \wp}, N) \bigr) \pmu \psi \Bigr) \\
 &= - \pj \Bigl( \sqrt{|h|} (h^{-1})^{j\mu} \pmu \psi \Bigr),
\end{align*}
where we used that 
\begin{equation*}
 \begin{aligned}
  \bfeta(N, N) - (h^{-1})^{\beta \nu} \bfeta(N, \Psi_{\beta; \wp}) \bfeta(\Psi_{\nu; \wp}, N) = \bfeta(N, n) = 1.
 \end{aligned}
\end{equation*}
The latter identity follows from the fact that $\{ \Psi_{\mu;\wp}, n_\wp \}$ forms a basis for the ambient space.
Using~\eqref{equ:deltaLdeltaphi_jnu_1storder} and~\eqref{equ:deltaLdeltaphi_munu_1storder}, it follows that the second and third terms on the right-hand side of~\eqref{equ:pt_dotphi_comp1} exactly cancel each other out. 
To evaluate the fourth term on the right-hand side of~\eqref{equ:pt_dotphi_comp1} we first need the following identity.

\begin{lemma}\label{lem:secondff}
 Let $\secondff$ be the second fundamental form of the embedding $\Psi_\wp|_{\substack{\dot{\ell}=0, \dot{\xi}=\ell}}$.
 Then we have 
 \begin{equation}
  \frac{1}{\sqrt{|h|}} \bfeta \biggl( \pj \Bigl( \frac{ \delta (\sqrt{|g|} (g^{-1})^{j \nu} \pnu \Phi) }{\delta  \psi} \bigg|_{\substack{\psi = 0, \\ k = h}} \Bigr), N \biggr) = |\secondff|^2.
 \end{equation}
\end{lemma}
\begin{proof}
First, we observe that
\begin{equation*}
 \begin{aligned}
  \frac{1}{\sqrt{|h|}} \bfeta \biggl( \pj \Bigl( \frac{ \delta (\sqrt{|g|} (g^{-1})^{j \nu} \pnu \Phi) }{\delta  \psi} \bigg|_{\substack{\psi = 0, \\ k = h}} \Bigr), N \biggr) &= \bfeta \biggl( \frac{\delta \bigl( \Box_g \Phi \bigr)}{\delta \psi} \bigg|_{\substack{\psi = 0, \\ k = h}} , N \biggr).
 \end{aligned}
\end{equation*}
We now split the evaluation of the right-hand side into several steps. In what follows, we write $N = n_\wp + W$, where $W$ denotes the tangential part of $N$. In what follows, $\nabla$ denotes the covariant derivative with respect to the embedding $\Phi$.

\medskip 

\noindent {\it Step 1: Computation of the part of $\Box_g \Phi$ that is linear in $\psi$.}
We begin by expanding
\begin{equation*}
 \begin{aligned}
  \Box_g \Phi &= (g^{-1})^{\mu \nu} \nabla_\mu \nabla_\nu \bigl( \Psi_\wp + \psi N \bigr) \\
  &= \Box_h \Psi_\wp + (\dot{g}^{-1})^{\mu \nu} \nabla_\mu \nabla_\nu \Psi_\wp - (h^{-1})^{\mu \nu} \dot{\Gamma}^\lambda_{\mu \nu} \partial_\lambda \Psi_\wp + (h^{-1})^{\mu \nu} \nabla_\mu \nabla_\nu \bigl( \psi N \bigr)  + \calO \bigl( (\psi, \partial \psi)^2 \bigr) \\
  &= \Box_{h} \Psi_\wp + (\dot{g}^{-1})^{\mu \nu} \nabla_\mu \nabla_\nu \Psi_\wp - (h^{-1})^{\mu \nu} \dot{\Gamma}^\lambda_{\mu \nu} \partial_\lambda \Psi_\wp \\
  &\quad + (h^{-1})^{\mu \nu} \partial_\mu \partial_\nu \bigl( \psi N \bigr) - (h^{-1})^{\mu \nu} \Gamma_{\mu \nu}^\lambda \partial_\lambda \bigl( \psi N \bigr) + \calO \bigl( (\psi, \partial \psi)^2 \bigr),
 \end{aligned}
\end{equation*}
where
\begin{equation*}
 \begin{aligned}
  (\dot{g}^{-1})^{\mu \nu} &= - (h^{-1})^{\mu \alpha} \bfeta( \partial_\alpha \Psi_\wp, \partial_\beta N ) (h^{-1})^{\beta \nu} \psi - (h^{-1})^{\mu \alpha} \bfeta( \partial_\alpha N, \partial_\beta \Psi_\wp ) (h^{-1})^{\beta \nu} \psi + \calO( \partial \psi ) + \calO \bigl( (\psi, \partial \psi)^2 \bigr)
 \end{aligned}
\end{equation*}
and 
\begin{equation*}
 \begin{aligned}
  \dot{\Gamma}_{\mu \nu}^\lambda &= (h^{-1})^{\lambda k} \bigl( \bfeta( \partial_\mu \partial_\nu \Psi_\wp, \partial_k N ) + \bfeta( \partial_k \Psi_\wp, \partial_\mu \partial_\nu N ) \bigr) \psi \\
  &\quad - (h^{-1})^{\lambda \rho} \Gamma_{\mu\nu}^\sigma \bigl( \bfeta( \partial_\rho \Psi_\wp, \partial_\sigma N ) + \bfeta( \partial_\sigma \Psi_\wp, \partial_\rho N ) \bigr) \psi + \calO( \partial \psi ) + \calO \bigl( (\psi, \partial \psi)^2 \bigr).
 \end{aligned}
\end{equation*}
Observe that
\begin{equation*}
 \begin{aligned}
  - (h^{-1})^{\mu \nu} \dot{\Gamma}_{\mu \nu}^\lambda &= - (h^{-1})^{\mu \nu} (h^{-1})^{\lambda k} \bfeta( \partial_\mu \partial_\nu \Psi_\wp, \partial_k N ) \psi + (h^{-1})^{\mu \nu} (h^{-1})^{\lambda \rho} \Gamma_{\mu \nu}^\sigma \bfeta( \partial_\sigma \Psi_\wp, \partial_\rho N ) \\
  &\quad - (h^{-1})^{\mu\nu} (h^{-1})^{\lambda k} \bfeta( \partial_k \Psi_\wp, \partial_\mu \partial_\nu N ) \psi + (h^{-1})^{\mu \nu} (h^{-1})^{\lambda \rho} \Gamma^{\sigma}_{\mu\nu} \bfeta( \partial_\rho \Psi_\wp, \partial_\sigma N ) \\
  &\quad + \calO( \partial \psi ) + \calO \bigl( (\psi, \partial \psi)^2 \bigr) \\
  &= - \bfeta( (h^{-1})^{\mu \nu} \partial_\mu \partial_\nu \Psi_\wp, \nabla^\lambda N ) \psi + \bfeta( (h^{-1})^{\mu \nu} \Gamma_{\mu \nu}^\sigma \partial_\sigma \Psi_\wp, \nabla^\lambda N ) \psi \\
  &\quad - \bfeta( \nabla^\lambda \Psi_\wp, (h^{-1})^{\mu\nu} \partial_\mu \partial_\nu N ) \psi + \bfeta( \nabla^\lambda \Psi_\wp, (h^{-1})^{\mu\nu} \Gamma_{\mu\nu}^\sigma \partial_\sigma N ) \psi \\
  &\quad + \calO( \partial \psi ) + \calO \bigl( (\psi, \partial \psi)^2 \bigr) \\
  &= - \bfeta( \Box_{h} \Psi_\wp, \nabla^\lambda N ) \psi - \bfeta( \nabla^\lambda \Psi_\wp, \Box_h N ) + \calO( \partial \psi ) + \calO \bigl( (\psi, \partial \psi)^2 \bigr).
 \end{aligned}
\end{equation*}

In what follows, we use the notation
\begin{equation*}
\begin{aligned}
 \tilde{\nabla}_\mu \tilde{\nabla}_\nu \Psi_\wp &:= \partial_\mu \partial_\nu \Psi_\wp - (h^{-1})^{\mu \nu} \tilde{\Gamma}_{\mu \nu}^\lambda \partial_\lambda \Psi_\wp, \\
 \tilde{\Gamma}_{\mu \nu}^\lambda &:= \frac12 (h^{-1})^{\lambda \kappa} \bigl( \partial_\mu h_{\kappa \nu} + \partial_\nu h_{\mu\kappa} - \partial_\kappa h_{\mu \nu} \bigr), \\
 \tilde{\nabla}^\mu \Psi_\wp &:= (h^{-1}){}^{\mu\nu} \partial_\nu \Psi_\wp, \\
 [ \tilde{\nabla}_\mu, \tilde{\nabla}_\nu ] \tilde{\nabla}_\lambda \Psi_\wp &= \tilde{R}_{\mu\nu\lambda\sigma} \tilde{\nabla}^\sigma \Psi_\wp, \\
 \tilde{R}_{\mu\nu} &= (h^{-1})^{\lambda \sigma} \tilde{R}_{\mu\lambda\nu\sigma}.
\end{aligned}
\end{equation*}

Correspondingly, we obtain 
\begin{equation*}
 \begin{aligned}
  \frac{\delta \bigl( \Box_g \Phi \bigr)}{\delta \psi} \bigg|_{\substack{\psi = 0, \\ k = h}} &= \Box_h N + \frac{ \delta \bigl( (\dot{g}^{-1})^{\mu \nu} \bigr) }{\delta \psi} \bigg|_{\substack{\psi = 0, \\ k = h}} \tilde{\nabla}_\mu \tilde{\nabla}_\nu \Psi_\wp - \frac{\delta \bigl( (h^{-1})^{\mu \nu} \dot{\Gamma}^\lambda_{\mu \nu} \partial_\lambda \Psi_\wp \bigr)}{\delta \psi} \bigg|_{\substack{\psi = 0, \\ k = h}} \\
  &= \Box_{h} N - \bfeta( \tilde{\nabla}^\mu \Psi_\wp, \tilde{\nabla}^\nu N ) (\tilde{\nabla}_\mu \tilde{\nabla}_\nu \Psi_\wp) - \bfeta( \tilde{\nabla}^\nu \Psi_\wp, \tilde{\nabla}^\mu N ) (\tilde{\nabla}_\mu \tilde{\nabla}_\nu \Psi_\wp) \\
  &\quad - \bfeta( \underbrace{\Box_{h} \Psi_\wp}_{=0}, \partial^\lambda N ) \partial_\lambda \Psi_\wp - \bfeta( \partial^\lambda \Psi_\wp, \Box N ) \partial_\lambda \Psi_\wp \\
  &= \Box_h N - \bfeta( \tilde{\nabla}^\mu \Psi_\wp, \tilde{\nabla}^\nu N ) (\tilde{\nabla}_\mu \tilde{\nabla}_\nu \Psi_\wp) - \bfeta( \tilde{\nabla}^\nu \Psi_\wp, \tilde{\nabla}^\mu N ) (\tilde{\nabla}_\mu \tilde{\nabla}_\nu \Psi_\wp) \\
  &\quad - \bfeta( \tilde{\nabla}^\lambda \Psi_\wp, \Box_h N ) \partial_\lambda \Psi_\wp.
 \end{aligned}
\end{equation*}

\noindent {\it Step 2: Computation of $\bfeta( \Box_h W, n_\wp )$ and $\bfeta( \Box_h N, n_\wp )$.}
Since $W$ is tangential, we may write 
\begin{equation*}
 \begin{aligned}
  W = \bfeta( W, \partial_\mu \Psi_\wp ) \tilde{\nabla}^\mu \Psi_\wp.
 \end{aligned}
\end{equation*}
In what follows we will use that for all $\mu, \nu, \sigma$ 
\begin{equation} \label{equ:vanishing_nabla2_Psi_tested_partial_Psi}
 \begin{aligned}
  \bfeta( \tilde{\nabla}_\mu \tilde{\nabla}_\nu \Psi_\wp, \partial_\sigma \Psi_\wp ) &= 0.
 \end{aligned}
\end{equation}
To see this we expand 
\begin{equation*}
 \begin{aligned}
  \bfeta( \tilde{\nabla}_\mu \tilde{\nabla}_\nu \Psi_\wp, \partial_\sigma \Psi_\wp ) &= \bfeta( \partial_\mu \partial_\nu \Psi_\wp, \partial_\sigma \Psi_\wp ) - \tilde{\Gamma}^\lambda_{\mu \nu} h_{\lambda \sigma} \\
  &= \bfeta( \partial_\mu \partial_\nu \Psi_\wp, \partial_\sigma \Psi_\wp ) - \frac12 (h^{-1})^{\lambda \kappa} \bigl( \partial_\mu h_{\kappa \nu} + \partial_\nu h_{\mu \kappa} - \partial_\kappa h_{\mu \nu} \bigr) h_{\lambda \sigma} \\
  &= \bfeta( \partial_\mu \partial_\nu \Psi_\wp, \partial_\sigma \Psi_\wp ) - \frac12 \bigl( \partial_\mu h_{\sigma \nu} + \partial_\nu h_{\mu \sigma} - \partial_\sigma h_{\mu \nu} \bigr) \\
  &= \bfeta( \partial_\mu \partial_\nu \Psi_\wp, \partial_\sigma \Psi_\wp ) - \bfeta( \partial_\mu \partial_\nu \Psi_\wp, \partial_\sigma \Psi_\wp ) \\
  &= 0.
 \end{aligned}
\end{equation*}
Using~\eqref{equ:vanishing_nabla2_Psi_tested_partial_Psi}, we obtain by direct computation 
\begin{equation*}
 \begin{aligned}
  \Box_h W &= \bfeta( W, \partial_\mu \Psi_\wp ) \tilde{\nabla}^\mu \underbrace{\Box_h \Psi_\wp}_{= \, 0} + \bfeta( W, \partial_\mu \Psi_\wp ) \tilR^{\mu \lambda} \partial_\lambda \Psi_\wp + \bigl( \Box_h \bfeta( W, \partial_\mu \Psi_\wp ) \bigr) \tilde{\nabla}^\mu \Psi_\wp \\
  &\quad + 2 \bfeta( \partial_\nu W, \partial_\mu \Psi_\wp )  \tilde{\nabla}^\nu \tilde{\nabla}^\mu \Psi_\wp + 2 \underbrace{\bfeta( W, \tilde{\nabla}_\nu \tilde{\nabla}_\mu \Psi_\wp )}_{= \, 0} \tilde{\nabla}^\nu \tilde{\nabla}^\mu \Psi_\wp.
 \end{aligned}
\end{equation*}
Note that $\bfeta( W, \tilde{\nabla}_\nu \tilde{\nabla}_\mu \Psi_\wp ) = 0$ follows from \eqref{equ:vanishing_nabla2_Psi_tested_partial_Psi} since $W$ is tangential. 
Testing against $n_\wp$ and inserting the relation $W = N - n_\wp$, we arrive at the identity 
\begin{equation*}
 \begin{aligned}
  \bfeta( \Box_h W, n_\wp ) &= 2 \bfeta( \partial_\nu W, \partial_\mu \Psi_\wp ) \bfeta( \tilde{\nabla}^\nu \tilde{\nabla}^\mu \Psi_\wp, n_\wp ) \\
  &= 2 \bfeta( \partial_\nu N, \partial_\mu \Psi_\wp ) \bfeta( \tilde{\nabla}^\nu \tilde{\nabla}^\mu \Psi_\wp, n_\wp ) - 2 \bfeta( \partial_\nu n_\wp, \partial_\mu \Psi_\wp ) \bfeta( \tilde{\nabla}^\nu \tilde{\nabla}^\mu \Psi_\wp, n_\wp ) \\
  &= 2 \bfeta( \partial_\nu N, \partial_\mu \Psi_\wp ) \bfeta( \tilde{\nabla}^\nu \tilde{\nabla}^\mu \Psi_\wp, n_\wp ) + 2 \bfeta( \partial_\nu n_\wp, \partial_\mu \Psi_\wp ) \bfeta( \tilde{\nabla}^\mu \Psi_\wp, \tilde{\nabla}^\nu n_\wp ).
 \end{aligned}
\end{equation*}
Moreover, using that $\Box_h n_\wp = - |\secondff|^2 n_\wp$ by \cite[Corollary II]{RV70}, we obtain  
\begin{equation*}
 \begin{aligned}
  \bfeta( \Box_h N, n_\wp ) &= \bfeta( \Box_h n_\wp, n_\wp ) + \bfeta( \Box_h W, n_\wp ) \\
  &= - |\secondff|^2 + 2 \bfeta( \partial_\nu N, \partial_\mu \Psi_\wp ) \bfeta( \tilde{\nabla}^\nu \tilde{\nabla}^\mu \Psi_\wp, n_\wp ) + 2 \bfeta( \partial_\nu n_\wp, \partial_\mu \Psi_\wp ) \bfeta( \tilde{\nabla}^\mu \Psi_\wp, \tilde{\nabla}^\nu n_\wp ).
 \end{aligned}
\end{equation*}

\medskip 

\noindent {\it Step 3: Final computation.}
We decompose into
\begin{equation*}
 \begin{aligned}
  \bfeta \biggl( \frac{\delta \bigl( \Box_g \Phi \bigr)}{\delta \psi} \bigg|_{\substack{\psi = 0, \\ k = h}}, N \biggr) = \bfeta \biggl( \frac{\delta \bigl( \Box_g \Phi \bigr)}{\delta \psi} \bigg|_{\substack{\psi = 0, \\ k = h}}, n_\wp \biggr) + \bfeta \biggl( \frac{\delta \bigl( \Box_g \Phi \bigr)}{\delta \psi} \bigg|_{\substack{\psi = 0, \\ k = h}}, W \biggr).
 \end{aligned}
\end{equation*}
Then on the one hand we have
\begin{equation*}
 \begin{aligned}
  \bfeta \biggl( \frac{\delta \bigl( \Box_g \Phi \bigr)}{\delta \psi} \bigg|_{\substack{\psi = 0, \\ k = h}}, n_\wp \biggr) &= \bfeta( \Box N, n_\wp ) - \bfeta( \tilde{\nabla}^\mu \Psi_\wp, \tilde{\nabla}^\nu N ) \bfeta( \tilde{\nabla}_\mu \tilde{\nabla}_\nu \Psi_\wp, n_\wp ) \\
  &\quad - \bfeta( \tilde{\nabla}^\nu \Psi_\wp, \tilde{\nabla}^\mu N ) \bfeta( \tilde{\nabla}_\mu \tilde{\nabla}_\nu \Psi_\wp, n_\wp ) - \bfeta( \tilde{\nabla}^\lambda \Psi_\wp, \Box N ) \underbrace{\bfeta( \partial_\lambda \Psi_\wp, n_\wp )}_{= \, 0} \\
  &= - |\secondff|^2 + 2 \bfeta( \partial_\mu \Psi_\wp, \partial_\nu N ) \bfeta( \tilde{\nabla}^\nu \tilde{\nabla}^\mu \Psi_\wp, n_\wp ) + 2 \bfeta( \partial_\mu \Psi_\wp, \partial_\nu n_\wp ) \bfeta( \tilde{\nabla}^\mu \Psi_\wp, \tilde{\nabla}^\nu n_\wp ) \\
  &\quad - \bfeta( \tilde{\nabla}^\mu \Psi_\wp, \tilde{\nabla}^\nu N ) \bfeta( \tilde{\nabla}_\mu \tilde{\nabla}_\nu \Psi_\wp, n_\wp ) - \bfeta( \tilde{\nabla}^\nu \Psi_\wp, \tilde{\nabla}^\mu N ) \bfeta( \tilde{\nabla}_\mu \tilde{\nabla}_\nu \Psi_\wp, n_\wp ) \\
  &= - |\secondff|^2 + 2 \bfeta( \partial_\mu \Psi_\wp, \partial_\nu n_\wp ) \bfeta( \tilde{\nabla}^\mu \Psi_\wp, \tilde{\nabla}^\nu n_\wp ) \\
  &= |\secondff|^2.
 \end{aligned}
\end{equation*}
On the other hand, using that $\bfeta( \tilde{\nabla}_\mu \tilde{\nabla}_\nu \Psi_\wp, \nabla_\lambda \Psi_\wp ) = 0$, we find 
\begin{equation*}
 \begin{aligned}
  \bfeta \biggl( \frac{\delta \bigl( \Box_g \Phi \bigr)}{\delta \psi} \bigg|_{\substack{\psi = 0, \\ k = h}}, W \biggr) &= \bfeta( \Box_h N, W ) - \bfeta( \tilde{\nabla}^\mu \Psi_\wp, \tilde{\nabla}^\nu N ) \underbrace{\bfeta( \tilde{\nabla}_\mu \tilde{\nabla}_\nu \Psi_\wp, W )}_{= \, 0} \\
  &\quad - \bfeta( \tilde{\nabla}^\nu \Psi_\wp, \tilde{\nabla}^\mu N ) \underbrace{\bfeta( \tilde{\nabla}_\mu \tilde{\nabla}_\nu \Psi_\wp, W )}_{= \, 0} - \underbrace{\bfeta( \tilde{\nabla}^\lambda \Psi_\wp, \Box_h N ) \bfeta( \partial_\lambda \Psi_\wp, W )}_{= \, \bfeta( \Box_h N, W )} \\
  &= 0.
 \end{aligned}
\end{equation*}
This finishes the proof.
\end{proof}

Using the preceding lemma, we find that the fourth term on the right-hand side of~\eqref{equ:pt_dotphi_comp1} simplifies to
\begin{equation*}
 \begin{aligned}
  &- \bfeta \biggl( \pmu \Bigl( \frac{ \delta (\sqrt{|g|} (g^{-1})^{\mu \nu} \pnu \Phi) }{\delta  \psi} \bigg|_{\substack{\psi = 0, \\ k = h}} \Bigr) \psi, N \biggr) \\
  &= - \bfeta \biggl( \pj \Bigl( \frac{ \delta (\sqrt{|g|} (g^{-1})^{j \nu} \pnu \Phi) }{\delta  \psi} \bigg|_{\substack{\psi = 0, \\ k = h}} \Bigr) \psi, N \biggr) - \bfeta \biggl( \pt \Bigl( \frac{ \delta (\sqrt{|g|} (g^{-1})^{0 \nu} \pnu \Phi) }{\delta  \psi} \bigg|_{\substack{\psi = 0, \\ k = h}} \Bigr) \psi, N \biggr) \\
  &= - \sqrt{|h|} |\secondff|^2 \psi + \calE_{11}
 \end{aligned}
\end{equation*}
with remainder term
\begin{equation*}
 \begin{aligned}
  \calE_{11} := - \bfeta \biggl( \pt \Bigl( \frac{ \delta (\sqrt{|g|} (g^{-1})^{0 \nu} \pnu \Phi) }{\delta  \psi} \bigg|_{\substack{\psi = 0, \\ k = h}} \Bigr) \psi, N \biggr) = \calO\bigl( \dotwp \psi \bigr). 
 \end{aligned}
\end{equation*}

We arrive at the equation 
\begin{equation*}
 \begin{aligned}
  \pt \dot{\psi} &= - \pj \bigl( \sqrt{|h|} (h^{-1})^{j\mu} \pmu \psi \bigr) - \sqrt{|h|} |\secondff|^2 \psi - \bfeta \bigl( \sqrt{|h|} (h^{-1})^{00} (0, \dot{\ell}), N \bigr) + \calE_5 + \ldots + \calE_{11}.
 \end{aligned}
\end{equation*}
Finally, inserting for $\pt \psi$ on the right-hand side the relation~\eqref{equ:relation_ptphi_dotphi} between $\dot{\psi}$ and $\pt \psi$, we obtain 
\begin{equation} \label{equ:pt_dotphi_final}
 \begin{aligned}
  \pt \dot{\psi} &= - \sqrt{|h|} L \psi - \pj \biggl( \frac{(h^{-1})^{j0}}{(h^{-1})^{00}} \dot{\psi} \biggr) - \bfeta \bigl( \sqrt{|h|} (h^{-1})^{00} (0, \dot{\ell}), N \bigr) + \dot{f},
 \end{aligned}
\end{equation}
where we introduce the linear operator
\begin{equation*}
 \begin{aligned}
  L := \frac{1}{\sqrt{|h|}} \pj \bigl( \sqrt{|h|} (\underline{h}^{-1})^{jk} \pk \bigr) + |\secondff|^2 
 \end{aligned}
\end{equation*}
with
\begin{equation*}
 \begin{aligned}
  (\underline{h}^{-1})^{jk} := (h^{-1})^{jk} - \frac{(h^{-1})^{j0} (h^{-1})^{0k}}{(h^{-1})^{00}},
 \end{aligned}
\end{equation*}
and where
\begin{equation*}
 \begin{aligned}
  \dot{f} &:= \pj \biggl(  \sqrt{|h|} (h^{-1})^{j0} \bfeta \Bigl( \bigl( \dot{\ell} \cdot \nabla_\ell + (\dot{\xi}-\ell) \cdot \nabla_\xi \bigr) \Psi_\wp, N \Bigr) \biggr) - \pj \bigl(  \sqrt{|h|} (h^{-1})^{j0} f \bigr) + \calE_5 + \ldots + \calE_{11}.
 \end{aligned}
\end{equation*}
Note that the first term in the preceding definition of $\dot{f}$ is of the form $\calO(\ell \dotwp)$ since $(h^{-1})^{j0} = \calO(\ell)$ by Remark~\ref{rem:hinv0j_smallness}, which justifies its treatment as a quadratic error.

\medskip

Introducing the matrix operator 
\begin{equation} \label{equ:definition_matrix_operator}
 \begin{aligned}
  M := \begin{pmatrix}
        - \frac{(h^{-1})^{0j}}{(h^{-1})^{00}} \pj & \frac{1}{\sqrt{|h|} (h^{-1})^{00}} \\ 
        -\sqrt{|h|} L & - \pj \Bigl( \frac{(h^{-1})^{j0}}{(h^{-1})^{00}} \Bigr)
       \end{pmatrix}
 \end{aligned}
\end{equation}
and setting
\begin{equation*}
 \begin{aligned}
  \vec{K} := \begin{pmatrix}
             K \\ \dot{K}
            \end{pmatrix}
          = \begin{pmatrix}
             - \bfeta \Bigl( \bigl( \dot{\ell} \cdot \nabla_\ell + (\dot{\xi}-\ell) \cdot \nabla_\xi \bigr) \Psi_\wp, N \Bigr) \\ - \bfeta \bigl( \sqrt{|h|} (h^{-1})^{00} (0, \dot{\ell}), N \bigr)
            \end{pmatrix}, 
  \qquad 
  \vec{f} := \begin{pmatrix}
              f \\ \dot{f}
             \end{pmatrix},
 \end{aligned}
\end{equation*}
we obtain from \eqref{equ:relation_ptphi_dotphi} and \eqref{equ:pt_dotphi_final} the following first-order formulation of the HVMC equation for $\vec{\psi}$
\begin{equation} \label{equ:first_order_linearized_HVMC}
 \begin{aligned}
  (\pt - M) \vec{\psi} = \vec{K} + \vec{f}.
 \end{aligned}
\end{equation}

\medskip 

We end this subsection by computing the second order equation for $\psi$ coming from~\eqref{equ:first_order_linearized_HVMC}. 
Upon rearranging, we obtain from~\eqref{equ:first_order_linearized_HVMC} that
\begin{equation} \label{equ:dotpsi_from_matrix_equation}
 \dot{\psi} = \sqrt{|h|}(h^{-1})^{0\nu} \partial_\nu \psi - \sqrt{|h|}(h^{-1})^{00} (K+f),
\end{equation}
and 
\begin{align*}
\begin{split}
 \partial_t \dot{\psi} + \partial_j\bigl( \sqrt{|h|}(h^{-1})^{ij} \partial_i\psi \bigr) - \partial_j \biggl( \sqrt{|h|} \frac{(h^{-1})^{0i}(h^{-1})^{0j}}{(h^{-1})^{00}} \partial_i\psi \biggr) + \partial_j \biggl( \frac{(h^{-1})^{0j}}{(h^{-1})^{00}} \dot{\psi} \biggr) + \sqrt{|h|} |\secondff|^2 \psi = \dot{K} + \dot{f}.
\end{split}
\end{align*}
Substituting~\eqref{equ:dotpsi_from_matrix_equation} for $\dot{\psi}$ in the preceding identity, we arrive at
\begin{align*}
\begin{split}
\partial_\mu \bigl( \sqrt{|h|}(h^{-1})^{\mu\nu} \partial_\nu \psi \bigr) + \sqrt{|h|} |\secondff|^2 \psi = (\dot{K} + \dot{f}) + \partial_\nu \bigl(\sqrt{|h|}(h^{-1})^{0\nu} (K+f) \bigr).
\end{split}
\end{align*}
We conclude that if $\vec{\psi}$ satisfies~\eqref{equ:first_order_linearized_HVMC}, then $\psi$ satisfies the wave equation 
\begin{align} \label{eq:phi1}
\begin{split}
 \frac{1}{\sqrt{|h|}} \partial_\mu \bigl( \sqrt{|h|} (h^{-1})^{\mu\nu} \partial_\nu \psi \bigr) + |\secondff|^2 \psi = \frac{1}{\sqrt{|h|}} (\dot{K} + \dot{f}) + \frac{1}{\sqrt{|h|}} \partial_\nu \bigl( \sqrt{|h|}(h^{-1})^{0\nu} (K+f) \bigr).
\end{split}
\end{align}

%%%%%%%%%%%%%%%%%%%%
%%%%%%%%%%%%%%%%%%%%
%%%%%%%%%%%%%%%%%%%%
\subsection{Eigenfunctions} \label{subsec:eigenfunctions}
%%%%%%%%%%%%%%%%%%%%
%%%%%%%%%%%%%%%%%%%%
%%%%%%%%%%%%%%%%%%%%

In this subsection we determine the eigenfunctions and generalized eigenfunctions of the matrix operator $M$ defined in~\eqref{equ:definition_matrix_operator} 
when the parameter $\ell(t)$ is time-independent, i.e., $\ell(t) = \ell$ for some fixed $\ell \in \bbR^{n+1}$ with $|\ell| < 1$, and $\xi(t)=a+t\ell$ for some fixed $ a\in \bbR^{n+1}$. We will denote the particular choice of $\ell$ for which we want to compute the eigenfunctions by $\ell_0$, and use the notation $M_{\ell_0}$ for the corresponding operator.
To find the eigenfunctions and generalized eigenfunctions of $M_{\ell_0}$, we consider the following maximal embeddings
\begin{align*}
\begin{split}
\Psi_{\xi_0,\ell_0} &:= \bigl( t, \xi_0 + \gamma_0^{-1}P_{\ell_0}F+P_{\ell_0}^\perp F \bigr), \qquad \xi_0 = a_0+t\ell_0, \\
\Psi_{\xi,\ell} &:= \bigl( t, \xi + \gamma^{-1} P_\ell F+P_\ell^\perp F \bigr), \qquad \quad \, \xi = a+t\ell,
\end{split}
\end{align*} 
for fixed $(a_0, \ell_0)$ and $(a, \ell)$.
For each $(a,\ell)$ we write
\begin{align*}
 \Psi_{\xi,\ell} = \Psi_{\xi_0,\ell_0} + \psi_{\xi,\ell}N_{\ell_0},
\end{align*}
where $N_{\ell_0}$ is the normal to $\text{im}(\Psi_{\xi_0,\ell_0}(t)) \cap \bsUpsigma_t$ viewed as a subspace of $\bsUpsigma_t$.
Then we have $\psi_{\xi_0,\ell_0} \equiv 0$.
The metric $\Psi_{\xi_0,\ell_0}^\ast\bfeta$ is denoted by $h$ and the metric $\Psi_{\xi,\ell}^\ast\bfeta$ by $g$ (note that in this subsection there is no difference between $h$ and what would be $k$). 
In view of~\eqref{equ:definition_B}, we define
\begin{align*}
\begin{split}
B_0 &:= \sqrt{|h|} (h^{-1})^{0 j}\bfeta(\partial_j N_{\ell_0},N_{\ell_0}) \\
&\qquad - \sqrt{|h|} (h^{-1})^{0 \kappa} (h^{-1})^{\nu j}\bfeta(\partial_\kappa\Psi_{\xi_0,\ell_0},\partial_j N_{\ell_0})\bfeta(\partial_\nu\Psi_{\xi_0,\ell_0},N_{\ell_0}) \\
&\qquad - \sqrt{|h|} (h^{-1})^{0 j} (h^{-1})^{\nu\lambda}\bfeta(\partial_\lambda\Psi_{\xi_0,\ell_0},\partial_j N_{\ell_0})\bfeta(\partial_\nu\Psi_{\xi_0,\ell_0},N_{\ell_0}) \\
&\qquad + \sqrt{|h|} (h^{-1})^{0 \nu} (h^{-1})^{\kappa j}\bfeta(\partial_\kappa\Psi_{\xi_0,\ell_0},\partial_j N_{\ell_0})\bfeta(\partial_\nu\Psi_{\xi_0,\ell_0},N_{\ell_0})
\end{split}
\end{align*}
and let 
\begin{align*}
\begin{split}
\dotpsi_{\xi,\ell} := \bfeta\bigl( \sqrt{|g|} (g^{-1})^{0\nu}\partial_\nu\Psi_{\xi,\ell}, N_{\ell_0} \bigr)  - \bfeta\bigl( \sqrt{|h|} (h^{-1})^{00} (1,\ell_0),N_{\ell_0}\bigr) - B_0 \psi_{\xi,\ell}.
\end{split}
\end{align*}
Then $\dotpsi_{\xi_0,\ell_0} = 0$ and $\vecpsi_{\xi,\ell} = (\psi_{\xi,\ell}, \dotpsi_{\xi,\ell})$ satisfies
\begin{align*}
(\partial_t-M_{\ell_0}) \vecpsi_{\xi,\ell} = \vec{\calF},
\end{align*}
where $\vec{\calF}$ depends at least quadratically on $\vecpsi_{\xi,\ell}$. In particular
\begin{align*}
\begin{split}
\frac{\delta \vec{\calF}}{\delta a} \Big|_{(a,\ell)=(a_0,\ell_0)} = \frac{\delta\vec{\calF}}{\delta\ell}\Big|_{(a,\ell)=(a_0,\ell_0)}=0.
\end{split}
\end{align*} 
It follows that $\frac{\delta\vecpsi_{\xi,\ell}}{\delta a^i}\vert_{(a,\ell)=(a_0,\ell_0)}$ and $\frac{\delta\vecpsi_{\xi,\ell}}{\delta \ell^i}\vert_{(a,\ell)=(a_0,\ell_0)}$ for $1 \leq i \leq n$ are solutions of 
\begin{align*}
 (\partial_t-M_{\ell_0}) \vecfy = 0.
\end{align*}
To compute these parameter derivatives more easily, we also observe that in view of~\eqref{equ:relation_ptphi_dotphi}
\begin{align*}
 \dotpsi_{\xi,\ell} = \sqrt{|h|} h^{0\nu} \partial_\nu \psi_{\xi,\ell} + \tilf,
\end{align*}
where $\tilf$ depends quadratically on $\psi_{\xi,\ell}$, whence $\frac{\delta \tilf}{\delta a^i}\vert_{(a,\ell)=(a_0,\ell_0)}=\frac{\delta\tilf}{\delta\ell^i}\vert_{(a,\ell)=(a_0,\ell_0)}=0$. Moreover, we note that $\frac{\delta}{\delta a^i}$ is the same as $\frac{\delta}{\delta\xi^i}$ and 
\begin{align*}
\begin{split}
 \frac{\delta \psi_{\xi,\ell}}{\delta \xi^i} \Bigr|_{\ell=\ell_0} = \bfeta\Bigl(\frac{\delta\Psi_{\xi,\ell}}{\delta \xi^i} \Bigr|_{\ell=\ell_0},|N_{\ell_0}|^{-2}N_{\ell_0}\Bigr), \qquad \frac{\delta \psi_{\xi,\ell}}{\delta \ell^i} \Bigr|_{\ell=\ell_0} = \bfeta\Bigl(\frac{\delta\Psi_{\xi,\ell}}{\delta \ell^i} \Bigr|_{\ell=\ell_0},|N_{\ell_0}|^{-2}N_{\ell_0}\Bigr).
\end{split}
\end{align*}
With these observations we compute
\begin{align*}
\begin{split}
 \vecfy_i = \begin{pmatrix} \fy_i \\ \dotfy_i \end{pmatrix} := \begin{pmatrix} \frac{\delta\psi_{\xi,\ell}}{\delta\xi^i} \bigl|_{\ell=\ell_0} \\ \frac{\delta\dotpsi_{\xi,\ell}}{\delta\xi^i} \bigl|_{\ell=\ell_0} \end{pmatrix} = \begin{pmatrix} \frac{\delta\psi_{\xi,\ell}}{\delta\xi^i} \bigl|_{\ell=\ell_0} \\ \sqrt{|h|} (h^{-1})^{0\nu} \bigl|_{\ell=\ell_0} \partial_\nu\frac{\delta\psi_{\xi,\ell}}{\delta\xi^i} \bigl|_{\ell=\ell_0} \end{pmatrix}, \quad 1 \leq i \leq n,
\end{split}
\end{align*}
to be 
\begin{equation*}
 \begin{aligned}
  \fy_i = |\ell_0|^{-2} (\gamma_0 - 1) (\ell_0 \cdot \nu) \ell_0^i + \nu^i, \qquad \dotfy_i = \sqrt{|h|} (h^{-1})^{0j} \bigl|_{\ell=\ell_0} \partial_j \fy_i
 \end{aligned}
\end{equation*}
with 
\begin{equation*}
 \gamma_0 = (1-|\ell_0|^2)^{-\frac12}.
\end{equation*}
Since $\vecfy_i$ is independent of $t$, we conclude that
\begin{align*}
 M_{\ell_0} \vecfy_i=0, \qquad 1 \leq i \leq n.
\end{align*}
Similarly, for the derivatives with respect to $\ell^i$ we have
(below we have used the fact that the first $n$ components of $\nu$ and $F$ are proportional, to simplify a bit) 
\begin{align*}
\begin{split}
 \frac{\delta \psi_{\xi,\ell}}{\delta \ell^i} \Bigr|_{\ell=\ell_0} = \bfeta\Bigl(\frac{\delta\Psi_{\xi,\ell}}{\delta \ell^i} \Bigr|_{\ell=\ell_0}, |N_{\ell_0}|^{-2}N_{\ell_0}\Bigr) 
 &= -\gamma_0(\ell_0 \cdot F) \nu^i - |\ell_0|^{-2} \gamma_0 (\gamma_0-1) (\ell_0 \cdot F)(\ell_0 \cdot \nu) \ell_0^i + t\fy_i.
\end{split}
\end{align*}
We set 
\begin{align*}
\begin{split}
 \fy_{n+i} := -\gamma_0 (\ell_0 \cdot F)\nu^i-|\ell_0|^{-2} \gamma_0 (\gamma_0-1)(\ell_0 \cdot F)(\ell_0 \cdot \nu) \ell_0^i, \quad 1 \leq i \leq n,
\end{split}
\end{align*}
and thus have 
\begin{equation*}
 \frac{\delta \psi_{\xi,\ell}}{\delta \ell^i} \Bigr|_{\ell=\ell_0} = \fy_{n+i} + t \fy_i, \quad 1 \leq i \leq n.
\end{equation*}
It follows that
\begin{align*}
\begin{split}
 \frac{\delta \dotpsi_{\xi,\ell}}{\delta \ell^i} \Bigr|_{\ell=\ell_0} &= \sqrt{|h|} (h^{-1})^{0\nu} \partial_\nu \Bigl( \frac{\delta \psi_{\xi,\ell}}{\delta \ell^i} \Bigr) \Bigr|_{\ell=\ell_0} \\
 &= \sqrt{|h|}(h^{-1})^{0j} \bigr|_{\ell=\ell_0} \partial_j \fy_{n+i} + \sqrt{|h|}(h^{-1})^{00} \bigr|_{\ell=\ell_0} \fy_i + t \sqrt{|h|} (h^{-1})^{0j} \bigr|_{\ell=\ell_0} \partial_j \fy_{i} \\
 &= \sqrt{|h|}(h^{-1})^{0j} \bigr|_{\ell=\ell_0} \partial_j \fy_{n+i} + \sqrt{|h|}(h^{-1})^{00} \bigr|_{\ell=\ell_0} \fy_i + t \dotfy_i.
\end{split}
\end{align*}
Hence, upon defining 
\begin{equation*}
 \dotfy_{n+i} := \sqrt{|h|}(h^{-1})^{0j} \bigr|_{\ell=\ell_0} \partial_j \fy_{n+i} + \sqrt{|h|}(h^{-1})^{00} \bigr|_{\ell=\ell_0} \fy_i, \quad 1 \leq i \leq n,
\end{equation*}
as well as 
\begin{equation*}
 \vecfy_{n+i} = \begin{pmatrix} \fy_{n+i} \\ \dotfy_{n+i} \end{pmatrix}, \quad 1 \leq i \leq n,
\end{equation*}
we have
\begin{equation*}
 \begin{pmatrix} \frac{\delta\psi_{\xi,\ell}}{\delta\ell^i} \Bigr|_{\ell=\ell_0} \\ \frac{\delta\dotpsi_{\xi,\ell}}{\delta\ell^i} \Bigr|_{\ell=\ell_0}  \end{pmatrix} = \vecfy_{n+i} + t \vecfy_i, \quad 1 \leq i \leq n.
\end{equation*}

Now recall that we have shown for time-independent $\ell = \ell_0$,
\begin{align*}
\begin{split}
(\partial_t - M_{\ell_0}) \begin{pmatrix} \frac{\delta\psi_{\xi,\ell}}{\delta\ell^i} \Bigr|_{\ell=\ell_0} \\ \frac{\delta\dotpsi_{\xi,\ell}}{\delta\ell^i} \Bigr|_{\ell=\ell_0} \end{pmatrix} = 0, \quad 1 \leq i \leq n.
\end{split}
\end{align*}
Since $\vecfy_i$ satisfies $(\partial_t-M_{\ell_0}) \vecfy_i = M_{\ell_0} \vecfy_i=0$ and since $\vecfy_{n+i}$ satisfies $\partial_t\vecfy_{n+i}=0$, we conclude 
\begin{align*}
 M_{\ell_0} \vecfy_{n+i}=\vecfy_i, \quad 1 \leq i \leq n.
\end{align*}

%%%%%%%%%%%%%%%%%%%%
%%%%%%%%%%%%%%%%%%%%
%%%%%%%%%%%%%%%%%%%%
\subsection{Modulation Equations}\label{subsec:modulationeqs}
%%%%%%%%%%%%%%%%%%%%
%%%%%%%%%%%%%%%%%%%%
%%%%%%%%%%%%%%%%%%%%

In this section we revert to the notation used before Section~\ref{subsec:eigenfunctions}. In particular, $\ell = \ell(t)$ and $\xi=\xi(t)$ are assumed to be time-dependent again, and $\xi(t)$ is no longer assumed to be of the form $a+t\ell$.

We begin with the definition of the symplectic form.
To this end we introduce the operator
\begin{equation*}
 \begin{aligned}
  J = \begin{pmatrix} 0 & 1 \\ -1 & 0 \end{pmatrix}.
 \end{aligned}
\end{equation*}
Note that $J^\ast = - J$, in the sense that $(J\vecu)\cdot\vecv=-\vecu\cdot(J\vecv)$.
We define the symplectic form $\bfOmega$ as
\begin{equation*}
 \begin{aligned}
  \bfOmega(\vec{u}, \vec{v}) := \langle \vec{u}, J \vec{v} \rangle, \quad \vec{u} = \begin{pmatrix} u_1 \\ u_2 \end{pmatrix}, \quad \vec{v} = \begin{pmatrix} v_1 \\ v_2 \end{pmatrix},
 \end{aligned}
\end{equation*}
where 
\begin{equation*}
 \langle \vec{u}, \vec{v} \rangle = \int \bigl( u_1 v_1 + u_2 v_2 \bigr) \, \ud \omega \, \ud \rho. 
\end{equation*}
We emphasize that the reason why we are using $\ud\omega \,\ud \rho$ for the volume form is that in our applications, $\sqrt{| h|}$ is incorporated in the definition of $\vecpsi$.
Recall the definition of the matrix operator
\begin{equation*}
 \begin{aligned}
  M := \begin{pmatrix}
        - \frac{(h^{-1})^{0j}}{(h^{-1})^{00}} \pj & \frac{1}{\sqrt{|h|} (h^{-1})^{00}} \\ 
        -\sqrt{|h|} L & - \pj \Bigl( \frac{(h^{-1})^{j0}}{(h^{-1})^{00}} \Bigr)
       \end{pmatrix}.
 \end{aligned}
\end{equation*}
Its adjoint with respect to the inner product $\langle \vec{u}, \vec{v} \rangle$ is given by 
\begin{equation*}
 \begin{aligned}
  M^\ast := \begin{pmatrix}
              \pj \Bigl( \frac{(h^{-1})^{j0}}{(h^{-1})^{00}} \Bigr) & -\sqrt{|h|} L \\
              \frac{1}{\sqrt{|h|} (h^{-1})^{00}} & \frac{(h^{-1})^{0j}}{(h^{-1})^{00}} \pj
       \end{pmatrix}.
 \end{aligned}
\end{equation*}
Then we have 
\begin{equation*}
 \begin{aligned}
  J M + M^\ast J = 0.
 \end{aligned}
\end{equation*}
In particular, it follows that
\begin{equation*}
 \bfOmega(\vecu, M \vecv) = - \bfOmega(M\vecu, \vecv).
\end{equation*}

Motivated by the discussion in Subsection~\ref{subsec:eigenfunctions}, we define for arbitrary $\ell \in \bbR^n$, $|\ell| < 1$, 
\begin{equation*}
 \begin{aligned}
  \vecfy_i := \begin{pmatrix} \fy_i \\ \dotfy_i \end{pmatrix}, \quad \vecfy_{n+i} := \begin{pmatrix} \fy_{n+i} \\ \dotfy_{n+i} \end{pmatrix}, \quad 1 \leq i \leq n,
 \end{aligned}
\end{equation*}
where 
\begin{equation*}
 \begin{aligned}
  \fy_i &:= |\ell|^{-2}(\gamma-1)(\ell\cdot\nu)\ell^i+\nu^i, \\
  \dotfy_i &:= \sqrt{|h|}(h^{-1})^{0j}\partial_j\fy_i, \\
  \fy_{n+i} &:= -\gamma(\ell\cdot F)\nu^i-|\ell|^{-2}\gamma(\gamma-1)(\ell\cdot F)(\ell\cdot\nu)\ell^i, \\
  \dotfy_{n+i} &:= \sqrt{|h|}(h^{-1})^{0j}\partial_j\fy_{n+i}+\sqrt{|h|}(h^{-1})^{00}\fy_i.
 \end{aligned}
\end{equation*}
Recall that here we no longer assume that $\frac{\ud\xi}{\ud t}=\ell$ and $\frac{\ud\ell}{\ud t}=0$.
From Subsection~\ref{subsec:eigenfunctions} we still obtain 
\begin{equation*}
 M \vecfy_i = 0, \quad M \vecfy_{n+i} = \vecfy_i, \quad 1 \leq i \leq n,
\end{equation*}
but $\vecfy_i$ and $\vecfy_{n+i}$ are no longer elements of the kernel, respectively generalized kernel, of $(\partial_t - M)$.

Next, we introduce truncated versions of the generalized eigenfunctions given by
\begin{equation*}
 \vecZ_i := \chi \vecvarphi_i, \quad \vecZ_{n+i} := \chi \vecvarphi_{n+i}, \quad i = 1, \ldots, n,
\end{equation*}
where the smooth cut-off function $\chi \in C_c^\infty(\bbR)$ satisfies $\chi(\rho) = 1$ for $|\rho| \leq \Reigenfunctioncutoffscale$ and $\chi(\rho) = 0$ for $|\rho| \geq 2\Reigenfunctioncutoffscale$.
Then, using \eqref{equ:first_order_linearized_HVMC}, we find for $i=1, \ldots, n$ that
\begin{equation*}
 \begin{aligned}
  \partial_t \bigl( \bfOmega( \vecpsi, \vecZ_i ) \bigr) &= \bfOmega(\vecK, \vecZ_i) + \bfOmega(\vecf, \vecZ_i) + \bfOmega(\vecpsi, (\pt-M) \vecZ_i), \\
  \partial_t \bigl( \bfOmega( \vecpsi, \vecZ_{n+i} ) \bigr) &= \bfOmega(\vecK, \vecZ_{n+i}) + \bfOmega(\vecf, \vecZ_{n+i}) + \bfOmega(\psi, (\pt-M)\vecZ_{n+i}).
 \end{aligned}
\end{equation*}
We determine the leading order behavior of $\bfOmega( \vecpsi, \vecZ_i )$ and $\bfOmega( \vecpsi, \vecZ_{n+i} )$.
To this end we first observe that to leading order
\begin{equation*}
 \begin{aligned}
  K &= -(\dot{\xi}-\ell) \cdot \bigl( \nu + \calO(|\ell|^2) \nu \bigr) + \dot{\ell} \cdot \calO(|\ell|), \\
  \dot{K} &= -\sqrt{|h|} (h^{-1})^{00} \dot{\ell} \cdot \bigl( \nu + \calO(|\ell|^2) \nu \bigr), 
 \end{aligned}
\end{equation*}
and 
\begin{equation*}
 \begin{aligned}
  \varphi_i &= \nu^i + \calO(|\ell|^2), \\
  \dot{\varphi}_i &= \sqrt{|h|} (h^{-1})^{0j} \partial_j \varphi_i = \calO(|\ell|), \\
  \varphi_{n+i} &= \calO(|\ell|), \\
  \dot{\varphi}_{n+i} &= \sqrt{|h|} (h^{-1})^{00} \nu^i + \calO(|\ell|).
 \end{aligned}
\end{equation*}
We obtain corresponding leading order expressions for $\vecZ_i$ and $\vecZ_{n+i}$.
Thus, we find that to leading order 
\begin{equation*}
 \begin{aligned}
  \bfOmega(\vecK, \vecZ_i) &= \sum_j \dot{\ell}_j ( d_{ij} + r_{ij} ) + \sum_j (\dot{\xi}_j - \ell_j) b_{ij}, \\
  \bfOmega(\vecK, \vecZ_{n+i}) &= \sum_j \dot{\ell}_j \tilb_{ij} - \sum_j (\dot{\xi}_j - \ell_j) ( d_{ij} + \tilr_{ij})
 \end{aligned}
\end{equation*}
with (see Section~\ref{subsubsec:normal})
\begin{equation} \label{equ:derivation_modul_equ_def_dij}
  d_{ij} = \int \chi \nu^i \nu^j \sqrt{|h|} (h^{-1})^{00} \, \ud \rho \, \ud \omega \, \simeq \, \left\{ \begin{aligned} 1, \quad i = j, \\
  o_{\Reigenfunctioncutoffscale,\ell}(1), \quad i \neq j \end{aligned} \right.
\end{equation}
and 
\begin{equation*}
 r_{ij} = b_{ij} = \tilr_{ij} = \tilb_{ij} = \calO(|\ell|).
\end{equation*}
Below we denote by $D$ the $n \times n$ matrix with entries $d_{ij}$ defined in~\eqref{equ:derivation_modul_equ_def_dij}. Clearly, $D$ is invertible for small $|\ell|$ and sufficiently large $\Reigenfunctioncutoffscale$.
Parts of the quantities $\bfOmega(\vecf, \vec{Z}_i)$ and $\bfOmega(\vecf, \vec{Z}_{n+i})$ are (at least) linear in $\dotwp$ with coefficients that are $\calO(\dotwp, \ell, \psi, \partial \psi)$. Note that not the entire quantities $\bfOmega(\vecf, \vec{Z}_i)$ and $\bfOmega(\vecf, \vec{Z}_{n+i})$ contain a factor of $\dotwp$, for instance $\calE_1$ does not. Only terms with $\dotwp$ will eventually need to be moved to the left-hand side of the ODEs for $\wp$.
Finally, we have 
\begin{equation*}
 \begin{aligned}
  \bfOmega(\vecpsi, (\partial_t - M) \vecZ_i) &= \bfOmega(\vecpsi, \chi \dot{\ell} \cdot \nabla_\ell \vec{Z}_i) - \bfOmega(\vecpsi, M(\chi \vec{Z}_i)), \\
  \bfOmega(\vecpsi, (\partial_t - M) \vecZ_{n+i}) &= \bfOmega(\vecpsi, \chi \dot{\ell} \cdot \nabla_\ell \vec{Z}_{n+i}) - \bfOmega(\vecpsi, M(\chi \vec{Z}_{n+i})).
 \end{aligned}
\end{equation*}

Putting the preceding observations together, we have found that schematically
\begin{equation} \label{equ:derivation_modul_equ1}
 \begin{aligned}
  \partial_t \vecbfOmega + \vecN = \begin{pmatrix} D + R & R \\ R & D + R \end{pmatrix} \dotwp + \vecH,
 \end{aligned}
\end{equation}
where (recall the notation from Section~\ref{subsec:prelimvfs})
\begin{equation*}
 \begin{aligned}
  \vecbfOmega &= \bigl( \bfOmega(\vecpsi, \vecZ_1), \ldots, \bfOmega(\vecpsi, \vecZ_{2n}) \bigr), \\
  \vecN &= \bigl( \bfOmega(\vecpsi, M \vecZ_1), \ldots, \bfOmega(\vecpsi, M \vecZ_{2n}) \bigr), \\
  R &= \calO(\psi, \partial \psi, \partial_\Sigma \partial \psi, \dotwp, \ell), \\
  \vecH &= \calO\bigl( (\psi, \partial \psi, \partial_\Sigma \partial \psi)^2 \bigr).
 \end{aligned}
\end{equation*}
Note that here we have separated $\vecN$ from $\vecH$ to emphasize that $\vecN$, which contains the linear contribution in $\vecpsi$, is the principal source term for $\partial_t\vecbfOmega$.

In what follows, we would like to view~\eqref{equ:derivation_modul_equ1} as an equation entirely in terms of $\vecpsi = (\psi, \dotpsi)$, $\dotwp$, and $\ell$. However, at this point the right-hand side of~\eqref{equ:derivation_modul_equ1} still involves $\partial_t \psi$. To remedy this, we use the relation~\eqref{equ:relation_ptphi_dotphi} between $\partial_t \psi$ and $\dotpsi$ to replace any $\partial_t \psi$ terms on the right-hand side of~\eqref{equ:derivation_modul_equ1}. However, some care has to be taken here because the quadratic error term $f$ in~\eqref{equ:relation_ptphi_dotphi} still contains terms involving $\partial_t \psi$, but they are of the form $\calO(\psi, \partial \psi, \ell, \dotwp) \partial_t \psi$.  
We can therefore use the implicit function theorem to infer from the relation~\eqref{equ:relation_ptphi_dotphi} that for sufficiently small $\psi$, $\partial \psi$, $\partial_\Sigma \partial \psi$, $\ell$, and $\dotwp$, we can write 
\begin{equation} \label{equ:derivation_modul_equ1_addon}
 \partial_t \psi = \frac{1}{\sqrt{|h|} (h^{-1})^{00}} \dot{\psi} - \frac{(h^{-1})^{0j}}{(h^{-1})^{00}} \pj \psi - \bfeta \Bigl( \bigl( \dot{\ell} \cdot \nabla_\ell + (\dot{\xi}-\ell) \cdot \nabla_\xi \bigr) \Psi_\wp, N \Bigr) + \calO\bigl( (\partial^{\leq 1}_\Sigma \vecpsi,  \dotwp)^2,  \ell\dotwp\bigr). 
\end{equation}

\begin{remark}
The smallness required for this application of the implicit function theorem will follow from the assumptions on the initial data and our bootstrap assumptions in Section~\ref{sec:bootstrap}, and will be assumed in the remainder of this section.
\end{remark}

Inserting the relation~\eqref{equ:derivation_modul_equ1_addon} into the right-hand side of~\eqref{equ:derivation_modul_equ1} and rearranging, we obtain
\begin{equation} \label{equ:derivation_modul_equ2}
 \begin{aligned}
  \partial_t \vecbfOmega + \vecN = \begin{pmatrix} D + \wtilR & \wtilR \\ \wtilR & D + \wtilR \end{pmatrix} \dotwp + \vec{\wtilH} =: \vecF(\partial^{\leq 2}_\Sigma \vecpsi, \ell, \dotwp),
 \end{aligned}
\end{equation}
where $\wtilR$ and $\wtilH$ are smooth functions of the form 
\begin{equation*}
 \begin{aligned}
  \wtilR &= \calO(\partial^{\leq 2}_\Sigma \vecpsi, \ell, \dotwp), \\
  \vec{\wtilH} &= \calO\bigl( (\partial^{\leq 2}_\Sigma \vecpsi)^2\bigr).
 \end{aligned}
\end{equation*}
As a consequence of the implicit function theorem, there exists a smooth function $\vecG$ such that for small $x$, $y$,  $w$, and $q$, we have $q = \vecF(x,y,w)$ if and only if $w = \vecG(x,y,q)$. Moreover, in view of the structure of $\vecF$ defined in~\eqref{equ:derivation_modul_equ2}, $\vecG$ must then also satisfy (uniformly for all sufficiently small $x$, $y$, and $q$) the estimate 
\begin{equation} \label{equ:derivation_modul_equ_bound_from_IFT}
  |\vecG(x,y,q)| \lesssim |x| + |q|.
\end{equation}
Thus, assuming $\partial_\Sigma^{\leq2} \vec{\psi}$, $\ell$, and $\dotwp$ are sufficiently small (in a pointwise sense) as remarked above, \eqref{equ:derivation_modul_equ2} holds if and only if we have 
\begin{equation} \label{equ:derivation_modul_equ3}
 \begin{aligned}
  \dotwp = \vecG(\partial_\Sigma^{\leq2} \vec{\psi}, \ell, \partial_t \vecbfOmega + \vecN).
 \end{aligned}
\end{equation}
Note that in this equation for $\dotwp$ second-order derivatives of $\vecpsi$ appear in the source term. This has the potential danger that after differentiating in $t$ (that is, for after eventually commuting the equation with the maximum number of $\partial_t$ derivatives), there is a loss of regularity. To avoid this issue, one can try to replace \eqref{equ:derivation_modul_equ3} by a smoothed out version of it. To achieve this, we could try to impose orthogonality conditions for $\vecbfOmega$ that lead to a differential equation of the form
\begin{equation}
 \begin{aligned}
  \dotwp = \vecG( S\partial_\Sigma^{\leq2} \vec{\psi}, \ell, S \vecN ),
 \end{aligned}
\end{equation}
where $S$ is a smoothing operator (in time) that will be defined shortly. While this seems feasible in principle, technically, it seems simpler (see \eqref{equ:derivation_modul_equ5} below) to allow more flexibility in the final differential equation for $\dotwp$, and to let it be of the form
\begin{equation}\label{equ:derivation_modul_equ3_alt}
 \begin{aligned}
  \dotwp = \vecG( S\partial_\Sigma^{\leq2} \vec{\psi}, \ell, S\vecN - \beta \vecomega),
 \end{aligned}
\end{equation}
where $\beta > 0$ and $\beta \vecomega$ is no larger than $\vecN$. 
Comparing with \eqref{equ:derivation_modul_equ3}, this is equivalent to
\begin{equation}
 \begin{aligned}
  \partial_t \vecbfOmega + \vecN = \vecF(\partial_\Sigma^{\leq2} \vec{\psi}, \ell, \vecG( S\partial_\Sigma^{\leq2} \vec{\psi}, \ell, S\vecN - \beta \vecomega) ).
 \end{aligned}
\end{equation}
Using that $\vecF(0,\ell, \vecG(0,\ell, q)) = q$, this is also equivalent to
\begin{equation}
 \begin{aligned}
  \partial_t \vecbfOmega = (S-I) \vecN - \beta \vecomega - \vecF_\omega, 
 \end{aligned}
\end{equation}
where 
\begin{equation}\label{eq:Fomegadef1}
 \begin{aligned}
  \vecF_\omega(\partial_\Sigma^{\leq2} \vec{\psi}, \ell, S\vecN-\beta\vecomega) &:= \vecF(0,\ell, \vecG(0,\ell, S\vecN-\beta\vecomega)) - \vecF(\partial_\Sigma^{\leq2} \vec{\psi}, \ell, \vecG(S\partial_\Sigma^{\leq2} \vec{\psi}, \ell, S\vecN-\beta\vecomega))
 \end{aligned}
\end{equation}
has additional smallness of order $\calO( \partial_\Sigma^{\leq2} \vec{\psi} )$. Indeed, by~\eqref{equ:derivation_modul_equ_bound_from_IFT} it follows from the preceding that 
\begin{equation}\label{eq:Fomegaestimate1}
 \begin{aligned}
  |\vecF_\omega(\partial_\Sigma^{\leq2} \vec{\psi}, \ell, S\vecN-\beta\vecomega)| \lesssim |\partial_\Sigma^{\leq2} \vec{\psi}| \bigl( |S\partial_\Sigma^{\leq2} \vec{\psi}| + |S\vecN-\beta\vecomega| \bigr).
 \end{aligned}
\end{equation}

We now seek to impose a decomposition 
\begin{equation} \label{equ:derivation_modul_equ4}
 \begin{aligned}
  \vecbfOmega = \vecUpomega + \vecomega
 \end{aligned}
\end{equation}
such that 
\begin{equation} \label{equ:derivation_modul_equ5}
 \begin{aligned}
  \partial_t\vecUpomega &= (S-I)(\vecN+\vecF_\omega), \\
  \partial_t\vecomega &= -\beta\vecomega-S\vecF_\omega.
 \end{aligned}
\end{equation}
The motivation for this decomposition can be explained as follows. We view $\vecN$ as the main source term for $\partial_t\vecbfOmega$. As we will see momentarily, the smoothing operator $S$ can be chosen to be almost local in time, so that for instance $S \psi$ and $\psi$ will satisfy comparable decay estimates in time. We will also see that with this choice, we can write $S-I=\partial_t\tilS$ where $\tilS$ is another almost local (but not smoothing) operator, so comparing with \eqref{equ:derivation_modul_equ5} we see that $\vecUpomega$ is of the same order as $\vecN$. This captures the main contribution to $\vecbfOmega$. For the remainder  $\vecomega$ we no longer have the structure $S-I$, but \eqref{equ:derivation_modul_equ3_alt} was conceived so that the equation in \eqref{equ:derivation_modul_equ5} satisfied by $\vecomega$ comes with the damping term $-\beta \vecomega$, which can be used to estimate $\vecomega$ in terms of $S\vecF_\omega$. The details of this argument are presented in Section~\ref{sec:parametercontrol}.
To achieve \eqref{equ:derivation_modul_equ4} and \eqref{equ:derivation_modul_equ5}, we will invoke the implicit function theorem on suitable Banach spaces of time-dependent curves on a time interval $J = [0, \tau_0]$.

We first introduce the definition of the smoothing operator (in time) $S$. Let $k \in C_c^\infty(\bbR)$ be a smooth bump function supported in the interval $[0,1]$ such that $\int_0^1 k(s) \, \ud s = 1$. For a given (locally integrable) function $h(t)$ defined for $t \geq -1$, we set 
\begin{equation*}
  (S h)(t) := \int_\bbR \chi_{[-1,\infty)}(s) h(s) k(t-s) \, \ud s, \quad t \geq -1.
\end{equation*}
Then $(S h)(t)$ is a smooth function for all $t \geq 0$. 
We also define an associated operator $\widetilde{S}$ with the property that $(S-I) h = \frac{\ud}{\ud t} (\widetilde{S} h)$. 
To this end we set $\tilde{k}(r) := 0$ for $r < 0$ and $\tilde{k}(r) := -\int_r^\infty k(s) \, \ud s$ for $r \geq 0$. Note that $\tilde{k}(r)$ is also supported in the interval $[0,1]$. Then the operator 
\begin{equation*}
 (\widetilde{S} h)(t) := \int_\bbR \chi_{[-1,\infty)}(s) h(s) \tilde{k}(t-s) \, \ud s, \quad t \geq -1,
\end{equation*}
has the desired properties.

Let $\Phi$ be a sufficiently regular (for instance $C^5$) solution to HVMC. Given $C^1$ curves $\xi(t)$ and $\ell(t)$ defined on some time interval $J = [0,\tau_0]$ in the domain of definition of $\Phi$, we denote by $\Psi_{\wp}$ the associated profile in the flat region, defined as in~\eqref{equ:definition_profile_flat_region}. We let 
\begin{equation}
 \psi_{\xi,\ell} := \bfeta( \Phi - \Psi_\wp, n_\wp ),
\end{equation}
and correspondingly define $\dot{\psi}_{\xi,\ell}$ as in~\eqref{equ:relation_ptphi_dotphi}. 

First, we trivially extend $\xi(t)$, $\ell(t)$, and $\partial_\Sigma^{\leq 2} \vec{\psi}_{\xi,\ell}$ to times $-1 \leq t \leq 0$. 
Then we define $\vec{\omega}(t)$ to be the solution of
\begin{equation}\label{eq:omegasolutionform1}
 \left\{ \begin{aligned}
  \vec{\omega}(t) &= - \int_0^t e^{-\beta(t-s)} S \vecF_\omega(\partial_\Sigma^{\leq 2} \vec{\psi}_{\xi,\ell}, \ell, S\vecN-\vecomega)(s) \, \ud s, &\quad t \geq 0, \\
  \vec{\omega}(t) &= 0, \quad &t < 0.
 \end{aligned} \right.
\end{equation}
We point out that since $\vecF_\omega$ satisfies the quadratic estimate \eqref{eq:Fomegaestimate1}, the existence of a solution $\vecomega$ to \eqref{eq:omegasolutionform1} can be shown by a fixed point argument.
Observe that $\vec{\omega}$ has additional smallness in view of the definition of $\vecF_\omega$.
Next, in view of \eqref{equ:derivation_modul_equ5}, we define 
\begin{equation}\label{eq:upOmegasolutionform1}
 \vecUpomega(t) := \wtilS \bigl( \vecN+\vecF_\omega \bigr)(t).
\end{equation}
Finally, we define\footnote{To eventually apply the implicit function theorem to $\Upsilon$ we need to specify its domain of definition. This can be done for instance as follows. Let $X=C^5_c(\{|\rho|< \tilR_1\};\bbR^{1+(n+1}))$ denote the space of $5$ times continuously differentiable functions (of $(\rho,\omega)$) supported in $\{|\rho|\leq \tilR_1\}$ for some large $\tilR_1$. We let
%%%%%%%
%%%%%%%
\begin{align*}
\begin{split}
\calX=C([-1,\tau_0];X)\cap C^1([-1,\tau_0];X),\quad \calY=C([-1,\tau_0];\bbR)\cap C^1([-1,\tau_0];\bbR).
\end{split}
\end{align*}
Then we can view $\Upsilon$ as a function
%%%%%%%
%%%%%%%
\begin{align*}
\begin{split}
\Upsilon: \calX\times \calY^{2n}\to \calY^{2n},
\end{split}
\end{align*}
where $\calY^{2n}$ represents the $2n$ components of $\ell$ and $\xi$. Note that by construction, if $\tilR_1$ is chosen such that $\chi$ in the definition of $\vecZ_i$ is supported on $\{|\rho|\leq \tilR_1\}$, then $\Upsilon(\chi \Psi_{0,0},0,0)=0$.
}
\begin{align*}
 \vec\Upsilon(\Phi, \xi, \ell) := \vecbfOmega(\vecpsi_{\xi,\ell})(t) - \vecUpomega(t) - \vecomega(t).
\end{align*}
Observe that by definition 
\begin{equation*}
 \vec\Upsilon(\Psi_{0,0},0,0) = 0,
\end{equation*}
where $\Psi_{0,0}(t,\rho,\omega) = (t, F(\rho,\omega))$ is the parametrization of the standard Lorentzian catenoid.

We now want to show that the Fr\'echet derivative $D_{\xi, \ell} \vec\Upsilon(\Psi_{0,0},0,0)$ is invertible.
Then the existence of the decomposition~\eqref{equ:derivation_modul_equ4} satisfying~\eqref{equ:derivation_modul_equ5} follows from the implicit function theorem under our bootstrap assumptions.
To this end, we observe that by the preceding definitions and in view of the computations in Subsection~\ref{subsec:eigenfunctions}, we have for $1 \leq i, j \leq n$ that 
\begin{equation*}
 \begin{aligned}
  \frac{\delta \bigl( \bfOmega(\vecpsi_{\xi,\ell}, \vecZ_i) \bigr)}{\delta \xi_j} \bigg|_{(\xi, \ell)=(0,0)} &= \bfOmega(\vecvarphi_j, \vecZ_i) \big|_{\ell=0} = 0, \\ 
  \frac{\delta \bigl( \bfOmega(\vecpsi_{\xi,\ell}, \vecZ_i) \bigr)}{\delta \ell_j} \bigg|_{(\xi, \ell)=(0,0)} &= \bfOmega(\vecvarphi_{n+j}, \vecZ_i) \big|_{\ell=0} = - \int \chi \nu^j \nu^i \sqrt{|h|}\big|_{\ell=0} \, \ud \omega \, \ud \rho, \\
  \frac{\delta \bigl( \bfOmega(\vecpsi_{\xi,\ell}, \vecZ_{n+i}) \bigr)}{\delta \xi_j} \bigg|_{(\xi, \ell)=(0,0)} &= \bfOmega(\vecvarphi_j, \vecZ_{n+i}) \big|_{\ell=0} = \int \chi \nu^j \nu^i \sqrt{|h|}\big|_{\ell=0} \, \ud \omega \, \ud \rho, \\
  \frac{\delta \bigl( \bfOmega(\vecpsi_{\xi,\ell}, \vecZ_{n+i}) \bigr)}{\delta \ell_j} \bigg|_{(\xi, \ell)=(0,0)} &= \bfOmega(\vecvarphi_{n+j}, \vecZ_{n+i}) \big|_{\ell=0} = 0.
 \end{aligned}
\end{equation*}
This determines the contributions of $\vecbfOmega(\vecpsi_{\xi,\ell})(t)$ to the Fr\'echet derivative $D_{\xi, \ell} \vec\Upsilon(\Psi_{0,0},0,0)$.
Since $\vecF_\omega$ has additional smallness, $\vec{\omega}$ does not contribute to $D_{\xi, \ell} \vec\Upsilon(\Psi_{0,0},0,0)$, and we only need to examine the part $(\widetilde{S} \vecN)(t)$ in $\vecUpomega(t)$ more carefully. Since $M \vecvarphi_i = 0$ for $1 \leq i \leq n$, and $\vecZ_i = \chi \vecvarphi_i$, we have 
that $M\vecZ_i$ is supported in $\{ |\rho| \simeq \Reigenfunctioncutoffscale\}$ and is of size $\calO(\Reigenfunctioncutoffscale^{-n+1},\ell)$. Thus, we find
\begin{equation*}
 \begin{aligned}
  \frac{\delta \bigl( \bfOmega(\vecpsi_{\xi,\ell}, M \vecZ_i) \bigr)}{\delta \xi_j} \bigg|_{(\xi, \ell)=(0,0)} = \frac{\delta \bigl( \bfOmega(\vecpsi_{\xi,\ell}, M \vecZ_i) \bigr)}{\delta \ell_j} \bigg|_{(\xi, \ell)=(0,0)} = o_\Reigenfunctioncutoffscale(1).
 \end{aligned}
\end{equation*}
Since $M \vecvarphi_{n+i} = \vecvarphi_i$, $1 \leq i \leq n$, we have $M \vecZ_{n+i} = \vecZ_i + \text{error}$ up to an error term that is supported in $\{ |\rho| \simeq \Reigenfunctioncutoffscale \}$ and is of size $\calO(\Reigenfunctioncutoffscale^{-2}, \ell)$. 
Thus,
\begin{equation*}
 \begin{aligned}
    \frac{\delta \bigl( \bfOmega(\vecpsi_{\xi,\ell}, M \vecZ_{n+i}) \bigr)}{\delta \xi_j} \bigg|_{(\ell,\xi)=(0,0)} &= o_\Reigenfunctioncutoffscale(1), \\
  \frac{\delta \bigl( \bfOmega(\vecpsi_{\xi,\ell}, M \vecZ_{n+i}) \bigr)}{\delta \ell_j} \bigg|_{(\ell,\xi)=(0,0)} &= \bfOmega(\vecvarphi_{n+j}, \vecZ_i) \big|_{\ell=0} + o_\Reigenfunctioncutoffscale(1) = - \int \chi \nu^j \nu^i \sqrt{|h|}\big|_{\ell=0} \, \ud \omega \, \ud \rho + o_\Reigenfunctioncutoffscale(1).
 \end{aligned}
\end{equation*}
It follows that the Fr\'echet derivative $D_{\xi, \ell} \vec\Upsilon(\Psi_{0,0},0,0)$ is a map of the form 
\begin{equation} \label{equ:frechet_derivative_form_schematic}
 \begin{aligned}
  \pmat{\delta\xi\\\delta\ell}\mapsto C \begin{bmatrix} 0 & Id \\ Id & Id \end{bmatrix}  \pmat{\delta\xi\\\delta\ell}+ \tilS \tilA  \pmat{\delta\xi\\\delta\ell},
 \end{aligned}
\end{equation}
where $C$ is a constant of order one as in Section~\ref{subsubsec:normal} and $\tilA$ is a time independent matrix of size~$o_\Reigenfunctioncutoffscale(1)$.
Clearly, the map~\eqref{equ:frechet_derivative_form_schematic} is invertible as a linear map of Banach spaces for sufficiently large $\Reigenfunctioncutoffscale$. Thus, in view of~\eqref{equ:frechet_derivative_form_schematic} the Fr\'echet derivative $D_{\xi, \ell} \vec\Upsilon(\Psi_{0,0},0,0)$ is invertible for sufficiently large $\Reigenfunctioncutoffscale$.

\subsection{Controlling the Unstable Mode} \label{subsec:unstableint}

Finally, we need to take into account the exponential instablity caused by the positive eigenvalue of the linearized operator $L$. 
At this point we assume that the modulation parameters $\ell(t)$ and $\xi(t)$ have been determined in terms of $\vecpsi$, so we treat these as given and enact a further decomposition of the perturbation $\vecpsi$.

Starting from the first-order equation~\eqref{equ:first_order_linearized_HVMC} and inserting the relation~\eqref{equ:derivation_modul_equ1_addon} between $\partial_t \psi$ and $\dot{\psi}$ (furnished by the implicit function theorem), we obtain a first-order evolution equation for the perturbation $\vec{\psi}$ of the form
\begin{equation} \label{equ:unstable_mode_first_order_psi_equation}
 (\partial_t - M) \vec{\psi} = \vec{F}_1\bigl(\partial_\Sigma^{\leq 2} \vec{\psi}, \ell, \dotvecwp\bigr).
\end{equation}
Recall from Section~\ref{sec:intro} that the linearized operator $\HHbar := \Delta_\barcalC + |\secondff|^2$ of the Riemannian catenoid has a positive eigenvalue $\mu^2 > 0$ with associated (exponentially decaying) eigenfunction $\varphibar_\mu$.
For the operator $M$ we introduce the time-independent ``almost eigenfunctions'' 
\begin{equation*}
 \vec{Z}_\pm := c_\pm \bigl( \chi \varphibar_\mu, \mp \mu \sqrt{|h|}\big|_{\ell=0} \chi \varphibar_\mu \bigr),
\end{equation*}
where $\chi \in C_c^\infty(\bbR)$ is the previously introduced smooth cut-off to $|\rho| \lesssim \Reigenfunctioncutoffscale$ and where $c_\pm$ are normalization constants such that 
\begin{equation*}
 \bfOmega(\vec{Z}_+, \vec{Z}_-) = - \bfOmega(\vec{Z}_-, \vec{Z}_+) = 1.
\end{equation*}
Then we have 
\begin{equation*}
 M \vecZ_\pm = \pm \mu \vecZ_\pm + \calE_\pm,
\end{equation*}
where the errors $\calE_\pm$ consist of terms that are supported around $|\rho| \simeq \Reigenfunctioncutoffscale$ or that have additional smallness in terms of the parameter $\ell$. 

We now enact a decomposition of $\vec{\psi}$ into 
\begin{equation} \label{equ:unstable_mode_decomposition}
 \vec{\psi} = \vec{\phi} + a_+ \vec{Z}_+ + a_- \vec{Z}_-,
\end{equation}
where the time-dependent parameters $a_+(t)$ and $a_-(t)$ will be defined shortly by imposing suitable orthogonality conditions. 
Inserting~\eqref{equ:unstable_mode_decomposition} into \eqref{equ:unstable_mode_first_order_psi_equation}, we find
\begin{align}\label{eq:wpoutline6}
\begin{split}
(a_{+}'-\mu a_{+})\vecZ_{+}+(a_{-}'+\mu a_{-})\vecZ_{-} &= a_{+}(M\vecZ_{+}-\mu\vecZ_{+})+a_{-}(M\vecZ_{-}+\mu\vecZ_{-}) \\
&\quad \quad -(\partial_t-M)\vecphi + \vecF_1(\partial_\Sigma^{\leq 2} \vec{\psi}, \ell, \dot{\vecwp}).
\end{split}
\end{align}
Note that the dependence of $\vecF_1(\partial_\Sigma^{\leq 2} \vec{\psi}, \ell, \dot{\vecwp})$ on $\vecpsi$ comes with additional smallness. Thus, if we replace $\vecpsi$ in the definition of $\vecF_1(\partial_\Sigma^{\leq 2} \vec{\psi}, \ell, \dot{\vecwp})$ in terms of $\vecphi$ and $a_{\pm}$, the terms involving $a_{\pm}$ come with extra smallness. Also note that at this point we have already determined the parameters $\ell(t)$ and $\xi(t)$, so we treat these as given. 
Now taking the $\bfOmega$ inner product of \eqref{eq:wpoutline6} with $\vecZ_{-}$, multiplying by $e^{-\mu t}$, and recalling that 
%%%%%%%
%%%%%%%
\begin{align*}
\begin{split}
\bfOmega(M\vecphi,\vecZ_{-})=-\bfOmega(\vecphi,M\vecZ_{-})=\mu\bfOmega(\vecphi,\vecZ_{-})-\bfOmega(\vecphi,M\vecZ_{-}+\mu\vecZ_{-}),
\end{split}
\end{align*}
we get
%%%%%%%
%%%%%%%
\begin{align}\label{eq:wpoutline7}
\begin{split}
\frac{\ud}{\ud t}\big(e^{-\mu t}a_{+}\big)=-\frac{\ud}{\ud t}\big(e^{-\mu t}\bfOmega(\vecphi,\vecZ_{-})\big)-e^{-\mu t}F_{+},
\end{split}
\end{align}
where
%%%%%%%
%%%%%%%
\begin{align}\label{eq:Fplusdef}
\begin{split}
-F_{+} &:= \bfOmega(\vecF_1,\vecZ_{-})-\bfOmega(\vecphi,M\vecZ_{-}+\mu\vecZ_{-})+a_{+}\bfOmega(M\vecZ_{+}-\mu\vecZ_{+},\vecZ_{+})\\
&\quad+a_{-}\bfOmega(M\vecZ_{-}+\mu\vecZ_{-},\vecZ_{-}).
\end{split}
\end{align}
Similarly, taking the $\bfOmega$ inner product of \eqref{eq:wpoutline6} with $\vecZ_{+}$ and multiplying by $e^{\mu t}$, we get
%%%%%%%
%%%%%%%
\begin{align}\label{eq:wpoutline8}
\begin{split}
\frac{\ud}{\ud t}\big(e^{\mu t}a_{-}\big)=\frac{\ud}{\ud t}\big(e^{\mu t}\bfOmega(\vecphi,\vecZ_{+})\big)+e^{\mu t}F_{-},
\end{split}
\end{align}
where
%%%%%%%
%%%%%%%
\begin{align*}
\begin{split}
-F_{-}=\bfOmega(\vecF_1,\vecZ_{+})-\bfOmega(\vecphi,M\vecZ_{+}-\mu\vecZ_{+})+a_{+}\bfOmega(M\vecZ_{+}-\mu\vecZ_{+},\vecZ_{+})+a_{-}\bfOmega(M\vecZ_{-}+\mu\vecZ_{-},\vecZ_{+}).
\end{split}
\end{align*}
Motivated by \eqref{eq:wpoutline7} and \eqref{eq:wpoutline8}, we require the orthogonality conditions 
%%%%%%%
%%%%%%%
\begin{align} \label{eq:wpoutline9}
\begin{split}
\bfOmega(\vecphi,\vecZ_{-}) = e^{\mu t}\tilS(e^{-\mu t}F_{+}), \qquad \bfOmega(\vecphi,\vecZ_{+}) = e^{-\mu t}\tilS(e^{\mu t}F_{-}).
\end{split}
\end{align}
In view of \eqref{eq:wpoutline7} and \eqref{eq:wpoutline8} and recalling that $\partial_t \tilS = S-Id$, the orthogonality conditions~\eqref{eq:wpoutline9} lead to the following equations for $a_{\pm}$
%%%%%%%
%%%%%%%
\begin{align}\label{eq:wpoutline10}
\begin{split}
\frac{\ud}{\ud t}\big(e^{-\mu t}a_{+}\big)=-S(e^{-\mu t}F_{+}),\qquad \frac{\ud}{\ud t}\big(e^{\mu t}a_{-}\big)=S(e^{\mu t}F_{-}).
\end{split}
\end{align}
Finally note that derivatives commute nicely with \eqref{eq:wpoutline9} in the sense that (using the product rule and the fact that $[\frac{\ud}{\ud t},\tilS]=0$)
%%%%%%%
%%%%%%%
\begin{align}\label{eq:wpoutline11}
\begin{split}
\frac{\ud}{\ud t}\bfOmega(\vecphi,\vecZ_{-})=e^{\mu t}\tilS(e^{-\mu t}F_{+}'),\qquad \frac{\ud}{\ud t}\bfOmega(\vecphi,\vecZ_{+})=e^{-\mu t}\tilS(e^{\mu t}F_{-}'),
\end{split}
\end{align}
and similarly for higher derivatives.

To conclude, we briefly explain how to obtain the decomposition~\eqref{equ:unstable_mode_decomposition} satisfying the orthogonality conditions \eqref{eq:wpoutline9}.
Given $\vec{\psi}$ and $a_\pm(t)$ on some time interval $J = [0,\tau_0]$, we first trivially extend these to times $-1 \leq t \leq 0$. 
Then we consider 
\begin{equation*}
 \Upsilon_\mu(\vec{\psi}, a_+, a_-) := \begin{pmatrix} \bfOmega\bigl( \vec{\psi} - a_+ \vec{Z}_+ - a_- \vec{Z}_-, \vec{Z}_- \bigr) - e^{\mu t}\tilS(e^{-\mu t}F_{+}) \\ \bfOmega\bigl( \vec{\psi} - a_+ \vec{Z}_+ - a_- \vec{Z}_-, \vec{Z}_+ \bigr) - e^{-\mu t}\tilS(e^{\mu t}F_{-}) \end{pmatrix}.
\end{equation*}
Note that the orthogonality conditions~\eqref{eq:wpoutline9} are equivalent to $\Upsilon_\mu(\vec{\psi}, a_+, a_-) = (0,0)$. We have $\Upsilon_\mu(0,0,0) = 0$ and we compute the Fr\'echet derivative
\begin{equation*}
 D_{a_+, a_-}\Upsilon_\mu(0,0,0) = \begin{pmatrix} -1 & 0 \\ 0 & 1 \end{pmatrix} + o_{\Reigenfunctioncutoffscale,\ell}(1).
\end{equation*}
Hence, for given $\vec{\psi}$ satisfying suitable bootstrap assumptions, the existence of the decomposition~\eqref{equ:unstable_mode_decomposition} obeying the orthogonality conditions~\eqref{eq:wpoutline9} follows (for sufficiently large $\Reigenfunctioncutoffscale$) from the implicit function theorem.

%%%%%%%%%%%%%%%%%%%%%%
%%%%%%%%%%%%%%%%%%%%%%
%%%%%%%%%%%%%%%%%%%%%%
\section{Coordinates, Vectorfields, and a More Precise Description of the Profile} \label{sec:coordinates}
%%%%%%%%%%%%%%%%%%%%%%
%%%%%%%%%%%%%%%%%%%%%%
%%%%%%%%%%%%%%%%%%%%%%

Given $\ell$,$\xi$, and $a_{\pm}$, we give a more detailed description of the foliation and the profile. Moreover, we obtain various expressions for the linear operator acting on $\phi$. Our starting point is to derive a parameterization of the profile.
%%%%%%%%%%%%%%
%%%%%%%%%%%%%%
\subsection{Parameterization of the Profile}\label{sec:profile2}
%%%%%%%%%%%%%%
%%%%%%%%%%%%%%
We give separate parameterizations of the profile $\cup_\sigma \Sigma_\sigma$ in the interior and exterior regions.
In the interior,  our parameterization is the same as in the first order formulation in Section~\ref{sec:interior}. That is, we parameterize the profile as
%%%%%%%
%%%%%%%
\begin{align*}
\begin{split}
(t,\xi+\gamma^{-1}P_\ell F(\rho,\omega)+P_\ell^\perp F(\rho,\omega)),
\end{split}
\end{align*}
where $\ell$, $\xi$, and $\gamma$ are functions of $t$. According to the definition of the profile in Section~\ref{sec:profileintro}, and with the notation used there, this parameterization is valid in the flat region $\calC_{\flatt}=\{X^0\geq \sigma_\temp(X)+\delta_1\}$, where as usual $X=(X^0,\dots X^{n+1})$ are the rectangular coordinates in the ambient $\bbR^{1+(n+1)}$.

In the exterior we eventually want to parameterize the VMC surface as a graph over a hyperplane, so we start by parameterizing the profile itself as a graph. For this, let the function $x^0(\sigma,x')$ be defined by the requirement that $(x^0(\sigma,x'),x')\in \barbsUpsigma_\sigma$. Note that in the hyperboloidal region $\calC_{\hyp}=\{X^0\leq \sigma_\temp(X)-\delta_1\}$
%%%%%%%
%%%%%%%
\begin{align}\label{eq:x0hyptemp1}
\begin{split}
x^0(\sigma,x')=\sigma-\gamma R+\sqrt{|x'-\xi+\gamma R\ell|^2+1},
\end{split}
\end{align}
while in the flat region $\{X^0\geq \sigma_\temp(\sigma)+\delta_1\}$
%%%%%%%
%%%%%%%
\begin{align*}
\begin{split}
x^0(\sigma,x')=\sigma.
\end{split}
\end{align*}
The expression for $x^0$ in the intermediate region $\{|X^0-\sigma_{\temp}(X)|<\delta_1\}$ is not explicit and depends on the choice of the smoothed out minimum function $\frakm$ in Section~\ref{sec:profileintro}. With $x^0$ determined, we define the function $\calQ(x^0,x')$ by the requirement that $(x^0(\sigma,x'),x',\calQ(x^0(\sigma,x'),x'))\in\Sigma_\sigma$. The map
%%%%%%%
%%%%%%%
\begin{align*}
\begin{split}
(\sigma,x')\mapsto (x^0(\sigma,x'),x',\calQ(x^0(\sigma,x'),x'))
\end{split}
\end{align*}
is then a parameterization of the profile in the exterior region $\{|x'|\gg1\}$ (more precisely, this parameterization is valid in a neighborhood of the support of $1-\chi$ in the definition \eqref{eq:tilNdefintro1} of $N$). We now want to derive more explicit expressions for $\calQ$ in the flat and hyperboloidal parts of $\bsUpsigma_\sigma$. First, for reasons that will become clear momentarily, we let
%%%%%%%
%%%%%%%
\begin{align*}
\begin{split}
\eta = \xi-\gamma R\ell,
\end{split}
\end{align*}
and define the non-geometric polar coordinates $(\tau,r,\theta)$ by
%%%%%%%
%%%%%%%
\begin{align*}
\begin{split}
\tau = \sigma-\gamma(\sigma)R\qquad \mand\qquad x'=\eta(\tau)+r\Theta(\theta).
\end{split}
\end{align*}
Here $\Theta$ denotes the standard parameterization of $\bbS^{n-1}\subseteq\bbR^n$, and here and in what follows, by a slight abuse of notation, we simply write $\eta(\tau)$ for $\eta(\sigma(\tau))$ (and similarly for other parameters $\ell$, $\xi$, etc.). Differentiation with respect to $\tau$ is denoted by a prime, and differentiation with respect to $\sigma$ by a dot, so for instance
%%%%%%%
%%%%%%%
\begin{align*}
\begin{split}
\eta' = \frac{\ud }{\ud \tau}\eta(\tau)= \frac{\ud\sigma}{\ud \tau}\frac{\ud}{\ud \sigma}\eta(\sigma)\vert_{\sigma=\sigma(\tau)},\qquad \doteta= \frac{\ud}{\ud \sigma}\eta(\sigma).
\end{split}
\end{align*}
Note that
%%%%%%%
%%%%%%%
\begin{align*}
\begin{split}
\frac{\ud\sigma}{\ud \tau}=1-\gamma' R\simeq 1.
\end{split}
\end{align*}
In the hyperboloidal region $\calC_\hyp$ we have
%%%%%%%
%%%%%%%
\begin{align*}
\begin{cases}
x^0=\tau+ \jap{r}\\
x'=\eta(\tau)+r\Theta(\theta)
\end{cases},
\end{align*}
while in the flat region $\calC_\flatt$
%%%%%%%
%%%%%%%
\begin{align*}
\begin{cases}
x^0=\tau+ \gamma R\\
x'=\eta(\tau)+r\Theta(\theta)
\end{cases}.
\end{align*}
In general, our parameterization of the profile in polar coordinates becomes
%%%%%%%
%%%%%%%
\begin{align*}
\begin{split}
(x^0(\tau,r),\eta(\tau)+r\Theta(\theta),\calQ(x^0(\tau,r),\eta(\tau)+r\Theta(\theta))).
\end{split}
\end{align*}
Now to motivate our definition of the polar coordinates, we investigate the form of $\calQ$ in the hyperboloidal region more closely. Note that if $(x^0,x')\in \Sigma_\sigma$, is given by (recall the definition of $\tilbsUpsigma_\sigma$ from Section~\ref{sec:profileintro})
%%%%%%%
%%%%%%%
\begin{align*}
\begin{split}
\pmat{x^0\\ x'}=\Lambda_{-\ell}\pmat{y^0\\y'}+\pmat{0\\\xi-\sigma\ell},
\end{split}
\end{align*}
then using \eqref{eq:x0hyptemp1}, 
%%%%%%%
%%%%%%%
\begin{align*}
\begin{split}
&y^0=\gamma^{-1}(\sigma-\gamma R) +\gamma(\sqrt{1+|x'-(\xi-\gamma R\ell)|^2}-(x'-(\xi-\gamma R\ell))\cdot \ell),\\
&y'= A_\ell(x'-(\xi-\gamma R\ell))-\gamma \sqrt{1+|x'-(\xi-\gamma R\ell)|^2} \ell,
\end{split}
\end{align*}
and the profile parameterization becomes 
%%%%%%%
%%%%%%%
\begin{align}\label{eq:graphpar1}
\begin{split}
(x^0,x',Q(A_\ell(x'-\eta(\sigma))-\gamma\ell \sqrt{1+|x'-\eta(\sigma)|^2})).
\end{split}
\end{align}
Here $Q(y)=Q(|y|)$ is given by the parameterization of the Riemannian catenoid and satisfies the ODE (after identifying it with a function of a single variable)
%%%%%%%
%%%%%%%
\begin{align}\label{eq:QRiem1}
\begin{split}
Q''(\tilr)+\frac{n-1}{\tilr}Q'(\tilr)-(1+(Q'(\tilr))^2)^{-1}(Q'(\tilr))^2Q''(\tilr)=0.
\end{split}
\end{align}
Therefore, in our polar coordinates these expressions take the simple forms
%%%%%%%
%%%%%%%
\begin{align}\label{eq:yx1}
\begin{split}
&y^0=\gamma^{-1}\tau +\gamma\jap{r}(1-\frac{r}{\jap{r}}\Theta\cdot\ell),\\
&y'=A_\ell(r\Theta-\jap{r}\ell)=rA_\ell\Theta-\gamma \jap{r}\ell,
\end{split}
\end{align}
and
%%%%%%%
%%%%%%%
\begin{align}\label{eq:extpar1}
\begin{split}
(\tau+\jap{r},\eta(\tau)+r\Theta,Q(rA_\ell\Theta-\gamma\jap{r}\ell)).
\end{split}
\end{align}
Sometimes we write $Q_\wp$ for $Q(rA_\ell\Theta-\gamma\jap{r}\ell)$ to emphasize the dependence on the parameters. For future use, we also record the following coordinate change formulas in the hyperboloidal region:
%%%%%%%
%%%%%%%
\begin{align}\label{eq:extchangeofvars1}
\begin{split}
&\partial_\tau = \partial_{x^0}+\eta',\\
&\partial_r = \frac{r}{\jap{r}}\partial_{x^0}+\Theta,\\
&\partial_a= r\Theta_a, ~a=1,\dots n-1,
\end{split}
\end{align}
and
%%%%%%%
%%%%%%%
\begin{align}\label{eq:extchangeofvars2}
\begin{split}
&\partial_{x^0}=\frac{1}{1-\frac{r}{\jap{r}}\Theta\cdot\eta'}\partial_\tau-\frac{\eta'\cdot\Theta}{1-\frac{r}{\jap{r}}\Theta\cdot\eta'}\partial_r-\frac{(\ringsg^{-1})^{ab}\Theta_b\cdot\eta'}{r(1-\frac{r}{\jap{r}}\Theta\cdot\eta')}\partial_a,\\
&\partial_{x^j}=-\frac{r\Theta^j}{\jap{r}(1-\frac{r}{\jap{r}}\Theta\cdot\eta')}\partial_\tau+\frac{\Theta^j}{1-\frac{r}{\jap{r}}\Theta\cdot\eta'}\partial_r+\Big(\frac{(\ringsg^{-1})^{ab}\Theta_b^j}{r}+\frac{(\ringsg^{-1})^{ab}\Theta_b\cdot\eta' \Theta^j}{\jap{r}(1-\frac{r}{\jap{r}}\Theta\cdot\eta')}\Big)\partial_a.
\end{split}
\end{align}

Next, we want to define the rotation and outgoing and incoming null vectorfields $\Omega$, $L$, $\Lbar$ and  the geometric radial function $\tilr$, corresponding to the Minkowski metric, at a point on $\barbsUpsigma_\sigma$. These will play an important role in the analysis in the exterior region, and will appear in our bootstrap assumptions in Section~\ref{sec:bootstrap} below. Let $(x^0,x')=(\tau+\jap{r},r\Theta+\eta)$ be a point on the hyperboloidal part of $\barbsUpsigma_\sigma$. With $\tau=\tau(\sigma)$, let $(y^0,y')$ be related to $(x^0,x')$ by \eqref{eq:yx1}. Note that by construction
%%%%%%%
%%%%%%%
\begin{align*}
\begin{split}
\pmat{x^0\\x'}= \Big(\Lambda_{-\ell}\pmat{y^0\\y'} +\pmat{0\\-\sigma\ell(\sigma)+\xi(\sigma)}\Big)\Big{\vert}_{\sigma=\sigma(\tau)}.
\end{split}
\end{align*}
We define our $L$, $\Lbar$, $\Omega$, and $T$ as the push forward by $\Lambda_{-\ell}$ of the corresponding vectorfields in the $y$ coordinates. That is, 
%%%%%%%
%%%%%%%
\begin{align}\label{eq:VFdef0}
\begin{split}
& T=\Lambda_{-\ell}\partial_{y^0},\\
&L= \Lambda_{-\ell}(\partial_{y^0}+\frac{y^i}{|y'|}\partial_{y^i}),\\
&\Lbar = \Lambda_{-\ell}(\partial_{y^0}-\frac{y^i}{|y'|}\partial_{y^i}),\\
&\Omega_{jk}=\Lambda_{-\ell}(y^j\partial_{y^k}-y^k\partial_{y^j}).
\end{split}
\end{align}
Using \eqref{eq:extchangeofvars2} we can find the following coordinate representations for these vectorfields:
%%%%%%%
%%%%%%%
\begin{align}\label{eq:VFexpansion1}
\begin{split}
&T=\callO(1)\partial_\tau+\callO(\dotwp)\partial_r+\callO(\dotwp r^{-1})\partial_a,\\
&L=\callO(r^{-2})\partial_\tau+\callO(1)\partial_r+\callO(r^{-3})\partial_a,\\
&\Lbar = \callO(1)\partial_\tau+\callO(1)\partial_r+\callO(\dotwp r^{-1}+r^{-3})\partial_a,\\
&\Omega_{jk}=\callO(\wp r)\partial_r+\callO(1) \partial_a.
\end{split}
\end{align}
Using the notation $\spartial^a=(\ringsg^{-1})^{ab}\partial_b$, for $T$ and $L$ also record the following more precise expressions
%%%%%%%
%%%%%%%
\begin{align}\label{eq:Tprecise1}
\begin{split}
T&= \gamma\Big(1+\frac{r\Theta\cdot(\eta'-\ell)}{\jap{r}(1-\frac{r}{\jap{r}}\Theta\cdot\eta')}\Big)\partial_\tau-\frac{\gamma(\eta'-\ell)\cdot\Theta}{1-\frac{r}{\jap{r}}\Theta\cdot\eta'}\partial_r\\
&\quad-\frac{\gamma}{r(1-\frac{r}{\jap{r}}\Theta\cdot\eta')}\big((\Theta\cdot\ell)\spartial^a\Theta\cdot(\eta'-\ell)-(\spartial^a\Theta\cdot\ell)\Theta\cdot(\eta'-\ell)-\spartial^a\Theta\cdot(\eta'-\ell)\big)\partial_a,\\
\end{split}
\end{align} 
and
%%%%%%%
%%%%%%%
\begin{align}\label{eq:Lprecise1}
\begin{split}
L&=(\frac{1}{2}\gamma^{-1}(1-\Theta\cdot\ell)^{-1}(1-\Theta\cdot\eta')^{-1}r^{-2}+\callO(r^{-4}))\partial_\tau\\
&\quad+(\gamma^{-1}(1-\Theta\cdot\ell)^{-1}+\callO(r^{-2}))\partial_r+\callO(r^{-3})\partial_a.
\end{split}
\end{align}
By inverting these relations we can also express the the coordinate derivatives in the $(\tau,r,\theta)$ coordinates in terms of $L$, $\Lbar$, and $\Omega$:
%%%%%%%
%%%%%%%
\begin{align}\label{eq:VFbasis1}
\begin{split}
&\partial_\tau=\callO(1)L+\callO(1)\Lbar,\\
&\partial_r= \callO(1)L+\callO(r^{-2})\Lbar+\callO( r^{-3})\Omega,\\
&\partial_\theta=\callO(\wp r)L+\callO(1)\Omega+\callO(\wp r^{-1})\Lbar.
\end{split}
\end{align}
Similarly, the geometric radial function is defined in terms of the $y$ variables as $\tilr=|y'|$ which in radial coordinates reads
%%%%%%%
%%%%%%%
\begin{align}\label{eq:tilrprecise1}
\begin{split}
\tilr =|rA_\ell\Theta-\gamma \jap{r}\ell|=\gamma(1-\Theta\cdot\ell)r+\callO(r^{-1}).
\end{split}
\end{align}
In particular $\tilr\simeq r$. 
%%%%%%%%%%
%%%%%%%%%%%
%%%%%%%%%%%
\subsection{Parameterization of the VMC Surface and Derivation of the Equations}
%%%%%%%%%%%
%%%%%%%%%%%
Recall from Section~\ref{sec:profileintro}, that to derive a parameterization of the VMC surface we first introduced an almost-normal vectorfield 
%%%%%%%
%%%%%%%
\begin{align*}
\begin{split}
N:\cup_\sigma \Sigma_\sigma\to \bbR^{1+(n+1)},
\end{split}
\end{align*}
and then defined
%%%%%%%
%%%%%%%
\begin{align*}
\begin{split}
\psi:\cup_{\sigma}\Sigma_\sigma\to \bbR
\end{split}
\end{align*}
by the requirement that $p+\psi(p)N(p)\in \calM$ for all $p\in \cup_{\sigma}\Sigma_\sigma$ (see also Lemma~\ref{rem:normalneighborhood}).  Then in Section~\ref{sec:interior} we further decomposed $\psi$ as
%%%%%%%
%%%%%%%
\begin{align*}
\begin{split}
\psi=\phi+a_\mu Z_\mu,\qquad a_\mu:=a_{+}+a_{-}.
\end{split}
\end{align*}
%%%%%%%%%%%%
%%%%%%%%%%%%
In view of the compact support of $Z_\mu$, the functions $\phi$ and $\psi$ agree outside a compact region of $\calC_{\flatt}$.  In the hyperboloidal region  $\calC_{\hyp}$  we will sometimes work with the following renormalized version of $\phi$:
%%%%%%%
%%%%%%%
\begin{align*}
\begin{split}
\varphi:=\angles{\phi N}{ \partial_{X^{n+1}}}.
\end{split}
\end{align*}
The advantage of $\varphi$ is that it satisfies a simpler equation, while the advantage of $\phi$ is that the linear part of the equation satisfied by it has the form which is familiar from the second variation of the area functional. The exact relation between $\phi$ and $\varphi$ can be calculated as follows. In the region  $\calC_{\hyp}$ the normal $n_\wp$ is given by
%%%%%%%
%%%%%%%
\begin{align}\label{eq:nwpext1}
\begin{split}
n_\wp= (1+(m^{-1})^{\mu\nu} Q_\mu Q_\nu)^{-\frac{1}{2}}((-m^{-1})^{\mu\nu}Q_\mu \partial_\nu,1).
\end{split}
\end{align}
Here indices run over the coordinates $(x^0,x')$, and (see \eqref{eq:graphpar1})
%%%%%%%
%%%%%%%
\begin{align*}
\begin{split}
Q_\nu = Q'\big(A_\ell(x'-\eta(\sigma))-\gamma\ell \sqrt{1+|x'-\eta(\sigma)|^2}\big)\partial_\nu \big(A_\ell(x'-\eta(\sigma))-\gamma\ell \sqrt{1+|x'-\eta(\sigma)|^2}\big),
\end{split}
\end{align*}
with the convention, as in Section~\ref{sec:interior}, that $\partial_{x^0}\eta =\ell$, $\partial_{x^0}\ell=0$, $\partial_{x'}\ell=\partial_{x'}\eta=0$.
It follows from the normalization $\bfeta(N,n_\wp)=1$ that in this region 
%%%%%%%
%%%%%%%
\begin{align*}
\begin{split}
N= (1+(m^{-1})^{\mu\nu} Q_\mu Q_\nu)^{\frac{1}{2}}\frac{\partial}{\partial X^{n+1}},
\end{split}
\end{align*}
and hence
%%%%%%%
%%%%%%%
\begin{align}\label{eq:varphiphi1}
\begin{split}
\varphi =  s\phi,\qquad s:=(1+(m^{-1})^{\mu\nu} Q_\mu Q_\nu)^{\frac{1}{2}}.
\end{split}
\end{align}
Since $(m^{-1})^{\mu\nu} Q_\mu Q_\nu=O(r^{-2(n-	1)})$, the relation \eqref{eq:varphiphi1} implies that the various energy norms for $\phi$ and $\varphi$ are equivalent.

In the remainder of this section we give more explicit expressions for $N$ using the parameterizations introduced in the previous section, and derive the equations satisfied by $\psi$, $\phi$, and $\varphi$ in the respective coordinates. Finally, we introduce a set of global coordinates and discuss the structure of the linearized operator in these various coordinate systems.
%%%%%%%%%%%
%%%%%%%%%%% 
\subsubsection{Interior Non-Geometric Coordinates}\label{sec:INGC}
%%%%%%%%%%%
%%%%%%%%%%%
Here $N$ is defined to lie on the $\{X^0=\mathrm{constant}\}$ hypersurfaces of the ambient space, and the equations satisfied by $\psi$ and $\phi$ can be read off from the first order formulation. In particular, according to \eqref{eq:phi1}, in the coordinate system $(t,\rho,\omega)$ introduced in Section~\ref{sec:interior} and in the region where $\chi\equiv1$ in \eqref{eq:tilNdefintro1}, the linear part of the equation satisfied by $\phi$ can be written as
%%%%%%%
%%%%%%%
\begin{align}\label{eq:LEDcalPint1}
\begin{split}
\calP \phi=\frac{1}{\sqrt{|h|}}\partial_\mu(\sqrt{|h|}(h^{-1})^{\mu\nu}\partial_\nu\phi)+V\phi+a^{\mu\nu}\partial^2_{\mu\nu}\phi+b^\mu\partial_\mu\phi+c\phi,
\end{split}
\end{align}
where $a$ is symmetric and $a,b,c=\callO(\dotwp^{\leq 2})$, and $V=|\secondff|^2$.

Besides the coordinates introduced above, in the flat part of the foliation we will often use the coordinates $(\tilt,\tilrho,\tilomega)$ defined by (note that this is a valid change of variables in the flat region where $\rho$ is bounded)
%%%%%%%
%%%%%%%
\begin{align}\label{eq:ttilt1}
\begin{split}
\tilt=\gamma^{-1}(t)t-\ell(t)\cdot F(\rho,\omega),\quad \tilrho=\rho,\quad \tilomega=\omega.
\end{split}
\end{align}
The corresponding coordinate derivatives are related as follows:
%%%%%%%
%%%%%%%
\begin{align}\label{eq:ttilt2}
\begin{split}
&\partial_t=(1+\upkappa)\partial_\tilt,\quad \partial_\rho=\ell\cdot F_\rho\partial_\tilt+\partial_\tilrho,\quad \partial_\omega= \ell\cdot F_\omega\partial_\tilt+\partial_\tilomega,\\
&\partial_\tilt=(1+\upkappa)^{-1}\partial_t,\quad \partial_\tilrho=-(1+\upkappa)^{-1}\ell\cdot F_\rho\partial_t+\partial_\rho,\quad \partial_\tilomega=-(1+\upkappa)^{-1}\ell\cdot F_\omega\partial_t+\partial_\omega,
\end{split}
\end{align}
where
%%%%%%%
%%%%%%%
\begin{align}\label{eq:kappadef1}
\begin{split}
\upkappa:=t\frac{\ud}{\ud t}\gamma^{-1}-\dotell\cdot F.
\end{split}
\end{align}
Note that by these relations $\det\frac{\partial(\tilt,\tilrho,\tilomega)}{\partial(t,\rho,\omega)}=1+\upkappa$. In the calculations in the flat region we will often use both the $(t,\rho,\omega)$ and the $(\tilt,\tilrho,\tilomega)$ coordinates. To emphasize which coordinate system is being used in each calculation, we will use a tilde to indicate that calculations are being done in the $(\tilt,\tilrho,\tilomega)$ coordinates. For instance, we write $h_{\mu\nu}$ for the components of $h$ in the $(t,\rho,\omega)$ coordinates and $\tilh_{\mu\nu}$ for its components in the $(\tilt,\tilrho,\tilomega)$ coordinates, and similarly for $|h|$ and $|\tilh|$.
%%%%%%%%%%%
%%%%%%%%%%% 
\subsubsection{Exterior Non-Geometric Coordinates}\label{sec:ENGC}
%%%%%%%%%%%
%%%%%%%%%%%
Since $\cup_{\sigma}\Sigma_\sigma$ can be parameterized by $\cup_{\sigma}\barbsUpsigma_\sigma$, by a slight abuse of notation we will often view $N$, $\nu$, $\phi$, and $\varphi$ as functions on $\cup_{\sigma}\barbsUpsigma_\sigma$ in the exterior region, and use coordinates, such as $(x^0,x')$ or $(\tau,r,\theta)$, as their arguments. In the graph formulation, the requirement that $\calM$ be a VMC surface is equivalent to the following PDEs for $u= Q_\wp+\varphi$:
%%%%%%%
%%%%%%%
\begin{align}\label{eq:uMinkowskieq1}
\begin{split}
\nabla_\mu\Big(\frac{\nabla^\mu u}{\sqrt{1+\nabla^\alpha u \nabla_\alpha u}}\Big)=\frac{1}{\sqrt{|m|}}\partial_\mu\Big(\frac{\sqrt{|m|}(m^{-1})^{\mu\nu}\partial_\nu u}{\sqrt{1+(m^{-1})^{\alpha\beta}\partial_\alpha u\partial_\beta u}}\Big)=0.
\end{split}
\end{align}
Here $m$ denotes the Minkowski metric
%%%%%%%
%%%%%%%
\begin{align*}
\begin{split}
m=-\ud x^0\otimes \ud x^0+\sum_{i=1}^n \ud x^i\otimes \ud x^i,
\end{split}
\end{align*}
and $\nabla$ the corresponding covariant derivative. The equation for $u$ can be expanded as
%%%%%%%
%%%%%%%
\begin{align}\label{eq:uext1}
\begin{split}
\Box_m u -(1+\nabla^\alpha u \nabla_\alpha u )^{-1}(\nabla^\nu u )(\nabla^\mu u) \nabla_\nu \nabla_\mu u=0.
\end{split}
\end{align}
Plugging in the decomposition for $u$ we arrive at the following equation for $\varphi$ and $\phi$ (see \eqref{eq:varphiphi1}):
%%%%%%%
%%%%%%%
\begin{align}\label{eq:abstractexteq1}
\begin{split}
s\calP\phi=\callP_\graph\varphi=\sum_{i=0}^3\calF_i,\qquad \calP:=\calP_\graph+s^{-1}[\calP_\graph,s].
\end{split}
\end{align}
Here $\calF_i$ denotes inhomogeneous terms of order $i$ in $\varphi$. A more explicit expression for $\calP$ is derived in \eqref{eq:extgeomlin1} below. One advantage of working with $\varphi$ rather than $\phi$ is that $\calF_i$ on the right-hand side are easier to compute. The source term $\calF_0$,which is independent of $\varphi$ (but depends on the derivatives of the parameters), is calculated in Lemma~\ref{lem:sourceext1} below.  The linear operator $\calP$ is given by (here $Q\equiv Q_{\wp}$)
%%%%%%%
%%%%%%%
\begin{align}\label{eq:callP1}
\begin{split}
\callP_\graph&=\Box_m-(1+\nabla^\alpha Q \nabla_\alpha Q)^{-1}(\nabla^\mu Q)(\nabla^\nu Q)\nabla_{\mu}\nabla_\nu-2(1+\nabla^\alpha Q \nabla_\alpha Q)^{-1}(\nabla^{\mu}\nabla^\nu Q)(\nabla_\mu Q)\nabla_\nu\\
&\quad+2(1+\nabla^\alpha Q\nabla_\alpha Q)^{-2}(\nabla^\nu Q)(\nabla^\mu Q)(\nabla^\lambda Q)(\nabla_\lambda \nabla_\mu Q)\nabla_\nu.
\end{split}
\end{align}
The quadratic and cubic terms are (here $Q\equiv Q_{\wp}$)
%%%%%%%
%%%%%%%
\begin{align}\label{eq:calF2_1}
\begin{split}
\calF_2&=-\frac{2\nabla^\mu Q\nabla^\nu \varphi\nabla^2_{\mu\nu}\varphi}{1+\nabla^\alpha u \nabla_\alpha u}-\frac{\nabla^2_{\mu\nu}Q\nabla^\mu\varphi\nabla^\nu\varphi}{1+\nabla_\alpha u \nabla^\alpha u}+\frac{\nabla^\mu Q \nabla^\nu Q \nabla^2_{\mu\nu}Q \nabla^\beta \varphi\nabla_\beta \varphi}{(1+\nabla^\alpha u \nabla_\alpha u)(1+\nabla^\alpha Q \nabla_\alpha Q)}\\
&\quad +\frac{2\nabla^\beta Q\nabla_\beta \varphi(\nabla^\mu Q\nabla^\nu Q \nabla^2_{\mu\nu} \varphi+2\nabla^2_{\mu\nu}Q\nabla^\mu Q \nabla^\nu \varphi)}{(1+\nabla^\alpha u \nabla_\alpha u)(1+\nabla^\alpha Q\nabla_\alpha Q)}-\frac{4(\nabla^\mu Q\nabla^\nu Q\nabla^2_{\mu\nu}Q)(\nabla^\beta Q\nabla_\beta \varphi)^2}{(1+\nabla^\alpha u\nabla_\alpha u)(1+\nabla^\alpha Q\nabla_\alpha Q)^2},
\end{split}
\end{align}
and
%%%%%%%
%%%%%%%
\begin{align}\label{eq:calF3_1}
\begin{split}
\calF_3&=-\frac{\nabla^\mu\varphi\nabla^\nu\varphi\nabla^2_{\mu\nu}\varphi}{1+\nabla^\alpha u \nabla_\alpha u}-\frac{\nabla^\beta\varphi\nabla_\beta\varphi(\nabla^\mu Q\nabla^\nu Q \nabla^2_{\mu\nu} \varphi+2\nabla^2_{\mu\nu}Q\nabla^\mu Q \nabla^\nu \varphi)}{(1+\nabla^\alpha u \nabla_\alpha u)(1+\nabla^\alpha Q\nabla_\alpha Q)}\\
&\quad+\frac{2(\nabla^\mu Q \nabla^\nu Q\nabla^2_{\mu\nu}Q)(\nabla^\beta Q \nabla_\beta\varphi)(\nabla^\gamma\varphi\nabla_\gamma\varphi)}{(1+\nabla^\alpha u \nabla_\alpha u)(1+\nabla^\alpha Q\nabla_\alpha Q)^2}.
\end{split}
\end{align}
%%%%%%%%%%
%%%%%%%%%%
In view of \eqref{eq:callP1}, the linearized operator $\callP$ has the expansion
%%%%%%%
%%%%%%%
\begin{align}\label{eq:callP2}
\begin{split}
\callP_\graph\psi&=\Box_m\psi+\Err_{\callP}(\psi),
\end{split}
\end{align}
where for some bounded symmetric coefficient  $p_2$ and a bounded coefficient $p_1$,
%%%%%%%
%%%%%%%
\begin{align}\label{eq:ErrcallP1}
\begin{split}
\Err_{\callP}\psi&= \jap{r}^{-4}p_2^{\mu\nu}\partial^2_{\mu\nu}\psi+\jap{r}^{-5}p_1^\mu \partial_\mu\psi.
\end{split}
\end{align}

Due to the asymptotic flatness of the metric of the catenoid, the Minkowski wave operator $\Box_m$  will play a major role in the exterior analysis. For this reason, it is convenient to derive expressions for $m$, $m^{-1}$, and $\Box_m$ in the $(\tau,r,\theta)$ coordinates. Starting with $m$, we have
%%%%%%%
%%%%%%%
\begin{align}
m&=-\tilgamma^{-2}\ud\tau\otimes \ud\tau-(1-\jap{r}^{-2}+\Theta\cdot\eta')(\ud\tau\otimes\ud r +\ud r\otimes \ud \tau)+r\Theta_a\cdot\eta'(\ud\tau\otimes \ud\theta^a+\ud \theta^a\otimes \ud\tau)\nonumber\\
&\quad +\jap{r}^{-2} \ud r\otimes \ud r + r^2\ringsg_{ab}\ud\theta^a\otimes \ud \theta^b,\label{eq:mform1}
\end{align}
where we have used the notation $\tilgamma=(1-|\eta'|^2)^{-\frac{1}{2}}$. In matrix form this is
%%%%%%%
%%%%%%%
\begin{align}\label{eq:mmatrix1}
\begin{split}
m= m_0+\jap{r}^{-2}m_1=\pmat{-\tilgamma^{-2}&-1+\Theta\cdot\eta'&r\Theta_\theta\cdot\eta'\\-1+\Theta\cdot\eta'&0&0\\r\Theta_\theta\cdot\eta'&0&r^2\ringsg}+\jap{r}^{-2}\pmat{0&1&0\\1&1&0\\0&0&0}.
\end{split}
\end{align}
It follows that (here $|m|=|\det m|$)\footnote{Here we have used the fact that $|\Theta_\theta\cdot\eta'|^2=1-(\Theta\cdot\eta')^2$, where$
|\Theta_\theta\cdot\eta'|^2=\sum_{a=1}^{n-1}(\Theta_a\cdot \eta') \det \ringsg_a$, and where
$\ringsg_a$ is $\ringsg$ with the $a^\mathrm{th}$ column replaced by $\Theta_\theta\cdot\eta'=(\Theta_1\cdot\eta',\dots,\Theta_{n-1}\cdot\eta')^\intercal$.}
%%%%%%%
%%%%%%%
\begin{align}\label{eq:mdvol1}
\begin{split}
| m|^{\frac{1}{2}} = (1-\Theta\cdot\eta')r^{n-1}|\ringsg|^{\frac{1}{2}}(1+\jap{r}^{-2}\frac{\tilgamma^{-2}+(1-(\Theta\cdot\eta')^2)}{(1-\Theta\cdot\eta')^2}).
\end{split}
\end{align}
The inverse $m^{-1}$ can be calculated using \eqref{eq:extchangeofvars2} (in this formula $\Theta^a=(\ringsg^{-1})^{ab}\Theta_b$):
%%%%%%%
%%%%%%%
\begin{align*}
\begin{split}
m^{-1}&=\pmat{\frac{-\jap{r}^{-2}}{(1-\frac{r}{\jap{r}}\Theta\cdot\eta')^2}&\frac{-1+\Theta\cdot\eta'+\jap{r}^{-2}}{(1-\frac{r}{\jap{r}}\Theta\cdot\eta')^2}&\frac{\jap{r}^{-2}\Theta^\theta\cdot\eta'}{r(1-\frac{r}{\jap{r}}\Theta\cdot\eta')^2}\\\frac{-1+\Theta\cdot\eta'+\jap{r}^{-2}}{(1-\frac{r}{\jap{r}}\Theta\cdot\eta')^2}&\frac{1-(\Theta\cdot\eta')^2}{(1-\frac{r}{\jap{r}}\Theta\cdot\eta')^2}&\frac{(1-\Theta\cdot\eta'-\jap{r}^{-2})\Theta^\theta\cdot\eta'}{r(1-\frac{r}{\jap{r}}\Theta\cdot\eta')^2}\\\frac{\jap{r}^{-2}\Theta^\theta\cdot\eta'}{r(1-\frac{r}{\jap{r}}\Theta\cdot\eta')^2}&\frac{(1-\Theta\cdot\eta'-\jap{r}^{-2})\Theta^\theta\cdot\eta'}{r(1-\frac{r}{\jap{r}}\Theta\cdot\eta')^2}&\frac{\ringsg^{-1}}{r^2}-\frac{\jap{r}^{-2}\Theta^a\cdot\eta'\Theta^b\cdot\eta'}{r^2(1-\frac{r}{\jap{r}}\Theta\cdot\eta')^2}}.
\end{split}
\end{align*}
Expanding in powers of $r$, we can write this as
%%%%%%%
%%%%%%%
\begin{align}\label{eq:minvdecomp1}
\begin{split}
m^{-1}=m_0^{-1}+\jap{r}^{-2}m_1,
\end{split}
\end{align}
where 
%%%%%%%%
%%%%%%%%
\begin{align}\label{eq:m02inv1}
\begin{split}
m_0^{-1}=\pmat{0&\frac{-1}{1-\Theta\cdot\eta'}&0\\\frac{-1}{1-\Theta\cdot\eta'}&\frac{1+\Theta\cdot\eta'}{1-\Theta\cdot\eta'}&\frac{\Theta^\theta\cdot\eta'}{r(1-\Theta\cdot\eta')}\\0&\frac{\Theta^\theta\cdot\eta'}{r(1-\Theta\cdot\eta')}&r^{-2}\ringsg^{-1}},
\end{split}
\end{align} 
and $m_1$ is a matrix of size $\callO(1)$. Next, to derive an expression for the wave operator $\Box_m$ we write
%%%%%%%
%%%%%%%
\begin{align*}
\begin{split}
\Box_m\psi&=\frac{1}{\sqrt{|m|}}\partial_\mu(\sqrt{|m|}\partial_\nu\psi)\\
&=\frac{1}{\sqrt{|m_0|}}\partial_\mu (\sqrt{|m_0|}\partial_\nu\psi)+\jap{r}^{-2}m_1^{\mu\nu}\partial_{\mu\nu}^2\psi\\
&\quad+\frac{1}{\sqrt{|m|}}\partial_\mu(\sqrt{|m_0|}((m^{-1})^{\mu\nu}-(m_0^{-1})^{\mu\nu})+(\sqrt{|m|}-\sqrt{|m_0|})(m^{-1})^{\mu\nu})\partial_\nu\psi\\
&\quad +(|m|^{-\frac{1}{2}}-|m_0|^{-\frac{1}{2}})\partial_\mu(\sqrt{|m_0|}(m_0^{-1})^{\mu\nu})\partial_\nu\psi.
\end{split}
\end{align*}
Using the fact that $\sqrt{|m_0|}(m_0^{-1})^{\tau\nu}$ is independent of $\tau$, we arrive at the expression
%%%%%%%
%%%%%%%
\begin{align}\label{eq:Boxm1}
\begin{split}
\Box_m\psi&=2(m_0^{-1})^{\tau r}\partial^2_{\tau r}\psi+\frac{n-1}{r}(m_0^{-1})^{\tau r}\partial_\tau\psi+2(m_0^{-1})^{\theta r}\partial^2_{\theta r}\psi\\
&\quad + (m_0^{-1})^{rr}\partial^2_r\psi+\frac{n-1}{r}(m_0^{-1})^{rr}\partial_r\psi+\frac{1}{\sqrt{|m_0|}}\partial_\theta(\sqrt{|m_0|}(m_0^{-1})^{\theta r})\partial_r\psi\\
&\quad
+\frac{1}{r^2}\sDelta_{\bbS^{n-1}}\psi+\frac{n-2}{r}(m_0^{-1})^{\theta r}\partial_\theta\psi+\frac{\Theta_\theta\cdot\eta'}{r^2(1+\Theta\cdot\eta')}\partial_\theta\psi+\Err_{\Box_m}(\psi),
\end{split}
\end{align}
where 
%%%%%%%
%%%%%%%
\begin{align}\label{eq:ErrBoxm1}
\begin{split}
\Err_{\Box_m}\psi&= \jap{r}^{-2}m_1^{\mu\nu}\partial^2_{\mu\nu}\psi+(\callO(r^{-2}|\dotwp|)+\callO(r^{-3}))\partial_\tau\psi+(\callO(r^{-2}|\dotwp|) +\callO(r^{-3}))\partial_r\psi\\
&\quad+\callO((r^{-3}|\dotwp|)+\callO(r^{-4}))\partial_\theta\psi.
\end{split}
\end{align}
%%%%%%%%%%%%%
%%%%%%%%%%%%%
\begin{remark}
When working in the exterior region, it is often easier to work with the metric $m$ rather than the induced metric on the leaves of the foliation, and treat the difference as an error. In such cases we will often also use the volume form coming from the coordinate expression for $m$. Due to the asymptotic flatness of the induced metric this volume form is comparable in size with the geometrically induced volume form, and therefore the various inequalities we derive remain valid if we change to the geometric volume form. 
\end{remark}
%%%%%%%%%%%%%
%%%%%%%%%%%%%

For future reference, we end this subsection by deriving a geometric expression for the linear operator in the equation satisfied by $\phi$ (rather than $\varphi$). Recall equation \eqref{eq:uMinkowskieq1}, where $u=Q+\varphi$. Writing (see \eqref{eq:varphiphi1})
%%%%%%%
%%%%%%%
\begin{align*}
\begin{split}
\varphi= s\phi,\qquad s=\sqrt{1+(m^{-1})^{\alpha\beta} Q_\alpha Q_\beta},
\end{split}
\end{align*}
our goal is to derive the linear part of \eqref{eq:uMinkowskieq1} in terms of $\phi$. The computation is similar to those in Section~\ref{sec:interior} and we will be brief
%%%%%%%
%%%%%%%
\begin{align}\label{eq:kmetrixextdef1}
\begin{split}
h_{\alpha\beta}= \bfeta((\partial_\alpha,Q_\alpha), (\partial_\beta, Q_\beta))= m_{\alpha\beta}+Q_\alpha  Q_\beta.
\end{split}
\end{align} 
Then by direct computation 
%%%%%%%
%%%%%%%
\begin{align}\label{eq:exteriorgeomlintemp1}
\begin{split}
(h^{-1})^{\alpha\beta} = (m^{-1})^{\alpha\beta}-\frac{ Q^\alpha  Q^\beta}{1+|\nabla Q|^2},\qquad |h|= |m|(1+|\nabla Q|^2),
\end{split}
\end{align}
where we have used the notation $|\nabla Q|^2:= (m^{-1})^{\alpha\beta} Q_\alpha  Q_\beta$ and $ Q^\alpha:=(m^{-1})^{\alpha\beta} Q_\beta$. It follows from the expression for $h^{-1}$ in \eqref{eq:exteriorgeomlintemp1} that the linear part of \eqref{eq:uMinkowskieq1} is
%%%%%%%
%%%%%%%
\begin{align}\label{eq:exteriorgeomlintemp2}
\begin{split}
\Box_h \phi+s^{-1}(\Box_h s-2s^{-1}(h^{-1})^{\mu\nu}\partial_\mu s\partial_\nu s)\phi+\callO(\dotwp^{\leq 2}r^{-6})\partial^{\leq 2}\phi.
\end{split}
\end{align}
Now we claim that in the case when $Q$ corresponds to a true parameterization of a boosted and translated catenoid (that is, when $\dotell=0$ and $\dotxi=\ell$), the coefficient of $\phi$ above is precisely $V=|\secondff|^2$. To see this, recall from Lemma~\ref{lem:secondff} (more precisely from \cite[Corollary II]{RV70}) that in this case we would have
%%%%%%%
%%%%%%%
\begin{align*}
\begin{split}
\Box_hn_\wp=-|\secondff|^2n_\wp.
\end{split}
\end{align*}
Since, with $N=s(0,1)$, we have $\bfeta(n_{\wp},N)=1$, it follows from the expression \eqref{eq:nwpext1} that in this case (that is, when $\dotell=0$ and $\dotxi=\ell$)
%%%%%%%
%%%%%%%
\begin{align*}
\begin{split}
-|\secondff|^2=\eta(\Box_h n_\wp,N)= s\Box_h s^{-1}=-s^{-1}\Box_hs+2s^{-2}(h^{-1})^{\mu\nu}\partial_\mu s\partial_\nu s, 
\end{split}
\end{align*}
which proves our claim. Returning to \eqref{eq:exteriorgeomlintemp2}, since the only errors come from when derivatives fall on parameters, we see that the linear part of \eqref{eq:uMinkowskieq1} in terms of $\phi$ is, with $V=|\secondff|^2$,
%%%%%%%
%%%%%%%
\begin{align}\label{eq:extgeomlin1}
\begin{split}
\Box_h\phi + V\phi +\callO(\dotwp^{\leq 2}r^{-6})\partial^{\leq2}\phi.
\end{split}
\end{align}
%%%%%%%%%%
%%%%%%%%%%
%%%%%%%%%%
\subsubsection{Global Non-Geometric Coordinates}\label{sec:GNGC}
%%%%%%%%%%
%%%%%%%%%%
%%%%%%%%%%
Under the assumption that $\wp$ is sufficiently small, we introduce a global set of coordinates $(\uptau,\uprho,\uptheta)$ which glue the interior and exterior coordinates introduced earlier. The smallness of $\wp$ will be guaranteed by the choice of initial data and the bootstrap assumptions (to be described in Section~\ref{sec:bootstrap}). The procedure is as follows. Let $\Psi_\Int$ and $\Psi_\Ext$, defined on overlapping open sets $V_\Int$ and $V_\Ext$, denote the coordinate maps associated to the rectangular coordinates of $(t,\rho,\omega)$ and $(\sigma,r,\theta)$ respectively. More precisely, for a point $p=(t,\xi(t)+\gamma^{-1}P_{\ell(t)} F(\rho,\omega)+P_{\ell(t)}^\perp F(\rho,\omega))$ in $V_\Int \cap \big(\cup \Sigma_\sigma)$ let
%%%%%%%
%%%%%%%
\begin{align*}
\begin{split}
\Psi_\Int(p)= (y^0,\dots,y^n), \qquad y^0=t, ~y^i= \rho\Theta^i(\omega); i=1,\dots,n.
\end{split}
\end{align*}
The coordinate map $\Psi_\Ext$ is defined similarly. In the overlapping region which is, by assumption, contained in $\calC_\flatt$, for a point $p=(\sigma,\eta(\sigma)+r\Theta(\theta))$,
%%%%%%%
%%%%%%%
\begin{align*}
\Psi_\Ext(p)=(y^0,\dots,y^n),\qquad y^0=\sigma,~y^i=r\Theta^i(\theta); i=1,\dots,n.
\end{align*}
Let $\chi$ be a cutoff function which is supported in $V_\Int$ and is equal to one on $V_\Int\backslash(V_\Int\cap V_\Ext)$. We define the global rectangular coordinates by 
%%%%%%%
%%%%%%%
\begin{align*}
\begin{split}
y=\Psi(p):= \chi(p)\Psi_\Int(p)+(1-\chi(p))\Psi_\Ext(p),
\end{split}
\end{align*}
and the global polar coordinates by expressing $y$ in polar coordinates $(\uptau,\uprho,\uptheta)$. To see that $\Psi$ defines a coordinate map we need to check that $d\Psi(p)$ is invertible for all $p\in V_\Int\cap V_\Ext$ and that $\Psi$ is one to one. For the derivatives, it suffices to show that $d(\Psi\circ \Psi^{-1}_\Int)$ is invertible. By the definition of $\Psi$,
%%%%%%%
%%%%%%%
\begin{align*}
\begin{split}
\Psi\circ\Psi_\Int^{-1}(x)=\chi\circ\Psi^{-1}_\Int(x)x+(1-\chi\circ\Psi^{-1}_\Int(x))\Psi_\ext\circ\Psi_\Int^{-1}(x),
\end{split}
\end{align*}
so (here $I$ denotes the identity matrix)
%%%%%%%
%%%%%%%
\begin{align*}
\begin{split}
d(\Psi\circ\Psi_\Int^{-1})(x)&=d(\Psi_\Ext\circ\Psi_\Int^{-1})(x)\\
&\quad+\chi\circ\Psi_\Int(x)(I-d(\Psi_\Ext\circ\Psi_\Int^{-1})(x))+\nabla(\chi\circ\Psi_\Int^{-1}(x)) (x-\Psi_\Ext\circ\Psi_\Int^{-1}(x)).
\end{split}
\end{align*}
Since $x-\Psi_\Ext\circ\Psi_\Int^{-1}(x)$ is small for $x\in \Psi_\Int(V_\Int\cap V_\Ext)$, it suffices to show that $I-d(\Psi_\Ext\circ\Psi_\Int^{-1})(x)$ is also small for such $x$. To compute the derivative we write $(\tau,z)$ for $\Psi_\Ext\circ\Psi_\Int^{-1}(x)$ and $(t,y)$ for $x$. Then according to the above formulas for $\Psi_\Int$ and $\Psi_\Ext$, $(t,y)$ and $(\sigma,z)$ are related by
%%%%%%%
%%%%%%%
\begin{align*}
\begin{split}
t= \sigma,\quad \xi(t)+\gamma^{-1}(t) P_{\ell(t)}(\jap{y}|y|^{-1}y)+P_{\ell(t)}^\perp(\jap{y}|y|^{-1}y)= \eta(\sigma)+z.
\end{split}
\end{align*}
The desired invertibility then follows by implicitly differentiating these relations to get
%%%%%%%
%%%%%%%
\begin{align*}
\begin{split}
&\frac{\partial\sigma}{\partial t}=1,\quad \frac{\partial\sigma}{\partial y^j}=0,\\
&\frac{\partial z}{\partial t}+\eta'(\sigma)\frac{\partial\sigma}{\partial t}=-\gamma'(t)\gamma^{-2}(t)P_\ell(\jap{y}|y|^{-1}y)+(\gamma^{-1}-1)\frac{\jap{y}}{|y|}\big(\frac{y\cdot\ell}{|\ell|^2}\dotell+\frac{y\cdot\dotell}{|\ell|^2}\ell-2\frac{(\dotell\cdot\ell)(y\cdot\ell)}{|\ell|^4}\ell\big)+\xi'(t),\\
&\frac{\partial z}{\partial y^j}=-\gamma^{-1}P_\ell(|y|^{-3}\jap{y}^{-1}y^jy)-P_\ell^\perp(|y|^{-3}\jap{y}^{-1}y^jy)+\gamma^{-1}P_\ell(\jap{y}|y|^{-1}e_j)+P_\ell(\jap{y}|y|^{-1}e_j),
\end{split}
\end{align*}
and using the smallness of $\ell$ and $\eta'(\tau)-\xi'(t)$ (by assumption). The fact that $\Psi$ is one to one can be shown using similar considerations. 

We also remark that since in the overlapping region $t(p)=\sigma(p) = X^0(p)$, the coordinate $\uptau$ satisfies $\uptau=t=\sigma$. The normalized vectorfield $$\RbfT:=\partial_\uptau$$ plays a distinguished role in this work a globally defined almost Killing and unit timelike vectorfield. Note that $T$, defined in \eqref{eq:VFdef0} in the exterior, and $\bfT$ differ only by terms which have $\uptau$ decay (see~\eqref{eq:Tprecise1}).

%%%%%%%%%%%%
%%%%%%%%%%%%
\begin{remark}\label{rem:nongeomglobal1}
Since the global coordinates $(\uptau,\uprho,\uptheta)$ agree with the coordinates introduced in the previous two subsections in the respective regions, and in view of the invariant form $\Box+V$ appearing in \eqref{eq:LEDcalPint1} and \eqref{eq:extgeomlin1}, by inspection of the calculations in the interior and exterior regions (see \eqref{eq:LEDcalPint1}, \eqref{eq:varphiphi1},  \eqref{eq:callP2}, \eqref{eq:ErrcallP1}, \eqref{eq:Boxm1}, \eqref{eq:ErrBoxm1}, \eqref{eq:extgeomlin1}), the operator $\calP$ in \eqref{eq:LEDcalPint1} and \eqref{eq:callP2} satisfies the following properties:
\begin{enumerate}
\item $\calP$ admits the decomposition
%%%%%%%
%%%%%%%
\begin{align*}
\begin{split}
\calP=\calP_\uptau+\calP_\Ell,
\end{split}
\end{align*}
where $\calP_\Ell$ is elliptic and does not contain $\partial_{\uptau}$ derivatives. $\calP_\Ell$ can be further decomposed as
%%%%%%%
%%%%%%%
\begin{align*}
\begin{split}
\calP_\Ell=\Delta_\barcalC+V+\calP_\Ell^{\pert},
\end{split}
\end{align*}
where
\begin{align*}
\begin{split}
\Delta_\barcalC=\frac{1}{\jap{\uprho}^{n-1}|F_{\uprho}|}\partial_\uprho(\jap{\uprho}^{n-1}|F_\uprho|^{-1}\partial_\rho)+\frac{1}{\jap{\uprho}^2}\ringsDelta,
\end{split}
\end{align*}
and (recall the notation from Section~\ref{subsec:prelimvfs})
%%%%%%%
%%%%%%%
\begin{align*}
\begin{split}
\calP_\Ell^\pert=o_{\wp,\Rone}(1)(\partial_\Sigma^2+\jap{\uprho}^{-1}\partial_\Sigma+\jap{\uprho}^{-2})+\callO(\dotwp)\partial_\Sigma+\callO(\dotwp)\jap{\uprho}^{-2}.
\end{split}
\end{align*}
Here $o_{\wp,\Rone}(1)$ denotes coefficients which are $\callO(\wp)$ or can be made arbitrarily small by taking $\Rone$ (the transition region from the flat to hyperboloidal foliation) large. Finally, $\calP_\uptau$ has the form
%%%%%%%
%%%%%%%
\begin{align*}
\begin{split}
\calP_\uptau=\callO(1)\partial \RbfT+\callO(\jap{\uprho^{-1}})\RbfT.
\end{split}
\end{align*}
\item If $|\dotwp^{\leq 2}|\lesssim \epsilon \uptau^{-\gamma-1}$, for some $\gamma>1$, then the operator $\calP$ takes the form
%%%%%%%
%%%%%%%
\begin{align*}
\begin{split}
\calP\psi=\uppi_q^{\mu\nu}\partial^2_{\mu\nu}\psi+\uppi_l^\mu\partial_\mu\psi+\uppi_c\psi,
\end{split}
\end{align*}
where the coefficients satisfy the following properties. In view of the invariant form $\Box_h+V$ appearing in \eqref{eq:LEDcalPint1} and \eqref{eq:extgeomlin1}, let
%%%%%%%
%%%%%%%
\begin{align*}
\begin{split}
&\pibar_q^{\mu\nu}=(h^{-1})^{\mu\nu}\vert_{\uptau=t_2}, \quad \pibar_l^\nu=|h|^{-\frac{1}{2}}\partial_\mu(|h|^{\frac{1}{2}}(h^{-1})^{\mu\nu})\vert_{\uptau=t_2},\quad \pibar_c=V\vert_{\substack{\dot{\ell} = 0 \\ \dot{\xi} = \ell\\ \uptau=t_2}},\\
& \ringpi_q=\uppi_q-\pibar,\quad \ringpi_l=\uppi_l-\pibar,\quad \ringpi_c=\uppi_c-\pibar,
\end{split}
\end{align*}
and correspondingly decompose $\calP$ as $\calP=\calP_0+\calP_\pert$ with
%%%%%%%
%%%%%%%
\begin{align}\label{eq:calPP0Ppertdecomp1}
\begin{split}
\calP_0=\pibar_q^{\mu\nu}\partial^2_{\mu\nu}+\pibar_l^\mu\partial_\mu+\pibar_c, \qquad \calP_\pert=\ringpi_q^{\mu\nu}\partial^2_{\mu\nu}+\ringpi_l^\mu\partial_\mu+\ringpi_c.
\end{split}
\end{align}
Then $|\ringpi_q|, |\ringpi_l|, |\ringpi_c|\lesssim \jap{\uptau}^{-\gamma}$, and, with $y$ denoting the spatial variables $(\uprho,\uptheta)$,
%%%%%%%
%%%%%%%
\begin{align*}
\begin{split}
\sup_y(\jap{\uprho}^{2}|\ringpi_q^{\uptau\uptau}|+|\ringpi_q^{\uptau y}|+|\ringpi_q^{yy}|)\lesssim \epsilon \jap{\uptau}^{-\gamma}. 
\end{split}
\end{align*}
Moreover, in the hyperboloidal part of the foliation, $\calP_\pert$ has the more precise structure
%%%%%%%
%%%%%%%
\begin{align*}
\begin{split}
\ringa\partial_\uptau(\partial_\uprho+\frac{n-1}{2\uprho})+\ringa^{\mu\nu}\partial_{\mu\nu}+\ringb^\mu \partial_\mu +\ringc,
\end{split}
\end{align*}
where,
%%%%%%%
%%%%%%%
\begin{align}\label{eq:calPP0Ppertdecomp2}
\begin{split}
&|\ringa|, |\ringa^{yy}|\lesssim \epsilon \jap{\uptau}^{-\gamma}, \quad  |\partial_y\ringa^{yy}|, |\ringa^{\uptau y}|, |\ringb^y|\lesssim \epsilon\jap{\uptau}^{-\gamma}\uprho^{-1},\quad |\ringa^{\uptau\uptau}|, |\ringb^\uptau|\lesssim \epsilon \jap{\uptau}^{-\gamma}\uprho^{-2},\\
&|\ringc|\lesssim \epsilon\jap{\uptau}^{-\gamma} \uprho^{-4}.
\end{split}
\end{align} 
\item $\calP$ can be written as 
%%%%%%%
%%%%%%%
\begin{align*}
\begin{split}
\calP\uppsi=\frac{1}{\sqrt{|\bfh|}}\partial_\mu (\sqrt{|\bfh|}(\bfh^{-1})^{\mu\nu}\partial_\nu\uppsi)+V\uppsi+\tilcalbfP, 
\end{split}
\end{align*}
where the Lorentzian metric $\bfh$ agrees with $h$ in \eqref{eq:LEDcalPint1} in the region where $(\uptau,\uprho,\uptheta)$ agrees with $(t,\rho,\omega)$, and with $m$ in \eqref{eq:mform1} in the region where $(\uptau,\uprho,\uptheta)$ agrees with $(\tau,r,\theta)$. Moreover, the perturbative part $\tilcalbfP$ can be written as
%%%%%%%
%%%%%%%
\begin{align*}
\begin{split}
\tilcalbfP= \bfp^{\mu\nu}\partial_{\mu\nu}^2+\bfq^\mu \partial_\mu+\bfs
\end{split}
\end{align*}
for a symmetric tensor $\bfp$, a vectorfield $\bfq$, and a scalar $\bfs$ satisfying (in the notation of item~(1) above)
%%%%%%%
%%%%%%%
\begin{align*}
\begin{split}
\jap{\uprho}^2|\bfp|+\jap{\uprho}^3|\bfq|+\jap{\uprho}^4|\bfs|=o_{\wp,\Rone}(1)+\callO(\dotwp^{\leq 2}).
\end{split}
\end{align*}
\end{enumerate}
\end{remark}
%%%%%%%%%%%%
%%%%%%%%%%%%

%%%%%%%%%%
%%%%%%%%%%
%%%%%%%%%%
\subsubsection{Global Geometric Coordinates}\label{sec:GGC}
%%%%%%%%%%
%%%%%%%%%%
%%%%%%%%%%
Here we introduce a new set of coordinates $(\tiluptau,\tiluprho,\tiluptheta)$ to which we refer as \emph{geometric global coordinates}. Their main property of interest to us is that the operator $\calP_0$ introduced in Remark~\ref{rem:nongeomglobal1} above has the following expression in these coordinates:
%%%%%%%
%%%%%%%
\begin{align*}
\begin{split}
\calP_0=-\partial_{\tiluptau}^2+\tilDelta+V(\tilrho),
\end{split}
\end{align*}
Here, with $\ringsDelta$ denoting the Laplacian on the round sphere $\bbS^{n-1}$,
%%%%%%%
%%%%%%%
\begin{align*}
\begin{split}
\tilDelta=\frac{1}{\jap{\tiluprho}^{n-1}|F_{\tiluprho}|}\partial_\tiluprho(\jap{\tiluprho}^{n-1}|F_\tiluprho|^{-1}\partial_\tiluprho)+\frac{1}{\jap{\tiluprho}^2}\ringsDelta.
\end{split}
\end{align*}
The geometric global coordinates can be defined as follows. Let $\ellbar=\ell(t_2)$, $\xibar=\xi(t_2)$ and $\gammabar=\gamma(t_2)$. With these choices we consider two parameterizations of the catenoid defined in \eqref{eq:calCsigmadefintro1}, where $\ell\equiv \ellbar$ and $\xi\equiv\xibar$. The first parameterization is exactly by the non-geometric global coordinates $(\uptau,\uprho,\uptheta)$ from Section~\ref{sec:GNGC}, corresponding the choice of parameters $\ell\equiv \ellbar$, $\xi\equiv\xibar+\sigma\ellbar$. The second parameterization is simply
%%%%%%%
%%%%%%%
\begin{align*}
\begin{split}
\Lambda_{-\ellbar}(\tiluptau,F(\tiluprho,\tiluptheta))+(0,\xibar),
\end{split}
\end{align*}
where $F$ and $\Lambda$ are as in \eqref{eq:Fdef} and \eqref{eq:Lambdadef1}. The coordinate change between $(\uptau,\uprho,\uptheta)$ and $(\tiluptau,\tiluprho,\tiluptheta)$ is then obtained by equating the $X^i$, $i=0,\dots,n$ coordinates of the ambient space with respect to these two parameterizations. The desired form of $\calP_0$ follows from the coordinate invariance of this operator.  Explicit formulas for the coordinate transformation can also be given in the regions where the non-geometric coordinates agree with the coordinates $(t,\rho,\omega)$ and $(\tau,r,\theta)$, and are given respectively by 
%%%%%%%
%%%%%%%
\begin{align}\label{eq:coordinatetransformationrelation1}
\begin{split}
\tiluptau=\gammabar^{-1}\uptau-\ellbar\cdot F(\uprho,\uptheta), \quad \tiluprho=\uprho,\quad \tiluptheta=\uptheta
\end{split}
\end{align}
in the interior, and by
%%%%%%%
%%%%%%%
\begin{align*}
\begin{split}
\uptau+\jap{\uprho}=\gammabar\tiluptau+\gammabar\ellbar\cdot F(\tiluprho,\tiluptheta), \quad \uptau\ellbar+\uprho\Theta(\uptheta)=\gammabar\tiluptau\ellbar+A_\ellbar F(\tiluprho,\tiluptheta)
\end{split}
\end{align*}
in the exterior.
%%%%%%%%%%%%
%%%%%%%%%%%%
%%%%%%%%%%%%%%%%%%%%%%
%%%%%%%%%%%%%%%%%%%%%%
%%%%%%%%%%%%%%%%%%%%%%
\section{Main Bootstrap Argument and the Proof of  Theorem~\ref{thm:main}}\label{sec:bootstrap}
%%%%%%%%%%%%%%%%%%%%%%
%%%%%%%%%%%%%%%%%%%%%%
%%%%%%%%%%%%%%%%%%%%%%

In the first part of this section we set up the bootstrap assumptions and state the propositions which assert that the bootstrap regime is trapped. These  are Propositions~\ref{prop:bootstrappar1} and~\ref{prop:bootstrapphi1}, and their proofs will occupy most of the remainder of the paper. In the rest of this section we will prove Theorem~\ref{thm:main} assuming Propositions~\ref{prop:bootstrappar1} and~\ref{prop:bootstrapphi1}. 

For concreteness, we set $n=5$ for the remainder of the paper, but our arguments are easily adaptable to higher dimensions. We can now state our bootstrap assumptions. We assume that there exist $\xi,\ell$ defined on $[0,\tau_f)$ and a parameterization \eqref{eq:psidefintro1} for $\tau\in[0,\tau_f)$ such that the orthogonality conditions \eqref{equ:derivation_modul_equ4} and \eqref{eq:wpoutline9} are satisfied. We also assume that the following \emph{trapping assumption} (see \eqref{eq:Fplusdef} for the definition of $F_{+}$ and \eqref{equ:unstable_mode_decomposition} for the relation between $\psi$ and $\phi$): 
\begin{equation} \label{eq:a+trap}
|\mu a_{+}(\tau)-e^{\mu \tau }S(e^{-\mu \tau}F_{+})|\leq  C_{trap} \delta_{\wp} \epsilon\tau^{-3},
\end{equation}
and the following estimates hold for all $\tau,\sigma_1,\sigma_2\in[0,\tau_f)$ (recall from Section~\ref{subsec:prelimvfs} that $\tilpartial_\Sigma$ denotes size one tangential derivatives $\partial_\Sigma$ or $\jap{\uprho}^{-1} \RbfT$ and that in the exterior $X$ denotes any of the vecotrfields $\tilr L$, $\Omega$, or $T$ introduced in Section~\ref{sec:profile2}):
%%%%%%%
%%%%%%%
\begin{align}
&|a_{+}^{(k)}|\leq 2 C_{k} \delta_{\wp}\epsilon \tau^{-\frac{5}{2}+\kappa}, \quad \forall  k\geq 0.\label{eq:a+b1}\\
&|a_{-}^{(k)}|\leq 2C_k \delta_{\wp} \epsilon\tau^{-\frac{5}{2}+\kappa},\quad \forall  k\geq 0.\label{eq:a-b1}\\
&|\dotwp^{(k)}|\leq 2C_k \delta_{\wp}\epsilon \tau^{-\frac{5}{2}+\kappa},\quad \forall  k\geq 1.\label{eq:wpb1}\\
&|\phi|+\chi_{\geq R}|X^k\phi|\leq 2C\epsilon\tau^{-\frac{9}{4}+\frac{\kappa}{2}},\quad k\leq M-7.\label{eq:phiptwiseb1}\\
&|\partial^k\phi|+\chi_{\geq R}|\partial^{k-j}X^j\phi|\leq 2C\epsilon\tau^{-\frac{5}{2}+\kappa},\quad 1\leq k\leq M-8,~j<k.\label{eq:dphiptwiseb1}\\
&\chi_{\geq R}|\partial^{k-j}X^j\phi|\leq 2C\epsilon\jap{r}^{-\frac{3}{2}}\tau^{-1+\kappa},\quad 1\leq k\leq M-8,~j<k.\label{eq:dphiptwiseb2}\\
&\chi_{\geq R}|\partial^{k-j}X^j\phi|\leq 2C\epsilon\jap{r}^{-2}\tau^{-\frac{1}{2}+\kappa},\quad 1\leq k\leq M-8,~j<k.\label{eq:dphiptwiseb3}\\
&\| \chi_{\leq R}\partial^k\phi\|_{L^2(\Sigma_\tau)}+\|\jap{r}^{-\frac{5}{2}+\kappa}(\chi_{\geq R}X^k\phi)\|_{L^2(\Sigma_\tau)}\leq 2C \epsilon\tau^{-\frac{5}{2}+\kappa},\quad 0\leq k\leq M-4.\label{eq:dphiL2b1}
\end{align}
\begin{align}
&\|\partial^2_\Sigma (\chi_{\leq R}\partial^k\RbfT\phi)\|_{L^2(\Sigma_\tau)}+\|\partial^2_\Sigma(\chi_{\geq R} X^k\RbfT\phi)\|_{L^2(\Sigma_\tau)}\leq 2C \epsilon\tau^{-3},\quad 0\leq k \leq M-4.\label{eq:d2Tphienergyb1}\\
&\|\chi_{\leq R}\partial_\Sigma \partial^{k}\RbfT^j\phi\|_{L^2(\Sigma_\tau)}+\|\chi_{\geq R} \tilpartial_\Sigma X^{k}\RbfT^j\phi\|_{L^2(\Sigma_\tau)}\leq 2C \epsilon\tau^{-1-j},\quad 0\leq j\leq 2,~ k+3j \leq M-2.\label{eq:Tjphienergyb1}\\
&\|\chi_{\geq R}r^{p} (\partial_r+\frac{n-1}{2r}) X^{k}\RbfT^j\phi\|_{L^2(\Sigma_\tau)}\leq 2C \epsilon\tau^{-1-j+\frac{p}{2}},\quad 0\leq p\leq 2,~ k+3j \leq M-2, 0\leq j\leq 2.\label{eq:Tjphienergyb2}
\end{align}
We close our bootstrap assumptions in a few steps. First, in Proposition~\ref{prop:bootstrappar1}, we close the bootstrap assumptions for the parameters, but with a suboptimal rate for $\mu a_{+}-e^{\mu \tau}S(e^{-\mu\tau}F_{+})$ in \eqref{eq:a+trap}. We then use this in Proposition~\ref{prop:bootstrapphi1} to improve the bootstrap bounds on $\phi$.  Finally, in the proof of Theorem~\ref{thm:main} we show that the initial data and parameters can be chosen such that the trapping assumption \eqref{eq:a+trap} is satisfied. Before stating Propositions~\ref{prop:bootstrappar1} and~\ref{prop:bootstrapphi1}, we remark that in view of equation~\eqref{eq:wpoutline10} satisfied by $a_{+}$ (see also \eqref{eq:higheraplustemp1}), the trapping assumption \eqref{eq:a+trap} is stated at the level of the derivative of $a_{+}$. The reason for this is that we need improved decay for the derivative to close the remaining bootstrap assumptions.
%%%%%%%%%%%
%%%%%%%%%%%
\begin{proposition}\label{prop:bootstrappar1}
Suppose the estimates \eqref{eq:a+trap}--\eqref{eq:Tjphienergyb2} and orthogonality conditions \eqref{equ:derivation_modul_equ4} and \eqref{eq:wpoutline9} are satisfied. If $\epsilon$ is sufficiently small and $C,C_k$ appearing on the right-hand side of \eqref{eq:a+trap}--\eqref{eq:Tjphienergyb2} are sufficiently large (compared to $C_{trap}$), then the following improved estimates hold:
\begin{align} 
&|a_{+}^{(k)}|\leq C_{k} \delta_{\wp} \epsilon\tau^{-\frac{5}{2}+\kappa}, \quad \forall  k\geq 0.\label{eq:a+1}\\
&|a_{-}^{(k)}|\leq C_k \delta_{\wp} \epsilon\tau^{-\frac{5}{2}+\kappa},\quad \forall  k\geq 0.\label{eq:a-1}\\
&|\dotwp^{(k)}|\leq C_k \delta_{\wp} \epsilon\tau^{-\frac{5}{2}+\kappa},\quad \forall  k\geq 1.\label{eq:wp1}
\end{align}
\end{proposition}
%%%%%%%%%%%
%%%%%%%%%%%
%%%%%%%%%%%
%%%%%%%%%%%
\begin{proposition}\label{prop:bootstrapphi1}
Suppose the estimates \eqref{eq:a+trap}--\eqref{eq:Tjphienergyb2} orthogonality conditions \eqref{equ:derivation_modul_equ4} and \eqref{eq:wpoutline9} are satisfied. If $\epsilon$ is sufficiently small and $C,C_k$ appearing on the right-hand side of \eqref{eq:a+trap}--\eqref{eq:Tjphienergyb2} are sufficiently large (compared to $C_{trap}$), then the following improved estimates hold:
%%%%%%%
%%%%%%%
\begin{align}
&|\phi|+\chi_{\geq R}|X^k\phi|\leq C\epsilon\tau^{-\frac{9}{4}+\frac{\kappa}{2}},\quad k\leq M-7.\label{eq:phiptwise1}\\
&|\partial^k\phi|+\chi_{\geq R}|\partial^{k-j}X^j\phi|\leq C\epsilon\tau^{-\frac{5}{2}+\kappa},\quad 1\leq k\leq M-8,~j<k.\label{eq:dphiptwise1}\\
&|\partial^k\phi|+\chi_{\geq R}|\partial^{k-j}X^j\phi|\leq C\epsilon\jap{r}^{-\frac{3}{2}}\tau^{-1+\kappa},\quad 1\leq k\leq M-8,~j<k.\label{eq:dphiptwise2}\\
&|\partial^k\phi|+\chi_{\geq R}|\partial^{k-j}X^j\phi|\leq C\epsilon\jap{r}^{-2}\tau^{-\frac{1}{2}+\kappa},\quad 1\leq k\leq M-8,~j<k.\label{eq:dphiptwise3}\\
&\| \jap{r}^{-\frac{5}{2}+\kappa}(\chi_{\leq R}\partial^k\phi)\|_{L^2(\Sigma_\tau)}+\|\jap{r}^{-\frac{5}{2}+\kappa}(\chi_{\geq R}X^k\phi)\|_{L^2(\Sigma_\tau)}\leq C \epsilon\tau^{-\frac{5}{2}+\kappa},\quad 0\leq k\leq M-4.\label{eq:dphiL21}\\
&\|\partial^2_\Sigma (\chi_{\leq R}\RbfT\partial^k\phi)\|_{L^2(\Sigma_\tau)}+\|\partial^2_\Sigma(\chi_{\geq R}\RbfT X^k\phi)\|_{L^2(\Sigma_\tau)}\leq C \epsilon\tau^{-3},\quad 0\leq k \leq M-4.\label{eq:d2Tphienergy1}\\
&\|\chi_{\leq R}\partial_\Sigma \partial^{k}\RbfT^j\phi\|_{L^2(\Sigma_\tau)}+\|\chi_{\geq R} \tilpartial_\Sigma X^{k}\RbfT^j\phi\|_{L^2(\Sigma_\tau)}\leq C \epsilon\tau^{-1-j},\quad 0\leq j\leq 2,~k+3j \leq M-2.\label{eq:Tjphienergy1}\\
&\|\chi_{\geq R}r^{p} (\partial_r+\frac{n-1}{2r}) X^{k}\RbfT^j\phi\|_{L^2(\Sigma_\tau)}\leq C \epsilon\tau^{-1-j+\frac{p}{2}},\quad 0\leq p\leq 2,~ k+3j \leq M-2, 0\leq j\leq 2.\label{eq:Tjphienergy2}
\end{align}
\end{proposition}
%%%%%%%%%%%
%%%%%%%%%%%
%%%%%%%%%%%
%%%%%%%%%%%
Propositions~\ref{prop:bootstrappar1} and \ref{prop:bootstrapphi1} will be proved in Sections~\ref{sec:parametercontrol} and \ref{sec:exterior} respectively. Observe that Proposition~\ref{prop:bootstrappar1} does \emph{not} improve the trapping assumption \eqref{eq:a+trap}, but only the bound \eqref{eq:a+b1}. We will employ a topological (shooting) argument to find a global solution for which~\eqref{eq:a+trap} holds for all times, by choosing the initial data appropriately.
%%%%%%%%%%%
%%%%%%%%%%%

%%%%%%%%%%%%%%
%%%%%%%%%%%%%%
\subsection{Proof of Theorem~\ref{thm:main}}\label{subsec:proofofmaintheorem}
%%%%%%%%%%%%%%
%%%%%%%%%%%%%%

Assuming Propositions~\ref{prop:bootstrappar1} and~\ref{prop:bootstrapphi1}, we will prove Theorem~\ref{thm:main}. 
%%%%%%%%%%%%
%%%%%%%%%%%%
\begin{proof}[Proof of Theorem~\ref{thm:main}]
 The proof consists of two steps as we now explain. Given $(\psi_0,\psi_1)$ for each $b$ (see the statement of Theorem~\ref{thm:main}) we let $\tau_f(b)$ be the maximal time on which there is a solution parameterized as in \eqref{eq:psidefintro1} such that the bootstrap assumptions \eqref{eq:a+trap}--\eqref{eq:Tjphienergyb2} and the orthogonality conditions \eqref{equ:derivation_modul_equ4} and \eqref{eq:wpoutline9} are satisfied. By local well-posedness (Proposition~\ref{prop:LWP}), the normal neighborhood lemma (Lemma~\ref{rem:normalneighborhood}), and the implicit function theorem arguments in Sections~\ref{subsec:modulationeqs} and~\ref{subsec:unstableint}, we know that $\tau_f(b)$ is strictly positive for each choice of $b$. Our goal is to show that $\tau_f(b)$ is infinite for some choice of $b$. Suppose $\tau_f(b)$ is finite for all choices of $b$. In the first step we show that condition \eqref{eq:a+trap} must get saturated, that is the inequality must be an equality, at $\tau=\tau_f$. In the second step, we show that if $(\psi_0,\psi_1)$ satisfy a certain codimension one condition, equation \eqref{eq:codim1}, then there is a choice of $b$ for which \eqref{eq:a+trap} is not saturated by $\tau_f(b)$, and this is the desired contradiction.

Before turning to the details of Steps 1 and 2, we clarify one point. Because of the nonlocal nature of the orthogonality conditions \eqref{equ:derivation_modul_equ4} and \eqref{eq:wpoutline9} (coming from the smoothing operator $S$), in order to apply the implicit function theorem to get parameters which guarantee \eqref{equ:derivation_modul_equ4} and \eqref{eq:wpoutline9}, we first need to define $\psi$ for $\tau\in[-1,0]$. For this, we fix an extension procedure which is continuous (say with respect to some finite regularity Sobolev norm) in $(\psi_0,\psi_1)$, for instance $\psi_0(t)=\psi_0+t\psi_1$ and $\psi_1(t)=\psi_1$, and work with this fixed extension throughout the proof.
 
{\underline{\emph{Step 1.}}}
Fix $b$ and let $\tau^\ast\in(0,\infty)$ be such that the bootstrap conditions described above (including the orthogonality conditions and the parameterization \eqref{eq:psidefintro1}) are satisfied on $[0,\tau^\ast]$. We want to show that if \eqref{eq:a+trap} is strict on $[0,\tau^\ast]$ then $\tau_f(b)>\tau^\ast$, where as above $\tau_f(b)$ is the maximal time on which the bootstrap conditions are satisfied. By Propositions~\ref{prop:bootstrappar1} and~\ref{prop:bootstrapphi1} (with $\tau_f$ replaced by $\tau^\ast$) we can improve the bootstrap assumptions \eqref{eq:a+b1}--\eqref{eq:Tjphienergyb2} on $[0,\tau^\ast]$. Now suppose \eqref{eq:a+trap} is strict on $[0,\tau^\ast]$. By Proposition~\ref{prop:LWP} applied with $\ell_0$ and $\xi_0$ fixed at values of $\ell$ and $\xi$ close to $\tau^\ast$, we can extend the solution on an interval of size of order one beyond $\tau^\ast$. Applying Lemma~\ref{rem:normalneighborhood} and the implicit function theorems in Sections~\ref{subsec:modulationeqs} and~\ref{subsec:unstableint} (note that while the details of the proofs there were carried out for zero $\ell_0$ and $\xi_0$, identical arguments can be used for other choices) we can extend $\xi$ and $\ell$ and the parameterization \eqref{eq:psidefintro1} beyond $\tau^\ast$ such that the orthogonality conditions \eqref{equ:derivation_modul_equ4} and \eqref{eq:wpoutline9} are still satisfied. Moreover, since \eqref{eq:a+trap} is strict on $[0,\tau^\ast]$, by continuity it is still satisfied on a larger interval. It follows that on this larger interval all the bootstrap conditions are satisfied and hence $\tau_f(b)>\tau^\ast$.  

{\underline{\emph{Step 2.}}} Assume, for contradiction, that $\tau_f(b)$ is finite for every choice of $b$. To simplify notation let
%%%%%%%
%%%%%%%
\begin{align*}
\begin{split}
q(\tau):=\mu a_{+}(\tau)-e^{\mu \tau}S(e^{-\mu \tau}F_{+}(\tau)),\qquad \uplambda(\tau):=C_\mathrm{trap}\delta_\wp\epsilon \jap{\tau}^{-3},\qquad \uplambda_0:=C_\mathrm{trap}\delta_\wp\epsilon.
\end{split}
\end{align*}
Note that $q=\dota_{+}$. 

{\underline{\emph{Step 2a.}}} We claim that if $(\psi_0,\psi_1)$ satisfy an orthogonality condition (see \eqref{eq:codim1}), then for each $|q_0|\leq \lambda_0$ there is a choice of $b$ in a neighborhood of zero for which $q(0)=q_0$. From Section~\ref{subsec:unstableint}, the orthogonality condition \eqref{eq:wpoutline9} determines $a_{+}$ such that
%%%%%%%
%%%%%%%
\begin{align*}
\begin{split}
a_+(t)=\bfOmega(\vecpsi,\vecZ_{-})-e^{\mu t}\tilS(e^{-\mu t}F_{+}).
\end{split}
\end{align*}
Recalling the extension procedure to $[-1,0]$ described at the beginning of the proof of Theorem~\ref{thm:main}, we define a map  $\calZ:C^\infty(\barcalC)\times C^\infty(\barcalC)\times I\to \bbR$, where $I$ is a neighborhood of zero in $\bbR$, by
%%%%%%%
%%%%%%%
\begin{align*}
\begin{split}
\calZ(\psi_0,\psi_1,b)=q(0),
\end{split}
\end{align*}
where $q$ is determined using initial data
%%%%%%%
%%%%%%%
\begin{align}\label{eq:shootingbdata1}
\begin{split}
 \Phi\vert_{\{t=0\}}=\Phi_0[\epsilon(\psi_0+b\tilvarphi_\mu)]\mand \partial_t\Phi\vert_{\{t=0\}}=\Phi_1[\epsilon(\psi_1-\mu b\tilvarphi_\mu)],
\end{split}
\end{align}
as in the statement of Theorem~\ref{thm:main}. We then restrict attention to $(\psi_0,\psi_1)$ satisfying the codimension one condition
%%%%%%%
%%%%%%%
\begin{align}\label{eq:codim1}
\begin{split}
\calZ(\psi_0,\psi_1,0)=0.
\end{split}
\end{align}
Note that $\calZ$ depends nonlinearly on $(\psi_0,\psi_1)$ because of $F_{+}$. Now since $\bfOmega((\tilvarphi_\mu,-\mu\tilvarphi_\mu)^{\intercal},\vecZ_{-})\simeq 1$, by \eqref{eq:shootingbdata1} and a similar argument as for the implicit function theorem in Section~\ref{subsec:unstableint}, we see that $\big|\frac{\partial q_0}{\partial b}\vert_{(\psi_0,\psi_1,0)}\big|\gtrsim1$. Our claim then follows from the implicit function theorem and \eqref{eq:codim1}.
 
{\underline{\emph{Step 2b.}}} By Step 1 and Step 2a, and our contradiction assumption, for every choice of $q_0$ there is $\tau_\trap(q_0)$ such that the corresponding solution satisfies $|q(\tau)|<\uplambda(\tau)$ for $\tau<\tau_\trap(q_0)$ and $|q(\tau_\trap(q_0))|=\uplambda(\tau_\trap(q_0))$. We use a standard shooting argument (see for instance \cite{DKSW,OP1}) to derive a contradiction from this. The main observation is that if 
$
\frac{1}{2}\uplambda(\tau)<| q(\tau)|<\uplambda(\tau),
$
for some $\tau\leq \tau_f$, then 
%%%%%%%
%%%%%%%
\begin{align}\label{eq:outgoing1}
\begin{split}
\frac{\ud}{\ud\tau}q^2(\tau)\geq \mu q^2(\tau).
\end{split}
\end{align}
Indeed, rewriting the equation $\frac{\ud}{\ud \tau}(e^{-\mu \tau}\dota_{+})=-S(e^{-\mu \tau}\dotF_{+})$ as
%%%%%%%
%%%%%%%
\begin{align*}
\begin{split}
\dot{q}(\tau)= \mu q(\tau)-e^{\mu\tau} S(e^{-\mu \tau}\dotF_{+}(\tau)),
\end{split}
\end{align*}
and multiplying by $2q(\tau)$, the first term on the right gives $2\mu q^2(\tau)$. On the other hand, by the arguments in Section~\ref{sec:parametercontrol},
%%%%%%%
%%%%%%%
\begin{align*}
\begin{split}
|e^{\mu\tau} S(e^{-\mu \tau}\dotF_{+}(\tau))|\leq c \uplambda(\tau)< c q(\tau),
\end{split}
\end{align*}
for some $c\ll \mu$, proving \eqref{eq:outgoing1}.  We will show that the map $\Lambda:(-\uplambda_0,\uplambda_0)\to\{\pm\uplambda_0\}$, $\Lambda(q_0)=q(\tau_\trap(q_0))\jap{\tau_\trap(q_0)}^3$ is continuous. Since, by \eqref{eq:outgoing1}, $\Lambda(q_0)=-\uplambda_0$ if $q_0$ is close to $-\uplambda_0$ and $\Lambda(q_0)=\uplambda_0$ if $q_0$ is close to $\uplambda_0$, the continuity of $\Lambda$ contradicts the intermediate value theorem. By continuous dependence on initial data, it suffices to prove that $\tau_\trap(\cdot)$ is continuous. Fix $q_0\in(-\uplambda_0,\uplambda_0)$ and let $q$ denote the corresponding solution. By \eqref{eq:outgoing1}, given $\upvarepsilon>0$ there exists $\updelta\in(0,1)$ such that if $(1-\updelta)\uplambda(\tau)<|q(\tau)|<\uplambda(\tau)$ for some $\tau<\tau_f$, then $|\tau_\trap(q_0)-\tau|<\upvarepsilon$. Let $\tau_1<\tau_f$ be such that $(1-\updelta^2)\uplambda(\tau_1)<|q(\tau_1)|<(1-\updelta^3)\uplambda(\tau_1)$, and note that if $q_1$ is sufficiently close to $q_0$ then the solution $\tilq$ corresponding to $q_1$ satisfies $(1-\updelta)\uplambda(\tau_1)<|\tilq(\tau_1)|<\uplambda(\tau_1)$, and hence $|\tau_\trap(q_0)-\tau_\trap(q_1)|\leq |\tau_\trap(q_0)-\tau_1|+|\tau_\trap(q_1)-\tau_1|<2\upvarepsilon$.
\end{proof}
%%%%%%%%%%%%
%%%%%%%%%%%%

%%%%%%%%%%%%%%%%%%%
%%%%%%%%%%%%%%%%%%%
%%%%%%%%%%%%%%%%%%%
\section{Parameter Control}\label{sec:parametercontrol}
%%%%%%%%%%%%%%%%%%%
%%%%%%%%%%%%%%%%%%%
%%%%%%%%%%%%%%%%%%%
In this section we prove Proposition~\ref{prop:bootstrappar1}. In the process we will also derive estimates on $\bfOmega_i(\RbfT^k\phi):=\bfOmega(\RbfT^k\vecphi,\vecZ_i)$, $i\in\{\pm\mu,1,\dots,2n\}$, $k\geq0$, which are of independent interest for the local-energy decay estimate. Recall equations~\eqref{equ:derivation_modul_equ3_alt} and \eqref{eq:wpoutline10} from Section~\ref{sec:interior} for $\dotvecwp=(\dotell,\dotxi-\ell)^{\intercal}$, $a_{-}$, and $a_{+}$,
%%%%%%%
%%%%%%%
\begin{align}\label{eq:pareqsrepeat1}
\begin{split}
\dotvecwp = \vecF_{\wp}(S\vecpsi,\ell,S\vecN-\beta\vecomega),\qquad \frac{\ud}{\ud t}(e^{-\mu t}a_{+})=-S(e^{-\mu t}F_{+}),\qquad \frac{\ud}{\ud t}(e^{\mu t}a_{-})=S(e^{\mu t}F_{-}),
\end{split}
\end{align}
where by a slight abuse of notation we have suppressed the derivatives on $\vecpsi$ in \eqref{equ:derivation_modul_equ3_alt} and written  $\vecF_{\wp}(S\vecpsi,\ell,S\veccalN_\wp-\beta\vecomega):= \vecG( S\partial_\Sigma^{\leq2} \vec{\psi}, \ell, S\vecN - \beta \vecomega)$. Here $\vecomega$ is defined as in \eqref{equ:derivation_modul_equ5} (see also \eqref{eq:omegasolutionform1}) as the solution of
%%%%%%%
%%%%%%%
\begin{align}\label{eq:vecomegarepeat1}
\begin{cases}
\partial_t\vecomega+\beta\vecomega=-S\vecF_\omega(\vecpsi,\ell,\vecomega)\quad &t> 0\\
\vecomega=0\quad &t\leq0
\end{cases},
\end{align}
 with $\vecF_\omega$ as in \eqref{eq:Fomegadef1} (where again we have suppressed the derivatives on $\vecpsi$ in the notation). In view of the spatial support of the test functions in imposing orthogonality conditions, all the integrations appearing in the definitions of $\vecF_\wp$, $F_{\pm}$, and $\vecF_\omega$ are over the region $\{\rho\leq \Reigenfunctioncutoffscale\}$. We will also need the equations for $\bfOmega_i(\psi)=\bfOmega(\vecpsi,\vecZ_i)$, $i=1,\dots,2n$, and $\bfOmega_{\pm}(\phi)=\bfOmega(\vecphi,\vecZ_{\pm}^\mu)$. First, recall that $\bfOmega_i(\psi)=\Upomega_i+\omega_i$, where $\vecomega=(\omega_1,\dots,\omega_{2n})$ is as in \eqref{eq:vecomegarepeat1} and $\vec{\Upomega}=(\Upomega_1,\dots\Upomega_{2n})$ satisfies
%%%%%%%
%%%%%%%
\begin{align}\label{eq:Upomegarepeat1}
\begin{split}
\vec{\Upomega}=\tilS(\vecN+\vecF_\omega),
\end{split}
\end{align}
with $\vecF_\omega$ and $N$ are as above (see \eqref{equ:derivation_modul_equ5} and \eqref{eq:upOmegasolutionform1}). Similarly, with $\tilF_{\pm}=(S-I)F_{\pm}$,\footnote{To be precise we should write $(S(e^{-\mu \cdot}\tilF_{+}))(t)$ instead of $S(e^{-\mu t}\tilF_{+}(t))$.}
%%%%%%%
%%%%%%%
\begin{align}\label{eq:bfOmegapmrepeat1}
\begin{split}
&\bfOmega_{+}(\phi)(t)=e^{\mu t}S(e^{-\mu t}\tilF_{+}(t)),\\
&\bfOmega_{-}(\phi)(t)=e^{-\mu t}S(e^{\mu t}\tilF_{-}(t)),
\end{split}
\end{align}
and
%%%%%%%
%%%%%%%
\begin{align}\label{eq:bfOmegapmrepeat2}
\begin{split}
&\frac{\ud}{\ud t}\bfOmega_{+}(\phi)(t)=e^{\mu t}S(e^{-\mu t}\tilF_{+}'(t)),\\
&\frac{\ud^2}{\udt^2}\bfOmega_{+}(\phi)(t)=e^{\mu t}S(e^{-\mu t}\tilF_{+}''(t)),
\end{split}
\end{align}
and
%%%%%%%
%%%%%%%
\begin{align}\label{eq:bfOmegapmrepeat3}
\begin{split}
&\frac{\ud}{\ud t}\bfOmega_{-}(\phi)(t)=e^{-\mu t}S(e^{\mu t}\tilF_{-}'(t)),\\
&\frac{\ud^2}{\udt^2}\bfOmega_{-}(\phi)(t)=e^{-\mu t}S(e^{\mu t}\tilF_{-}''(t)).
\end{split}
\end{align}
Finally recall that the smoothing operator $S$ and the operator $\tilS$ are given by
%%%%%%%
%%%%%%%
\begin{align*}
\begin{split}
Sf(t)=\int K_S(s)f(t-s)\ud s,\qquad \tilS f(t)=\int K_\tilS(s)f(t-s)\ud s,
\end{split}
\end{align*}
where the smooth kernel $K_S$ and the non-smooth kernel $K_\tilS$ are supported in $[0,1]$. In particular, if $|f(t)|\lesssim \jap{t}^{-\gamma}$ for some $\gamma>0$, then we also have $$|\tilS f(t)|+ |Sf(t)|\lesssim \jap{t}^{-\gamma}.$$ We will use this observation in this section without further mention. Our starting point  is to estimate $\vecomega$.
%%%%%%%%%%
%%%%%%%%%%
\begin{lemma}\label{lem:vecomega1}
Under the bootstrap assumptions \eqref{eq:a+trap}--\eqref{eq:Tjphienergyb2}, if $\Reigenfunctioncutoffscale$ is sufficiently large, then
%%%%%%%
%%%%%%%
\begin{align*}
\begin{split}
\big(\frac{\ud}{\ud t}\big)^k\vecomega\lesssim \epsilon^2 \jap{t}^{-\frac{9}{2}+\kappa},\quad \forall k\geq 0.
\end{split}
\end{align*}
\end{lemma}
%%%%%%%%%%
%%%%%%%%%%
\begin{proof}
Integrating equation \eqref{eq:vecomegarepeat1} gives
%%%%%%%
%%%%%%%
\begin{align*}
\begin{split}
\vecomega(t)=\int_0^t e^{-\beta(t-s)}(S\vecF_\omega(\vecpsi,\ell,\vecomega))(s)\ud s.
\end{split}
\end{align*}
Recalling that $|\vecF_\omega(x,p,w)|\lesssim |x| (|x|+|w|)$ (see \eqref{eq:Fomegaestimate1}), the desired estimate for $k=0$ follows from the assumptions \eqref{eq:a+trap}--\eqref{eq:Tjphienergyb2}. Here a point that deserves further clarification is the relation between $\vecpsi$ and $\phi$ and its derivatives. First, note that by writing $\vecpsi=\vecphi+a_{+}\vecZ_\mu^{+}+a_{-}\vecZ_\mu^{-}$ and using the bootstrap assumptions on $a_{\pm}$, we can reduce the estimate on $\vecpsi$ to that on $\vecphi$. Then observe that by the first component of equation \eqref{eq:wpoutline6} (viewed as an equation for $\vecphi$) we can estimate $|\dotphi|\lesssim |\partial\phi|+|\dotvecwp|+|\phi|^2$. Finally, the higher order estimates $k\geq 1$ follow by exactly the same argument after differentiating equation \eqref{eq:vecomegarepeat1} and absorbing the time derivatives by the smoothing operator $S$.
\end{proof}
%%%%%%%%%%
%%%%%%%%%%
For the proof of Proposition~\ref{prop:bootstrappar1} we also need to prove some estimates for $\bfOmega_i(\phi)$.
%%%%%%%%%%
%%%%%%%%%%
\begin{lemma}\label{lem:Omega1}
Under the bootstrap assumptions \eqref{eq:a+trap}--\eqref{eq:Tjphienergyb2} and for $\Reigenfunctioncutoffscale$ sufficiently large,
%%%%%%%
%%%%%%%
\begin{align*}
\begin{split}
|\bfOmega_i(\phi)|\lesssim o_{\Reigenfunctioncutoffscale}(1)\epsilon\jap{t}^{-\frac{5}{2}+\kappa}+\callO(\wp)\epsilon\jap{t}^{-\frac{5}{2}+\kappa}+\epsilon\jap{t}^{-\frac{9}{2}+\kappa},\quad i\in\{\pm,1,\dots,2n\},
\end{split}
\end{align*}
where $o_{\Reigenfunctioncutoffscale}(1)$ denotes a constant that goes to zero as $\Reigenfunctioncutoffscale\to \infty$.
\end{lemma}
%%%%%%%%%%
%%%%%%%%%%
\begin{proof}
Starting with $\bfOmega_1(\psi),\dots,\bfOmega_{2n}(\psi)$, observe that in view of Lemma~\ref{lem:vecomega1} it suffices to estimate $\vec{\Upomega}$. For this we distinguish between the first $n$ components and the last $n$ components of $\vec{\Upomega}$ by representing them as $\Upomega_i$ and $\Upomega_{n+i}$, respectively, with $i=1,\dots,n$. Appealing again to Lemma~\ref{lem:vecomega1} and the bootstrap assumptions \eqref{eq:a+trap}--\eqref{eq:Tjphienergyb2}, the only terms on the right-hand side of \eqref{eq:Upomegarepeat1} which need special treatment are the linear terms in $\vec{\psi}$. To use the bootstrap assumptions to draw this conclusion, note that one can account for the difference between $\vecpsi$ and $\phi$ and its derivatives in the same way as in the proof of Lemma~\ref{lem:vecomega1}. For the linear terms, recall from the form of $\vecN$ from Section~\ref{sec:interior} (see \eqref{equ:derivation_modul_equ1}) that the only term that is not bounded by $\callO(\wp)\epsilon\jap{t}^{-\frac{5}{2}+\kappa}$ directly by the bootstrap assumptions is $\bfOmega(\vecpsi,M\vecZ_i)$. But, $M\vecZ_i$ would be zero, if it were not for the cutoff function in the definition of $\vecZ_i$. Treating the difference between $\vecpsi$ and $\phi$ as before, since every term in the definition of $\vecZ_i$ comes with a decay of $r^{-n+1}$ or a factor of $\ell$, our task reduces to estimating $\Reigenfunctioncutoffscale^{-1}\angles{\chi_{\{r\simeq \Reigenfunctioncutoffscale\}}\partial\phi}{\callO(r^{-n+1})}$. This is then bounded by (recall that $n=5$)
%%%%%%%
%%%%%%%
\begin{align}\label{eq:lemOmega1temp1}
\begin{split}
\Reigenfunctioncutoffscale^{-1}\|\chi_{\{r\simeq \Reigenfunctioncutoffscale\}}r^{-\frac{5}{2}+\kappa}\partial\phi\|_{L^2}\Big(\int_{\{r\simeq \Reigenfunctioncutoffscale\}}r^{-5+1+5-2\kappa}\ud r\Big)^{\frac{1}{2}}\lesssim \epsilon \Reigenfunctioncutoffscale^{-\kappa}\jap{t}^{-\frac{5}{2}+\kappa},
\end{split}
\end{align}
where for the last estimate we have used the bootstrap assumption \eqref{eq:dphiL2b1}. The case of $\Upomega_{n+i}$ is similar, with the difference that now we need to estimate $\bfOmega(\vecpsi,M\vecZ_{n+i})$ instead of $\bfOmega(\vecpsi,M\vecZ_i)$. Since $M\vecZ_{n+i}$ would be $\vecZ_{i}$ if it were not for the cutoffs in the definitions of $\vecZ_i$ and $\vecZ_{n+i}$, this leads to estimating $\bfOmega(\vecpsi,\vecZ_i)$ and $\Reigenfunctioncutoffscale^{-2}\angles{\chi_{\{r\simeq \Reigenfunctioncutoffscale\}}\phi}{\callO(r^{-n+1})}$. The first term was already treated above, and the second term is bounded, using \eqref{eq:dphiL2b1}, as (recall that $n=5$)
%%%%%%%
%%%%%%%
\begin{align}\label{eq:lemOmega1temp2}
\begin{split}
\Reigenfunctioncutoffscale^{-2}\|\chi_{\{r\simeq \Reigenfunctioncutoffscale\}}r^{-\frac{5}{2}+\kappa}\phi\|_{L^2}\Big(\int_{\{r\simeq \Reigenfunctioncutoffscale\}}r^{-5+1+5-2\kappa}\ud r\Big)^{\frac{1}{2}}\lesssim \epsilon \Reigenfunctioncutoffscale^{-1-\kappa}\jap{t}^{-\frac{5}{2}+\kappa}.
\end{split}
\end{align}
This proves the estimates for $\bfOmega_{i}(\psi)$ and the passage to $\bfOmega(\phi)$ is again by decomposing $\vecpsi$ in terms of $\vecphi$ and $a_{\pm}$ and observing the extra smallness in $\Reigenfunctioncutoffscale$ coming from $\angles{\vecZ_i}{\vecZ_{\pm}}$, $i=1,\dots,2n$.  The estimates for $\bfOmega_{\pm}(\phi)$ using \eqref{eq:bfOmegapmrepeat1} are similar, where again we use the smallness of $\angles{\vecZ_i}{\vecZ_{\pm}}$, $i=1,\dots,2n$, when estimating the linear contributions of $\dotwp$ in $F_{\pm}$.
\end{proof}
%%%%%%%%%%
%%%%%%%%%%
We are now ready to prove Proposition~\ref{prop:bootstrappar1}.
%%%%%%%%%%
%%%%%%%%%%
\begin{proof}[Proof of Proposition~\ref{prop:bootstrappar1}]
We start with the estimates for $\dotvecwp$ for which we use the first equation in~\eqref{eq:pareqsrepeat1}. As in the proof of Lemma~\ref{lem:vecomega1}, in view of the presence of the smoothing operator $S$, the higher derivatives are treated in the same way as $\dotwp$. For $\dotwp$, in view of the estimate for $\vecomega$ from Lemma~\ref{lem:vecomega1} and the bootstrap assumptions \eqref{eq:a+trap}--\eqref{eq:Tjphienergyb2}, the only terms on the right-hand side of the equation $\dotvecwp=\vecF_\wp$ which need special attention are the linear terms in $\vecpsi$. Here the difference between $\vecpsi$ and $\phi$ and its derivatives is accounted for in the same way as in the proof of Lemma~\ref{lem:vecomega1}. Turning to these linear terms, recall that as above (see \eqref{equ:derivation_modul_equ1}) they are given by $\bfOmega(\vecpsi,M\vecZ_i)$ and $\bfOmega(\vecpsi,M\vecZ_{n+i})$, $i=1,\dots,n$. But, these can be estimated exactly as in \eqref{eq:lemOmega1temp1} and \eqref{eq:lemOmega1temp2}. The estimate for $a_{-}$ is similar, where now we use the last equation in \eqref{eq:pareqsrepeat1} which leads to the representation 
%%%%%%%
%%%%%%%
\begin{align*}
\begin{split}
a_{-}(t)=a_{-}(0)e^{-\mu t}+\int_0^te^{-\mu t}S(e^{\mu s}F_{-}(s))\ud s.
\end{split}
\end{align*}
The first term already has better decay than we need. For the second term we can again consider the linear and quadratic and higher order contributions of $F_{-}$ separately, and gain smallness in $\Reigenfunctioncutoffscale$ for the linear terms from the smallness of $\angles{\vecZ_i}{\vecZ_{\pm}}$. The higher derivatives are treated similarly, where in the case where derivatives fall on $SF_{-}$ we can absorb them in the smoothing operator $S$. To estimate $a_{+}$, we use the triangle inequality and the bootstrap assumption \eqref{eq:a+trap} to bound,
%%%%%%%
%%%%%%%
\begin{align*}
\begin{split}
|a_{+}(t)|\lesssim \delta_{\wp}t^{-3}+e^{\mu t}S(e^{-\mu t}F_{+}).
\end{split}
\end{align*}
The desired estimate then follows by estimating $F_{+}$ using similar considerations as for $F_{-}$. The higher derivative estimates for $a_{+}$ also follow similarly by using equation \eqref{eq:pareqsrepeat1} to express $\dota_{+}$ algebraically in terms of $a_{+}$ and $F_{+}$ as $\dota_{+}=\mu a_{+}+e^{\mu t}S(e^{\mu t}F_{+})$.
\end{proof}
%%%%%%%%%%
%%%%%%%%%%
In addition to Proposition~\ref{prop:bootstrappar1}, we will also need some integrated estimates on $\bfOmega_i(\RbfT^k\phi)$, $a_\pm^{(k)}$, and $\dotwp^{(k)}$ for our local-energy decay estimate. These estimates are the content of the next lemma.
%%%%%%%%%%
%%%%%%%%%%
\begin{lemma}\label{lem:OmegaiTkphi}
Under the bootstrap assumptions \eqref{eq:a+trap}--\eqref{eq:Tjphienergyb2}, if $\Reigenfunctioncutoffscale$ is sufficiently large, then for $k=1,2$, $j\geq 0$, and $i\in\{\pm,1,\dots,2n\}$,
%%%%%%%
%%%%%%%
\begin{align}\label{eq:OmegaiTkphi}
\begin{split}
&\|\bfOmega_i(\RbfT^k\phi)\|_{L^2([t_1,t_2])}+\|\dota_{\pm}^{(k+j)}\|_{L^2([t_1,t_2])}+\|\dotwp^{(k+1+j)}\|_{L^2([t_1,t_2])}\\
&\lesssim (\delta_\wp+o_{\Reigenfunctioncutoffscale}(1)+\callO(\wp))\epsilon \jap{t_1}^{-3}+\epsilon^2\jap{t_1}^{-\frac{9}{2}+\kappa}\\
&\quad+(o_{\Reigenfunctioncutoffscale}(1)+\callO(\wp))(\|\RbfT^k\phi\|_{LE([t_1,t_2])}+\sup_{t_1\leq \tau\leq t_2}\|\RbfT^k\phi\|_{E(\Sigma_\tau)}),\\
\end{split}
\end{align}
where $o_{\Reigenfunctioncutoffscale}(1)$ denotes a constant that goes to zero as $\Reigenfunctioncutoffscale\to \infty$. With the same notation and for $k>2$
%%%%%%%
%%%%%%%
\begin{align*}
\begin{split}
&\|\bfOmega_i(\RbfT^k\phi)\|_{L^2([t_1,t_2])}+\|\dota_{\pm}^{(k+j)}\|_{L^2([t_1,t_2])}+\|\dotwp^{(k+1+j)}\|_{L^2([t_1,t_2])}\\
&\lesssim (\delta_\wp+o_{\Reigenfunctioncutoffscale}(1)+\callO(\wp))\epsilon \jap{t_1}^{-3}+\epsilon^2\jap{t_1}^{-\frac{9}{2}+\kappa}\\
&\quad+(o_{\Reigenfunctioncutoffscale}(1)+\callO(\wp))\sum_{j=2}^k(\|\RbfT^j\phi\|_{LE([t_1,t_2])}+\sup_{t_1\leq \tau\leq t_2}\|\RbfT^j\phi\|_{E(\Sigma_\tau)}).\\
\end{split}
\end{align*}
\end{lemma}
%%%%%%%%%%
%%%%%%%%%%
\begin{proof}
We start with the estimate for $\bfOmega_i(\RbfT^k\psi)$, $i\in\{1,\dots,n\}$. This is achieved by differentiating the expression $\bfOmega_i(\psi)=\omega_i+\Upomega_i$. The desired estimate is then derived similarly to the proof of Lemma~\ref{lem:Omega1}. Indeed, using the bootstrap assumptions to estimate the quadratic and higher order terms by $\epsilon^2t_1^{-\frac{9}{2}+\kappa}$, we see that except for the contribution of $\bfOmega(\RbfT^k\vecpsi,M\vecZ_i)$ the remaining terms are bounded by 
%%%%%%%
%%%%%%%
\begin{align*}
\begin{split}
\callO(\wp)\|\RbfT^k\phi\|_{L^2([t_1,t_2])}+\callO(\wp)\|\dota_{\pm}^{(k)}\|_{L^2([t_1,t_2])}+\callO(\wp)\|\dotwp^{k+1}\|_{L^2([t_1,t_2])}.
\end{split}
\end{align*}
Here as usual we have expressed $\vecpsi$ in terms of $\phi$ and its derivatives as well as $a_{\pm}$ and $\dotwp$. Note that the last two terms above can be absorbed on the left-hand side of \eqref{eq:OmegaiTkphi}. Similarly, using the smallness of $MZ=\vecZ_i$, $i=1,\dots,n$, we can estimate the contribution of $\bfOmega_i(\RbfT^k\vecpsi,M\vecZ_i)$ by 
%%%%%%%
%%%%%%%
\begin{align*}
\begin{split}
o_{\Reigenfunctioncutoffscale}(1)\|\RbfT^k\phi\|_{L^2([t_1,t_2])}+o_{\Reigenfunctioncutoffscale}(1)\|\dota_{\pm}^{(k)}\|_{L^2([t_1,t_2])}+o_{\Reigenfunctioncutoffscale}(1)\|\dotwp^{k+1}\|_{L^2([t_1,t_2])}.
\end{split}
\end{align*}
Note that even though $\tilS$ in \eqref{eq:Upomegarepeat1} is not a smoothing operator, since only $\vecpsi$ and not $\RbfT\vecpsi$ appears in $\vecN$ and $\vecF_\omega$ (and similarly in $\tilF_{\pm}$ in the discussion for $\bfOmega_{\pm}(\RbfT^k\phi)$ below) and spatial derivatives can be integrated by parts to the lower order terms, there is no loss of regularity in these estimates. The passage from $\bfOmega_i(\RbfT^k\psi)$ to $\bfOmega_i(\RbfT^k\phi)$ follows as usual. The estimate for $\bfOmega_{n+i}(\RbfT^k\phi)$ now also follows from that of $\bfOmega_{i}(\RbfT^k\phi)$ in the same way as in the proof of Lemma~\ref{lem:Omega1}. The estimates for $\bfOmega_{\pm}(\RbfT^k\phi)$ are proved similarly where now we use the differentiated equations \eqref{eq:bfOmegapmrepeat2} and \eqref{eq:bfOmegapmrepeat3}. Once the estimates for $\bfOmega_i(\RbfT^k\phi)$, $i\in\{\pm,1,\dots,2n\}$, are established, the estimates for $\dota_{-}^{(k+j)}$ and $\dotwp^{(k+1+j)}$ follow as in the proof of Proposition~\ref{prop:bootstrappar1} by differentiating the corresponding equations in \eqref{eq:pareqsrepeat1}. As usual, any excess derivatives can be absorbed by the smoothing operator $S$. The argument for $\dota_{+}$ is more delicate, and it is here that the improved decay in \eqref{eq:a+trap} comes in. Differentiating the differential equation for $a_{+}$ from  \eqref{eq:pareqsrepeat1}, and using the notation $\dotF_{+}=\frac{\ud}{\ud t}F_{+}$, gives
%%%%%%%
%%%%%%%
\begin{align*}
\begin{split}
\frac{\ud}{\ud t}(e^{-\mu t}\dota_{+})=-S(e^{-\mu t}\dotF_{+}).
\end{split}
\end{align*}
Integrating from the final bootstrap time $\tau_f$ and using the algebraic relation $\dota_{+}=\mu a_{+}-e^{\mu t}S(e^{-\mu t}F_{+})$ (which is a rewriting of the differential equation for $a_{+}$), for any $t\leq \tau_f$ we get
%%%%%%%
%%%%%%%
\begin{align}\label{eq:higheraplustemp1}
\begin{split}
\dota_{+}(t)=(\mu a_{+}(\tau_f)-e^{\mu\tau_f}S(e^{-\mu \tau_f}F_{+}(\tau_f)))e^{-\mu(\tau_f-t)}-\int_{t}^{\tau_f}e^{-\mu (s-t)}(e^{\mu s}S(e^{-\mu s}\dotF_{+}(s)))\ud s.
\end{split}
\end{align}
The desired estimate for $\dota_{+}$ now follows from \eqref{eq:a+trap} and an application of Schur's test with the kernel $e^{-\mu(s-t)}\chi_{s\geq t}$, where we use similar considerations as before to bound $\dotF_{+}$. The higher order estimates for $a_{+}$ are proved similarly by differentiating equation \eqref{eq:higheraplustemp1}. For this purpose note that differentiation of the integral on the right-hand side of \eqref{eq:higheraplustemp1} gives
%%%%%%%
%%%%%%%
\begin{align*}
\begin{split}
&e^{\mu t}S(e^{-\mu t}\dotF_{+}(t))+\int_t^{\tau_f} \big(\frac{\ud}{\ud s}e^{-\mu(s-t)}\big)e^{\mu s}S(e^{-\mu s}\dotF_{+}(s))\ud s\\
&=e^{-\mu(\tau_f-t)}e^{\mu \tau_f}S(e^{-\mu \tau_f}\dotF_{+}(\tau_f))-\int_t^{\tau_f} e^{-\mu(s-t)}e^{\mu s}S(e^{-\mu s}\ddot{F}_{+}(s))\ud s.
\end{split}
\end{align*}
This can be estimated by the same considerations as above, where for the boundary term we also use \eqref{eq:d2Tphienergyb1} with $\partial_\Sigma^2$ replaced by $\jap{\rho}^{-2}$ (which can be done in dimension $5$ and higher; see for instance Lemma~\ref{lem:ellipticestimate1}.)
\end{proof}
%%%%%%%%%%
%%%%%%%%%%

%%%%%%%%%%%%%%%%%%%
%%%%%%%%%%%%%%%%%%%
%%%%%%%%%%%%%%%%%%%
\section{Local Energy Decay}\label{sec:LED}
%%%%%%%%%%%%%%%%%%%
%%%%%%%%%%%%%%%%%%%
%%%%%%%%%%%%%%%%%%%
In this section we prove linear energy and local energy decay estimates. The relatively straightforward nonlinear applications are postponed to the next section after the nonlinearity and source terms of the equation are calculated more carefully. 

For any $\tau_1<\tau_2$, let
%%%%%%%
%%%%%%%
\begin{align*}
\begin{split}
\Sigma_{\tau_1}^{\tau_2}= \cup_{\tau=\tau_1}^{\tau_2}\Sigma_\tau.
\end{split}
\end{align*}
and consider two functions $\uppsi,f:\Sigma_0^T\to \bbR$, satisfying
%%%%%%%
%%%%%%%
\begin{align}\label{eq:LEDlinearmodel1}
\begin{split}
\calP\uppsi=f.
\end{split}
\end{align}
We make the following assumptions on $\calP$, which are consistent with the linear operator arising in our problem: In the global non-geometric coordinates from Section~\ref{sec:GNGC}, $\calP$ satisfies the properties stated in Remark~\ref{rem:nongeomglobal1}, while $\calP_0$ in Remark~\ref{rem:nongeomglobal1} takes the form given in Section~\ref{sec:GGC} in the global geometric coordinates defined there. In the interior coordinates from Section~\ref{sec:INGC} and the exterior coordinate from Section~\ref{sec:ENGC}, we assume that $\calP$ takes the forms \eqref{eq:LEDcalPint1} and \eqref{eq:abstractexteq1}, \eqref{eq:callP2}, \eqref{eq:ErrcallP1}, \eqref{eq:extgeomlin1} respectively. We use $K_\Int$ and $K_\ext$ to denote two large compact regions in $\Sigma_0^T$ with $K_\ext\subseteq K_\Int$, such that the coordinates $(t,\rho,\omega)$ (Section~\ref{sec:INGC}) are defined in a neighborhood $U_\Int$ of $K_\Int$, and the coordinates $(\tau,\rho,\theta)$ (Section~\ref{sec:ENGC}) are defined in a neighborhood of $\overline{K_\ext^c}$. We assume that $\Reigenfunctioncutoffscale$ in Section~\ref{sec:interior} (see for instance Section~\ref{subsec:modulationeqs}) is such that the region $\{\rho\leq \Reigenfunctioncutoffscale\}$ is much larger than $K_\Int$.

For any $\tau$ the energy norm of $\uppsi$ on $\Sigma_\tau$ is defined by
%%%%%%%
%%%%%%%
\begin{align}\label{eq:standardenergydef1}
\begin{split}
E[\uppsi](\tau)\equiv\|\uppsi\|_{E(\Sigma_\tau)}^2&:=\int_{\Sigma_\tau}\chi_{\leq \tilR}(|\partial\uppsi|^2+\jap{\rho}^{-2}|\uppsi|^2)\ud V\\
&\quad+\int_{\Sigma_\tau}\chi_{\geq \tilR}(|\partial_\Sigma \uppsi|^2+r^{-2}|T\uppsi|^2+r^{-2}|\uppsi|^2) \ud V.
\end{split}
\end{align}
Here $\chi_{\geq \tilR}$ is a cutoff function supported in $\calC_\hyp$, for some fixed large $\tilR\gg1$, and $\chi_{\leq \tilR}=1-\chi_{\geq \tilR}$. The local energy norm on any (space-time) region $\calR$ of the domain of definition of $\uppsi$ is defined by
%%%%%%%
%%%%%%%
\begin{align*}
\begin{split}
\|\uppsi\|_{LE(\calR)}^2=\int_{\calR}\chi_{\leq \tilR} ((\rho\tilpsi)^2+(\rho\partial \tilpsi)^2+( \partial_\rho\tilpsi)^2)\ud V+\int_{\calR}\chi_{\geq \tilR}( r^{-3-\alpha}\uppsi^2+r^{-1-\alpha}(\partial\uppsi)^2)\ud V.
\end{split}
\end{align*}
Here $0<\alpha\ll 1$ is a fixed small positive number, and $\rho$ and $r$ are the radial coordinates introduced in Section~\ref{sec:coordinates}. The dual local energy norm  is defined by 
%%%%%%%
%%%%%%%
\begin{align*}
\begin{split}
\|f\|_{LE^\ast(\calR)}^2=\int_{\calR}\chi_{\leq \tilR} f^2\ud V+\int_{\calR}\chi_{\geq \tilR}r^{1+\alpha}f^2\ud V.
\end{split}
\end{align*}
We use the notation
%%%%%%%
%%%%%%%
\begin{align*}
\begin{split}
\|\uppsi\|_{L^pL^q(\Sigma_{\tau_1}^{\tau_2})}=\Big(\int_{\tau_1}^{\tau_2}\|\uppsi\|_{L^q(\Sigma_\uptau)}^p\ud \tau\Big)^{\frac{1}{p}},
\end{split}
\end{align*}
with the usual modificaion when $p=\infty$. When $p=q$ we simply write $\|\uppsi\|_{L^p(\Sigma_{\tau_1}^{\tau_2})}$, and similarly with $\Sigma_{\tau_1}^{{\tau_2}}$ replaced by any other region. We also occasionally use the notation
%%%%%%%
%%%%%%%
\begin{align*}
\begin{split}
\angles{\uppsi_1}{\uppsi_2}\equiv\angles{\uppsi_1}{\uppsi_2}_{\Sigma_{\tau}}=\int_{\Sigma_{\tau}}\uppsi_1\uppsi_2\,\ud V_{\Sigma_\tau}.
\end{split}
\end{align*} Since our focus in this section is on linear estimates, we introduce $\bfUpomega_k$ as a linear proxy for $\bfOmega_k$  (which was defined nonlinearly in terms of $\vecpsi$ and $\vecphi$ in Section~\ref{sec:interior}):
%%%%%%%
%%%%%%%
\begin{align*}
\begin{split}
&\bfUpomega_k(\uppsi)(c)=-\int_{\{\uptau= c\}}Z_kn^\alpha\partial_\alpha \uppsi \sqrt{|h|}\ud y,\quad \bfUpomega_{n+k}(\uppsi(c))=\int_{\{\uptau= c\}}\uppsi Z_kn^\alpha\partial_\alpha \tiluptau \sqrt{|h|}\ud y,\quad k=1,\dots,n,\\
&\bfUpomega_{\mu}^{\pm}(\uppsi)(c)=\int_{\{\uptau= c\}}(\pm\mu \uppsi Z_\mu \partial_\alpha \tiluptau-Z_\mu \partial_\alpha \uppsi)n^\alpha \sqrt{|h|}\ud y.
\end{split}
\end{align*}
Here $n$ denotes the normal to $\Sigma_c$ with respect to $h$, and $y$ denotes the spatial variables (say $(\uprho,\uptheta)$) on $\Sigma_c$. Our goal in this section is to prove the following two estimates. The first is the energy estimate. 
%%%%%%%%%%%%%
%%%%%%%%%%%%%
\begin{proposition}\label{prop:energyestimate1}
Suppose $\uppsi$ satisfies $\calP\uppsi=f$, and $\sum_{k\in\{\pm\mu,1,\dots,2n\}}|\bfUpomega_k(\uppsi(t))|\leq \delta\|\uppsi\|_{E(\Sigma_t)}$. Then if $\delta$ is sufficiently small, for any $t_1<t_2$ and $\varepsilon\ll 1$, $\uppsi$ satisfies the estimates
%%%%%%%
%%%%%%%
\begin{align}\label{eq:linenergyestimate1}
\begin{split}
&\sup_{t_1\leq t\leq t_2}\|\uppsi\|_{E(\Sigma_{t})}\lesssim \|\uppsi\|_{E(\Sigma_{t_1})}+C_\varepsilon \|f\|_{L^1L^2(\Sigma_{t_1}^{t_2})},\\
&\sup_{t_1\leq t\leq t_2}\|\uppsi\|_{E(\Sigma_{t})}\lesssim \|\uppsi\|_{E(\Sigma_{t_1})}+C_\varepsilon \|f\|_{LE^\ast(\Sigma_{t_1}^{t_2})}+\|\RbfT f\|_{LE^\ast(\Sigma_{t_1}^{t_2})}\\
&\phantom{\sup_{t_1\leq t\leq t_2}\|\uppsi\|_{E(\Sigma_{t})}\lesssim }+\|f\|_{L^\infty L^2(\Sigma_{t_1}^{t_2})}+\varepsilon \|\uppsi\|_{LE(\Sigma_{t_1}^{t_2})}.
\end{split}
\end{align}
\end{proposition}
%%%%%%%%%%%%%
%%%%%%%%%%%%%
The second estimate we will prove in this section is a local energy decay (LED) estimate.
%%%%%%%%%%%%%
%%%%%%%%%%%%%
\begin{proposition}\label{prop:LED1}
Suppose $\uppsi$ satisfies $\calP\uppsi=f$, and $\sum_{k\in\{\pm\mu,1,\dots,2n\}}|\bfUpomega_k(\uppsi(t))|\leq \delta\|\uppsi\|_{E(\Sigma_t)}$. Then for any $t_1<t_2$ and $\varepsilon\ll 1$, $\uppsi$ satisfies the estimates
%%%%%%%
%%%%%%%
\begin{align*}
\begin{split}
&\|\uppsi\|_{LE(\Sigma_{t_1}^{t_2})}\lesssim \sum_{k\in\{\pm\mu,1,\dots,2n\}}\|\bfUpomega_k(\uppsi)\|_{L^2([t_1,t_2])}+\|\uppsi\|_{E(\Sigma_{t_1})}+\|f\|_{L^1L^2(\Sigma_{t_1}^{t_2})},\\
&\|\uppsi\|_{LE(\Sigma_{t_1}^{t_2})}\lesssim \sum_{k\in\{\pm\mu,1,\dots,2n\}}\|\bfUpomega_k(\uppsi)\|_{L^2([t_1,t_2])}+\|\uppsi\|_{E(\Sigma_{t_1})}+\|f\|_{LE^\ast(\Sigma_{t_1}^{t_2})}\\
&\phantom{\|\uppsi\|_{LE(\Sigma_{t_1}^{t_2})}\lesssim}+\|\RbfT f\|_{LE^\ast(\Sigma_{t_1}^{t_2})}+\|f\|_{L^\infty L^2(\Sigma_{t_1}^{t_2})}.
\end{split}
\end{align*}
\end{proposition}
%%%%%%%%%%%%%
%%%%%%%%%%%%%
\begin{remark}\label{rem:Omegalinear1}
As mentioned earlier $\bfUpomega_k$ is a linear substitute for $\bfOmega_k$. It is easy to see from our proofs that in Propositions~\ref{prop:energyestimate1} and~\ref{prop:LED1} one can replace $\bfUpomega_k$ by any other choice $\tilde{\bfUpomega}_k$ as long as $\|\tilde{\bfUpomega}_k(\uppsi)-\bfUpomega_k(\uppsi)\|_{L^2([t_1,t_2])}$ is bounded by a small multiple of the $LE$ norm of $\uppsi$. In our nonlinear applications we will use this observation to apply these propositions with $\bfUpomega_k$ replaced by $\bfOmega_k$. The condition $|\bfOmega_k|(\uppsi(t))\leq \delta\|\uppsi\|_{E(\Sigma_t)}$ will always be satisfied in our applications as a consequence of the orthogonality conditions. See for instance the arguments in Lemmas~\ref{lem:Omega1} and~\ref{lem:OmegaiTkphi}.
\end{remark}
%%%%%%%%%%%%%
%%%%%%%%%%%%%
%%%%%%%%%%%%%
%%%%%%%%%%%%%
\begin{remark}\label{rem:fLED1}
The proof of Proposition~\ref{prop:LED1} requires several multiplier identities. In applications, where we consider the equation after commuting derivatives, we may want to perform some integration by parts in the term $fQ\uppsi$, where $Q\uppsi$ denotes the multiplier, before placing $f$ in $LE^\ast(\Sigma_{t_1}^{t_2})$ or $L^1L^2(\Sigma_{t_1}^{t_2})$. This is the case for instance where $f$ is of the form $\partial_\Sigma^2g$, where $g$ denotes the unknown with fewer commuted derivatives. While such integration by parts manipulations are not explicitly contained  in  the statement of Proposition~\ref{prop:LED1}, they can be easily incorporated by an inspection of the proof. Specifically, they can be performed in the treatment of equation \eqref{eq:phifardef1} in Lemma~\ref{lem:phifar1}.
\end{remark}
%%%%%%%%%%%%
%%%%%%%%%%%%
\begin{remark}
The explanation for the second estimate in \eqref{eq:linenergyestimate1} is the same as for the corresponding estimate in Proposition~\ref{prop:LEDproduct}. See Remark~\ref{rem:LEDproduct1}.
\end{remark}
%%%%%%%%%%%%
%%%%%%%%%%%%
%%%%%%%%%%%%%
%%%%%%%%%%%%%
We start with the proof of Proposition~\ref{prop:energyestimate1}.
%%%%%%%%%%%%%
%%%%%%%%%%%%%
\begin{proof}[Proof of Proposition~\ref{prop:energyestimate1}]
Recall from Remark~\ref{rem:nongeomglobal1}, part (3), that in the global coordinates $(\uptau,\uprho,\upomega)$
%%%%%%%
%%%%%%%
\begin{align*}
\begin{split}
\calP\uppsi=\frac{1}{\sqrt{|\bfh|}}\partial_\mu (\sqrt{|\bfh|}(\bfh^{-1})^{\mu\nu}\partial_\nu\uppsi)+V\uppsi+\tilcalbfP, 
\end{split}
\end{align*}
where $\tilcalbfP$ has the structure given in Remark~\ref{rem:nongeomglobal1}, and $|\partial_\uptau \bfh|\lesssim \epsilon \uptau^{-\gamma}$ for some $\gamma>1$. We multiply equation \eqref{eq:LEDlinearmodel1} by $\partial_\uptau\uppsi\sqrt{|\bfh|}$ and integrate. Note that the contribution of $f\partial_\uptau\uppsi$ can be estimated by the right-hand side of each estimate in \eqref{eq:linenergyestimate1} plus a small multiple of the corresponding left-hand side, as in the proof of Proposition~\ref{prop:LEDproduct}. The main term in $\calP\uppsi \partial_\uptau\uppsi\sqrt{|\bfh|}$ is
%%%%%%%
%%%%%%%
\begin{align}\label{eq:linenergyestimatetemp1}
\begin{split}
(\calP-\tilcalbfP)\uppsi\partial_\uptau\uppsi \sqrt{|\bfh|}&=\partial_\mu\big(\sqrt{|\bfh|}(\bfh^{-1})^{\mu\nu}\partial_\mu\uppsi\partial_\uptau\uppsi\big)+\frac{1}{2}\partial_\uptau\big(\sqrt{|\bfh|}V\uppsi^2-\sqrt{|\bfh|}(\bfh^{-1})^{\mu\nu}\partial_\mu\uppsi\partial_\nu\uppsi\big)\\
&\quad+\partial_\uptau(\sqrt{|\bfh|}(\bfh^{-1})^{\mu\nu})\partial_\mu\uppsi\partial_\nu\uppsi.
\end{split}
\end{align}
In view of the assumption on $\bfUpomega_k(\uppsi)$, the first line gives us the desired control of the energy of $\uppsi$. Indeed, we can write $\uppsi=\uppsi^\perp+\sum_k\angles{\uppsi}{\Zbar_i}_{\Sigma_\tau}\Zbar_i$ where $\Zbar_i$  denote truncated eigenfunctions $\chi \varphi_i$ of $\Delta_\barcalC+V$ supported in some region $\{\uprho\leq \uprho_1\}$ with $\uprho_1$ large (specifically, $\uprho_1\geq \Reigenfunctioncutoffscale$), normalized to have $L^2$ norm equal to one, and where $\angles{\uppsi}{Z_i}_{\Sigma_\tau}=\int_{\Sigma_\tau} \uppsi \Zbar_i \ud V$. The first line of \eqref{eq:linenergyestimatetemp1} then bounds the energy of $\uppsi^\perp$ and the energy of $\uppsi$ can be bounded in terms of that of $\uppsi^\perp$ using the assumption on $\bfUpomega_k(\uppsi)$. The second line of \eqref{eq:linenergyestimatetemp1} can be absorbed by a small multiple of the energy in view of the $\uptau$ decay of $\partial_\uptau(\sqrt{|\bfh|}(\bfh^{-1})^{\mu\nu})$. Here note that in view of the form of $\bfh$ from Remark~\ref{rem:nongeomglobal1} the terms in $\partial_\uptau(\sqrt{|\bfh|}(\bfh^{-1})^{\mu\nu})$ where at least one of $\mu,\nu$ is $\uptau$, in particular $\partial_\uptau(\sqrt{|\bfh|}(\bfh^{-1})^{\uptau\uprho})$, come with extra $\uprho$ decay which allows us to bound the corresponding errors in the exterior region by the energy. Finally the contribution of $\tilcalbfP\uppsi$ can again be bounded by a small multiple of the energy in view of the $\uptau$ decay of the coefficients of $\tilcalbfP$. Here the only term in $\tilcalbfP$ that needs special attention is $\ringa (\partial_\uprho+\frac{n-1}{2\uprho})\partial_\uptau\uppsi$ which, after integration by parts, yields
%%%%%%%
%%%%%%%
\begin{align*}
\begin{split}
\frac{1}{2}\big(\partial_\uprho(\ringa \sqrt{|\bfh|})-\frac{n-1}{\rho}\ringa\sqrt{|\bfh|}\big)(\partial_\uptau\uppsi)^2.
\end{split}
\end{align*} 
Since $|\partial_\uprho(\ringa \sqrt{|\bfh|})-\frac{n-1}{\rho}\ringa\sqrt{|\bfh|}|\lesssim \uptau^{-\gamma}\uprho^{-2}$  for large $\uprho$, this contribution can be bounded by the energy as well.
\end{proof}
%%%%%%%%%%%%%
%%%%%%%%%%%%%
We next turn to the proof of Proposition~\ref{prop:LED1}. The proofs of the two estimates in this proposition are different only in which energy estimate from Proposition~\ref{prop:energyestimate1} we use to bound the fluxes that come up in the integration by parts, so we give the proof only for the first estimate. Our starting point is a local energy decay estimate allowing for an $L^2$ error in a bounded region. For this we need to define some cutoff functions and auxiliary potentials. We fix $\chi_1\equiv \chi_1(\tiluprho)$ to be a smooth, non-decreasing, non-negative cutoff function supported in $U_\ext$ that is equal to one on $\overline{K_\ext^c}$ and satisfies $(\sgn \tiluprho) \frac{\ud}{\ud \tiluprho} \chi_{1} \leq 0$. Similarly, $\chi_2\equiv\chi_2(\rho)$ is a smooth, non-negative cutoff supported in $U_\Int$ that is equal to one on $K_\Int$. Let $V_\temp\equiv V_\temp(\rho)$ be a compactly supported, non-negative, smooth potential such that 
%%%%%%%
%%%%%%%
\begin{align*}
\begin{split}
\supp\,V_\temp \subseteq U_\Int,\qquad (\sgn\rho)\frac{\ud}{\ud\rho}V_\temp(\rho)\leq0~\mathrm{in~}U_\Int,\qquad (\sgn\rho)\frac{\ud}{\ud\rho}V_\temp(\rho)\leq -v_0<0~\mathrm{in~} K_\Int,
\end{split}
\end{align*}
and for some large constant $M$ to be fixed later, let
%%%%%%%
%%%%%%%
\begin{align}\label{eq:Vfardef1}
\begin{split}
V_\far:=MV_\temp.
\end{split}
\end{align}
Let $\uppsi_\far$ be the solution to
%%%%%%%
%%%%%%%
\begin{align}\label{eq:phifardef1}
\begin{split}
(\calP-V_\far)\uppsi_\far = f,\qquad (\uppsi_\far,\partial_\uptau\uppsi_\far)\vert_{\Sigma_{t_1}}=(\uppsi,\partial_\uptau\uppsi)\vert_{\Sigma_{t_1}},
\end{split}
\end{align}
and $\uppsi_\near:=\uppsi-\uppsi_\far$. Note that $\uppsi_\near$ satisfies
%%%%%%%
%%%%%%%
\begin{align}\label{eq:calPphinear1}
\begin{split}
\calP\uppsi_\near=-V_\far\uppsi_\far,\qquad (\uppsi_\far,\partial_\uptau\uppsi_\near)\vert_{\Sigma_{t_1}}=(0,0).
\end{split}
\end{align}
%%%%%%%%%%%%
%%%%%%%%%%%%
\begin{lemma}\label{lem:phifar1}
$\uppsi_\far$ and $\uppsi_\near$ as defined above satisfy 
%%%%%%%
%%%%%%%
\begin{align}\label{eq:lemphifarbound1}
\begin{split}
\|\uppsi_\far\|_{LE(\Sigma_{t_1}^{t_2})}\lesssim  \|\uppsi\|_{E(\Sigma_{t_1})}+ \|f\|_{L^1L^2(\Sigma_{t_1}^{t_2})},
\end{split}
\end{align}
and 
%%%%%%%
%%%%%%%
\begin{align}\label{eq:lemphifarbound2}
\begin{split}
\|\uppsi_\near\|_{LE(\Sigma_{t_1}^{t_2})}&\lesssim \|\uppsi\|_{E(\Sigma_{t_1})}+ \|f\|_{L^1L^2(\Sigma_{t_1}^{t_2})}+\|\uppsi_\near\|_{L^2(K_\Ext)}.
\end{split}
\end{align}
\end{lemma}
%%%%%%%%%%%%
%%%%%%%%%%%%
\begin{proof}
The proof consists of two multiplier arguments, one in the exterior and one in the interior. The proofs for the estimates for $\uppsi_\near$ and $\uppsi_\far$ are similar so we carry out the details for $\uppsi_\far$ which is slightly more involved. 
To simplify notation we write $\upphi$ for $\uppsi_\far$ and $U$ for $V_\far-V$ in the remainder of the proof. In addition to the smooth cutoffs $\chi_1$ and $\chi_2$ introduced above, we will write $\chi_A$ to denote an appropriate cutoff with support in a set $A$. Starting with the exterior we use \eqref{eq:callP2} and \eqref{eq:ErrcallP1} to write the equation in the exterior as (to be precise, we have used the conjugation \eqref{eq:varphiphi1} and $\upphi$ corresponds to the conjugated variable, but the estimates are easy to transfer between the conjugated and original variables) 
%%%%%%%
%%%%%%%
\begin{align}\label{eq:psifarLEDtemp0}
\begin{split}
\Box_m\upphi-U\upphi+\Err_\calP(\upphi)=f.
\end{split}
\end{align}
Let $Q$ be the multiplier defined in the $(\tiluptau,\tiluprho,\tiluptheta)$ coordinates, relative to the parameter values at $t_2$, as
%%%%%%%
%%%%%%%
\begin{align}\label{eq:psifarLEDtemp0.5}
\begin{split}
Q=-2\beta_1(\partial_\tiluprho-\partial_\tiluptau)+(\beta_1'+\frac{n-1}{\tiluprho}\beta_1),
\end{split}
\end{align}
where (here $\chi_1$ is as defined before the statement of Lemma~\ref{lem:phifar1})
%%%%%%%
%%%%%%%
\begin{align*}
\begin{split}
\beta_1\equiv \beta_1(\tiluprho) = (\frac{\tiluprho}{\jap{\tiluprho}}-\frac{\delta \tiluprho}{\jap{\tilrho}^{1+\alpha}})\chi_1(\tiluprho),
\end{split}
\end{align*}
for suitable small constants $\alpha$ and $\delta$. We multiply equation \eqref{eq:psifarLEDtemp0} by $Q\upphi |m|^{\frac{1}{2}}$. Note that, except for the cutoff $\chi_1$, this choice of $Q$ is the standard multiplier for the proof of LED on Minkowski space near each asymptotically flat end $\tiluprho \to \pm \infty$. As usual, for concreteness, we focus on the end $\tiluprho \to \infty$. The main contribution comes from $\Box_m-U$. By direct computation, for any vectorfield $a^\mu\partial_\mu$ (here we use $i,j$ to denote tangential partial derivatives with respect to $\uprho$ and $\uptheta$),
%%%%%%%
%%%%%%%
\begin{align*}
(\Box_m-U) \upphi a^\lambda\partial_\lambda\upphi |m|^{\frac{1}{2}}
&=-\frac{1}{2}\partial_\uptau(|m|^{\frac{1}{2}}(m^{-1})^{ij}a^\tau \partial_i\upphi\partial_j\upphi-|m|^{\frac{1}{2}}U a^\tau \phi^2)-\frac{1}{2}\partial_\mu(U a^\mu|m|^{\frac{1}{2}})\upphi^2\\
&\quad+\partial_i((m^{-1})^{i\nu}|m|^{\frac{1}{2}}a^j\partial_\nu\upphi \partial_j\upphi-\frac{1}{2}(m^{-1})^{\mu\nu}|m|^{\frac{1}{2}}a^i\partial_\mu\upphi\partial_\nu\upphi+\frac{1}{2}|m|^{\frac{1}{2}}U a^j \phi^2)\\
&\quad+\frac{1}{2}\partial_\lambda((m^{-1})^{\mu\nu}|m|^{\frac{1}{2}}a^\lambda)\partial_\mu\upphi\partial_\nu\upphi-(m^{-1})^{\mu\nu}|m|^{\frac{1}{2}}(\partial_\mu a^\lambda)\partial_\nu\upphi\partial_\lambda\upphi,
\end{align*}
and for any scalar function $S$,
%%%%%%%
%%%%%%%
\begin{align*}
(\Box_m-U)\upphi S\upphi |m|^{\frac{1}{2}}
&=\partial_\tau(|m|^{\frac{1}{2}}(m^{-1})^{\tau\nu}\partial_\nu\upphi S\upphi-\frac{1}{2}|m|^{\frac{1}{2}}(m^{-1})^{\tau\nu}\partial_\nu S\upphi^2)\\
&\quad+\partial_j(|m|^{\frac{1}{2}}(m^{-1})^{j\nu}\partial_\nu\upphi S\upphi-\frac{1}{2}|m|^{\frac{1}{2}}(m^{-1})^{j\nu}\partial_\nu S\upphi^2)\\
&\quad+\frac{1}{2}|m|^{\frac{1}{2}}(\Box_m S)\upphi^2-|m|^{\frac{1}{2}}S(m^{-1})^{\mu\nu}\partial_\mu\upphi\partial_\nu\upphi-|m|^{\frac{1}{2}}US\upphi^2.
\end{align*}
We apply and add these identities in the $(\uptau,\uprho,\uptheta)$ coordinates with $a^\mu$ and $S$ determined by $Q$ above. It follows with $B^\mu[\upphi]$ determined through these identities,
%%%%%%%
%%%%%%%
\begin{align}\label{eq:psifarLEDtemp1.5}
\begin{split}
fQ\upphi|m|^{\frac{1}{2}}&=\partial_\mu B^\mu[\upphi]+\calP_\pert\upphi Q\upphi |m|^{\frac{1}{2}}-|m|^{\frac{1}{2}}(|m|^{-\frac{1}{2}}\frac{1}{2}\partial_\mu(U a^\mu|m|^{\frac{1}{2}})\upphi^2-US\upphi^2)\\
&\quad+|m|^{\frac{1}{2}}\big(\frac{1}{2}|m|^{-\frac{1}{2}}\partial_\lambda((m^{-1})^{\mu\nu}|m|^{\frac{1}{2}}a^\lambda)\partial_\mu\upphi\partial_\nu\upphi-(m^{-1})^{\mu\nu}(\partial_\mu a^\lambda)\partial_\nu\upphi\partial_\lambda\upphi\big)\\
&\quad +|m|^{\frac{1}{2}}(\frac{1}{2}(\Box_m S)\upphi^2-S(m^{-1})^{\mu\nu}\partial_\mu\upphi\partial_\nu\upphi).
\end{split}
\end{align}
The contribution of $B^\mu[\upphi]$ can be bounded by the energy (see \eqref{eq:minvdecomp1} and \eqref{eq:m02inv1} for the form of $m$).  To calculate the bulk terms, we first write $m=\mbar+\ringm$ where $\mbar$ is defined by freezing the $\uptau$ values of the coefficients at $\uptau=t_2$, and $\ringm:=m-\mbar$. In view of  \eqref{eq:minvdecomp1} and \eqref{eq:m02inv1}, and the $\uptau$ decay of $\ringm$, the contribution of $\ringm$ is bounded by the energy. Here note that the term $|m|^{\frac{1}{2}}(m^{-1})^{\uptau\uprho}$, which could lead to a transversal derivative with no spatial decay on the leaves $\Sigma_\uptau$, is independent of $\uptau$ to leading order in $\uprho$, so its leading order contribution to $\ringm$ vanishes (see \eqref{eq:mdvol1}, \eqref{eq:m02inv1}). For the contribution of $\mbar$, except for the multiplicative $|\mbar|^{\frac{1}{2}}$, the expression of the bulk terms is coordinate invariant, so we can calculate in the $(\tiluptau,\tiluprho,\tiluptheta)$ coordinates. But then, using the asymptotic flatness of the metric in these coordinates, and the fact that $\chi_1'\geq 0$, if $\delta$ is sufficiently small the last two lines of \eqref{eq:psifarLEDtemp1.5} give control of 
%%%%%%%
%%%%%%%
\begin{align}\label{eq:psifarLEDtemp1.75}
\begin{split}
-\uprho^{-1-\alpha}((\partial_\uprho\upphi)^2+(\uprho^{-1}\partial_\uptheta\upphi)^2+(\uprho^{-1}\upphi)^2).
\end{split}
\end{align}
Using similar considerations, the contributions of $\calP_\pert$ and $U$ to \eqref{eq:psifarLEDtemp1.5} can be bounded by the energy and a small multiple of the LE norm of $\upphi$. To get control of the remaining derivative $\partial
_\uptau\upphi$, we again multiply \eqref{eq:psifarLEDtemp0} by $\beta_1'\upphi |m|^{\frac{1}{2}}$, and manipulate as above to get, for the appropriate choice of $\tilB^\mu[\upphi]$,
%%%%%%%
%%%%%%%
\begin{align}\label{eq:psifarLEDtemp2.5}
\begin{split}
f\beta_1'\upphi |m|^{\frac{1}{2}}&=\partial_\mu \tilB^\mu[\upphi]+\calP_\pert\beta_1'\upphi |m|^{\frac{1}{2}}+|m|^{\frac{1}{2}}U\beta_1'\upphi^2\\
&\quad+|m|^{\frac{1}{2}}(\frac{1}{2}(\Box_m \beta_1')\upphi^2-\beta_1'(m^{-1})^{\mu\nu}\partial_\mu\upphi\partial_\nu\upphi).
\end{split}
\end{align}
Using similar arguments as above, and as in the standard Minkowski computation, this gives control of $\uprho^{-1-\alpha}(\partial_\uptau\upphi)^2$ in terms of \eqref{eq:psifarLEDtemp1.75}.
%%%%%%%%
%%%%%%%%
Note also that by a similar argument as in the proof of Proposition~\ref{prop:energyestimate1} we can prove an energy estimate for equation \eqref{eq:phifardef1}. Adding a suitably large multiple of \eqref{eq:psifarLEDtemp1.5} and the energy identity for \eqref{eq:phifardef1} to \eqref{eq:psifarLEDtemp2.5}, for a small constant $\varepsilon$ depending on the support of $\chi_1$, we get (note that the induced volume form and $\sqrt{|m|}$ are comparable in the support of $\chi_1$)
%%%%%%%
%%%%%%%
\begin{align}\label{eq:psifarLEDtemp3}
\begin{split}
\iint_{\Sigma_{t_1}^{t_2}}\chi_1 \tilr^{-1-\alpha}((\partial\upphi)^2+(\tilr^{-1}\upphi)^2)\ud V &\lesssim \|f\|^2_{L^1L^2(\Sigma_{t_1}^{t_2})}+\varepsilon \|\upphi\|^2_{LE(\Sigma_{t_1}^{t_2})}\\
&\quad+\iint_{\Sigma_{t_1}^{t_2}\cap \,\supp \chi_1'}|\upphi|^2\ud V.
\end{split}
\end{align}
For the interior we multiply the equation by two multipliers $Q\phi$ and $P\phi$ of the forms, 
%%%%%%%
%%%%%%%
\begin{align*}
\begin{split}
Q\phi :=q^\mu\partial_\mu\phi +|\tilh|^{-\frac{1}{2}}\partial_\mu (|\tilh|^{\frac{1}{2}}q^\mu \phi)\qquad\mand\qquad P\phi:= p\phi,
\end{split}
\end{align*}
where (recall the relations \eqref{eq:ttilt1}, \eqref{eq:ttilt2}, \eqref{eq:kappadef1})
%%%%%%%
%%%%%%%
\begin{align*}
\begin{split}
q^\mu=\delta_\rho^\mu\beta_2(\rho)-\gamma(t)\delta_{t}^\mu \ell(t)\cdot F_\rho(\rho,\omega) (1-\upkappa(1+\upkappa)^{-1})\beta_2(\rho), \qquad p= p_0 \rho^2\chi_2(\rho).
\end{split}
\end{align*}
Here  $p_0$ is a constant to be fixed later (and with $\chi_2$ as defined before the statement of Lemma~\ref{lem:phifar1}),
%%%%%%%
%%%%%%%
\begin{align*}
\begin{split}
\beta_2(\rho)=\rho \chi_2(\rho).
\end{split}
\end{align*}
It is helpful to keep in mind that $\ell\cdot F_\rho=\frac{\rho}{\jap{\rho}}\ell\cdot\Theta$ vanishes at $\rho=0$. Note that in the $(\tilt,\tilrho,\tilomega)$ coordinates $\tilq^\mu\partial_\mu =\beta_2 \partial_\tilrho$ (recall that in our notation we use $\tilq^\mu$ to denote the components of the vectorfield $q$ in the $(\tilt,\tilrho,\tilomega)$ coordinates). Also, recalling \eqref{eq:LEDcalPint1} we use the following notation $\ringcalP=a^{\mu\nu}\partial^2_{\mu\nu}+b^\mu \partial_\mu + c$ for the perturbation.
For the principal part of $\calP$ we can change variables to $(\tilt,\tilrho,\tilomega)$ to get
%%%%%%%
%%%%%%%
\begin{align}\label{eq:LEDintbulk1}
\begin{split}
(\Box_h -U)\phi Q\phi |\tilh|^{\frac{1}{2}}&= \partial_\mu\Big(2|\tilh|^{\frac{1}{2}}(\tilh^{-1})^{\mu\nu}\tilq^\lambda \partial_\lambda\phi\partial_\nu\phi-|\tilh|^{\frac{1}{2}}(\tilh^{-1})^{\lambda\nu}\tilq^\mu\partial_\lambda\phi\partial_\nu\phi-|\tilh|^{\frac{1}{2}}U\tilq^\lambda\phi^2\Big)\\
&\quad+\partial_\mu\Big((\tilh^{-1})^{\mu\nu}\partial_\lambda(|\tilh|^{\frac{1}{2}}\tilq^\lambda)\phi\partial_\nu\phi-\frac{1}{2}(\tilh^{-1})^{\mu\nu}\partial_\nu(|\tilh|^{-\frac{1}{2}}\partial_\lambda(|\tilh|^{\frac{1}{2}}\tilq^\lambda))|\tilh|^{\frac{1}{2}}\phi^2\Big)\\
&\quad-2\big((\tilh^{-1})^{\lambda\nu}(\partial_\lambda\tilq^\mu)-\frac{1}{2}\tilq^\lambda\partial_\lambda(\tilh^{-1})^{\mu\nu}\big)\partial_\mu\phi\partial_\nu\phi|\tilh|^{\frac{1}{2}}\\
&\quad+\frac{1}{2}\Box_\tilh(|\tilh|^{-\frac{1}{2}}\partial_\lambda(|\tilh|^{\frac{1}{2}}\tilq^\lambda))\phi^2|\tilh|^{\frac{1}{2}}+(\tilq^\lambda\partial_\lambda U)\phi^2|\tilh|^{\frac{1}{2}},
\end{split}
\end{align}
and
%%%%%%%
%%%%%%%
\begin{align}\label{eq:LEDintbulk2}
\begin{split}
(\Box_h-U) \phi P\phi |\tilh|^{\frac{1}{2}}&= \partial_\mu\big(|\tilh|^{\frac{1}{2}}(\tilh^{-1})^{\mu\nu}p\phi\partial_\nu\phi-\frac{1}{2}|\tilh|^{\frac{1}{2}}(\tilh^{-1})^{\mu\nu}\partial_\mu p \phi^2\big)\\
&\quad-p(\tilh)^{-1}\partial_\mu\phi\partial_\nu\phi|\tilh|^{\frac{1}{2}}+\big(\frac{1}{2}\Box_\tilh p-Up\big)\phi^2|\tilh|^{\frac{1}{2}}.
\end{split}
\end{align}
Recalling that, in view of \eqref{eq:ttilt2}, $(1+\upkappa)^{-1}|h|^{\frac{1}{2}}=|\tilh|^{\frac{1}{2}}$, by adding a small multiple $\epsilon_M$ of \eqref{eq:LEDintbulk2} to \eqref{eq:LEDintbulk1} and multiplying by $(1+\upkappa)^{-1}$ we get, for some constant $c_M=o(M)$, 
%%%%%%%
%%%%%%%
\begin{align}\label{eq:LEDintbulk3}
\begin{split}
(U-\Box_h) \phi (Q\phi+\epsilon_MP\phi) |h|^{\frac{1}{2}}&\geq C\chi_1(c_M\rho^2(T\phi)^2+(\partial_\Sigma\phi)^2+M\phi^2)\\
&\quad-O(1)\chi_{U_\Int\backslash K_\Int}((\partial\phi)^2+(\phi)^2)\\
&\quad+O(\epsilon \,t^{-5/4})\chi_{U_\Int}((\partial\phi)^2+(\phi/\rho)^2)\\
&\quad+\partial(O(1)\chi_{U_\Int}(\partial\phi)^2+O(1)\chi_{U_\Int}(\phi/\rho)^2).
\end{split}
\end{align}
In a similar manner, using the $t$ decay of $a,b,c$, we can see that
%%%%%%%
%%%%%%%
\begin{align}\label{eq:LEDintbulk4}
\begin{split}
P \phi (Q\phi+\epsilon_MP\phi) |h|^{\frac{1}{2}}&=O(\epsilon \,t^{-9/4})\chi_{U_\Int}((\partial\phi)^2+(\phi/\rho)^2)\\
&\quad+\partial(O(\epsilon)\chi_{U_\Int}(\partial\phi)^2+O(\epsilon)\chi_{U_\Int}(\phi/\rho)^2).
\end{split}
\end{align}
The desired estimate now follows by integrating \eqref{eq:LEDintbulk3} and \eqref{eq:LEDintbulk4} (with respect to the measure $\ud t \ud\rho \ud\omega$) and combining with the energy identity for \eqref{eq:phifardef1} and a suitably large multiple (independent of $M$) of \eqref{eq:psifarLEDtemp3}. Here note that the bulk error terms in \eqref{eq:LEDintbulk3} and \eqref{eq:LEDintbulk4} are absorbed by the suitably large multiple of \eqref{eq:psifarLEDtemp3}, while the bulk $L^2(\supp \chi_1')$ error terms in the latter are absorbed by \eqref{eq:LEDintbulk3} if $M$ is chosen sufficiently large. The argument for \eqref{eq:lemphifarbound2} is similar, where now we incur some $L^2$ errors in a compact region in view of the absence of $V_\far$ from the left-hand side of \eqref{eq:calPphinear1}. The contribution of the source term in \eqref{eq:calPphinear1} is bounded in $LE^\ast$ using the decay of $V_\far$ and \eqref{eq:lemphifarbound1}.
\end{proof}
%%%%%%%%%%%%
%%%%%%%%%%%%
%%%%%%%%%%%%
%%%%%%%%%%%%
At this point we use the coordinates $(\uptau,\uprho,\uptheta)$ and the decomposition $\calP=\calP_0+\calP_\pert$ (see the opening paragraphs of this section). For a globally defined function $u$, the frequency projections $P_{\leq N_0}$ and $P_N$ are defined by (suppressing the spatial variables $(\uprho,\uptheta)$)
%%%%%%%
%%%%%%%
\begin{align*}
\begin{split}
P_{\leq N_0}u(\uptau)=\int_{-\infty}^{\infty}2^{N_0}\chi(2^{N_0}\uptau')u(\uptau-\uptau')\ud \uptau',\qquad P_{N}u(\uptau)=\int_{-\infty}^{\infty}2^{N_0}\tilchi(2^{N_0}\uptau')u(\uptau-\uptau')\ud \uptau',
\end{split}
\end{align*}
where as usual $\hat\chi(\hat\uptau)=\int_{-\infty}^{\infty}\chi(\uptau)e^{-i\hat\uptau\uptau}\ud \uptau$ and $\hat{\tilchi}(\hat\uptau)=\int_{-\infty}^{\infty}\chi(\uptau)e^{-i\hat\uptau\uptau}\ud \uptau$ are supported in $\{|\hat\uptau|\lesssim 1\}$ and $\{|\hat\uptau|\simeq 1\}$ respectively. 

In order to apply frequency projections, we need to extend $\uppsi$, $\uppsi_\far$, and $\uppsi_\near$. For this, we view these as functions of the variables $(\uptau,\uprho,\upomega)$  and extend them outside of their current domain of definition by requiring that they satisfy
%%%%%%%
%%%%%%%
\begin{align*}
\begin{split}
\calP\uppsi-V_\far \uppsi=0,\qquad \calP\uppsi_\far- V_\far\uppsi_\far=0,
\end{split}
\end{align*}
in $( \Sigma_{t_1}^{t_2})^c$. Here $\calP$ is extended outside the original domain using the decomposition \eqref{eq:calPP0Ppertdecomp1}, by extending the coefficients of $\calP_0$ independently of $\uptau$ and those of $\calP_\pert$ smoothy such that the estimates \eqref{eq:calPP0Ppertdecomp1} are still satisfied.
It follows that $\uppsi_\near:=\uppsi-\uppsi_\far$ satisfies 
%%%%%%%
%%%%%%%
\begin{align*}
\begin{split}
\calP\uppsi_\near -V_\far \uppsi_\near=0,
\end{split}
\end{align*}
outside its original domain of definition. By a similar argument as in the proof of Lemma~\ref{lem:phifar1}, we can then replace $\|\uppsi_\far\|_{LE(\Sigma_{t_1}^{t_2})}$ and $\|\uppsi_{\near}\|_{LE(\Sigma_{t_1}^{t_2})}$ in the estimates \eqref{eq:lemphifarbound1} and \eqref{eq:lemphifarbound2} by $\|\uppsi_\far\|_{LE}$ and $\|\uppsi_{\near}\|_{LE}$, respectively, where $$LE\equiv LE(\cup_{\uptau}\Sigma_\uptau).$$
%%%%%%%%%%%%
%%%%%%%%%%%%
In view of Lemma~\ref{lem:phifar1} our task has reduced to estimating $\|\uppsi_\near\|_{L^2(K_\Ext)}$. The high frequency part of this error can already be absorbed by the $LE$ norm as shown in the next lemma.
%%%%%%%%%%%%
%%%%%%%%%%%%
\begin{lemma}\label{lem:LEDhighfreq1}
Given $\delta>0$, if $N_0$ is sufficiently large then $P_{>N_0}\uppsi_\near:=\uppsi_\near - P_{\leq N_0}\uppsi_\near$ satisfies
%%%%%%%
%%%%%%%
\begin{align*}
\begin{split}
\|P_{>N_0}\uppsi_\near\|_{L^2(K_{\Ext})}\leq \delta \|\uppsi_\near\|_{LE}.
\end{split}
\end{align*}
\end{lemma}
%%%%%%%%%%%%
%%%%%%%%%%%%
\begin{proof}
This would be immediate from the definition of $P_{\geq N_0}$ and $\|\cdot\|_{LE}$ if we had $\|P_{>N_0}\rho\uppsi_\near\|_{L^2(K_\Ext)}$ instead of $\|P_{>N_0}\uppsi_\near\|_{L^2(K_\Ext)}$ (note that in $K_\Ext$ the coordinates $(t,\rho,\omega)$ and $(\uptau,\uprho,\uptheta)$ agree). To insert the extra factor of $\rho$ we argue as follows. Let $u=P_{> N_0}\uppsi$. Then, with $\chi\equiv \chi(\rho)$ an appropriate cutoff, we need to estimate (recall that $|\tilh|\simeq 1$ in $K_\Ext$, in particular with no vanishing at $\rho=0$) 
%%%%%%%
%%%%%%%
\begin{align*}
\begin{split}
&\int_{t_1}^{t_2}\int_{\bbS^{n-1}}\int_{-\infty}^\infty u^2\chi \ud \rho\,\ud\omega\ud t= \int_{t_1}^{t_2}\int_{\bbS^{n-1}}\int_{-\infty}^\infty (\partial_\rho\rho)u^2\chi \ud \rho\,\ud\omega\ud t\\
&=-\int_{t_1}^{t_2}\int_{\bbS^{n-1}}\int_{-\infty}^{\infty} \rho u^2\partial_\rho\chi\ud \rho\,\ud\omega \ud t-2\int_{t_1}^{t_2}\int_{\bbS^{n-1}}\int_{-\infty}^{\infty}\rho u \partial_\rho u \chi \ud \rho\,\ud\omega \ud t.
\end{split}
\end{align*}
The first integral is supported away from $\{\rho=0\}$, so we can insert a factor of $\rho$ making it an acceptable $L^2$ error. For the second integral we use
%%%%%%%
%%%%%%%
\begin{align*}
\begin{split}
|\rho u \partial_\rho u| \leq C_\epsilon \rho^2 u^2 + \epsilon (\partial_\rho u)^2.
\end{split}
\end{align*}
Since the coefficient of $\partial_\rho u$ in the LE norm is non-degenerate we can absorb the last term on the right (note that $P_{> N_0}$ and $\partial_\rho$ commute and that $P_{>N_0}$ is bounded in $LE$). The first term on the right is exactly the term we had hoped for.
\end{proof}
%%%%%%%%%%%%
%%%%%%%%%%%%
It follows that 
%%%%%%%
%%%%%%%
\begin{align}\label{eq:phinearLEtemp1}
\begin{split}
\|\uppsi_\near\|_{LE}\lesssim  \|\uppsi\|_{E(\Sigma_{t_1})}+ \|f\|_{L^1L^2(\Sigma_{t_1}^{t_2})}+\|P_{\leq N_0}\uppsi_\near\|_{L^2(K_\Ext)}.
\end{split}
\end{align}
Now to estimate $P_{\leq N_0}\uppsi_\near$ we again apply the near-far decomposition, this time with respect to $\calP_0$. First note that
%%%%%%%
%%%%%%%
\begin{align*}
\begin{split}
\|P_{\leq N_0}\uppsi_\near\|_{L^2(K_\Ext)}\lesssim \|P_{\leq N_0}\uppsi_\near\|_{LE}.
\end{split}
\end{align*}
Let $\uppsi_{\near,\far}$ be defined by (note that the operator on the left-hand side is applied to $\uppsi_{\near,\far}$ while on right-hand side $\uppsi_\near$ and $\uppsi_\far$ appear and not $\uppsi_{\near,\far}$)
%%%%%%%
%%%%%%%
\begin{equation} \label{eq:psi-nearfar}
\begin{split}
\begin{cases}(\calP_0-V_\far)\uppsi_{\near,\far} = -\calP_\pert\uppsi_\near-V_\far\uppsi_\far\quad&\mathrm{in~}\Sigma_{t_1}^{t_2}\\
(\calP_0-V_\far)\uppsi_{\near,\far} =-\calP_\pert\uppsi_\near\quad&\mathrm{in~}(\Sigma_{t_1}^{t_2})^c
\end{cases},\qquad \uppsi_{\near,\far}\vert_{\Sigma_{t_1}}=0,
\end{split}
\end{equation}
so that $\uppsi_{\near,\near}:=\uppsi_\near-\uppsi_{\near,\far}$ satisfies
%%%%%%%
%%%%%%%
\begin{equation} \label{eq:psi-nearnear}
\begin{split}
\begin{cases}\calP_0\uppsi_{\near,\near}=-V_\far\uppsi_{\near,\far}\quad&\mathrm{~in~} \Sigma_{t_1}^{t_2}\\
(\calP_0-V_\far)\uppsi_{\near,\near}=0\quad&\mathrm{~in~} (\Sigma_{t_1}^{t_2})^c
\end{cases},
\qquad \uppsi_{\near,\near}\vert_{\Sigma_{t_1}}=0.
\end{split}
\end{equation}
Since $P_{\leq N_0}$ commutes with $\calP_0$ and $V_\far$, we also have
%%%%%%%
%%%%%%%
\begin{align*}
\begin{split}
P_{\leq N_0}\uppsi_\near = P_{\leq N_0}\uppsi_{\near,\near}+P_{\leq N_0}\uppsi_{\near,\far}
\end{split}
\end{align*}
and
%%%%%%%
%%%%%%%
\begin{align*}
\begin{split}
&(\calP_0-V_\far)P_{\leq N_0}\uppsi_{\near,\far}=P_{\leq N_0}f_{\near,\far},\\
&\calP_0P_{\leq N_0}\uppsi_{\near,\near} =P_{\leq N_0}f_{\near,\near},
\end{split}
\end{align*}
where
%%%%%%%
%%%%%%%
\begin{align*}
\begin{split}
f_{\near,\far}:=\begin{cases}-\calP_\pert\uppsi_\near-V_\far\uppsi_\far\quad&\mathrm{in~}\Sigma_{t_1}^{t_2}\\ -\calP_\pert\uppsi_\near\quad&\mathrm{in~}(\Sigma_{t_1}^{t_2})^c\end{cases},\quad f_{\near,\near}:=\begin{cases}-V_\far\uppsi_{\near,\far}\quad&\mathrm{in~}\Sigma_{t_1}^{t_2}\\ V_\far\uppsi_{\near,\near}\quad&\mathrm{in~}(\Sigma_{t_1}^{t_2})^c\end{cases}.
\end{split}
\end{align*}
By a slight abuse of notation we will sometimes write
%%%%%%%
%%%%%%%
\begin{align*}
\begin{split}
-P_{\leq N_0}(\calP_\pert \uppsi_\near)-P_{\leq N_0}(V_\far \uppsi_\far)
\end{split}
\end{align*}
for $P_{\leq N_0}f_{\near,\far}$, and write the equation for $P_{\leq N_0}\uppsi_{\near,\far}$ simply as
%%%%%%%
%%%%%%%
\begin{align*}
\begin{split}
(\calP_0-V_\far)P_{\leq N_0}\uppsi_{\near,\far}=-P_{\leq N_0}(\calP_\pert \uppsi_\near)-P_{\leq N_0}(V_\far \uppsi_\far).
\end{split}
\end{align*}
The proof of the following lemma will occupy much of the remainder of this section.
%%%%%%%%%%%
%%%%%%%%%%%
\begin{lemma}\label{lem:phinearfar1}
$\uppsi_{\near,\far}$ satisfies
%%%%%%%
%%%%%%%
\begin{align*}
\begin{split}
\|P_{\leq N_0}\uppsi_{\near,\far}\|_{L^2(K_\Ext)}\lesssim \|f\|_{L^1L^2(\Sigma_{t_1}^{t_2})}+\epsilon \sup_{t_1\leq t \leq t_2}\|\uppsi\|_{E(\Sigma_{t})}+\epsilon\|\uppsi_\near\|_{LE(\Sigma_{t_1}^{t_2})}.
\end{split}
\end{align*}
\end{lemma}
%%%%%%%%%%%
%%%%%%%%%%%
We postpone the proof of this lemma and proceed to prove Proposition~\ref{prop:LED1} using its statement. 
%%%%%%%%%%%
%%%%%%%%%%%
\begin{proof}[Proof of Proposition~\ref{prop:LED1}]
Throughout the proof, we use an underline to denote the parameters, or other functions depending on the parameters, with values fixed at $\uptau=t_2$. So for instance we write $\ellbar=\ell(t_2)$ and $\hbar$ for $h$ with $\ell$ replaced by $\underline{\ell}$. Let $\Reigenfunctioncutoffscale>0$ be a large constant (see for instance Section~\ref{subsec:modulationeqs}) so that $$K_\Ext\subseteq \calR_{t_1}^{t_2}:=\cup_{\uptau\in [t_1,t_2]}\Sigma_\uptau\cap\{\uprho\leq \Reigenfunctioncutoffscale\}.$$ In view of \eqref{eq:phinearLEtemp1} and Lemmas~\ref{lem:phifar1} and~\ref{lem:phinearfar1}, it suffices for us to prove the following estimate
%%%%%%%
%%%%%%%
\begin{align}\label{eq:phinearneartemp1}
\begin{split}
\|P_{\leq N_0}\uppsi_{\near,\near}\|_{L^2(K_\Ext)}&\lesssim \sum_k\|\bfUpomega_k(\uppsi)\|_{L^2_\uptau([t_1,t_2])}+ \epsilon \|\uppsi\|_{LE(\calR_{t_1}^{t_2})}\\
&\quad+\|\uppsi_\far\|_{LE}+\|\uppsi_{\near,\far}\|_{LE}+\|\uppsi\|_{LE((\Sigma_{t_1}^{t_2})^c)}.
\end{split}
\end{align}
To simplify notation let $u=P_{\leq N_0}\uppsi_{\near,\near}$ and $g=P_{\leq N_0}f_{\near,\near}$, so that the equation
%%%%%%%
%%%%%%%
\begin{align*}
\begin{split}
\calP_0 u = g
\end{split}
\end{align*}
is satisfied globally. 
We also recall that in the coordinates $(\tiluptau,\tiluprho,\tiluptheta)$ the operator $\calP_0$ takes the form
%%%%%%%
%%%%%%%
\begin{align*}
\begin{split}
\calP_0 = -\partial_\tiluptau^2+\tilDelta +V(\tiluprho),
\end{split}
\end{align*}
where $\tilDelta$ denotes the Laplacian on the Riemannian Catenoid in polar coordinates:
%%%%%%%
%%%%%%%
\begin{align*}
\begin{split}
\tilDelta=\frac{1}{\jap{\tiluprho}^{n-1}|F_{\tiluprho}|}\partial_\tiluprho(\jap{\tiluprho}^{n-1}|F_\tiluprho|^{-1}\partial_\rho)+\frac{1}{\jap{\tiluprho}^2}\ringsDelta.
\end{split}
\end{align*} 
By \eqref{eq:coordinatetransformationrelation1}, in the region $\{\uprho\leq \Reigenfunctioncutoffscale\}$ the two coordinates are related by
%%%%%%%
%%%%%%%
\begin{align*}
\begin{split}
\tiluptau=\gammabar^{-1}\uptau-\ellbar\cdot F(\uprho,\uptheta),\quad \tiluprho=\uprho,\quad \tiluptheta=\uptheta.
\end{split}
\end{align*}
We will use $\tily$ for the spatial coordinates $(\tiluprho,\tiluptheta)$ and use $\angles{\cdot}{\cdot}_\tily$ for the $L^2$ pairing with respect to\footnote{Following our convention in this section, by a slight abuse of notation, we write $\sqrt{|\tilhbar|}$ rather than $\sqrt{|\hbar|}$ to emphasize that we are working in the $(\tiluptau,\tiluprho,\tiluptheta)$ coordinates.} $\sqrt{|\hbar|}\ud \tily$  on the $\{\tiluptau=\mathrm{constant}\}$ hypersurfaces. On these hypersurfaces we define the spectral projection $\bbP_c$ by
%%%%%%%
%%%%%%%
\begin{align*}
\begin{split}
u= \bbP_cu+\sum_{j=1}^{n} \angles{u}{\uppsi_j}_\tily \uppsi_j+\angles{u}{\uppsi_\mu}_\tily \uppsi_\mu,
\end{split}
\end{align*}
where $\uppsi_j$, $j=1,\dots,n$, denote the eigenfunctions of $\tilDelta+V$ with eigenvalue zero, and $\uppsi_{\mu}$ the eigenfunction with eigenvalue $-\mu^2<0$. In what follows, unless otherwise specified, when summing over the eigenfunctions $\uppsi_j$ we always let $j$ vary over $\{\mu,1,\dots,n\}$ without distinguishing between the zero and $-\mu^2$ eigenvalues. As in the figure below,
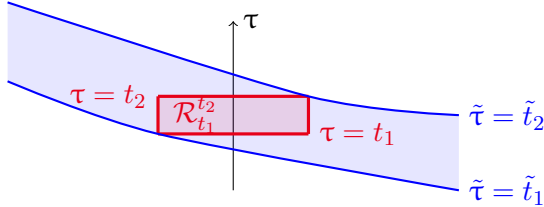
\begin{center}
\begin{tikzpicture}[scale=1,transform shape]
  \draw[->] (0,-0.25) -- (0,2) node[right] {$\uptau$};
  \draw[name path= C, red, very thick,decorate] (-1,0.5) -- (1,0.5) node[right] {$\uptau=t_1$};
  \draw[name path = D, red, very thick,-,decorate]  (-1,1) node[left] {$\uptau=t_2$}-- (1,1) ;
  \draw[red,very thick] (1,1) -- (1,0.5);
  \draw[red, very thick] (-1,1) -- (-1,0.5) (-0.5,0.76) node{$\calR_{t_1}^{t_2}$};
    \tikzfillbetween[of=C and D]{red, opacity=0.1};
    \coordinate  (A) at (-3,2.25);
\coordinate  (B) at (1,1);
\coordinate  (C) at (3,0.75);
\draw[name path=O, thick,blue] plot [smooth] coordinates { (A) (B) (C) };
    \coordinate  (D) at (-3,1.2);
\coordinate  (E) at (-1,0.5);
\coordinate  (F) at (3,-0.25);
\draw[name path = U, thick,blue] plot [smooth] coordinates { (D) (E) (F) };
\node[right] at (C) {$\Blue{\tiluptau=\tilt_2}$};
\node[right] at (F) {$\Blue{\tiluptau=\tilt_1}$};
\tikzfillbetween[of=O and U]{blue, opacity=0.1};
\end{tikzpicture}
\end{center}
let $\tilcalR_{\tilt_1}^{\tilt_2}=\{\tilt_1\leq \tiluptau\leq \tilt_2\}$ be the smallest infinite rectangle containing $\calR_{t_1}^{t_2}$, and observe that (note that the implicit constant is independent of $\Reigenfunctioncutoffscale$ and rather depends on the size of $K_\Ext$ which we can choose to be much smaller than $\Reigenfunctioncutoffscale$)
%%%%%%%
%%%%%%%
\begin{align*}
\begin{split}
\|u\|_{L^2(K_\Ext)}\lesssim\|u\|_{LE(\tilcalR_{\tilt_1}^{\tilt_2})}.
\end{split}
\end{align*}
Let $a_j:=\angles{u}{\uppsi_j}_\tily$, $a_\mu:=\angles{u}{\uppsi_\mu}_\tily$ and $a_j':=\frac{\ud}{\ud \tiluptau}\angles{u}{\uppsi_j}_\tily$, $a_\mu':=\frac{\ud}{\ud \tiluptau}\angles{u}{\uppsi_\mu}_\tily$, and denote by $\tilI$ the time interval $\{\tilt_1\leq \tiluptau \leq \tilt_2\}$. We now apply the LED estimate Proposition~\ref{prop:LEDproduct}, using the second and fourth estimates in the statement there. Note that since $u=P_{\leq N_0}\uppsi_{\near,\near}$, we can drop the time derivative from the last term on the right-hand side of the second estimate in  Proposition~\ref{prop:LEDproduct} and absorb the corresponding error by the left-hand side of the fourth estimate. Using this argument (recall that $j$ varies over $\{\mu,1,\dots,n\}$),
%%%%%%%
%%%%%%%
\begin{align}\label{eq:orthLEtemp1}
\begin{split}
\|u\|_{LE(\tilcalR_{\tilt_1}^{\tilt_2})}&\lesssim \|\bbP_cu\|_{LE(\tilcalR_{\tilt_1}^{\tilt_2})}+\sum_j(\|a_j\|_{L^2_\tilt (\tilI)}+\|a_j'\|_{L^2_\tilt (\tilI)})\\
&\lesssim \|\bbP_cg\|_{LE^\ast(\tilcalR_{\tilt_1}^{\tilt_2})}+\sum_j(\|a_j\|_{L^2_\tilt (\tilI)}+\|a_j'\|_{L^2_\tilt (\tilI)})\\
&\lesssim  \|g\|_{L^2(\tilcalR_{\tilt_1}^{\tilt_2})}+\sum_j(\|a_j\|_{L^2_\tilt (\tilI)}+\|a_j'\|_{L^2_\tilt (\tilI)}).
\end{split}
\end{align}
Here, to pass to the last line, we have used that $\jap{\tily}^{\frac{1+\alpha}{2}}\uppsi_j\in L^2_\tily$ (which holds for $n\geq 4$) to bound
%%%%%%%
%%%%%%%
\begin{align*}
\begin{split}
\|\bbP_cg\|_{LE^\ast(\tilcalR_{\tilt_1}^{\tilt_2})}\leq \|g\|_{LE^\ast(\tilcalR_{\tilt_1}^{\tilt_2})}+\sum_j\|\angles{g}{\uppsi_j}_\tily\|_{L^2_\tilt[\tilt_1,\tilt_2]}\|\jap{y}^{\frac{1+\alpha}{2}}\uppsi_j\|_{L^2_\tily}\lesssim  \|g\|_{L^2(\tilcalR_{\tilt_1}^{\tilt_2})}.
\end{split}
\end{align*}
To treat the last term on the right-hand side of \eqref{eq:orthLEtemp1} we introduce some more notation. For $k=\mu,1,\dots,n$, let $\Zbar_k=\chi \uppsi_k$ where $\chi\equiv \chi(\tiluprho)$ is supported in $\{\tiluprho<\Reigenfunctioncutoffscale/2\}$. Then let
%%%%%%%
%%%%%%%
\begin{align}\label{eq:tilbfOmegabardef1}
\begin{split}
&\tilbfOmegabar_k(v):=-\angles{\partial_\tiluptau v}{\Zbar_k}_\tily,\qquad\tilbfOmegabar_{n+k}(v):=\angles{v}{\Zbar_k}_\tily,\quad k=1,\dots,n\\
&\tilbfOmegabar^{+}_{\mu}(v):=\angles{v}{\mu \Zbar_\mu}_\tily-\angles{\partial_\tiluptau v}{\Zbar_\mu}_\tily,\qquad\tilbfOmegabar^{-}_{\mu}(v):=-\angles{v}{\mu \Zbar_\mu}_\tily-\angles{\partial_\tiluptau v}{\Zbar_\mu}_\tily.
\end{split}
\end{align}
For any $c$ we define $\calT_c$ to be the intersection of $\{\uprho\leq \Reigenfunctioncutoffscale\}$ with the region bounded between $\{\tiluptau=c\}$ and $\{\uptau=\gammabar c\}$, and let $\calT_{c,1}:=\calT_c\cap\{\uptau\leq \gammabar c\}$ and $\calT_{c,2}:=\calT_c\cap\{\uptau\geq \gammabar c\}$. See the figure below. 
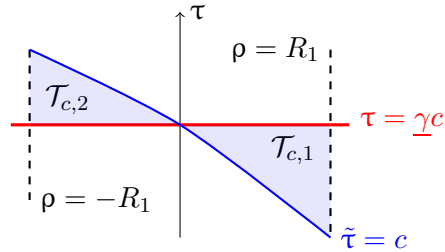
\begin{center}\begin{tikzpicture}[scale=1,transform shape]
  \draw[->] (0,-1.5) -- (0,1.5) node[right] {$\uptau$};
  \draw[name path = O, anchor=center,red, very thick,decorate] (-2,0)--(-1.5,0) node[above,black]{$\calT_{c,2}$}-- (1.5,0) node[below,black]{$\calT_{c,1}$} -- (2,0);
  \draw[red, very thick,decorate] (-2.25,0)--(-2,0) ;
  \draw[red, very thick,decorate] (2,0) -- (2.25,0)node[right] {$\uptau=\gammabar c$};
  \draw[dashed,thick] (2,1)node[left]{$\uprho=\Reigenfunctioncutoffscale$}--(2,-1.5);
  \draw[dashed,thick] (-2,1)--(-2,-1)node[right]{$\uprho=-\Reigenfunctioncutoffscale$};
\coordinate (A) at (-2,1);
\coordinate (B) at (0,0);
\coordinate (C) at (2,-1.5);
\draw [name path = U, thick,blue] plot [smooth] coordinates { (A) (B) (C) };
\node[right] at (C) {$\Blue{\tiluptau=c}$};
\tikzfillbetween[of=O and U]{blue, opacity=0.1};
\end{tikzpicture}
\end{center}
Finally, with $\nbar$ denoting the normal (with respect to $\hbar$) to $\{\uptau=c\}$, we let
%%%%%%%
%%%%%%%
\begin{align*}
\begin{split}
&\bfOmegabar_k(v)(c)=-\int_{\{\uptau= c\}}\Zbar_k\nbar^\alpha\partial_\alpha v \sqrt{|\hbar|}\ud y,\quad \bfOmegabar_{n+k}(v(c))=\int_{\{\uptau= c\}}v\Zbar_k\nbar^\alpha\partial_\alpha \tiluptau \sqrt{|\hbar|}\ud y,\quad k=1,\dots,n,\\
&\bfOmegabar_{\mu}^{\pm}(v)(c)=\int_{\{\uptau= c\}}(\pm\mu v \Zbar_\mu \partial_\alpha \tiluptau-\Zbar_\mu \partial_\alpha v)\nbar^\alpha \sqrt{|\hbar|}\ud y.
\end{split}
\end{align*}
Returning to the last term on the right-hand side of \eqref{eq:orthLEtemp1} note that
%%%%%%%
%%%%%%%
\begin{align*}
\begin{split}
&\tilbfOmegabar_k(u)=\tilbfOmegabar_k(\bbP_cu)-\sum_j a_j'\angles{\uppsi_j}{\Zbar_k}_\tily,\quad \tilbfOmegabar_{n+k}(u)=\tilbfOmegabar_{n+k}(\bbP_cu)+\sum_j a_j\angles{\uppsi_j}{\Zbar_k}_\tily,\quad k=1,\dots,n,\\
&\tilbfOmegabar_{\mu}^{\pm}(u)=\tilbfOmegabar_{\mu}^{\pm}(\bbP_cu)+\sum_j\big(\pm\mu a_j\angles{\uppsi_j}{ \Zbar_{\mu}}_\tily-a_j'\angles{\uppsi_j}{\Zbar_{\mu}}_\tily\big).
\end{split}
\end{align*}
Viewing this as a linear system for $a_j,a_j'$, $j\in \{\mu,1,\dots,n\}$, with invertible coefficient matrix, and since $\Zbar_k$, $k\in \{\mu,1,\dots,n\}$, are compactly supported, we get (below, each sum in $k$ is over~$\pm\mu,1,\dots,2n$)
%%%%%%%
%%%%%%%
\begin{align*}
\begin{split}
\sum_j(\|a_j\|_{L^2_\tiluptau (\tilI)}+\|a_j'\|_{L^2_\tiluptau (\tilI)})&\lesssim \sum_k(\|\tilbfOmegabar_k(u)\|_{L^2_\tiluptau(\tilI)}+\|\tilbfOmegabar_{k}(\bbP_cu)\|_{L^2_\tilt(\tilI)})\\
&\lesssim \|\bbP_cu\|_{LE(\tilcalR_{\tilt_1}^{\tilt_2})}+\sum_k\|\tilbfOmegabar_k(u)\|_{L^2_\tiluptau(\tilI)}\\
&\lesssim \|g\|_{L^2(\tilcalR_{\tilt_1}^{\tilt_2})}+\sum_k\|\tilbfOmegabar_k(u)\|_{L^2_\tiluptau(\tilI)},
\end{split}
\end{align*}
where to pass to the last line we have argued as in \eqref{eq:orthLEtemp1}. 
It remains to estimate $\|\tilbfOmegabar_k(u)\|_{L^2_\tiluptau(\tilI)}$. 
Note that by the divergence theorem (note that if $\tilnbar$ denotes the normal to $\{\tiluptau=\mathrm{constant}\}$, then $\partial_{\tiluptau}v$ in \eqref{eq:tilbfOmegabardef1} can be written as $\tilnbar^\alpha\partial_\alpha v$), for $k=1,\dots,n$,
%%%%%%%
%%%%%%%
\begin{align}\label{eq:orthodivtemp1}
\begin{split}
\tilbfOmegabar_k(u)(\tiluptau)=\bfOmegabar_k(u)(\gammabar\tiluptau)+\iint_{\calT_{\tiluptau}}(\Zbar_k\calP_0u-V\Zbar_k u+(\tilhbar^{-1})^{\alpha\beta}\partial_\alpha u \partial_\beta \Zbar_k)\sqrt{|\tilhbar|}\ud \tily \ud \tiluptau',
\end{split}
\end{align}
so
%%%%%%%
%%%%%%%
\begin{align}\label{eq:tilOmegabarL2temp1}
\begin{split}
\|\tilbfOmegabar_k(u)\|_{L^2_\tiluptau(\tilI)}&\leq \|\bfOmegabar_k(u)(\gammabar\tiluptau)\|_{L^2_\tiluptau(\tilI)}\\
&\quad+\Big\|\iint_{\calT_{\tiluptau}}(\Zbar_kg-V\Zbar_k u+(\tilhbar^{-1})^{\alpha\beta}\partial_\alpha u \partial_\beta \Zbar_k)\sqrt{|\tilhbar|}\ud \tily \ud \tiluptau'\Big\|_{L^2_\tiluptau(\tilI)}.
\end{split}
\end{align}
For the first term, after a change of variables, and recalling that $u=P_{\leq N_0}\phi_{\near,\near}= P_{\leq N_0}(\uppsi-\uppsi_\far-\uppsi_{\near,\far})$,
%%%%%%%
%%%%%%%
\begin{align*}
\begin{split}
\|\bfOmegabar_k(u)(\gammabar\tiluptau)\|_{L^2_\tiluptau(\tilI)}&\lesssim \|\bfOmegabar_k(u)(\uptau)\|_{L^2_\uptau([\gammabar\tilt_1,\gammabar\tilt_2])}\\
&\lesssim \|\bfOmegabar_k(P_{\leq N_0}\uppsi)\|_{L^2_\uptau([\gammabar\tilt_1,\gammabar\tilt_2])}+\|\bfOmegabar_k(P_{\leq N_0}(\uppsi_{\far}+\uppsi_{\near,\far}))\|_{L^2_\uptau([\gammabar\tilt_1,\gammabar\tilt_2])}\\
&\lesssim \|\bfOmegabar_k(P_{\leq N_0}\uppsi)\|_{L^2_t([\gammabar\tilt_1,\gammabar\tilt_2])}+\|\uppsi_\far\|_{LE}+\|\uppsi_{\near,\far}\|_{LE}.
\end{split}
\end{align*}
Here to pass to the last line we have used Lemmas~\ref{lem:phifar1} and~\ref{lem:phinearfar1}. To estimate $\|\bfOmegabar_k(P_{\leq N_0}\uppsi)\|_{L^2_t([\gammabar\tilt_1,\gammabar\tilt_2])}$ note that since $\Zbar_k$ and $\hbar$ are independent of $\uptau$ (recall that $(\uprho,\uptheta)=(\tiluprho,\tiluptheta)$ in the support of $\Zbar_k$), 
%%%%%%%
%%%%%%%
\begin{align*}
\begin{split}
\|\bfOmegabar_k(P_{\leq N_0}\uppsi)\|_{L^2_\uptau([\gammabar\tilt_1,\gammabar\tilt_2])}&=\|P_{\leq N_0}\bfOmegabar_k(\uppsi)\|_{L^2_\uptau([\gammabar\tilt_1,\gammabar\tilt_2])}\lesssim \|\bfOmegabar_k(\uppsi)\|_{L^2_\uptau}\\
&\leq \|\bfOmegabar_k(\uppsi)-\bfUpomega_k(\uppsi)\|_{L^2_\uptau([t_1,t_2])}+\|\bfUpomega_k(\uppsi)\|_{L^2_\uptau([t_1,t_2])}+\|\uppsi\|_{LE((\Sigma_{t_1}^{t_2})^c)}.
\end{split}
\end{align*}
Since
%%%%%%%
%%%%%%%
\begin{align*}
\begin{split}
\|\bfOmegabar_k(\uppsi)-\bfUpomega_k(\uppsi)\|_{L^2_\uptau([t_1,t_2])}\lesssim \epsilon \|\uppsi\|_{LE(\calR_{t_1}^{t_2})},
\end{split}
\end{align*}
combining the last few estimates we get,
%%%%%%%
%%%%%%%
\begin{align}\label{eq:tilOmegabarL2temp2}
\begin{split}
\|\bfOmegabar_k(u)(\gammabar\tilt)\|_{L^2_\tiluptau(\tilI)}&\lesssim \|\bfUpomega_k(\uppsi)\|_{L^2_\uptau([t_1,t_2])}+ \epsilon \|\uppsi\|_{LE(\calR_{t_1}^{t_2})}\\
&\quad+\|\uppsi_\far\|_{LE}+\|\uppsi_{\near,\far}\|_{LE}+\|\uppsi\|_{LE((\Sigma_{t_1}^{t_2})^c)}.
\end{split}
\end{align}
To treat the second term on the right in \eqref{eq:tilOmegabarL2temp1} we write $\calT_\tiluptau=\calT_{\tiluptau,1}\cup \calT_{\tiluptau,2}$ and treat the two regions separately. The estimates in these regions are similar so we carry out the details only for $\calT_{\tiluptau,1}$. For each $c$ define $\tilupsigma_\max(c)$ minimally and $\tilupsigma_\min(c)$ maximally such that $\tiluptau\in[\tilupsigma_\min(c),\tilupsigma_\max(c)]$ in $\calT_{c,1}$ and let
%%%%%%%
%%%%%%%
\begin{align*}
\begin{split}
\tilupsigma_\max:=\sup_{\tiluptau\in[\tilt_1,\tilt_2]}\tilupsigma_\max(\tiluptau)\quad\mand\quad\tilupsigma_{\min}:=\inf_{\tiluptau\in[\tilt_1,\tilt_2]}\tilupsigma_\min(\tiluptau).
\end{split}
\end{align*}
For each $\tilupsigma$ let
%%%%%%%
%%%%%%%
\begin{align*}
\begin{split}
w(\sigma):=\int_{\{\tiluptau=\tilupsigma\}}\big|\Zbar_kg-V\Zbar_k u+(\tilhbar^{-1})^{\alpha\beta}\partial_\alpha u \partial_\beta \Zbar_k\big|\sqrt{|\tilhbar|}\ud \tily.
\end{split}
\end{align*}
We can then bound the contribution of $\calT_{\tiluptau,1}$ to the last term on the right in \eqref{eq:tilOmegabarL2temp1} as
%%%%%%%
%%%%%%%
\begin{align*}
\begin{split}
\Big(\int_{\tilt_3}^{\tilt_4}\Big(\int_{\tilupsigma_\min(\tiluptau)}^{\tilupsigma_\max(\tiluptau)}w(\tilupsigma)\ud \tilupsigma\Big)^2\ud \tiluptau\Big)^{\frac{1}{2}}=\Big(\int_{\tilt_1}^{\tilt_2}\Big(\int_{\tilupsigma_\min}^{\tilupsigma_{\max}}\chi_{\{\tilupsigma_{\min}(\tiluptau)\leq \tilupsigma\leq\tilupsigma_\max(\tiluptau)\}}w(\tilupsigma)\ud \tilupsigma\Big)^2\ud \tiluptau\Big)^{\frac{1}{2}},
\end{split}
\end{align*}
where we have used $\chi_S$ to denote the characteristic function of a set $S$. Applying Schur's test and noting that
%%%%%%%
%%%%%%%
\begin{align*}
\begin{split}
\|\chi_{\{\tilupsigma_{\min}(\tiluptau)\leq \tilupsigma\leq\tilupsigma_\max(\tiluptau)\}}\|_{L^\infty_\tiluptau L^1_\tilupsigma\cap L^\infty_\tilupsigma L_\tiluptau^1}\lesssim_{\Reigenfunctioncutoffscale}|\ell|\lesssim \epsilon,
\end{split}
\end{align*}
we get
%%%%%%%
%%%%%%%
\begin{align*}
\begin{split}
&\Big\|\iint_{\calT_{\tilt,1}}\big(\Zbar_kg-V\Zbar_k u+(\tilhbar^{-1})^{\alpha\beta}\partial_\alpha u \partial_\beta \Zbar_k\big)\sqrt{|\tilhbar|}\ud \tily \ud \tiluptau'\Big\|_{L^2_\tiluptau(\tilI)}\\
&\lesssim \epsilon \Big(\int_{\tilupsigma_\min}^{\tilupsigma_\max}\Big(\int_{\{\tiluptau=\tiluptau'\}} \big|\Zbar_kg-V\Zbar_k u+(\tilhbar^{-1})^{\alpha\beta}\partial_\alpha u \partial_\beta \Zbar_k\big|\sqrt{|\tilhbar|}\ud \tily\Big)^2\ud\tiluptau'\Big)^{\frac{1}{2}}\\
&\lesssim \epsilon\big(\|u\|_{LE}+\|g\|_{L^2(\{\uprho\leq \Reigenfunctioncutoffscale\}}\big).
\end{split}
\end{align*}
Combining with \eqref{eq:tilOmegabarL2temp1} and \eqref{eq:tilOmegabarL2temp2} we have shown that 
%%%%%%%
%%%%%%%
\begin{align*}
\begin{split}
\|\tilbfOmegabar_k(u)\|_{L^2_\tiluptau(\tilI)}&\lesssim\bfUpomega_k(\uppsi)\|_{L^2_\uptau([t_1,t_2])}+ \epsilon \|\uppsi\|_{LE(\calR_{t_1}^{t_2})}\\
&\quad+\|\uppsi_\far\|_{LE}+\|\uppsi_{\near,\far}\|_{LE}+\|\uppsi\|_{LE((\Sigma_{t_1}^{t_2})^c)}.
\end{split}
\end{align*}
The estimate for $\|\tilbfOmegabar_{n+k}(u)\|_{L^2_\tiluptau(\tilI)}$ is similar except that when using the divergence identity to relate $\tilbfOmegabar_k(u)(\tiluptau)$ and $\bfOmegabar(u)(\gammabar\tiluptau)$, analogously to \eqref{eq:orthodivtemp1}, we need to integrate the quantity
%%%%%%%
%%%%%%%
\begin{align*}
\begin{split}
v\Zbar_k\Box \tiluptau+(\tilhbar^{-1})^{\alpha\beta}\partial_\alpha\tiluptau\partial_\beta(v\Zbar_k)
\end{split}
\end{align*}
over $\calT_{\tiluptau}$. The estimates for $\tilbfOmegabar_{\mu}^{\pm}(u)$ are obtained similarly, completing the proof of \eqref{eq:phinearneartemp1}.
\end{proof}
%%%%%%%%%%%
%%%%%%%%%%%
It remains to prove Lemma~\ref{lem:phinearfar1}. For this we will use the following technical lemma.
%%%%%%%%%%%
%%%%%%%%%%%
\begin{lemma}\label{lem:LEDcommutator1}
Suppose $a$, $b$, and $c$, satisfy
%%%%%%%
%%%%%%%
\begin{align*}
\begin{split}
\sup_y(\jap{\uprho}^{\frac{1+\alpha}{2}}|a|+|b|+|c|)\lesssim \epsilon \jap{\uptau}^{-\gamma}, 
\end{split}
\end{align*}
for some $\gamma>1$. Then
%%%%%%%
%%%%%%%
\begin{align*}
\begin{split}
\calX&:=\|[P_{\leq N_0},a]\partial^2_\uptau \uppsi_\near\|_{L^1_\uptau L^2_y\cap L^2_{\uptau,y}(I)}^2+\|[P_{\leq N_0},b]\partial^2_y \uppsi_\near\|_{L^1_\uptau L^2_y\cap L^2_{\uptau,y}(I)}^2\\
&\quad+\|[P_{\leq N_0},c]\partial^2_{\uptau y} \uppsi_\near\|_{L^1_\uptau L^2_y\cap L^2_{\uptau,y}(I)}^2
\end{split}
\end{align*}
satisfies
%%%%%%%
%%%%%%%
\begin{align*}
\begin{split}
\calX\lesssim \epsilon \|\uppsi_\near\|_{LE}^2+\epsilon\|\uppsi_\far\|_{LE(I)}^2+\epsilon\sup_\uptau \|\uppsi_\near\|_{E(\Sigma_\uptau)}^2.
\end{split}
\end{align*}
\end{lemma}
%%%%%%%%%%%
%%%%%%%%%%%
\begin{proof}
To simplify notation we will write $\phi$ instead of $\uppsi_\near$ during the proof. Note that since the coefficients $\ringpi_q^{\mu\nu}$ of $\calP_\pert$ satisfy the conditions assumed on $a$, $b$, $c$, we can simultaneously carry out our estimates with $a$, $b$, $c$, replaced by $\ringpi_q^{\mu\nu}$. This will allow us to absorb small multiplies of the quantity we are trying to estimate. With this in mind, to simplify notation, we simply assume that $a$, $b$, and $c$ are equal to $\ringpi_q^{\uptau\uptau}$, $\ringpi_q^{yy}$, and $\ringpi_q^{\uptau y}$, respectively. We will repeatedly use the following standard weighted commutator estimate
%%%%%%%
%%%%%%%
\begin{align}\label{eq:commtempgen1}
\begin{split}
\|[P_{N},h]g\|_{L^r_\uptau}\lesssim 2^{-N}(1+2^{-\beta N})\|\jap{\uptau}\partial_\uptau h\|_{L^q_\uptau}\|\jap{\uptau}^{-\beta}g\|_{L^p_\uptau},\qquad \frac{1}{p}+\frac{1}{q}=\frac{1}{r}.
\end{split}
\end{align}
Indeed, with $\chi_N(\uptau)=2^N\chi(2^N\uptau)$, for an appropriate Schwartz function $\chi$ we have
%%%%%%%
%%%%%%%
\begin{align*}
\begin{split}
\|[P_{N},h]g\|_{L^r_\uptau}&\leq 2^{-N}\int_{\bbR}\int_0^1\chi_N(s)\|\big(\jap{\uptau-s}^\alpha h'(\uptau-ts))(\jap{\uptau-s}^{-\alpha}g(\uptau-s))\|_{L^r_\uptau}\ud t \ud s\\
&\leq 2^{-N}\int_{\bbR} \chi_N(s)\int_0^1\|\jap{\uptau}^\beta h\|_{L^q_\uptau}\|\jap{\uptau}^{-\beta}g\|_{L^p_\uptau}\|\jap{\uptau-(1-t)s}^\beta\jap{\uptau}^{-\beta}\|_{L^\infty_\uptau}\ud t \ud s\\
&\lesssim 2^{-N}\|\jap{\uptau}^\beta h'\|_{L^q_\uptau}\|\jap{\uptau}^{-\beta}g\|_{L^p_\uptau}\int_{\bbR}\chi_N(s)(1+2^{-\beta N}|s|^\beta)\ud s,
\end{split}
\end{align*}
which proves the commutator estimate. In our applications $N\geq 1$ and we use $N$ instead of $ N^{-1}(1+N^{-\beta})$. Also note that the same estimate holds if $P_N$ is replaced by $P_{\leq N}$, $\bfP_N$, or $\bfP_{\leq N}$.

Let
%%%%%%%
%%%%%%%
\begin{align}\label{eq:calAcommdef1}
\begin{split}
\calA&:=\|[P_{\leq N_0},a]\partial^2_\uptau \phi\|_{L^1_\uptau L^2_y\cap L^2_{\uptau,y}(I)}^2+\sum_{N\geq N_0}\|[P_N,a]\partial^2_\uptau\phi\|_{L^1_\uptau L^2_y\cap L^2_{\uptau,y}(I)}^2\\
&\quad+\|[P_{\leq N_0},b]\partial^2_y \phi\|_{L^1_\uptau L^2_y\cap L^2_{\uptau,y}(I)}^2+\sum_{N\geq N_0}\|[P_N,b]\partial^2_y\phi\|_{L^1_\uptau L^2_y\cap L^2_{\uptau,y}(I)}^2\\
&\quad+\|[P_{\leq N_0},c]\partial^2_{\uptau y} \phi\|_{L^1_\uptau L^2_y\cap L^2_{\uptau,y}(I)}^2+\sum_{N\geq N_0}\|[P_N,c]\partial^2_{\uptau y}\phi\|_{L^1_\uptau L^2_y\cap L^2_{\uptau,y}(I)}^2.
\end{split}
\end{align}
Since $\calA\geq \calX$, it suffices to prove the estimate (the extra terms are added to absorb error terms that arise in the estimates)
%%%%%%%
%%%%%%%
\begin{align}\label{eq:calAbound}
\begin{split}
\calA\lesssim  \epsilon \|\uppsi_\near\|_{LE}^2+\epsilon\|\uppsi_\far\|_{LE(I)}^2+\epsilon\sup_\uptau \|\uppsi_\near\|_{E(\Sigma_\uptau)}^2.
\end{split}
\end{align}
Let us start with the second terms on each line of \eqref{eq:calAcommdef1}. Let $\bfP_N=\sum_{M=N-4}^{M+4}P_M$ and $\bfP_{\leq N}=P_{\leq N+4}$ be fattened projections, and decompose
%%%%%%%
%%%%%%%
\begin{align}\label{eq:comfreqdecomp1}
\begin{split}
[P_N,a]\partial^2_\uptau\phi&= \sum_{M=N+1}^{N+4}[P_N, \bfP_{\leq N}a]\partial^2_\uptau P_M\phi+[P_N,\bfP_{N}a]\partial^2_\uptau P_{\leq N}\phi+\sum_{M> N+4}[P_N,\bfP_Ma]\partial^2_\uptau P_M\phi\\
&\quad+[P_N,P_{\leq N+4}a]\partial^2_\uptau P_{\leq N}\phi+\sum_{M>N+4}[P_N,P_Ma]\partial^2_\uptau P_{\leq N+3}\phi,
\end{split}
\end{align}
with similar decompositions for $[P_N,b]\partial^2_y\phi$ and $[P_N,c]\partial^2_{\uptau y}\phi$. With $\alpha_0:=\frac{1}{2}(1+\alpha)$ and $\tilbfP_N=\sum_{M=N+1}^{N+4}P_M$, the contribution of the first term on the right-hand side of \eqref{eq:comfreqdecomp1} is bounded as 
%%%%%%%
%%%%%%%
\begin{align*}
\begin{split}
\sum_{N\geq N_0}\|[P_N, \bfP_{\leq N}a]\partial^2_\tau\tilbfP_N\phi\|_{L^2_{\uptau,y}}^2&\lesssim \sum_{N\geq N_0} 2^{-2N}\| \jap{\uprho}^{\alpha_0}\partial_\tau \bfP_{\leq N} a\|_{L^\infty_{\uptau,y}}^2\|\jap{\uprho}^{-\alpha_0}\partial^2_\tau \tilbfP_N\phi\|_{L^2_{\uptau,y}}^2\\
&\lesssim \epsilon\sum_{N\geq N_0}\|\jap{\uprho}^{-\alpha_0}\partial_\tau \tilbfP_{N}\phi\|_{L^2_{\uptau,y}}^2\lesssim \epsilon \|\phi\|_{LE}^2.
\end{split}
\end{align*}
The $L^1_\uptau L^2_y$ estimate is similar. The corresponding term for $c$ is bounded as
%%%%%%%
%%%%%%%
\begin{align*}
\begin{split}
\sum_{N\geq N_0}\|[P_N, \bfP_{\leq N}c]\partial^2_{\uptau,y}\tilbfP_N\phi\|_{L^1_\uptau L^2_y}^2&\lesssim \sum_{N\geq N_0} 2^{-2N}\| \jap{\uptau}^{\frac{1}{2}+\delta}\partial_\uptau \bfP_{\leq N} c\|_{L^2_\uptau L^\infty_{y}}^2\|\jap{\uptau}^{-\frac{1}{2}-\delta}\partial^2_{\uptau y} \tilbfP_N\phi\|_{L^2_{\uptau,y}}^2\\
&\lesssim \epsilon\sum_{N\geq N_0}2^{-2N}\big(\|\partial_\uptau \tilbfP_{N}\jap{\uptau}^{-\frac{1}{2}-\delta}\partial_y\phi\|_{L^2_{\uptau,y}}^2+\|[\jap{\uptau}^{-\frac{1}{2}-\delta},\tilbfP_N \partial_\uptau]\partial_y\phi\|_{L^2_{\uptau,y}}^2\big)\\
&\lesssim \epsilon \|\jap{\uptau}^{-\frac{1}{2}-\delta}\partial_y\phi\|_{L^2_{\uptau,y}}+\delta\sum_{N\geq N_0}2^{-2N}\|\partial_\uptau\jap{\uptau}^{-\frac{1}{2}-\delta}\|_{L^2_\uptau L^\infty_y}^2\|\partial_y\phi\|_{L^\infty_\uptau L^2_y}^2\\
&\lesssim \epsilon \|\partial_y\phi\|_{L^\infty_\uptau L^2_y}^2,
\end{split}
\end{align*}
which is bounded by the energy. In this calculation $[\jap{\uptau}^{-\frac{1}{2}-\delta},\bfP_N \partial_\uptau]$ was treated in a similar way as in the proof of \eqref{eq:commtempgen1}.  The estimate for the same sum in $L^2_{\tau,x}$ is the same, except that in the first step we place $\jap{\tau}^{\frac{1}{2}+\epsilon}\partial_\tau \bfP_{\leq N} c$ in $L^\infty_{\tau,x}$ instead of $L^2_\tau L^\infty_x$. Let us now turn to the more difficult term $b$:
%%%%%%%
%%%%%%%
\begin{align*}
\begin{split}
\sum_{N\geq N_0}\|[P_N, \bfP_{\leq N}b]\partial^2_y\tilbfP_N\phi\|_{L^1_\uptau L^2_y}^2&\lesssim \sum_{N\geq N_0} 2^{-2N}\| \jap{\uptau}^{\frac{1}{2}+\delta}\partial_\uptau \bfP_{\leq N} b\|_{L^2_\uptau L^\infty_{y}}^2\|\jap{\uptau}^{-\frac{1}{2}-\delta}\partial^2_{ y} \tilbfP_N\phi\|_{L^2_{\uptau,y}}^2\\
&\lesssim\epsilon\sum_{N \geq N_0}2^{-2N}\|\jap{\uptau}^{-\frac{1}{2}-\delta}\calP_\Ell \tilbfP_N\phi\|_{L^2_{\uptau,y}}^2\\
&\lesssim\epsilon\sum_{N_\geq N_0}2^{-2N}\Big(\|\jap{\uptau}^{-\frac{1}{2}-\delta}[\calP_\Ell ,\tilbfP_N]\phi\|_{L^2_{\uptau,y}}^2+\|\jap{\uptau}^{-\frac{1}{2}-\delta}\tilbfP_N \calP_\uptau\phi\|_{L^2_{\uptau,y}}^2\\
&\phantom{\lesssim\epsilon\sum_{N_\geq N_0}N^{-2}\Big(}+\|\jap{\uptau}^{-\frac{1}{2}-\delta}\tilbfP_N \calP\phi\|_{L^2_{\uptau,y}}^2\Big).
\end{split}
\end{align*}
The last term above is bounded by the right-hand side of \eqref{eq:calAbound} using the equation for $\phi=\uppsi_\near$. In the line before last, the first term can be absorbed by $\calA$, while the second term is treated as in the contribution of $a$ and $c$ above. The contribution of $\sum_{N\geq N_0}\|[P_N, \bfP_{\leq N}b]\partial^2_x\bfP_N\phi\|_{L^2_{\tau ,x}}^2$ can be handled similarly where now we bound $\jap{\tau}^{\frac{1}{2}+\epsilon}\partial_\tau \bfP_{\leq N} b$ in $L^{\infty}_{\uptau,y}$, instead of $L^2_\uptau L^\infty_y$, in the first step. 

We next consider the second term in the decomposition \eqref{eq:comfreqdecomp1}. For $a$ we have
%%%%%%%
%%%%%%%
\begin{align*}
\begin{split}
\sum_{N\geq N_0}\|[P_N,\bfP_{N}a]\partial^2_\tau P_{\leq N}\phi\|_{L^2_\uptau L^2_y}^2&\lesssim \sum_{N\geq N_0} 2^{-2N}\|\jap{\uprho}^{\alpha_0} 2^N\partial_\tau \bfP_{ N} a\|_{L^\infty_{\uptau,y}}^2\|\jap{\uprho}^{-\alpha_0}\partial_\uptau P_{\leq N}\phi\|_{L^2_{\uptau,y} }^2\\
&\lesssim \|\jap{\uprho}^{\alpha_0}\partial^2_\uptau a\|_{L^\infty_{\uptau,y}}^2\|\jap{\uprho}^{-\alpha_0}\partial_\uptau\phi\|_{L^2_{\uptau,y}}^2\sum_{N\geq N_0} 2^{-2N}\lesssim \epsilon \|\phi\|_{LE}^2.
\end{split}
\end{align*}
The estimate in $L^1_\uptau L^2_y$ is similar. For $c$,
%%%%%%%
%%%%%%%
\begin{align*}
\begin{split}
\sum_{N\geq N_0}\|[P_N, \bfP_{ N}c]\partial^2_{\uptau,y}P_{\leq N}\phi\|_{L^1_\uptau L^2_y}^2&\lesssim \sum_{N\geq N_0} 2^{-2N}\| \jap{\uptau}^{\frac{1}{2}+\epsilon}\partial_\uptau \bfP_{N} c\|_{L^2_\uptau L^\infty_{y}}^2\|\jap{\uptau}^{-\frac{1}{2}-\epsilon}\partial^2_{\uptau y} P_{\leq N}\phi\|_{L^2_{\uptau,y}}^2\\
&\lesssim \sum_{N\geq N_0}2^{-2N}\|\partial_\uptau P_{\leq N}\jap{\uptau}^{-\frac{1}{2}-\epsilon}\partial_y\phi\|_{L^2_{\uptau,y}}^2\|\jap{\uptau}^{\frac{1}{2}+\epsilon}\partial_\uptau\bfP_Nc\|_{L^2_\uptau L^\infty_y}^2\\
&\quad+\sum_{N\geq N_0}2^{-2N}\|[\jap{\uptau}^{-\frac{1}{2}-\epsilon},P_{\leq N} \partial_\uptau]\partial_y\phi\|_{L^2_{\uptau,y}}^2\|\jap{\uptau}^{\frac{1}{2}+\epsilon}\partial_\uptau\bfP_Nc\|_{L^2_\uptau L^\infty_y}^2\\
&\lesssim\|\partial_y\phi\|_{L^\infty_\uptau L^2_y}^2\sum_{N\geq N_0}\|\jap{\uptau}^{\frac{1}{2}+\epsilon}\bfP_Nc\|_{L^2_\uptau L^\infty_y}^2\lesssim \epsilon \|\partial_y\phi\|_{L^\infty_\uptau L^2_y}^2.
\end{split}
\end{align*}
The $L^2_{\uptau,y}$ estimate is similar. For $b$ we again use elliptic estimates and argue more carefully as
%%%%%%%
%%%%%%%
\begin{align*}
\begin{split}
\sum_{N\geq N_0}\|[P_N, \bfP_{ N}b]\partial^2_y P_{\leq N}\phi\|_{L^1_\uptau L^2_y}^2&\lesssim \sum_{N\geq N_0} 2^{-2N}\| \jap{\uptau}^{\frac{1}{2}+\delta}\partial_\uptau \bfP_{N} b\|_{L^2_\uptau L^\infty_{y}}^2 \|\jap{\uptau}^{-\frac{1}{2}-\delta}[\calP_{\Ell} P_{\leq N}]\phi\|_{L^2_{\uptau,y}}^2\\
&\quad+ \sum_{N\geq N_0} 2^{-2N}\| \jap{\uptau}^{\frac{1}{2}+\delta}\partial_\uptau \bfP_{N} b\|_{L^2_\uptau L^\infty_{y}}^2 \|\jap{\uptau}^{-\frac{1}{2}-\delta} P_{\leq N}\calP_\uptau\phi\|_{L^2_{\uptau,y}}^2\\
&\quad+ \sum_{N\geq N_0} 2^{-2N}\| \jap{\uptau}^{\frac{1}{2}+\delta}\partial_\uptau \bfP_{N} b\|_{L^2_\uptau L^\infty_{y}}^2 \|\jap{\uptau}^{-\frac{1}{2}-\delta}P_{\leq N}\calP\phi\|_{L^2_{\uptau,y}}^2.
\end{split}
\end{align*}
The last two lines can be handled as in the case of $a$ and $c$ above and using the equation for $\phi=\uppsi_\near$. The first line is bounded by
%%%%%%%%
%%%%%%%%
%%%%%%%
%%%%%%%
\begin{align*}
\begin{split}
&\sum_{N\geq N_0} 2^{-2N}\| \jap{\uptau}^{\frac{1}{2}+\epsilon}\partial_\uptau \bfP_{N} b\|_{L^2_\uptau L^\infty_{y}}^2 \|[\calP_{\Ell}, P_{\leq N_0}]\phi\|_{L^2_{\uptau,y}}^2\\
&+\sum_{N\geq N_0}\sum_{N_0< M\leq N} 2^{-2N}\| \jap{\uptau}^{\frac{1}{2}+\epsilon}\partial_\uptau \bfP_{N} b\|_{L^2_\uptau L^\infty_{y}}^2 \|[\calP_{\Ell}, P_{M}]\phi\|_{L^2_{\uptau,y}}^2\\
&\lesssim \calA \sum_{N\geq N_0}2^{-2N}\| \jap{\tau}^{\frac{1}{2}+\epsilon}\partial_\tau \bfP_{N} b\|_{L^2_\tau L^\infty_{x}}^2\lesssim \epsilon \calA,
\end{split}
\end{align*}
which can be absorbed. The estimate in $\sum_{N\geq N_0}\|[P_N, \bfP_{ N}b]\partial^2_yP_{\leq N}\phi\|_{ L^2_{\uptau,y}}^2$ is similar.

We consider the third term in the decomposition \eqref{eq:comfreqdecomp1} next. For $a$ we have
%%%%%%%
%%%%%%%
\begin{align*}
\begin{split}
\sum_{N\geq N_0}\Big\|\sum_{M> N+4}[P_N,\bfP_Ma]\partial^2_\uptau P_M\phi\Big\|_{L^2_{\uptau,y}}^2&\lesssim \sum_{N\geq N_0}\Big(\sum_{M> N}2^{M-N}\|\jap{\uprho}^{\alpha_0}\partial_\uptau\bfP_Ma\|_{L^\infty_{\uptau,y}}\|\jap{-\uprho}^{\alpha_0}\partial_\uptau P_M\phi\|_{L^2_{\uptau,y}}\Big)^2\\
&\lesssim \|\jap{r}^{\alpha}\partial_\uptau^3a\|_{L^\infty_{\uptau,y}}^2\|\phi\|_{LE}^2\sum_{N\geq N_0}2^{-4N}\lesssim \epsilon \|\phi\|_{LE}^2.
\end{split}
\end{align*}
The estimate in $L^1_{\uptau}L^2_y$ is similar. For $c$,
%%%%%%%
%%%%%%%
\begin{align*}
\begin{split}
&\sum_{N\geq N_0}\Big\|\sum_{M> N+4}[P_N, \bfP_{M}c]\partial^2_{\uptau,y}P_{M}\phi\Big\|_{L^1_\uptau L^2_y}^2\\
&\lesssim \sum_{N\geq N_0} N^{-2}\Big(\sum_{M> N}M\| \jap{\uptau}^{\frac{1}{2}+\delta}\partial_\uptau \bfP_{M} c\|_{L^2_\uptau L^\infty_{y}}\|\jap{\uptau}^{-\frac{1}{2}-\delta}\partial_{ y} \bfP_{M}\phi\|_{L^2_{\uptau,y}}\Big)^2\\
&\lesssim \|\jap{\uptau}^{\frac{1}{2}+\epsilon}\partial_\uptau^2c\|_{L^2_\uptau L^\infty_y}^2\|\partial_y\phi\|_{L^\infty_\uptau L^2_y}^2\lesssim \epsilon\|\partial_y\phi\|_{L^\infty_\uptau L^2_y}^2,
\end{split}
\end{align*}
and the contribution in $L^2_{\uptau,y}$ is bounded in a similar way. For $b$ we again apply elliptic estimates. Except for the commutator with $\calP_\Ell$, the resulting terms can be bounded as was done for $a$ and $c$, and using the equation for $\phi=\uppsi_\near$, and the commutator term is bounded, using similar arguments as earlier, as (for the $L^1_\uptau L^2_y$ estimate; the $L^2_{\uptau,y}$ estimate is similar)
%%%%%%%
%%%%%%%
\begin{align*}
\begin{split}
&\sum_{N\geq N_0} N^{-2}\Big(\sum_{M\geq N}\|\jap{\uptau}^{\frac{1}{2}+\epsilon}\partial_\uptau \bfP_{M} b\|_{L^2_\uptau L^\infty_{y}} \|[\calP_{\Ell}, P_{M}]\phi\|_{L^2_{\uptau,y}}\Big)^2\\
&\lesssim \|\jap{\uptau}^{\frac{1}{2}+\epsilon}\partial_\uptau^2b\|_{L^2_\uptau L^\infty_y}^2\sum_{M> N_0}\|[\calP_\Ell,P_M]\phi\|_{L^2_{\uptau,y}}^2\lesssim \epsilon \calA,
\end{split}
\end{align*}
which can be absorbed.

The first term on the second line of \eqref{eq:comfreqdecomp1} can be further decomposed as
%%%%%%%
%%%%%%%
\begin{align*}
\begin{split}
\sum_{N\geq N_0}[P_N,P_{\leq N+4}a]\partial_{\uptau}^2P_{\leq N}\phi&=\sum_{N=N_0}^{N_0+10}[P_{N},P_{\leq N +4}a]P_{\leq N}\partial_{\uptau}^2\phi+\sum_{N>N_0+10}[P_N,\bfP_Na]\partial_{\uptau}^2P_{\leq N-4}\phi\\
&\quad+\sum_{N>N_0+10}\sum_{M=N-3}^{N}[P_N,P_{\leq N+3}a]\partial_{\uptau}^2P_M\phi,
\end{split}
\end{align*}
with similar decompositions for the contributions of $b$ and $c$. Each of the terms above can then be bounded using similar arguments to the earlier ones. Similarly, the sum over $N\geq N_0$ of the last term in \eqref{eq:comfreqdecomp1} can be decomposed as (with similar decompositions for $b$ and $c$ contributions)
%%%%%%%
%%%%%%%
\begin{align*}
\begin{split}
\sum_{N=N_0}^{N_0+10}\sum_{M>N+4}[P_N,P_Ma]\partial_{\uptau}^2P_{\leq N+3}\phi+\sum_{N=N_0}^{N_0+10}\sum_{M>N+4}\sum_{L=N-3}^{N+3}[P_N,P_Ma]\partial_{\uptau}^2P_{L}\phi,
\end{split}
\end{align*}
and each of these terms can be estimated as before. 

It remains to estimate the first term on each line of the definition of $\calA$. For this we use the following decomposition, where $\tilbfP_{\leq N_0}=P_{\leq N_0+2}$,
%%%%%%%
%%%%%%%
\begin{align}\label{eq:comfreqdecomp2}
\begin{split}
[P_{\leq N_0},a]\partial^2_\uptau\phi&=[P_{\leq N_0}, \bfP_{\leq N_0}a]\partial^2_\uptau \tilbfP_{\leq N_0}\phi+\sum_{N>N_0+2}[P_{\leq N_0},\bfP_Na]\partial^2_\uptau P_N\phi\\
&\quad+\sum_{N>N_0+4}[P_{\leq N_0},P_Na]\partial^2_\uptau \tilbfP_{\leq N_0}\phi,
\end{split}
\end{align}
and similarly for $b$ and $c$. The contribution of the first term is bounded as
%%%%%%%
%%%%%%%
\begin{align*}
\begin{split}
\|[P_{\leq N_0}, \bfP_{\leq N_0}a]\partial^2_\uptau\tilbfP_{\leq N_0}\phi\|_{L^2_{\uptau,y}}^2\lesssim \|\jap{\uprho}^{\alpha_0} \partial_\uptau \bfP_{\leq N_0}a\|_{L^\infty_{\uptau,y}}^2\|\jap{\uprho}^{-\alpha_0}\partial_\uptau^2\tilbfP_{\leq N_0}\phi\|_{L^2_{\uptau,y}}^2\lesssim \epsilon \|\phi\|_{LE}^2,
\end{split}
\end{align*}
with the $L^1_\uptau L^2_y$ estimate being similar. The corresponding contribution for $c$ is bounded as
%%%%%%%
%%%%%%%
\begin{align*}
\begin{split}
\|[P_{\leq N_0}, \bfP_{\leq N_0}c]\partial^2_{\uptau y}\tilbfP_{\leq N_0}\phi\|_{L^1_\uptau L^2_y}\lesssim \|\jap{\uptau}^{\frac{1}{2}+\epsilon}\partial_\uptau \bfP_{\leq N_0}c\|_{L^2_\uptau L^\infty_y}^2\|\jap{\uptau}^{-\frac{1}{2}-\epsilon}\partial_{\uptau y}\tilbfP_{\leq N_0}\phi\|_{L^2_{\uptau,y}}\lesssim \epsilon \|\partial_y\phi\|_{L^\infty_\uptau L^2_y}^2,
\end{split}
\end{align*}
and the corresponding estimate in $L^2_{\tau,x}$ is handled similarly. For $b$ we apply elliptic estimates and as usual use similar arguments as for $a$ and $c$ above as well as the equation for $\phi=\uppsi_\near$ to handle all terms except the commutator with $\calP_\Ell$, while this commutator error is  bounded by $\epsilon \calA$ which can be absorbed.

Turning to the second term in \eqref{eq:comfreqdecomp2}, the contribution of $a$ in $L^2_{\uptau,y}$ is bounded as
%%%%%%%
%%%%%%%
\begin{align*}
\begin{split}
\Big\|\sum_{N> N_0+2}[P_{\leq N_0},\bfP_Na]\partial^2_\tau P_N\phi\Big\|_{L^2_{\uptau ,y}}^2&\lesssim \Big(\sum_{N> N_0}\|\jap{\uprho}^{\alpha_0}\partial_\tau N\bfP_Na\|_{L^\infty_{\tau,y}}\|\jap{\uprho}^{-\alpha_0}\partial_\tau P_N\phi\|_{L^2_{\uptau,y} }\Big)^2\\
&\lesssim \|\jap{\uprho}^{\alpha_0}\partial_\tau^3a\|_{L^\infty_{\tau,y}}^2\|\phi\|_{LE}^2\lesssim \epsilon \|\phi\|_{LE}^2,
\end{split}
\end{align*}
and the $L^1_\uptau L^2_y$ contribution is similar. For $c$
%%%%%%%
%%%%%%%
\begin{align*}
\begin{split}
\Big\|\sum_{N> N_0+2}[P_{\leq N_0},\bfP_Nc]\partial^2_{\uptau y}P_N\phi\Big\|_{L^1_\uptau L^2_y}&\lesssim \Big(\sum_{N\geq N_0}\|\jap{\uptau}^{\frac{1}{2}+\delta}2^N\partial_\uptau \bfP_{N}c\|_{L^2_\uptau L^\infty_y}\|\jap{\uptau}^{-\frac{1}{2}-\delta}\partial_yP_N\phi\|_{L_{\uptau,y}^2}\Big)^2\\
&\lesssim \epsilon \|\partial_y\phi\|_{L^\infty_\uptau L^2_y}^2.
\end{split}
\end{align*}
The estimate for the corresponding term in $L^2_{\uptau,y}$ is similar. The contribution of $b$ is also handled using elliptic estimates as usual. Finally, the last term in \eqref{eq:comfreqdecomp2} can also be handled using similar arguments as above, and we omit the details. This completes the proof of \eqref{eq:calAbound}. 
\end{proof}
%%%%%%%%%%%
%%%%%%%%%%%
We can now prove Lemma~\ref{lem:phinearfar1}.
%%%%%%%%%%%
%%%%%%%%%%%
\begin{proof}[Proof of Lemma~\ref{lem:phinearfar1}]
In this proof, when working in the $(\uptau,\uprho,\upomega)$ coordinates, we will use $y$ for the coordinates $(\uprho,\upomega)$. The notation $\iint$ is used for integration over $\Sigma_{t_1}^{t_2}$. We will also use the notation $I=[t_1,t_2]$ with $LE(I)=LE(\Sigma_{t_1}^{t_2})$ and $LE=LE(\cup_{\uptau}\Sigma_{\uptau})$. Note that it suffices to estimate $\|P_{\leq N_0}\uppsi_{\near,\far}\|_{LE(I)}$. By a similar multiplier argument as in the proof of Lemma~\ref{lem:phifar1}, and with multipliers $Q_j^\Int P_{\leq N_0}\uppsi_{\near,\far}$ and $Q_j^\Ext P_{\leq N_0}\uppsi_{\near,\far}$, $j=1,2$, we get (recall that, by a slight abuse of notation, we write $P_{\leq N_0}\calP_\pert\uppsi_{\near}$ for $P_{\leq N_0}\uppsi$, where $\uppsi=\calP_\pert\uppsi_{\near}$ in~$\Sigma_{t_1}^{t_2}$ and $\uppsi=0$ in $(\Sigma_{t_1}^{t_2})^c$) 
%%%%%%%
%%%%%%%
\begin{align*}
\begin{split}
\|P_{\leq N_0}\uppsi_{\near,\far}\|_{LE(I)}^2&\lesssim \|V_\far\uppsi_\far\|_{LE^\ast(I)}^2+\sum_{j=1}^2\Big|\iint (P_{\leq N_0}\calP_\pert\uppsi_\near)(Q^\Int_jP_{\leq N_0}\uppsi_{\near,\far}))\sqrt{|\hbar|}\ud y\ud\tau\Big|\\
&\quad +\sum_{j=1}^2\Big|\iint (P_{\leq N_0}\calP_\pert\uppsi_\near)(Q^\Ext_jP_{\leq N_0}\uppsi_{\near,\far}))\sqrt{|\hbar|}\ud y\ud\tau\Big|\\
&\lesssim\|\uppsi\|_{E(\Sigma_{t_1})}^2+\|f\|_{L^1L^2(\Sigma_{\tau_1}^{\tau_2})}^2\\
&\quad+\Big|\iint (P_{\leq N_0}\calP_\pert\uppsi_\near)(QP_{\leq N_0}\uppsi_{\near,\far}))\sqrt{|\hbar|}\ud y\ud\tau\Big|\\
&\quad +\sum_{j=1}^2\Big|\iint (P_{\leq N_0}\calP_\pert\uppsi_\near)(Q^\Ext_jP_{\leq N_0}\uppsi_{\near,\far}))\sqrt{|\hbar|}\ud y\ud\tau\Big|.
\end{split}
\end{align*}
Here $Q_1^\Int$ and $Q_1^\Ext$ denote the first order interior and exterior multipliers (see \eqref{eq:psifarLEDtemp1.5} and \eqref{eq:LEDintbulk1}), respectively, and $Q_2^\Int$ and $Q_2^\Ext$ are the corresponding order zero multipliers (see \eqref{eq:psifarLEDtemp2.5} and \eqref{eq:LEDintbulk2}; we have not used the notation $P$ for the order zero multipliers to prevent confusion with the frequency projections). Therefore, our task is reduced to estimating 
%%%%%%%
%%%%%%%
\begin{align}\label{eq:mainerror}
\begin{split}
\Big|\iint (P_{\leq N_0}\calP_\pert\uppsi_\near)(Q_j^{\Int,\Ext} P_{\leq N_0}\uppsi_{\near,\far})\sqrt{|\hbar|}\ud y\ud\tau\Big|.
\end{split}
\end{align}
For the interior we place both terms in $LE$ where for $P_{\leq N_0}\calP_\pert \uppsi_{\near}$ we use the frequency projection to remove one derivative. More precisely, let $\chi_K\equiv\chi_K(\uprho)$ be a cutoff to a large compact region $K$ containing the supports of $Q_j^\Int$, and let $\chi_{K^c}=1-\chi_K$. Then (note that we can always insert a factor of $\uprho$ in the order zero terms using the same argument as in Lemma~\ref{lem:LEDhighfreq1})
%%%%%%%
%%%%%%%
\begin{align*}
\begin{split}
&\Big|\iint \chi_K(P_{\leq N_0}\calP_\pert\uppsi_\near)(Q_j^\Int P_{\leq N_0}\uppsi_{\near,\far}))\sqrt{|\hbar|}\ud y\ud\tau\Big|\\
&\lesssim \delta\|P_{\leq N_0}\uppsi_{\near,\far}\|_{LE(I)}^2+C_\delta\iint(P_{\leq N_0}(\chi_K\calP_\pert\uppsi_\near))^2\sqrt{|\hbar|}\ud y\ud\tau.
\end{split}
\end{align*}
The first term can be absorbed if $\delta$ is small. For the last term we write $\calP_\pert$  as $$\calP_\pert=\ringpi_q^{\mu\nu}\partial^2_{\mu\nu}+\ringpi_l^\mu\partial_\mu+\ringpi_c,$$
where $|\ringpi_q|,|\ringpi_l|,|\ringpi_c|\lesssim \epsilon$. The contribution of $\ringpi_l^\mu\partial_\mu + \ringpi_c$ is then already bounded by $\epsilon\|\uppsi_\near\|_{LE}^2$, which is admissible. For the second order part, if at least one of $\mu$ or $\nu$ is $\uptau$ then we can use the frequency projection to drop a $\partial_\uptau$ derivative and argue as before. When both derivatives are with respect to the spatial variables $y$, we use elliptic estimates using $\calP$. That is, let
%%%%%%%
%%%%%%%
\begin{align*}
\begin{split}
\calP= \calP_\uptau+\calP_\Ell,
\end{split}
\end{align*}
where $\calP_\uptau$ contains the terms with at least one $\partial_\uptau$ derivative and $\calP_\Ell$ is the remainder which is elliptic. Then, with $\tilchi_K\equiv\tilchi_K(\uprho)$ a cutoff with slightly larger support than $\chi_K$, recalling equation~\eqref{eq:calPphinear1}, and using elliptic regularity, we can bound $\|P_{\leq N_0}(\chi_K\ringpi^{yy}\partial_y^2\uppsi_\near)\|_{L^2_{\tau,y}(I)}$ by
%%%%%%%
%%%%%%%
\begin{align*}
\begin{split}
& \|[P_{\leq N_0},\chi_K \ringpi^{yy}\partial_y^2]\uppsi_\near\|_{L^2_{\tau,y}(I)}+\epsilon\|\tilchi_K\calP_\Ell P_{\leq N_0}\uppsi_\near\|_{L^2_{\tau,y}(I)}+\epsilon\|\tilchi_K P_{\leq N_0}\uppsi_\near\|_{L^2_{\tau,y}(I)}\\
&\lesssim \|[P_{\leq N_0},\chi_K \ringpi^{yy}\partial_y^2]\uppsi_\near\|_{L^2_{\tau,y}(I)}+\epsilon\|[\tilchi_K\calP_\Ell,P_{\leq N_0}]\uppsi_\near\|_{L^2_{\tau,y}(I)}\\
&\quad+\epsilon\|P_{\leq N_0}(\tilchi_K\calP_\uptau\uppsi_\near)\|_{L^2_{\tau,y}(I)}+\epsilon\|P_{\leq N_0}(\tilchi_KV_\far\uppsi_\far)\|_{L^2_{\tau,y}(I)}+\epsilon\|\tilchi_K P_{\leq N_0}\uppsi_\near\|_{L^2_{\tau,y}(I)}.
\end{split}
\end{align*}
The last line contributes an admissible error by using the frequency projection to drop $\partial_\uptau$ derivatives in the first term and using Lemma~\ref{lem:phifar1} for the second term. The line before last with the commutators also contributes an admissible error by Lemma~\ref{lem:LEDcommutator1}. 

For the exterior, note that by choosing the compact set $K$ above sufficiently large, we may assume that the coordinates $(\uptau,\uprho,\upomega)$ and $(\tau,r,\theta)$ agree in $K^c$. We will therefore use the notation $(\tau,r,\theta)$ instead of  $(\uptau,\uprho,\upomega)$. In this region we need an extra weighted energy estimate for $\uppsi_\near$ which allows us to put more $r$ weights on $(\partial_r+\frac{n-1}{2r})\uppsi_\near$ (see \underline{\emph{Case 2}} below). The details are as follows. First recall that $\calP_\pert$ is of the form
%%%%%%%
%%%%%%%
\begin{align*}
\begin{split}
\ringa\partial_\tau(\partial_r+\frac{n-1}{2r})+\ringa^{\mu\nu}\partial_{\mu\nu}+\ringb^\mu \partial_\mu +\ringc,
\end{split}
\end{align*}
where for some $\gamma>1$,
%%%%%%%
%%%%%%%
\begin{align*}
\begin{split}
|\ringa|, |\ringa^{yy}|,\lesssim \epsilon \tau^{-\gamma}, \quad  |\partial_y\ringa^{yy}|, |\ringa^{\tau y}|, |\ringb^y|\lesssim \epsilon\tau^{-\gamma}r^{-1},\quad |\ringa^{\tau\tau}|, |\ringb^\tau|\lesssim \epsilon \tau^{-\gamma}r^{-2},\quad |\ringc|\lesssim \epsilon\tau^{-\gamma} r^{-4}.
\end{split}
\end{align*} 
while $Q_j^\Ext$ have the structure
\begin{align*}
\begin{split}
Q_j^\Ext=\beta^\tau \partial_\tau+\beta^y \partial_y +\frac{\beta}{r},
\end{split}
\end{align*}
with
%%%%%%%
%%%%%%%
\begin{align*}
\begin{split}
|\beta|, |\beta^\tau|, |\beta^y|\lesssim 1, \quad |\partial_y\beta|, |\partial_y\beta^\tau|, |\partial_y\beta^y|\lesssim r^{-1-\alpha},
\end{split}
\end{align*}
and $\beta$, $\beta^\tau$, $\beta^y$ supported outside of some large compact set $\tilK$. We now consider the contribution of all combinations to \eqref{eq:mainerror}. Below we will use the shorthand notation $\angles{f}{g}=\iint fg \sqrt{\hbar}\ud y \ud \uptau.$

{\underline{\emph{Case 1:}}} First we consider the contribution of $\partial_\tau$ in $Q_j^\Ext$ and every term except $2a\partial_\tau(\partial_r+\frac{n-1}{2r})$ in $\calP_\pert$. For $\ringa^{yy}$, after one integration by parts in $\partial_y$ we get
%%%%%%% 
%%%%%%%
\begin{align*}
\begin{split}
&|\angles{P_{\leq N_0}\ringa^{yy}\partial_y^2\uppsi_\near}{\beta^\tau\partial_\tau P_{\leq N_0}\uppsi_{\near,\far}}|\\
&\lesssim \angles{|P_{\leq N_0}\ringa^{yy}\partial_r\uppsi_\near|}{|\partial_\tau P_{\leq N_0}\partial_r\uppsi_{\near,\far}|}+ \angles{|P_{\leq N_0}\ringa^{yy}\partial_r\uppsi_\near|}{|\partial_\tau P_{\leq N_0}r^{-1}\uppsi_{\near,\far}|}\\
&\quad+ \angles{|P_{\leq N_0} r\partial_y\ringa^{yy}\partial_r\uppsi_\near|}{|\partial_\tau P_{\leq N_0}r^{-1}\uppsi_{\near,\far}|}+ \epsilon(\sup_{\tau}\|\uppsi_\near\|_{E}+\sup_{\tau\in I}\|P_{\leq N_0}\uppsi_{\near,\far}\|_E)\\
&\lesssim \epsilon(\sup_{\tau}\|\uppsi_\near\|_{E}+\sup_{\tau\in I}\|P_{\leq N_0}\uppsi_{\near,\far}\|_E),
\end{split}
\end{align*}
where in the last estimate we have used $P_{\leq N_0}$ to drop the $\tau$ derivatives and the $\tau$ decay of $\ringa^{yy}$ and $r\partial_y\ringa^{yy}$ to integrate in $\tau$. The contributions of the other terms in $\ringa^{\mu\nu}\partial_{\mu\nu}$, $\ringb^\mu \partial_\mu$, $\ringc$ are handled similarly, where instead of integrating by parts we use the decay of $r$ and move one factor of $r^{-1}$ to $\uppsi_{\near,\far}$.

{\underline{\emph{Case 2:}}} For the contribution of $\partial_\tau$ to $Q_j^\Ext$ and $2\ringa\partial_\tau(\partial_r+\frac{n-1}{2r})$ to $\calP_\pert$ we use the $\tau$ decay and smallness of $\ringa$ to estimate (here for $\uppsi_{\near,\far}$ we do not drop $\partial_\tau$)
%%%%%%%
%%%%%%%
\begin{align*}
\begin{split}
&|\angles{P_{\leq N_0}\ringa\partial_\tau(\partial_r+\frac{n-1}{2r})\uppsi_\near}{\beta\partial_\tau P_{\leq N_0}\uppsi_{\near,\far}}|\\
&\lesssim \epsilon\|P_{\leq N_0}\uppsi_{\near,\far}\|_{LE(I)}^2+\epsilon^{-1}\|\chi_{\tilK^c}P_{\leq N_0}\ringa\partial_\tau r^{\frac{1+\alpha}{2}}(\partial_r+\frac{n-1}{2r})\uppsi_\near\|_{L^2_{\tau,y}(I)}^2\\
&\lesssim \epsilon\|P_{\leq N_0}\uppsi_{\near,\far}\|_{LE(I)}^2+\epsilon  \sup_{\tau}\|\chi_{\tilK^c}r^{\frac{1+\alpha}{2}}(\partial_r+\frac{n-1}{2r})\uppsi_\near\|_{L^2_y}^2.
\end{split}
\end{align*}
For the last term we will need a weighted energy estimate for $\uppsi_\near$, which we will discuss below.

{\underline{\emph{Case 3:}}} Next we consider the contribution of $\beta^y\partial_y+r^{-1}\beta$ to $Q_j^\Ext$ and the contribution of every term except $\ringa^{yy}\partial^2_y$ to $\calP_\pert$. Here for $\uppsi_{\near,\far}$ we simply drop $P_{\leq N_0}$ while for $\uppsi_\near$ we use $P_{\leq N_0}$ to drop one $\partial_\tau$ in the second order terms in $\calP_\pert$. Using the decay and smallness of the coefficients of $\calP_\pert$, the corresponding contributions are then bounded by $$\epsilon(\sup_{\tau}\|\uppsi_\near\|_E^2+\sup_{\tau\in I}\|P_{\leq N_0}\uppsi_{\near,\far}\|_E^2).$$

{\underline{\emph{Case 4:}}} For the contribution of $\ringa\partial^2_y$ to $\calP_\pert$ and $\beta^y\partial_u+r^{-1}\beta$ to $Q_j^\Ext$ we use the equation $\calP\uppsi_\near = -V_\far \uppsi_\far$ and elliptic estimates for $\calP_\Ell$. Note that, unlike the case of the interior above, in view of the $r$ decay of order zero term in $\calP_\Ell$ we do not need to add an $L^2_y$ term when applying elliptic estimates for $\calP_\Ell$. Using this observation and Lemma~\ref{lem:LEDcommutator1}, and with $\tilK_1$ a compact region contained in $\tilK$, the corresponding contribution is bounded by
%%%%%%%
%%%%%%%
\begin{align*}
\begin{split}
&\|P_{\leq N_0}\uppsi_{\near,\far}\|_{L^\infty_\tau E(I)}\|\chi_{\tilK^c}P_{\leq N_0}\ringa^{yy}\partial_y^2\uppsi_\near\|_{L_\tau^1L_y^2(I)}\\
&\lesssim \|P_{\leq N_0}\uppsi_{\near,\far}\|_{L^\infty_\tau E(I)}(\|\chi_{\tilK^c}[P_{\leq N_0},\ringa^{yy}]\partial_y^2\uppsi_\near\|_{L_\tau^1L_y^2(I)}+\epsilon\|\tau^{-\gamma}\chi_{\tilK^c}\partial_y^2P_{\leq N_0}\uppsi_{\near}\|_{L^1_\tau L^2_y(I)})\\
&\lesssim \|P_{\leq N_0}\uppsi_{\near,\far}\|_{L^\infty_\tau E(I)}(\|\chi_{\tilK^c}[P_{\leq N_0},\ringa^{yy}]\partial_y^2\uppsi_\near\|_{L_\tau^1L_y^2(I)}+\epsilon\|\tau^{-\gamma}\chi_{\tilK^c_1}\calP_\Ell P_{\leq N_0}\uppsi_{\near}\|_{L^1_\tau L^2_y(I)})\\
&\lesssim \epsilon\|P_{\leq N_0}\uppsi_{\near,\far}\|_{L^\infty_\tau E(I)}^2+\epsilon^{-1} \|\chi_{\tilK^c}[P_{\leq N_0},\ringa^{yy}]\partial_y^2\uppsi_\near\|_{L_\tau^1L_y^2(I)}^2+\epsilon\|\tau^{-\gamma} P_{\leq N_0}V_\far\uppsi_{\far}\|_{L^1_\tau L_y^2(I)}^2\\
&\quad+\epsilon\|\tau^{-\gamma}P_{\leq N_0}\chi_{\tilK_1^c}\calP_\tau\uppsi_\near\|_{L^1_\tau E(I)}^2+\epsilon \|\tau^{-\gamma}[\chi_{\tilK^c}\calP_\Ell, P_{\leq N_0}]\uppsi_{\near}\|_{L^1_\tau L_y^2(I)}^2\\
&\lesssim \epsilon(\sup_{\tau}\|\uppsi_\near\|_E^2+\sup_{\tau\in I}\|P_{\leq N_0}\uppsi_{\near,\far}\|_E^2+\sup_{\tau\in I}\|\uppsi_\far\|_E^2+\|\uppsi_\far\|_{LE}^2).
\end{split}
\end{align*}

Putting everything together we have shown that
%%%%%%%
%%%%%%%
\begin{align*}
\begin{split}
\|P_{\leq N_0}\uppsi_{\near,\far}\|_{LE(I)}&\lesssim \|f\|_{LE^\ast(I)}+\|\uppsi\|_{L^\infty_\tau E(I)}+\epsilon\|\uppsi_{\near}\|_{LE}+\epsilon\|P_{\leq N_0}\uppsi_{\near\,\far}\|_{L^\infty_\tau E(I)}\\
&\quad+\epsilon\|\chi_{\tilK^c}r^{\frac{1+\alpha}{2}}(\partial_r+\frac{n-1}{2r})\uppsi_\near\|_{L^\infty_\tau L^2_y}.
\end{split}
\end{align*}
Using a similar argument with the multiplier $\partial_\tau P_{\leq N_0}\uppsi_{\near,\far}$, we can also prove an energy estimate for $P_{\leq N_0}\uppsi_{\near,\far}$ which allows us to absorb the last term on the first line above, and get
%%%%%%%
%%%%%%%
\begin{align}\label{eq:LEDnearcircletemp1}
\begin{split}
\|P_{\leq N_0}\uppsi_{\near,\far}\|_{LE(I)}&\lesssim \|f\|_{LE^\ast(I)}+\|\uppsi\|_{L^\infty_\tau E(I)}+\epsilon\|\uppsi_{\near}\|_{LE}\\
&\quad+\epsilon\|\chi_{\tilK^c}r^{\frac{1+\alpha}{2}}(\partial_r+\frac{n-1}{2r})\uppsi_\near\|_{L^\infty_\tau L^2_y}.
\end{split}
\end{align}
It remains to control $\|\chi_{\tilK^c} r^{\frac{1+\alpha}{2}}(\partial_r+\frac{n-1}{2r})\uppsi_\near\|_{L^\infty_\tau L^2_y}$. This is achieved by the same argument as we will later use to prove $r^p$-energy estimates in the exterior. The more complete version of the argument is worked out in Lemma~\ref{lem:rpmult1} below, which holds independently of the results in this section. Without repeating the details, multiplying the equation $\calP\uppsi_\near=-V_\far\uppsi_\far$ by $\chi_{\tilK^c}r^p(\partial_r+\frac{n-1}{2r})\uppsi_\near$ (alternatively, by $\chi_{\tilK^c}L(\tilr^{\frac{n-1}{2}}\uppsi_\near)$), with $1+\alpha\leq p \leq 2$, and a few integration by parts yield the estimate (recall that $V_\far$ is compactly supported)
%%%%%%%
%%%%%%%
\begin{align*}
\begin{split}
\|\chi_{\tilK^c}r^{p/2}(\partial_r+\frac{n-1}{2r})\uppsi_\near\|_{L^\infty_\tau L^2_y}&\lesssim \|\uppsi_\near\|_{LE}+ \|\uppsi_\near\|_{L^\infty_\tau E}\\
&\quad+|\angles{V_\far \uppsi_\far}{\chi_{\tilK^c} r^{p}(\partial_r+\frac{n-1}{2r})\uppsi_\near}_{L^2_{\tau,y}}|\\
&\lesssim \|\uppsi_\near\|_{LE}+ \|\uppsi_\near\|_{L^\infty_\tau E}+\|\uppsi_\far\|_{LE}.
\end{split}
\end{align*}
Plugging this back into \eqref{eq:LEDnearcircletemp1} completes the proof of the lemma.
\end{proof}
%%%%%%%%%%%
%%%%%%%%%%%

%%%%%%%%%%%%%%%%%%%%%%
%%%%%%%%%%%%%%%%%%%%%%
%%%%%%%%%%%%%%%%%%%%%%
\section{Exterior}\label{sec:exterior}
%%%%%%%%%%%%%%%%%%%%%%
%%%%%%%%%%%%%%%%%%%%%%
%%%%%%%%%%%%%%%%%%%%%%
%%%%%%%%%%%%%%
%%%%%%%%%%%%%%
This section contains the proof of the decay estimates on $\phi$. We will use the $r^p$ weighted vectorfield method to derive decay estimates for the energy of $\phi$ and improved decay for the energy of $T^k\phi$, $k=1,2$. Then using elliptic and interpolation estimates we use the decay of the energies to deduce pointwise bounds, and in particular complete the proof of Proposition~\ref{prop:bootstrapphi1}. We will start by deriving some commutator formulas and expressing the equation in terms of the vectofields $L$, $\Lbar$, $\Omega$. 

%%%%%%%%%%%%%%%
%%%%%%%%%%%%%%%
%%%%%%%%%%%%%%%
\subsection{Frame Decomposition of the Operator and Commutation Relations}
%%%%%%%%%%%%%%%
%%%%%%%%%%%%%%%
%%%%%%%%%%%%%%%

We use the relations derived in Section~\ref{sec:profile2} to calculate the commutators among the vectorfields $T$, $L$, $\Lbar$, $\Omega$. The calculations in this subsection are valid in the hyperboloidal region $\calC_\hyp$, where these vectorfields are defined.
%%%%%%%%%
%%%%%%%%%
\begin{lemma}\label{lem:commutators1}
%%%%%%%
%%%%%%%
The following commutation relations hold among the vectorfields $L,\Lbar,\Omega$:
\begin{align}\label{eq:VFcomm1}
\begin{split}
&[\Lbar,L]= \callO(\dotwp^{\leq 2})L+\callO(\dotwp^{\leq 2}r^{-2})\Lbar+\callO(\dotwp^{\leq 2} r^{-2})\Omega,\\
&[\Lbar,\Omega]= \callO(\dotwp)\Omega+\callO(\dotwp r) L+\callO(\dotwp)\Lbar,\\
&[L,\Omega]=\callO(\dotwp r^{-1})L+\callO(\dotwp r^{-2})\Omega+\callO(\dotwp r^{-2})\Lbar,\\
&[\Omega_{ij},\Omega_{k\ell}]=\delta_{[ki}\Omega_{\ell j]},\\
&[T,L]=-[T,\Lbar]= \callO(\dotwp^{\leq 2})L+\callO(\dotwp^{\leq 2}r^{-2})\Lbar+\callO(\dotwp^{\leq 2} r^{-2})\Omega,\\
&[T,\Omega]=\callO(\dotwp)\Omega+\callO(\dotwp r) L+\callO(\dotwp)\Lbar.\\
\end{split}
\end{align}
\end{lemma}
%%%%%%%%%%
%%%%%%%%%%
\begin{proof}
Starting with $[\Lbar,L]$, note that since $\Lbar = 2T-L$ this commutator is the same as $2[T,L]$. The desired structure then follows from the relations \eqref{eq:VFexpansion1}, \eqref{eq:Tprecise1}, and \eqref{eq:VFbasis1}. For $[\Omega_{ij},\Omega_{k\ell}]$ the desired relation follows from the fact that $\Omega_{ij}=\Lambda_\ell \tilOmega_{ij}$, where $\tilOmega_{ij}=y^i\partial_{y^j}-y^j\partial_{y^i}$ are tangential to the reference hyperboloids $\calH_\sigma$. This tangentiality implies that $[\Omega_{ij},\Omega_{k\ell}]=\Lambda_\ell[\tilOmega_{ij},\tilOmega_{k\ell}]$. By the same reasoning, to compute the commutator $[L,\Omega]$ we decompose $\tilL=\partial_{y^0}+\frac{y^i}{|y'|}\partial_{y^i}$ as (recall that $T= \Lambda_\ell \partial_{y^0}$)
%%%%%%%
%%%%%%%
\begin{align*}
\begin{split}
\tilL = (1-\tilc)\partial_{y^0}+\frac{y^i}{|y'|}\partial_{y^i}+\tilc\, \partial_{y^0},
\end{split}
\end{align*}
where $\tilc$ is chosen so that $\tilL_\temp:=\tilL-\tilc\,\partial_{y^0}$ is tangential to $\calH_\sigma$. Let $c=\tilc(y(\tau,r,\theta))$ where $y$ is as in \eqref{eq:yx1}. Since $\tilOmega$ is tangential to $\calH_\sigma$, it follows that $\Omega( c(y(\tau,r,\theta))) = (\tilOmega\tilc)(y(\tau,r,\theta))$, and therefore
%%%%%%%
%%%%%%%
\begin{align*}
\begin{split}
[L,\Omega] &= \Lambda_\ell[\tilL_\temp,\tilOmega]+[cT,\Omega]=\Lambda_\ell[\tilL_\temp,\tilOmega]+(\Omega c) T+c[T,\Omega]\\
&=\Lambda_\ell([\tilL_\temp,\tilOmega]+(\tilOmega \tilc) \partial_{y^0})+c[T,\Omega]=\Lambda_\ell[\tilL,\tilOmega]+c[T,\Omega]\\
&=c[T,\Omega].
\end{split}
\end{align*}
To find $\tilc$, recall that the defining equation for $\sigma=\sigma(\tau)$ is $(y^0-\gamma^{-1}\tau)^2-|y'|^2=1$, and therefore $\tilc$ must be such that $\tilL_\temp$ is Minkowski perpendicular to $(y^0-\gamma^{-1}\tau,y')$. It follows that
%%%%%%%
%%%%%%%
\begin{align*}
\begin{split}
\tilc = 1-\frac{|y'|}{y^0-\gamma^{-1}\tau}=\callO(\tilr^{-2}),
\end{split}
\end{align*}
where the last estimate follows from \eqref{eq:yx1}. The last two observations together give $[L,\Omega]=\callO(r^{-2})[T,\Omega]$, and the desired expansion follows from \eqref{eq:VFexpansion1}, \eqref{eq:Tprecise1}, and \eqref{eq:VFbasis1}. Finally, the expansions for $[\Lbar,\Omega]$, $[T,L]$, $[T,\Lbar]$, and $[T,\Omega]$ follow from the previous ones and the observation that $\Lbar=2T-L$.
\end{proof}
%%%%%%%%%%
%%%%%%%%%%
With these preparations we can turn to the calculation of the wave operator $\Box_m$ in terms of $\{L, \Lbar, \Omega\}$. 
%%%%%%%%%%
%%%%%%%%%%
\begin{lemma}\label{lem:BoxVF1}
For any function $\uppsi$
%%%%%%%
%%%%%%%
\begin{align}\label{eq:BoxVF1}
\begin{split}
\Box_m\uppsi&= -\Lbar L \uppsi -\frac{n-1}{2\tilr}(\Lbar-L)\uppsi+\frac{1}{\tilr^2}\sum_{\Omega}\Omega^2 \uppsi\\
&\quad+\callO(\dotwp^{\leq 2})L\uppsi+\callO(\dotwp^{\leq 2} r^{-2})\Lbar \uppsi+\callO(\dotwp^{\leq 2} r^{-2})\Omega \uppsi,
\end{split}
\end{align}
and, with $\tilupsi:=\tilr^{\frac{n-1}{2}}\uppsi$,
%%%%%%%
%%%%%%%
\begin{align}\label{eq:BoxVF2}
\begin{split}
\tilr^{\frac{n-1}{2}}\Box_m\uppsi&= -\Lbar L \tilupsi+\frac{1}{\tilr^2}\sum_{\Omega}\Omega^2\tilupsi-\frac{(n-1)(n-3)}{4\tilr^2}\tilupsi\\
&\quad+\callO(\dotwp^{\leq 2})L\tilupsi+\callO(\dotwp^{\leq 2} r^{-2})\Lbar\tilupsi+\callO(\dotwp^{\leq 2} r^{-2})\Omega\tilupsi +\callO(\dotwp^{\leq 2}r^{-2})\tilupsi.
\end{split}
\end{align}
\end{lemma}
%%%%%%%%%%
%%%%%%%%%%
\begin{proof}
We work with an orthonormal frame $\{e_I\}=\{\tilL, L, e_A;~a=1,2\}$, where $e_A=\Lambda_\ell \tile_A$, and $\{\tile_1,\tile_2\}$ is a local orthonormal frame for the reference spheres on $\calH_\sigma$. With $\Gamma_{IJ}^K$ the corresponding connection coefficients, that is $\nabla_{e_I}e_J=\Gamma_{IJ}^Ke_K$, and $D_I$ denoting scalar differentiation along $e_I$, the wave operator can be written as
%%%%%%%
%%%%%%%
\begin{align}\label{eq:BoxVFtemp1}
\begin{split}
\Box_m \uppsi&=(m^{-1})^{IJ} (D_I D_J-\Gamma_{IJ}^KD_K)\uppsi\\
&=-\frac{1}{2}\Lbar L  \uppsi-\frac{1}{2}L\Lbar \uppsi + \sum_A D_A^2 \uppsi-(m^{-1})^{IJ}(\Gamma^L_{IJ}L\uppsi+\Gamma^\Lbar_{IJ}\Lbar \uppsi+\Gamma^A_{IJ}D_A\uppsi )\\
&=-\Lbar L  \uppsi-\frac{1}{2}[L,\Lbar] \uppsi + \frac{1}{\tilr^2}\sum_\Omega \Omega^2 \uppsi-(m^{-1})^{IJ}(\Gamma^L_{IJ}L\uppsi+\Gamma^\Lbar_{IJ}\Lbar \uppsi+\Gamma^A_{IJ}D_A\uppsi ).
\end{split}
\end{align}
Here to pass to the last line we have used the fact that $e_A$ are tangential to the foliation, so $\sum_A D_A^2\uppsi(y(\tau,r,\theta))= \sum_A (\tile_A^2\uppsi)(y(\tau,r,\theta))$. To calculate the connection coefficients we use Koszul's formula, which for an orthonormal frame reads
%%%%%%%
%%%%%%%
\begin{align*}
\begin{split}
\Gamma_{IJ}^M= \frac{1}{2}(m^{-1})^{KM}(m([e_I,e_J],e_K)-m([e_I,e_K],e_J)-m([e_J,e_K],e_I)).
\end{split}
\end{align*}
These can now be computed using Lemma~\ref{lem:commutators1}. Here note that the commutators with $e_A$ can be calculated in the same way as in the proof of Lemma~\ref{lem:commutators1}, by noting that $e_A$ is a linear combination of $\Omega_{ij}$ with coefficients of size $r^{-1}$. In particular, with $\tilL,\tilLbar$ such that $L=\Lambda_\ell \tilL$ and $\Lbar=\Lambda_\ell \tilLbar$ (see~\eqref{eq:VFdef0}), 
%%%%%%%
%%%%%%%
\begin{align*}
\begin{split}
&[\Lbar,e_A]= \Lambda_\ell[\tilLbar,\tile_A]+\callO(\dotwp r^{-1})\Omega+\callO(\dotwp ) L+\callO(\dotwp r^{-1})\Lbar,\\
&[L,e_A]=\Lambda_\ell[\tilL,\tile_A]+\callO(\dotwp r^{-2})L+\callO(\dotwp r^{-3})\Omega+\callO(\dotwp r^{-3})\Lbar,\\
&[e_A,e_B]=\Lambda_\ell[\tile_A,\tile_B].
\end{split}
\end{align*}
It follows from this and Lemma~\ref{lem:commutators1} that
%%%%%%%
%%%%%%%
\begin{align*}
\begin{split}
&\Gamma^L_{\Lbar L}=-\frac{1}{2}m([\Lbar,L],\Lbar)=\callO(\dotwp^{\leq 2}),\\
&\Gamma^L_{L\Lbar}=-\frac{1}{4}(m([L,\Lbar],\Lbar)-m([L,\Lbar],\Lbar))=0,\\
&\Gamma^{L}_{AA}=\frac{1}{2}m([e_A,\Lbar],e_A).
\end{split}
\end{align*}
For the last term, for the purpose of deriving \eqref{eq:BoxVF1} it suffices to observe that $m([e_A,\Lbar],e_A)=m([\tile_A,\tilLbar],\tile_A)+\callO(\dotwp)$. But, for \eqref{eq:BoxVF2} we will need the better estimate
%%%%%%%
%%%%%%%
\begin{align}\label{eq:GammaLAAtemp1}
\begin{split}
\Gamma^{L}_{AA}= \frac{1}{2}m([\tile_A,\tilLbar],\tile_A)+\callO(\dotwp r^{-1}).
\end{split}
\end{align}
To prove \eqref{eq:GammaLAAtemp1}, using our usual notation as in the proof of Lemma~\ref{lem:commutators1}, note that 
%%%%%%%
%%%%%%%
\begin{align*}
\begin{split}
\frac{1}{2}m([e_A,\Lbar],e_A)=m([e_A,T],e_A)-\frac{1}{2}m([e_A,L],e_A)=m([e_A,T],e_A)+\frac{1}{2}m([\tile_A,\tilLbar],\tile_A)+\callO(\dotwp r^{-2}).
\end{split}
\end{align*}
For the first term on the right, we write $[e_A,T]=\big(e_A(T^\mu)-T(e_A^\mu)\big)\partial_\mu$. On the other hand, in view of the expansions \eqref{eq:VFbasis1}, we have $m(\partial_r,e_A)=\callO(r^{-2})$ and $m(\partial_\tau,e_A)=0$. It then follows from~\eqref{eq:Tprecise1} and the expansion $e_A=\callO(r^{-1})\partial_a+\callO(1)\partial_r$ that
%%%%%%%
%%%%%%%
\begin{align*}
\begin{split}
m((e_AT^\mu)\partial_\mu,e_A)=\callO(\dotwp r^{-1}),
\end{split}
\end{align*}
and
%%%%%%%
%%%%%%%
\begin{align*}
\begin{split}
m(T(e_A^\mu)\partial_\mu,e_A)=\callO(\dotwp r^{-1})+\callO(1)m((\partial_\tau e_A^\mu)\partial_\mu,e_A).
\end{split}
\end{align*}
For the last term observe that by \eqref{eq:extchangeofvars1}, the metric components $m_{rr}=1-\frac{r}{\jap{r}}$, $m_{ra}=0$, and $m_{ab}=r^2\ringsg_{ab}$ are independent of $\tau$, so since $m(e_A,e_A)=1$ and $e_A=e_A^r\partial_r+e_A^a\partial_a$,
%%%%%%%
%%%%%%%
\begin{align*}
\begin{split}
m((\partial_\tau e_A^\mu)\partial_\mu,e_A)=m_{\mu\nu}e_A^\nu\partial_\tau e_A^\mu=\frac{1}{2}\partial_\tau (m_{\mu\nu}e_A^\mu e_A^\nu)=0.
\end{split}
\end{align*}
This completes the proof of \eqref{eq:GammaLAAtemp1}. Returning to the other connection coefficients, by the Koszul formula,
%%%%%%%
%%%%%%%
\begin{align*}
\begin{split}
&\Gamma_{\Lbar L}^\Lbar=0,\quad \Gamma_{L\Lbar}^\Lbar=\frac{1}{2}m([\Lbar,L],L)=\callO(\dotwp^{\leq 2} r^{-2}),\\
&\Gamma_{AA}^\Lbar=\frac{1}{2}m([e_A,L],e_A)=\frac{1}{2}m([\tile_A,\tilL],\tile_A)+\callO(\dotwp r^{-2}),
\end{split}
\end{align*}
and
%%%%%%%
%%%%%%%
\begin{align*}
\begin{split}
&\Gamma^B_{L\Lbar}=\frac{1}{2}(m([L,\Lbar],e_B)-m([L,e_B],\Lbar)-m([\Lbar,e_B],L))=\callO(\dotwp r^{-1})\\
&\Gamma^B_{L\Lbar}=\frac{1}{2}(m([\Lbar,L],e_B)-m([\Lbar,e_B],L)-m([L,e_B],\Lbar))=\callO(\dotwp r^{-1}),\\
&\Gamma^B_{AA}=-m([\tile_A,\tile_B],\tile_A).
\end{split}
\end{align*}
Inserting the expressions we have derived for the connection coefficients into \eqref{eq:BoxVFtemp1} and using the relations \eqref{eq:VFcomm1} gives
%%%%%%%
%%%%%%%
\begin{align}\label{eq:BoxVFtemp2}
\begin{split}
\Box_m\uppsi&= -\Lbar L \uppsi -\frac{n-1}{2\tilr}(\Lbar-L)\uppsi+\frac{1}{\tilr^2}\sum_{\Omega}\Omega^2 \uppsi-[L,\Lbar]^LL\uppsi\\
&\quad+\callO(\dotwp^{\leq 2}r^{-1})L\uppsi+\callO(\dotwp^{\leq 2} r^{-2})\Lbar \uppsi+\callO(\dotwp^{\leq 2} r^{-2})\Omega \uppsi,
\end{split}
\end{align}
The expansion \eqref{eq:BoxVF1} follows from \eqref{eq:BoxVFtemp2} and the fact that, by \eqref{eq:VFcomm1}, $[L,\Lbar]^L=\callO(\dotwp^{\leq2})$, but we will need a more precise expression for this commutator to derive \eqref{eq:BoxVF2}. To  prove \eqref{eq:BoxVF2} first note that, with the notation $\tilupsi=\tilr^{\frac{n-1}{2}}\uppsi$,
%%%%%%%
%%%%%%%
\begin{align*}
\begin{split}
\tilr^{\frac{n-1}{2}}\Box \uppsi=\Box\tilupsi-\Box(\tilr^{\frac{n-1}{2}})\uppsi - 2 m^{IJ} (e_I \tilr^{\frac{n-1}{2}})(e_J\uppsi)=:I+II+III.
\end{split}
\end{align*}
In expanding the terms $I$, $II$, $III$ we use the notation $\Err$ to denote error terms which are acceptable on the right-hand side of \eqref{eq:BoxVF2}. Starting with $I$, by \eqref{eq:BoxVF1},
%%%%%%%
%%%%%%%
\begin{align*}
\begin{split}
I=-\Lbar L\tilupsi+\frac{1}{\tilr^2}\sum\Omega^2\tilupsi-\frac{n-1}{2\tilr}\tilr^{\frac{n-1}{2}}(\Lbar-L)\uppsi-\frac{n-1}{2\tilr}\uppsi(\Lbar-L)\tilr^{\frac{n-1}{2}}+\Err.
\end{split}
\end{align*}
For $II$ we use \eqref{eq:VFcomm1}, the more precise expression \eqref{eq:BoxVFtemp2} for $\Box_m$ (applied to $\tilr^{\frac{n-1}{2}}$), and the usual decomposition $L=L_\temp+\callO(r^{-2})T$, with $L_\temp\tilr=1$, to write
%%%%%%%
%%%%%%%
\begin{align*}
\begin{split}
II&=\frac{n-1}{2}\uppsi\Lbar \tilr^{\frac{n-1}{2}-1}+\frac{n-1}{2\tilr} \uppsi(\Lbar-L)\tilr^{\frac{n-1}{2}}+\Err\\
&=\frac{n-1}{2\tilr}\uppsi(\Lbar-L)\tilr^{\frac{n-1}{2}}-\frac{(n-1)(n-3)}{4\tilr^2}\tilupsi+\frac{n-1}{2\tilr} \uppsi(\Lbar+\frac{n-1}{2\tilr})\tilr^{\frac{n-1}{2}}\\
&\quad+\frac{n-1}{2}\tilupsi(\Lbar-\frac{1}{\tilr})\tilr^{-1}+\uppsi[L,\Lbar]^LL\tilr^{\frac{n-1}{2}}+\Err.
\end{split}
\end{align*}
To treat the last line as an error, we need to use the more precise expansions \eqref{eq:Lprecise1} and \eqref{eq:tilrprecise1} to get (note that each term by itself is only $\callO(\dotwp^{\leq 2}r^{-1})\tilupsi$)
%%%%%%%
%%%%%%%
\begin{align*}
\begin{split}
&\frac{n-1}{2}\tilupsi(\Lbar-\frac{1}{\tilr})\tilr^{-1}+\uppsi[L,\Lbar]^LL\tilr^{\frac{n-1}{2}}\\
&=\frac{n-1}{2\tilr^2}(L\tilr -1)\tilupsi+\frac{n-1}{\tilr}[L,T]^L(L\tilr-1)\tilupsi+\frac{n-1}{\tilr}([L,T]^L-\tilr^{-1}T\tilr)\tilupsi=\callO(\dotwp^{\leq 2}r^{-2})\tilupsi,
\end{split}
\end{align*}
and hence
%%%%%%%
%%%%%%%
\begin{align*}
\begin{split}
II=\frac{n-1}{2\tilr}\uppsi(\Lbar-L)\tilr^{\frac{n-1}{2}}-\frac{(n-1)(n-3)}{4\tilr^2}\tilupsi+\frac{n-1}{2\tilr}\uppsi(\Lbar+\frac{n-1}{2\tilr})\tilr^{\frac{n-1}{2}}+\Err.
\end{split}
\end{align*}
For the term $III$, since $e_A\tilr=0$,
%%%%%%%
%%%%%%%
\begin{align*}
\begin{split}
III&=(L\tilr^{\frac{n-1}{2}})\Lbar \uppsi+(\Lbar \tilr^{\frac{n-1}{2}})L\uppsi\\
&=\frac{n-1}{2\tilr}\tilr^{\frac{n-1}{2}}(\Lbar-L)\uppsi+(L\uppsi)(\Lbar+\frac{n-1}{2\tilr})\tilr^{\frac{n-1}{2}}+\Err.
\end{split}
\end{align*}
Equation \eqref{eq:BoxVF2} now follows by adding the expansions for $I$, $II$, and $III$, and using the observation that $\tilr^{\frac{n-1}{2}}(L+\frac{n-1}{2\tilr})\uppsi=L\tilupsi+\Err$.
\end{proof}
%%%%%%%%%%
%%%%%%%%%%
Lemma~\ref{lem:BoxVF1} and equations \eqref{eq:callP2} and \eqref{eq:ErrcallP1} yield the following representation for $\calP_\graph \uppsi$:
%%%%%%%
%%%%%%%
\begin{align}
\tilr^{\frac{n-1}{2}}\calP_\graph \uppsi&= -(1+\callO(r^{-4}))\Lbar L\tilupsi+\frac{1+\callO(r^{-5})}{\tilr^2}\sum\Omega^2\tilupsi-\frac{(n-1)(n-3)}{4\tilr^2}\tilupsi\nonumber\\
&\quad+\callO(r^{-4})L^2\tilupsi+\callO(r^{-4})\Lbar^2\tilupsi+\callO(r^{-5})\Omega L\tilupsi+\callO(r^{-5})\Lbar\Omega\tilupsi \label{eq:calPVF1}\\
&\quad+\callO(\dotwp+r^{-5})L\tilupsi+\callO(\dotwp r^{-2}+r^{-5})\Lbar\tilupsi+\callO(\dotwp r^{-2}+r^{-6})\Omega\tilupsi+\callO(\dotwp r^{-2}+r^{-6})\tilupsi.\nonumber
\end{align}
This is the representation we will use to derive multiplier identities in the exterior. Before starting on these multiplier identities, we calculate the equations satisfied by higher order derivatives of $\tilupsi$. Our goal is to calculate the analogous equation to \eqref{eq:calPVF1} satisfied by $\tilr L$, $T$, and $\Omega$ applied to $\tilupsi$. Since, due of the presence of parameters, these vectorfields do not commute, we start by deriving an estimate for the commutator of a string of them, valid in the hyperboloidal region where they are defined.
%%%%%%%%%%
%%%%%%%%%%
\begin{lemma}\label{lem:commutators2}
If $X_1,\dots,X_k\in\{\tilr L,\Omega,T\}$ are $k$ vectorfields with $k_1$ factors of $\tilr L$, $k_2$ factors of $\Omega$ and $k_3$ factors of $T$, then for any function $\uppsi$,
%%%%%%%
%%%%%%%
\begin{align*}
\begin{split}
X_k\dots X_1\uppsi= (\tilr L)^{k_1}\Omega^{k_2} T^{k_3}\uppsi+\sum_{j_1+j_2+j_3\leq k-1}\callO(\dotwp^{\leq 2k-2(j_1+j_2+j_3)})(\tilr L)^{j_1}\Omega^{j_2}T^{j_3}\uppsi.
\end{split}
\end{align*}
\end{lemma}
%%%%%%%%%%
%%%%%%%%%%
\begin{proof}
The proof is by induction on $k$. For $k=2$ the statement follows from Lemma~\ref{lem:commutators1}. For the induction step, suppose there are $\tilk_1$, $\tilk_2$, and $\tilk_3$ factors of $\tilr L$, $\Omega$, and $T$, respectively, among $X_1,\dots,X_{k-1}$. Then
%%%%%%%
%%%%%%%
\begin{align*}
\begin{split}
X_k\dots X_1\uppsi=X_k(\tilr L)^{k_1}\Omega^k_2 T^{k_3}\uppsi+X_k\sum_{j_1+j_2+j_3\leq k-2}\callO(\dotwp^{\leq 2k-2-2(j_1+j_2+j_3)})(\tilr L)^{j_1}\Omega^{j_2}T^{j_3}\uppsi=I+II.
\end{split}
\end{align*}
The term $II$ can be put in the desired form using the induction hypothesis. For the first term, if $X_k=\tilr L$, if $X_k=\Omega$ and $\tilk_1=0$, or if $X_k=T$ and $\tilk_1=\tilk_2=0$, then this is already of the desired form. If $X_k=\Omega$ and $\tilk_1\neq0$, then by Lemma~\ref{lem:commutators1} the first term is 
%%%%%%%
%%%%%%%
\begin{align*}
\begin{split}
(\tilr L) \Omega (\tilr L)^{\tilk_1-1}\Omega^{\tilk_2} T^{\tilk_3}\uppsi+(\callO(\dotwp r^{-1})\tilr L +\callO(\dotwp r^{-2})\Omega+\callO(\dotwp r^{-2})T)(\tilr L)^{\tilk_1-1}\Omega^{\tilk_2} T^{\tilk_3}\uppsi.
\end{split}
\end{align*}
The second term can again be put in the desired form by the induction hypothesis. Similarly by the induction hypothesis we can rearrange the first $k-1$ derivatives in the first term to put this in the desired form. If $X_k=T$ and $\tilk_1=0$ but $\tilk_2\neq 0$ then by Lemma~\ref{lem:commutators1} we can write $I$ as
%%%%%%%
%%%%%%%
\begin{align*}
\begin{split}
\Omega T\Omega^{\tilk_2-1}T^{\tilk_1}+(\callO(\dotwp)\tilr L + \callO(\dotwp)\Omega+\callO(\dotwp)T)\Omega^{\tilk_2-1}T^{\tilk_1}\uppsi,
\end{split}
\end{align*}
which, by the induction hypothesis, can be arranged into the desired form, using the same argument as above. The case where $X_k=T$ and $\tilk_1\neq0$ is similar.
\end{proof}
%%%%%%%%%%
%%%%%%%%%%
In view of Lemma~\ref{lem:commutators2}, in order to estimate $X_k\dots X_1\uppsi$, with $X_i$ as in the lemma, it suffices to consider only the rearrangement $(\tilr L)^{k_1}\Omega^{k_2}T^{k_3}\uppsi$, so it suffices to consider commutators with \eqref{eq:calPVF1} in this order. This commutator is calculated in the next lemma. For this, we let $\tilcalP$ denote the non-perturbative part of the operator on the right-hand side of \eqref{eq:calPVF1}, that is,
%%%%%%%
%%%%%%%
\begin{align*}
\begin{split}
\tilcalP:=-\Lbar L+\frac{1}{\tilr^2}\sum\Omega^2-\frac{(n-1)(n-3)}{4\tilr^2}.
\end{split}
\end{align*}
%%%%%%%%%%
%%%%%%%%%%
\begin{lemma}\label{lem:higherordereqs1}
For any function $\uppsi$ and any integers $k_1,k_2,k_3\geq0$, and with $\tilupsi=\tilr^{\frac{n-1}{2}}\uppsi$ and $k=k_1+k_2+k_3$,
%%%%%%%%
%%%%%%%%
%%%%%%%
%%%%%%%
\begin{equation}\label{eq:higher1}
(\tilr L+1)^{k_1}\Omega^{k_2}T^{k_3}(\tilr^{\frac{n-1}{2}}\calP_\graph \uppsi)=\tilcalP((\tilr L)^{k_1}\Omega^{k_2}T^{k_3}\tilupsi)+\Err_{k_1,k_2,k_3}[\tilupsi],\\
\end{equation}
where
%%%%%%%
%%%%%%%
\begin{align}
\Err_{k_1,k_2,k_3}[\tilupsi]&=\sum_{j=0}^{k_1-1}c_{j,k_1}\tilcalP_1((\tilr L)^{j}\Omega^{k_2}T^{k_3}\tilupsi)+\callO(r^{-4})L^2(\tilr L)^{k_1}\Omega^{k_2}T^{k_3}\tilupsi+\callO(r^{-4})\Lbar^2(\tilr L)^{k_1}\Omega^{k_2}T^{k_3}\tilupsi\nonumber\\
&\quad+\callO(r^{-5})\Omega L(\tilr L)^{k_1}\Omega^{k_2}T^{k_3}\tilupsi+\callO(r^{-5})\Lbar\Omega(\tilr L)^{k_1}\Omega^{k_2}T^{k_3}\tilupsi\nonumber\\
&\quad+\callO(\dotwp^{\leq 2k+2}+r^{-5})\sum_{j_1+j_2+j_3\leq k}L(\tilr L)^{j_1}\Omega^{j_2}T^{j_3}\tilupsi\nonumber\\
&\quad+\callO(\dotwp^{\leq 2k+2} r^{-2}+r^{-4})\sum_{j_1+j_2+j_3\leq k}\Lbar(\tilr L)^{j_1}\Omega^{j_2}T^{j_3}\tilupsi\label{eq:errorhigher1}\\
&\quad+\callO(\dotwp^{\leq 2k+2} r^{-2}+r^{-6})\sum_{j_1+j_2+j_3\leq k}\Omega(\tilr L)^{j_1}\Omega^{j_2}T^{j_3}\tilupsi\nonumber\\
&\quad+\callO(\dotwp^{\leq 2k+2} r^{-2}+r^{-5})\sum_{j_1+j_2+j_3\leq k}(\tilr L)^{j_1}\Omega^{j_2}T^{j_3}\tilupsi,\nonumber
\end{align}
for some constants $c_{j,k_1}$ (which are nonzero only if $k_1\geq 1$), with $c_{k_1-1,k_1}=-k_1$, and where
%%%%%%%
%%%%%%%
\begin{align*}
\begin{split}
\tilcalP_1=LL+\frac{1}{\tilr^2}\sum\Omega^2-\frac{(n-1)(n-3)}{4\tilr^2}.
\end{split}
\end{align*}
\end{lemma}
%%%%%%%%%%
%%%%%%%%%%
\begin{proof}
Note that the terms involving $\tilcalP_1$ are what would come up by commuting $(\tilr L+1)^{k_1}$ if the parameters were treated as fixed. The proof is by induction on $k$, applied to every term on the right-hand side of \eqref{eq:calPVF1}. The treatment of the different terms is similar, so here we present the details only for $\Lbar L\tilupsi$. Starting with $T$, by Lemma~\ref{lem:commutators1}, and recalling that $\Lbar=2T-L$,
%%%%%%%
%%%%%%%
\begin{align*}
\begin{split}
T\Lbar L\tilupsi&= \Lbar L T\tilupsi+[T,\Lbar]L\tilupsi+\Lbar[T,L]\tilupsi\\
&=\Lbar L T\tilupsi+\callO(\dotwp^{\leq2})L L\tilupsi + \callO(\dotwp^{\leq 2})TL\tilupsi+\callO(\dotwp^{\leq2})\Lbar T\tilupsi+\callO(\dotwp^{\leq2}r^{-2})\Lbar\Omega\tilupsi+\callO(\dotwp^{\leq2}r^{-2})\Omega L\tilupsi\\
&\quad+\callO(\dotwp^{\leq3})L\tilupsi+\callO(\dotwp^{\leq3}r^{-2})\Lbar\tilupsi+\callO(\dotwp^{\leq3}r^{-2})\Omega\tilupsi\\
&=\Lbar L T\tilupsi+\callO(\dotwp^{\leq2})L L\tilupsi + \callO(\dotwp^{\leq 2})LT\tilupsi+\callO(\dotwp^{\leq2})\Lbar T\tilupsi+\callO(\dotwp^{\leq2}r^{-2})\Lbar\Omega\tilupsi+\callO(\dotwp^{\leq2}r^{-2})\Omega L\tilupsi\\
&\quad+\callO(\dotwp^{\leq3})L\tilupsi+\callO(\dotwp^{\leq3}r^{-2})\Lbar\tilupsi+\callO(\dotwp^{\leq3}r^{-2})\Omega\tilupsi,
\end{split}
\end{align*}
which has the desired structure. We can then inductively apply this same identity together with Lemmas~\ref{lem:commutators1} and~\ref{lem:commutators2} to conclude that
%%%%%%%
%%%%%%%
\begin{align}\label{eq:hotemp1}
\begin{split}
T^{k_3}\Lbar L\tilupsi=\Lbar LT^{k_3}\tilupsi+\Err,
\end{split}
\end{align}
where $\Err$ has the structure given in the statement of the lemma. Next we apply $\Omega$ to \eqref{eq:hotemp1}. Note that $\Omega$ applied to $\Err$ in \eqref{eq:hotemp1} has the desired structure by Lemmas~\ref{lem:commutators1} and~\ref{lem:commutators2}. For the main term, again by Lemma~\ref{lem:commutators1}, and with $\Psi=T^k\tilupsi$,
%%%%%%%
%%%%%%%
\begin{align*}
\begin{split}
\Omega\Lbar L \Psi&= \Lbar L \Omega\Psi+[\Omega,\Lbar]L\Psi+\Lbar[\Omega,L]\Psi\\
&=\Lbar L \Omega\Psi+\callO(\dotwp)\tilr L L\Psi+\callO(\dotwp)\Omega L\Psi+\callO(\dotwp)TL\Psi+\callO(\dotwp r^{-2})\Lbar\Omega\Psi+\callO(\dotwp r^{-2})\Lbar T\Psi\\
&\quad+\callO(\dotwp^{\leq2}r^{-1})L\Psi+\callO(\dotwp^{\leq2}r^{-2})\Omega\Psi+\callO(\dotwp^{\leq 2}r^{-2})\Lbar\Psi\\
&=\Lbar L \Omega\Psi+\callO(\dotwp)L(\tilr L)\Psi+\callO(\dotwp) L\Omega\Psi+\callO(\dotwp)LT\Psi+\callO(\dotwp r^{-2})\Lbar\Omega\Psi+\callO(\dotwp r^{-2})\Lbar T\Psi\\
&\quad+\callO(\dotwp^{\leq3})L\Psi+\callO(\dotwp^{\leq3}r^{-2})\Omega\Psi+\callO(\dotwp^{\leq 3}r^{-2})\Lbar\Psi,
\end{split}
\end{align*}
which has the desired structure. Here we have used the fact that $\tilr L^2\Psi= L(\tilr L)\Psi+\callO(1+\dotwp )L\Psi$. As for $T^{k_3}$ we can apply this identity inductively and use Lemmas~\ref{lem:commutators1} and~\ref{lem:commutators2} to conclude that
%%%%%%%
%%%%%%%
\begin{align}\label{eq:hotemp2}
\begin{split}
\Omega^{k_2}T^{k_3}\Lbar L\tilupsi=\Lbar L\Omega^{k_2}T^{k_3}\tilupsi+\Err,
\end{split}
\end{align}
where $\Err$ has the structure given in the statement of the lemma. Finally we apply $\tilr L+1$ to \eqref{eq:hotemp2}. Again by Lemmas~\ref{lem:commutators1} and~\ref{lem:commutators2} the contribution of $\Err$ in \eqref{eq:hotemp2} has the desired form. For the main term we have, using Lemma~\ref{lem:commutators1} and with $\Psi=\Omega^{k_2}T^{k_1}\tilupsi$,
%%%%%%%
%%%%%%%
\begin{align*}
\begin{split}
(\tilr L+1)\Lbar L \Psi&= \Lbar L(\tilr L+1)\Psi+\tilr[L,\Lbar]L\Psi+[\tilr,\Lbar]L^2\Psi+\Lbar([\tilr,L]L\Psi)\\
&=\Lbar L(\tilr L+1)\Psi-\Lbar L\Psi+LL\Psi+\callO(\dotwp^{\leq2})\tilr L L\Psi+\callO(\dotwp r^{-1})TL\Psi\\
&\quad+\callO(\dotwp^{\leq2}r^{-1})\Omega L\Psi+\callO(\dotwp r^{-2})L\Psi\\
&=\Lbar L(\tilr L+1)\Psi-\Lbar L\Psi+LL\Psi+\callO(\dotwp^{\leq2})L (\tilr L)\Psi+\callO(\dotwp r^{-1})LT\Psi\\
&\quad+\callO(\dotwp^{\leq2}r^{-1}) L\Omega\Psi+\callO(\dotwp^{\leq3})L\Psi+\callO(\dotwp^{\leq3}r^{-2})\Lbar\Psi+\callO(\dotwp^{\leq3}r^{-3})\Omega\Psi.
\end{split}
\end{align*}
The terms $-\Lbar L\Psi+LL\Psi$ will contribute to the terms involving $\calP_1$ on the right-hand side of \eqref{eq:higher1} and the remaining terms have the expected form. The desired structure now follows by inductively applying this identity and using Lemmas~\ref{lem:commutators1} and~\ref{lem:commutators2}. Here the commutators with $L^2\Psi$ and the remaining terms on the right-hand side of \eqref{eq:calPVF1} are treated inductively in a similar way as with $\Lbar L$ above.
\end{proof}
%%%%%%%%%%
%%%%%%%%%%
%%%%%%%%%%%%
%%%%%%%%%%%%
%%%%%%%%%%%%
%%%%%%%%%%%%
We end this subsection by deriving expansions for the source and the cubic terms in equation \eqref{eq:uext1}. The cubic term refers to the part of the term (recall that $\varphi$ and $\phi$ are related by the conjugation \eqref{eq:varphiphi1})
%%%%%%%
%%%%%%%
\begin{align}\label{eq:purelycubic1}
\begin{split}
\frac{\nabla^\mu\varphi\nabla^\nu\varphi}{1+\nabla^\alpha v \nabla_\alpha v}\nabla^2_{\mu\nu}\varphi
\end{split}
\end{align}
in \eqref{eq:calF3_1} where no factors of $Q$ appear, which we expect to be the most difficult term in the nonlinearity. Recall from \eqref{eq:extpar1}, that in the exterior region
%%%%%%%
%%%%%%%
\begin{align*}
\begin{split}
Q\equiv Q_\wp=Q(rA_\ell\Theta-\gamma\jap{r}\ell)=Q(\tilr).
\end{split}
\end{align*}
In view of \eqref{eq:uext1} the source term $\calF_0$ is given by
%%%%%%%
%%%%%%%
\begin{align}\label{eq:sourcetermrp1}
\begin{split}
\calF_0=\Box_m Q-(1+\nabla^\alpha Q \nabla_\alpha Q)^{-1}\nabla^\mu Q\nabla^\nu Q\nabla^2_{\mu\nu}Q.
\end{split}
\end{align}
The more precise structure of $\calF_0$ is calculated in the next lemma.
%%%%%%%%%%%
%%%%%%%%%%%
\begin{lemma}\label{lem:sourceext1}
The source term $\calF_0$ satisfies the following estimate in the hyperboloidal region $\calC_\hyp$:
%%%%%%%
%%%%%%%
\begin{align*}
\begin{split}
\RbfT^k\calF_0= \callO(\dotwp^{1+k\leq\cdot\leq 3+k}r^{-n+1}),\quad k=1,2,3.
\end{split}
\end{align*}
\end{lemma}
%%%%%%%%%%%
%%%%%%%%%%%
\begin{proof}
Recall that if $Q$ were a maximal embedding then $\calF_0$ would vanish. In particular (by a slight abuse of notation we are identifying $Q(y)=Q(|y|)$ with a function of a single variable), $Q$ satisfies equation \eqref{eq:QRiem1}.
Moreover, $\Omega Q=0$ by construction. It follow from these facts and Lemma~\ref{lem:BoxVF1} that
%%%%%%%
%%%%%%%
\begin{align*}
\begin{split}
\calF_0=\callO(\dotwp^{\leq 3})Q'+\callO(\dotwp r)Q'',
\end{split}
\end{align*}
which proves the desired bound for $k=0$. The higher order bounds are obtained similarly by differentiating the equation.
\end{proof}
%%%%%%%%%%%
%%%%%%%%%%%
Turning to \eqref{eq:purelycubic1}, we have the following expansion of the purely cubic part of this nonlinearity.
%%%%%%%%%%%
%%%%%%%%%%%
\begin{lemma}\label{lem:purelycubic1}
$\nabla^\mu\varphi\nabla^\nu\varphi\nabla^2_{\mu\nu}\varphi$ can be written as a linear combination of terms of the following forms in the hyperboloidal region $\calC_\hyp$:
\begin{enumerate}
\item Quasilinear terms: $(L\varphi)^2\Lbar^2\varphi$, $(L\varphi\Lbar\varphi)\Lbar L\varphi$, $(\Lbar\varphi)^2L^2\varphi$, $(L\varphi e_A\varphi)e_A\Lbar\varphi$, $(\Lbar\varphi e_A\varphi)e_AL\varphi$, $(e_A\varphi e_B\varphi)e_Ae_B\varphi$.
\item Semilinear terms: $\callO(\dotwp^{\leq 2}r^{-3})(\Lbar\varphi)^2L\fy$, $\callO(\dotwp r^{-3})e_A\varphi$, $\callO(\dotwp^{\leq 2})(L\fy)^2\Lbar\fy$, \\$\callO(\dotwp^{\leq2}r^{-1})\Lbar\fy L\fy e_A \fy$, $\callO(\dotwp r^{-1})\Lbar\fy e_A\fy e_B\fy$, $\callO(\dotwp)L\fy e_A\fy e_B\fy$, $\callO(1)e_A\fy e_B\fy e_C\fy$.
\end{enumerate}
\end{lemma}
%%%%%%%%%%%
%%%%%%%%%%%
\begin{proof}
This follows by writing this expression as
%%%%%%%
%%%%%%%
\begin{align*}
\begin{split}
(m^{-1})^{II'}(m^{-1})^{JJ'}(D_{I'}\fy D_{J'}\fy)(D_{I}D_J\fy-\Gamma_{IJ}^KD_K\fy).
\end{split}
\end{align*}
The relevant connection coefficients can be calculated using the Koszul formula as in the proof of Lemma~\ref{eq:BoxVF2} and are given by
%%%%%%%
%%%%%%%
\begin{align*}
\begin{split}
&\Gamma_{LL}^\Lbar=0,\quad \Gamma_{LL}^L=\callO(\dotwp^{\leq 2}r^{-2}),\quad \Gamma_{LL}^A=\callO(\dotwp r^{-3}),\quad \Gamma_{\Lbar L}^\Lbar=0, \quad\Gamma_{\Lbar L}^L=\callO(\dotwp^{\leq2}), \quad\Gamma_{\Lbar L}^\Lbar=0,\\
&\Gamma_{\Lbar L}^A=\callO(\dotwp^{\leq 2}r^{-1}),\quad \Gamma_{A L}^L=\callO(\dotwp^{-2}r^{-1}),\quad \Gamma_{AL}^\Lbar=0,\quad \Gamma_{AL}^B=\callO(\dotwp r^{-1}), \quad \Gamma_{A\Lbar}^\Lbar=\callO(\dotwp^{\leq 2}r^{-1}),\\
&\Gamma_{A\Lbar}^L=0,\quad \Gamma_{A\Lbar}^B=\callO(\dotwp), \quad \Gamma_{AB}^L=\callO(\dotwp),\quad \Gamma_{AB}^\Lbar=\callO(\dotwp r^{-1}), \quad \Gamma_{AB}^C=\callO(1).\qedhere
\end{split}
\end{align*}
\end{proof}
%%%%%%%%%%%
%%%%%%%%%%%
%%%%%%%%%%%%%%
%%%%%%%%%%%%%%
\subsection{The Main $r^p$ Multiplier Identity}  
%%%%%%%%%%%%%%
%%%%%%%%%%%%%%
This section contains the main $r^p$ multiplier identity for $\calP_\graph$. Recall that this operator arises when using the conjugated variable $\varphi$ defined in terms of $\phi$ in \eqref{eq:varphiphi1}. Since we will be interested in the exterior hyperboloidal region, we  fix a cutoff function $\chi_{\geq \tilR}$ supported in the region $\{r\geq \tilR\}\subseteq \calC_\hyp$. We will also use the notation $\chi_{\leq \tilR}=1-\chi_{\geq \tilR}$. Given $\uppsi$ with
%%%%%%%
%%%%%%%
\begin{align*}
\begin{split}
\calP_\graph \uppsi=f,
\end{split}
\end{align*}
we let $\tilupsi=\tilr^{\frac{n-1}{2}}\uppsi$ and $\tilf=\tilr^{\frac{n-1}{2}}f$ as usual. Suppose $X_1,\dots,X_k$, $k=k_1+k_2+k_3$, are a collection of vectorfields from $\{\tilr L, \Omega, T\}$, with $X_1,\dots X_{k_3}=T$, $X_{k_3+1},\dots X_{k_3+k_2}=\Omega$, and~$X_{k_3+k_2+1},\dots X_k=\tilr L$. We let
%%%%%%%
%%%%%%%
\begin{align}\label{eq:tilphikdef1}
\begin{split}
\tilupsi_k=X_k\dots X_1\tilupsi,\qquad \tilf_k = X_k\dots X_1\tilf,
\end{split}
\end{align}
and if the precise choice of the vectorfields is important we write
%%%%%%%
%%%%%%%
\begin{align}\label{eq:tilphikdef2}
\begin{split}
\tilupsi_k=\tilupsi_{k_1,k_2,k_2},\qquad \tilf_{k}=\tilf_{k_1,k_2,k_3}.
\end{split}
\end{align}
The basic $r^p$ boundary and bulk energies are defined as follows. For any $p\in[0,2]$, 
%%%%%%%
%%%%%%%
\begin{align}\label{eq:rpenergiesdef1}
\begin{split}
&\calE^p_k(\sigma)\equiv\calE_k^p[\uppsi](\sigma):=\int_{\Sigma_\sigma} \chi_{\geq\tilR} \tilr^p (L\tilupsi_k)^2 \ud \theta \ud r,\\
&\calB_{k}^{p}(\sigma_1,\sigma_2)\equiv\calB_{k}^{p}[\uppsi](\sigma_1,\sigma_2):=\int_{\sigma_1}^{\sigma_2}\int_{\Sigma_\tau}\chi_{\geq\tilR}\tilr^{p-1}\big((L\tilupsi_k)^2+\big(\frac{2-p}{\tilr^2}\big)(|\Omega\tilupsi_k|^2+\tilupsi_k^2)\big)\ud S \ud r \ud\tau.
\end{split}
\end{align}
When there is a need to distinguish between the vectorfields applied to $\tilupsi$ we write
%%%%%%%
%%%%%%%
\begin{align*}
\begin{split}
\calE_{k_1,k_2,k_3}^p(\sigma) \mand \calB_{k_1,k_2,k_3}^p(\sigma_1,\sigma_2)
\end{split}
\end{align*}
for the corresponding energies. We also define the standard energy (note that the definition agrees with $E$ in \eqref{eq:standardenergydef1} when $k=0$)
%%%%%%% 
%%%%%%%
\begin{align*}
\begin{split}
E_k(\tau)&\equiv E_k[\uppsi](\tau)\\
&:=\int_{\Sigma_\tau}\chi_{\leq \tilR}(|\partial\partial^k\uppsi|^2+\jap{\rho}^{-2}|\partial^k\uppsi|^2)\ud V+\int_{\Sigma_\tau}\chi_{\geq \tilR}(|\partial_\Sigma X^k \uppsi|^2+r^{-2}|TX^k\uppsi|^2+r^{-2}|X^k\uppsi|^2) \ud V.
\end{split}
\end{align*}
%%%%%%%%%%%
%%%%%%%%%%%
\begin{lemma}\label{lem:rpmult1}
Suppose $\calP_\graph \uppsi=f$ and let $\tilupsi=\tilr^{\frac{n-1}{2}}\uppsi$ and $\tilf=\tilr^{\frac{n-1}{2}}f$. If the bootstrap assumptions \eqref{eq:a+trap}--\eqref{eq:wpb1} hold, then with the notation introduced in \eqref{eq:tilphikdef1}, \eqref{eq:tilphikdef1}, \eqref{eq:rpenergiesdef1}, with $k=k_1+k_2+k_3$, and for any $0\leq p\leq 2$ and any $\tau_1<\tau_2$, 
%%%%%%%
%%%%%%%
\begin{align}\label{eq:rpmult1}
\begin{split}
&\sum_{j\leq k_1}\big(\sup_{\tau\in[\tau_1,\tau_2]}\calE_{j,k_2,k_3}^p(\tau)+\calB_{j,k_2,k_3}^{p-1}(\tau_1,\tau_2)\big)\\
&\leq C\sum_{j\leq k_1}\calE_{j,k_2,k_3}^{p}(\tau_1)+C\sum_{j\leq k_1}\int_{\tau_1}^{\tau_2}\int_{\Sigma_\tau}\chi_{\geq\tilR}\tilr^p \tilf_k (L\tilupsi_{j,k_2,k_3} +\callO(r^{-5})\Omega\tilupsi_{j,k_2,k_3})\ud \theta \ud r \ud\tau\\
&\quad+C_\tilR\int_{\tau_1}^{\tau_2}\int_{\Sigma_\tau}|\partial\chi_{\geq\tilR}|(|\partial_{\tau,x}\tilupsi_k|^2+|\tilupsi_k|^2)\ud \theta \ud r \ud \tau\\
&\quad+C\sum_{j\leq k}\sup_{\tau\in[\tau_1,\tau_2]}E_j(\tau)+C\delta \sup_{\tau\in[\tau_1,\tau_2]}\sum_{j\leq k}\calE_j^p(\tau)\\
&\quad+C\delta\sum_{j\leq k}\int_{\tau_1}^{\tau_2}\int_{\Sigma_\tau}\chi_{\geq\tilR}(|\partial_{\tau,x}\tilupsi_j|^2+\tilr^{-2}|\tilupsi_j|^2)\tilr^{-1-\alpha}\ud \theta \ud r \ud\tau.
\end{split}
\end{align}
Here $C$ and $C_\tilR$ are large constants constants and $\delta=o(\epsilon)+o(\tilR)$ is a small constant that is independent of $C$ and $C_\tilR$.
\end{lemma}
%%%%%%%%%%%
%%%%%%%%%%%
\begin{proof}
To simplify notation we write $\chi$ for $\chi_{\geq\tilR}$ and $\tilupsi_k$ for $\tilupsi_{k_1,k_2,k_3}$, and multiply each term in the expansion \eqref{eq:higher1},~\eqref{eq:errorhigher1} by $\chi\tilr^p L\tilupsi$.
%%%%%%%
%%%%%%%
\begin{align*}
\begin{split}
-(1+\callO(r^{-4}))\chi\Lbar L\tilupsi_k L\tilupsi_k \tilr^p&=-\frac{1}{2}\Lbar((1+\callO(r^{-4}))\chi(L\tilupsi_k)^2\tilr^p)-\frac{p}{2}\chi\tilr^{p-1}(L\tilupsi_k)^2\\
&\quad+\callO(\dotwp)(L\tilupsi_k)^2\tilr^p+\callO(r^{-4})\tilr^{p-1}(L\tilupsi_k)^2+\callO(1)(\Lbar\chi)\tilr^p(L\tilupsi_k)^2.
\end{split}
\end{align*}
Similarly, using also \eqref{eq:VFcomm1} and Cauchy-Schwarz,
%%%%%%%
%%%%%%%
\begin{align*}
\begin{split}
(1+\callO(r^{-4}))\chi\Omega^2\tilupsi_k L \tilupsi_k \tilr^{p-2}&=-\frac{1}{2}L((1+\callO(r^{-4}))\chi(\frac{\Omega}{\tilr}\tilupsi_k)^2\tilr^{p})-\frac{2-p}{2}\chi(\frac{\Omega}{\tilr}\tilupsi_k)^2\tilr^{p-1}\\
&\quad+\Omega((1+\callO(r^{-4}))\chi\tilr^{p-2}\Omega\tilupsi L\tilupsi_k)\\
&\quad+\callO(1)(\Omega\chi)\Omega\tilupsi_kL\tilupsi_k\tilr^{p-1}+\callO(1)(L\chi)(\frac{1}{\tilr}\Omega\tilupsi_k)^2\tilr^{p}\\
&\quad+\callO(r^{-4})\chi(\frac{\Omega}{\tilr}\tilupsi_k)^2\tilr^{p-1}+\callO(\dotwp r^{-2})\chi(L\tilupsi_k)^2\tilr^p\\
&\quad +\callO(\dotwp r^{-2})\chi(\frac{\Omega}{\tilr}\tilupsi)^2\tilr^p+\callO(\dotwp r^{-4})\chi(\Lbar\tilupsi)^2\tilr^p,
\end{split}
\end{align*}
and
%%%%%%%
%%%%%%%
\begin{align*}
\begin{split}
-(1+\callO(r^{-4}))\chi\tilupsi_k L\tilupsi_k \tilr^{p-2} &=-\frac{1}{2}L((1+\callO(r^{-4}))\chi(\frac{\tilupsi_k}{\tilr})^2\tilr^p)-\frac{2-p}{2}\chi(\frac{\tilupsi_k}{\tilr})^2\tilr^{p-1}\\
&\quad+\callO(1)(L\chi)\tilupsi_k^2\tilr^{p-2}+\callO(r^{-4})\chi(\frac{\tilupsi_k}{\tilr})^2\tilr^{p-1}+\callO(\dotwp r^{-2})(\frac{\tilupsi_k}{\tilr})^2\tilr^p.
\end{split}
\end{align*}
Note that integrating the last three identities already gives the desired control on the left-hand side of \eqref{eq:rpmult1}. Here the terms involving $\dotwp$ can be integrated in $\tau$ and absorbed by the left-hand side of \eqref{eq:rpmult1} or the standard energy $E_k$. Turning to the error terms, we first consider
%%%%%%%
%%%%%%%
\begin{align*}
\begin{split}
\callO(r^{-4})\chi L^2\tilupsi_k^2L\tilupsi_k \tilr^p&=L(\callO(r^{-4})\chi(L\tilupsi_k)^2\tilr^p)+\callO(r^{-4})\chi(L\tilupsi_k)^2\tilr^{p-1}+\callO(1)(L\chi)(L\tilupsi_k)^2.
\end{split}
\end{align*}
Similarly, using also \eqref{eq:VFcomm1},
%%%%%%%
%%%%%%%
\begin{align*}
\begin{split}
\callO(r^{-4})\chi\Lbar^2\tilupsi_k L\tilupsi_k \tilr^p&=\Lbar(\callO(r^{-4})\chi\Lbar\tilupsi_k L\tilupsi_k\tilr^p)+L(\callO(r^{-4})\chi(\Lbar\tilupsi_k)^2\tilr^p)+\callO(r^{-4})(\Lbar\tilupsi_k)^2\tilr^{p-1}\\
&\quad+\callO(r^{-5})\chi L\tilupsi_k\Lbar\tilupsi_k \tilr^{p-1}+\callO(\dotwp r^{-4})L\tilupsi_k\Lbar\tilupsi_k \tilr^{p}+\callO(r^{-4})(\Lbar\chi)L\tilupsi_k\Lbar \tilupsi_k\tilr^{p}\\
&\quad \callO(\dotwp r^{-5})(\frac{\Omega}{\tilr}\tilupsi_k)\Lbar\tilupsi_k\tilr^p+\callO(r^{-4})(L\chi)(\Lbar\tilupsi_k)^2\tilr^p,
\end{split}
\end{align*}
and
%%%%%%%
%%%%%%%
\begin{align*}
\begin{split}
\callO(r^{-5})\chi\Omega L\tilupsi_k L\tilupsi_k \tilr^p=\Omega(\callO(r)^{-5}\chi(L\tilupsi_k)^2\tilr^p)+\callO(r^{-4})\chi(L\tilupsi_k)^2\tilr^{p-1}+\callO(r^{-5})(\Omega\chi)(L\tilupsi_k)^2\tilr^p.
\end{split}
\end{align*}
For the last term on the second line of \eqref{eq:errorhigher1} first observe that
%%%%%%%
%%%%%%%
\begin{align*}
\begin{split}
\callO(r^{-5})\chi\Lbar\Omega\tilupsi_k L\tilupsi_k \tilr^p &= \Lbar(\callO(r^{-5})\chi\Omega\tilupsi_k L\tilupsi_k \tilr^p)+\callO(r^{-5})\chi\Omega\tilupsi_k\Lbar L\tilupsi_k\tilr^p+\callO(\dotwp r^{-4})\chi(\frac{\Omega}{\tilr}\tilupsi_k)L\tilupsi_k\tilr^p\\
&\quad+\callO(r^{-5})\chi(\frac{\Omega}{\tilr}\tilupsi_k)L\tilupsi_k\tilr^p+\callO( r^{-4})(\Lbar\chi)(\frac{\Omega}{\tilr}\tilupsi_k)L\tilupsi_k\tilr^p
\end{split}
\end{align*}
To treat the second term on the right, we use equation \eqref{eq:higher1} to solve for $\Lbar L\tilupsi$, and replace this term by
%%%%%%%
%%%%%%%
\begin{align*}
\begin{split}
\callO(r^{-5})\chi\Omega\tilupsi_k(\calP_\graph+\Lbar L)\tilupsi_k\tilr^p+\callO(r^{-5})\chi\Omega\tilupsi_k \Err_k [\tilupsi]\tilr^p.
\end{split}
\end{align*}
These terms can then be treated using similar considerations as above and below. The term $\chi \Err_k[\tilupsi]L \tilupsi_k\tilr^p$ in \eqref{eq:higher1} is treated using repeated applications of the product rule (for integration by parts) and Lemmas~\ref{lem:commutators1} and~\ref{lem:commutators2} to write this term as a total derivative plus acceptable terms. We discuss only the contribution of $\tilcalP_1$, where by an abuse of notation we write $\tilupsi_{j}$ for $\tilupsi_{j,k_2,k_3}$. For the top order term $\tilcalP_1\tilupsi_{k_1-1}$ the favorable sign of the coefficient $c_{k_1-1,k_1}$ is important, as we will illustrate with the term $L^2\tilupsi_{k_1-1}$ appearing in $\tilcalP_1\tilupsi_{k_1-1}$, for which we write
%%%%%%%
%%%%%%%
\begin{align*}
\begin{split}
L^2\tilupsi_{k_1-1}L\tilupsi_{k_1} \tilr^2= L(\tilr^{-1}\tilupsi_{k_1})L\tilupsi_{k_1}\tilr^2= (L\tilupsi_{k_1})^2\tilr-\frac{1}{2}L(\tilr^2(L\tilupsi_{k_1})^2)+\callO(\dotwp)\tilupsi_{k_1} L \tilupsi_{k_1}.
\end{split}
\end{align*}
Here the first term on the right has a favorable sign after multiplication by $c_{k_1-1,k_1}$ and the other terms can be bounded by the energy fluxes. The other terms in $\tilcalP_1\tilupsi_{k_1-1}$ are treated similarly, with a few more integration by parts and commutations between $\Omega$ and $L$, again with the sign of $c_{k_1-1,k_1}$ playing an important role for the main bulk term. The contributions of $\tilcalP_1\tilupsi_{j}$, $j\leq k_1-2$, are treated similarly where the error terms are absorbed inductively by adding a suitable multiple of the estimates for lower values of $k_1$. For instance
%%%%%%%
%%%%%%%
\begin{align*}
\begin{split}
\Omega^2\tilupsi_j L\tilupsi_{k_1}= \frac{1}{\tilr}\Omega\tilupsi_j\Omega\tilupsi_{k_1}+\Omega(\Omega\tilupsi_{j}L\tilupsi_{k_1})-L(\Omega\tilupsi_{j+1}\Omega\tilupsi_{k_1})+[L,\Omega]\tilupsi_j\Omega\tilupsi_{k_1}-\Omega\tilupsi_{j}[\Omega,L]\tilupsi_k,
\end{split}
\end{align*}
and the first term on the right is bounded by $\delta |\tilr^{-1}\Omega\tilupsi_{k_1}|^2\tilr+C_\delta|\tilr^{-1}\Omega\tilupsi_{j+1}|^2\tilr$. The first term is absorbed by a corresponding term coming from $\tilcalP_1\tilupsi_{k_1-1}$ and the second term by a similar term in the multiplier identity for $\tilupsi_{j+1}$ (instead of $\tilupsi_{k_1}$), after adding a suitably large multiple of that identity. The contribution of the last four lines of \eqref{eq:errorhigher1} need no further manipulations. Putting everything together and applying Cauchy-Schwartz we obtain (to be precise, as explained above, we should write this first for $\tilupsi=\tilupsi_{1,k_2,k_3}$ and derive the corresponding estimate, and then inductively build up to $\tilupsi=\tilupsi_{k_1,k_2,k_3}$)
%%%%%%%%
%%%%%%%%
%%%%%%%
%%%%%%%
\begin{align}\label{eq:rpmulttemp1}
\begin{split}
\tilf L\tilupsi \tilr^p&=-\frac{1}{2}\Lbar\big(\chi(L\tilupsi)^2\tilr^p+\Err_\Lbar\big)-L\big(\tilr^{p-2} \big((\Omega\tilupsi)^2+\frac{(n-1)(n-3)}{4}\tilupsi^2\big)+\Err_{L}\big)\\
&\quad+\Omega( \Err_\Omega)+\Err_{\mathrm{int}}+\callO(\jap{r}^{-5})\Omega\tilupsi \tilf\tilr^p\\
&\quad-\frac{p}{2}\chi\tilr^{p-1}(L\tilupsi)^2-\frac{2-p}{2}\chi(\frac{\Omega}{\tilr}\tilupsi)^2\tilr^{p-1}-\frac{(2-p)(n-1)(n-3)}{8}\chi(\tilr^{-1}\tilupsi)^2\tilr^{p-1},
\end{split}
\end{align}
where
%%%%%%%
%%%%%%%
\begin{align*}
\begin{split}
&\Err_\Lbar\leq C\chi\sum_{j\leq k} ((\partial_{\tau,x}\tilupsi_j)^2+\tilr^{-2}\tilupsi_j^2),\\
&\Err_{L}\leq C\chi\tilr^p\sum_{j\leq k}\sum_{j\leq k} ((L\tilupsi)^2+\tilr^{-2}(\Omega\tilupsi)^2+\tilr^{-2}\tilupsi_j^2+\tilr^{-2}(\Lbar\tilupsi)^2),\\
&\Err_{\Omega}\leq C\chi\tilr^{p-1}\sum_{j\leq k}\sum_{j\leq k} ((L\tilupsi)^2+\tilr^{-2}(\Omega\tilupsi)^2+\tilr^{-2}\tilupsi_j^2+\tilr^{-2}(\Lbar\tilupsi)^2),\\
&\Err_{\mathrm{int}}\leq C_\tilR\sum_{j\leq k}|\partial\chi|(|\partial_{\tau,x}\tilupsi_j|^2+|\tilupsi_j|^2)+C\chi\sum_{j\leq k}(|\partial_{\tau,x}\tilupsi_j|^2+\tilr^{-2}|\tilupsi_j|^2)\tilr^{-1-\alpha}\\
&\phantom{\Err_{\mathrm{int}}\leq }+\callO(\dotwp)\sum_{j\leq k}\big((L\tilupsi_j)^2\tilr^p+(\tilr^{-1}\Omega\tilupsi_j)^2+(\tilr^{-1}\tilupsi_j)^2+\tilr^{-2}(\Lbar\tilupsi_j)^2\big).
\end{split}
\end{align*}
The desired estimate \eqref{eq:rpmult1} now follows from integrating \eqref{eq:rpmulttemp1}.  Here note when integrating \eqref{eq:rpmulttemp1} we also encounter a term involving $\div V$, for $V=L$, $\Lbar$, or $\Omega$ (from the difference of $V^\mu\partial_\mu u$ and $\partial_\mu(V^\mu u$)), but these terms come with $\dotwp$, which has extra $\tau$ integrability, and can be absorbed.
\end{proof}
%%%%%%%%%%%
%%%%%%%%%%%
%%%%%%%%%%%%
%%%%%%%%%%%%

%%%%%%%%%%%%%%%%%%%
%%%%%%%%%%%%%%%%%%%
%%%%%%%%%%%%%%%%%%%
\subsection{Nonlinear Energy and Local Energy Decay Estimates}
%%%%%%%%%%%%%%%%%%%
%%%%%%%%%%%%%%%%%%%
%%%%%%%%%%%%%%%%%%%
In this section we again use the variable $\phi$, not the conjugated version $\varphi$ (see \eqref{eq:varphiphi1}). However, in view of the definition \eqref{eq:varphiphi1}, and under our bootstrap assumptions,  the estimates on $\varphi$ easily transfer to estimates on $\phi$. As a first step in the proof of Proposition~\ref{prop:bootstrapphi1} we apply the results of Section~\ref{sec:LED} to derive energy and local energy decay estimates for $\phi$. Let
%%%%%%%
%%%%%%%
\begin{align*}
\begin{split}
\|\phi\|_{LE_k[\sigma_1,\sigma_2]}^2:=\|\chi_{\leq R}\partial^k\phi\|_{LE[\sigma_1,\sigma_2]}^2+\|\chi_{\geq R}X^k\phi\|_{LE[\sigma_1,\sigma_2]}^2.
\end{split}
\end{align*}
Our goal is to prove the following result.
%%%%%%%%%%%%
%%%%%%%%%%%%
\begin{proposition}\label{prop:energyLEDnonlin1}
If the bootstrap assumptions \eqref{eq:a+trap}--\eqref{eq:Tjphienergyb2} are satisfied and $\epsilon$ is sufficiently small, then 
%%%%%%%
%%%%%%%
\begin{align}
&\sup_{\sigma_1\leq \sigma \leq \sigma_2} E_k[\phi](\sigma)+\|\phi\|_{LE_k[\sigma_1,\sigma_2]}^2\lesssim \sum_{i\leq k}E_i[\phi](\sigma_1)+\epsilon^2\bfupsigma_0(\sigma_1), \quad k\leq M,\label{eq:phienergyLEDboundtemp1}\\
&\sup_{\sigma_1\leq \sigma \leq \sigma_2} E_k[\RbfT\phi](\sigma)+\|\RbfT\phi\|_{LE_k[\sigma_1,\sigma_2]}^2\lesssim \sum_{i\leq k}E_i[\RbfT\phi](\sigma_1)+\epsilon^2\bfupsigma_1(\sigma_1), \quad k\leq M-1,\label{eq:TphienergyLEDboundtemp1}\\
&\sup_{\sigma_1\leq \sigma \leq \sigma_2} E_k[\RbfT^2\phi](\sigma)+\|\RbfT^2\phi\|_{LE_k[\sigma_1,\sigma_2]}^2\lesssim \sum_{i\leq k}E_i[\RbfT^2\phi](\sigma_1)+\epsilon^2\bfupsigma_2(\sigma_1), \quad k\leq M-2,\label{eq:T2phienergyLEDboundtemp1}
\end{align}
where $\bfupsigma_j(\sigma_1)=\sigma_1^{-2-2j}$ if $k+3j\leq M-2$ and $\upsigma_j(\sigma_1)=1$ otherwise.
\end{proposition}
%%%%%%%%%%%%
%%%%%%%%%%%%
We start with some estimates on the source term defined in \eqref{eq:sourcetermrp1} (see also Lemma~\ref{lem:sourceext1}).
%%%%%%%%%%%%
%%%%%%%%%%%%
\begin{lemma}\label{lem:sourceLE1}
Under the assumptions of Proposition~\ref{prop:energyLEDnonlin1}, and with $\calR_{\sigma_1}^{\sigma_2}=\cup_{\sigma=\sigma_1}^{\sigma_2}\Sigma_\sigma$,
\begin{align}
&\|\jap{r}(\chi_{r\geq R}|X^k\calF_0|+\chi_{r\leq R}|\partial^k\calF_0|)\|_{L^2(\Sigma_\tau)}\lesssim \delta_\wp \epsilon \tau^{-\frac{5}{2}+\kappa},\label{eq:calF0constanttautemp1}\\
&\|\jap{r}(\chi_{r\geq R}|X^k\RbfT\calF_0|+\chi_{r\leq R}|\partial^k\RbfT\calF_0|)\|_{L^2(\Sigma_\tau)}\lesssim \delta_\wp \epsilon \tau^{-3},\label{eq:TcalF0constanttautemp1}\\
&\|\jap{r}(\chi_{r\geq R}|X^k\RbfT^j\calF_0|+\chi_{r\leq R}|\partial^k\RbfT^j\calF_0|)\|_{L^2(\calR_{\sigma_1}^{\sigma_2})}\lesssim \delta_\wp \epsilon\sigma_1^{-4+2\kappa}+\delta_\wp\|\RbfT^j\phi\|_{LE([\sigma_1,\sigma_2])},\nonumber\\
&\phantom{\|\jap{r}(\chi_{r\geq R}|X^k\RbfT^j\calF_0|+\chi_{r\leq R}|\partial^k\RbfT^j\calF_0|)\|_{L^2(\calR_{\sigma_1}^{\sigma_2})}\lesssim} j=0,1,2.\label{eq:calF0tauintegratedtemp1}
\end{align}
\end{lemma}
%%%%%%%%%%%%
%%%%%%%%%%%%
\begin{proof}
The exterior estimates follow from Lemma~\ref{lem:sourceext1}, Proposition~\ref{prop:bootstrappar1}, and Lemma~\ref{lem:OmegaiTkphi}. For the interior the spatial decay of the source term is not important, so the estimates follow simply by counting the number of $\RbfT$ derivatives and Proposition~\ref{prop:bootstrappar1} and Lemma~\ref{lem:OmegaiTkphi}. 
\end{proof}
%%%%%%%%%%%%
%%%%%%%%%%%%
Note that even though the source term $\calF_0$ was derived using the conjugated variable $\varphi$ in \eqref{eq:varphiphi1}, bounding $s$ in \eqref{eq:varphiphi1} using our bootstrap assumptions, the same estimates are satisfied by the source term in the equation for $\phi$ (see \eqref{eq:abstractexteq1}). We can now prove Proposition~\ref{prop:energyLEDnonlin1}.
%%%%%%%%%%%%
%%%%%%%%%%%%
\begin{proof}[Proof of Proposition~\ref{prop:energyLEDnonlin1}]
Using the global coordinates $(\uptau,\uprho,\uptheta)$, after commuting any number of $\partial_\uptau=\RbfT$ derivatives we collect the leading order terms to write the equation in the form \eqref{eq:LEDlinearmodel1} with $\calP$ as in Section~\ref{sec:LED}. We start with the proof of the estimates when $k=0$. Applying Propositions~\ref{prop:energyestimate1} and~\ref{prop:LED1}, the nonlinear terms (including products of derivatives of $\dotwp$ and $\phi$) can be estimated using the bootstrap assumptions \eqref{eq:a+trap}--\eqref{eq:Tjphienergyb2}, simply by treating them as quadratic. The contribution of the source terms is estimated using Lemma~\ref{lem:sourceLE1}, and the contribution of $\bfOmega_k$ (see Remark~\ref{rem:Omegalinear1}) using Lemma~\ref{lem:OmegaiTkphi}, where we use the smallness of $\delta_\wp$ to absorb the $LE$ norms appearing in \eqref{eq:OmegaiTkphi} and~\eqref{eq:calF0tauintegratedtemp1}. Note that by the same procedure we can prove the estimates for higher powers of $\RbfT$ without gaining extra decay (that is, by treating powers of $\RbfT$ which are higher than two as arbitrary derivatives). There is one point that deserves further explanation in this process. Among the error terms after commuting $\partial_\uptau^k$, there will be terms of the forms (recall that the part of the equation without spatial decay is given by \eqref{eq:Boxm1};  below $\partial_y$ denotes an arbitrary tangential derivative of size one)
%%%%%%%
%%%%%%%
\begin{align*}
\begin{split}
\callO(\dotwp^{(j)})\partial^2_y\RbfT^{k-j}\phi\quad\mand\quad \callO(\dotwp^{(j)})(\partial_\uprho+\frac{n-1}{2\uprho})\partial_\uptau \RbfT^{k-j}\phi,
\end{split}
\end{align*}
with $j\geq1$, and the multipliers for the energy and LED estimates contain terms of the form $\callO(1)\partial_\uptau \RbfT^k\phi$. Since $\partial_\tau\RbfT^k\phi$ cannot be placed in the energy flux for $\RbfT^k\phi$ in the hyperboloidal part of the foliation, some integration by parts are necessary to deal with these terms. For the error terms of the form $\callO(\dotwp^{(j)})\partial^2_y\RbfT^{k-j}\phi$, we can integrate by parts twice to obtain terms of the forms 
%%%%%%%
%%%%%%%
\begin{align*}
\begin{split}
\callO(\dotwp^{(j)})\partial_y\RbfT^{k-j+1}\phi\partial_y \RbfT^k\phi, \quad \callO(\dotwp^{(j+1)})\partial_y\RbfT^{k-j}\phi\partial_y \RbfT^k\phi,\quad \callO(\dotwp^{(j)}\uprho^{-1})\partial_y\RbfT^{k-j}\phi \partial_\uptau \RbfT^k\phi.
\end{split}
\end{align*}
These can be bounded, respectively, as (where $L^2_y$ denotes $L^2(\Sigma_\uptau)$)
%%%%%%%
%%%%%%%
\begin{align*}
\begin{split}
&\|\dotwp^{(j)}\|_{L^1_\uptau}\|\partial_y\RbfT^{k-j+1}\phi\|_{L^\infty_{\uptau}L^2_y}\|\partial_y \RbfT^k\phi\|_{L^\infty_\uptau L^2_y},\quad \|\dotwp^{(j+1)}\|_{L^1_\uptau}\|\partial_y\RbfT^{k-j}\phi\|_{L^\infty_{\uptau}L^2_y}\|\partial_y \RbfT^k\phi\|_{L^\infty_\uptau L^2_y},\\
& \|\dotwp^{(j)}\|_{L^2_\uptau}\|\partial_y\RbfT^{k-j}\phi\|_{L^\infty_{\uptau}L^2_y}\| \RbfT^k\phi\|_{LE}.
\end{split}
\end{align*}
For the error terms of the form $\callO(\dotwp^{(j)})(\partial_\uprho+\frac{n-1}{2\uprho})\partial_\uptau \RbfT^{k-j}\phi$, we use the equation for $\RbfT^{k-j}\phi$ (again see \eqref{eq:Boxm1}) to replace them by terms of the forms that were already handled above, or have better spatial decay.

Next, we use elliptic estimates to obtain energy and local energy estimates for arbitrary, size one, derivatives applied on $\phi$. For this, recall the decomposition of the operator $\calP$ as
%%%%%%%
%%%%%%%
\begin{align*}
\begin{split}
\calP=\calP_\Ell+\calP_\uptau,\qquad \mwhere\qquad \calP_\uptau=\callO(1)\partial \RbfT+\callO(\jap{\uprho^{-1}})\RbfT.
\end{split}
\end{align*}
Using the estimates for $\RbfT\phi$, we can use elliptic estimates to bound $$\sup_{\sigma_1\leq \sigma \leq\sigma_2}\|\partial_\Sigma\phi\|_{E(\Sigma_\sigma)}+\|\partial_\Sigma\phi\|_{LE[\sigma_1,\sigma_2]}$$ by the right-hand side of \eqref{eq:phienergyLEDboundtemp1}. Here for the $LE$ norm the spatial norms can be inserted in the elliptic bound by writing 
%%%%%%%
%%%%%%%
\begin{align*}
\begin{split}
\phi=\chi_{\leq 1}\phi+\sum_{j\geq 1}\chi_j \phi,
\end{split}
\end{align*}
where $\chi_{\leq 1}(\uprho)$ is supported in the region $\{|\uprho|\leq 1\}$, and $\chi_j(\uprho)$ in the region $\{|\uprho|\simeq 2^j\}$, and applying the elliptic estimate on each annulus separately. The estimates for $\partial_\Sigma^k\RbfT^j\phi$ are proved similarly. 

To upgrade the size one derivatives $\partial_\Sigma$ to vectorfield derivatives $X$ in the exterior, we argue as follows. Suppose we are commuting $X^k$ for some $k>0$. In view of Lemma~\ref{lem:commutators2} we can arrange to commute first the $T$ vectorfields, then the $\Omega$ vectorfields, and last the $\tilr L$ vectorfields. The case where all vectorfields are $T$ was already discussed above. When there are $\Omega$ vectorfields, but no $\tilr L$ vectorfields the argument is similar with the following points to keep in mind. First, since the equation in Lemma~\ref{lem:higherordereqs1} was calculated in terms of the conjugated variable $\tilvarphi=\tilr^{\frac{n-1}{2}}\varphi$, we can directly carry out the energy and LED multiplier arguments in this setting. Indeed, with $\chi$ denoting a cutoff supported in the hyperboloidal region of the foliation, for the energy estimate we multiply the equation by $\chi TX^k\tilvarphi$ while for the LED we use the two multipliers $\chi \beta (L-\Lbar)X^l\tilvarphi$ (for the analogue of \eqref{eq:psifarLEDtemp1.5}) and $\chi \beta' \tilvarphi$ (for the analogue of \eqref{eq:psifarLEDtemp2.5}). The errors which result from the derivatives falling on $\chi$ during the integration by parts are then absorbed by the LED estimate for the size one derivatives. Except for the treatment of the error terms of the form $\callO(\dotwp^{(\ell)})LX^m\tilvarphi$ on the right-hand side of \eqref{eq:errorhigher1} (note that since we are not yet commuting $\tilr L$ the terms involving $\tilcalP_1$ are not present), the remainder of the energy and LED estimate are similar to what has already been carried out, so we omit the details (see also below for the case $X=\tilr L$ where more details are worked out). The difficulty with  $\callO(\dotwp^{(\ell)})LX^m\tilvarphi$ errors is that for the part of the multiplier which is of the form $\callO(1)\Lbar X^k\tilvarphi$ (with $k\neq m$) we cannot simply use the $\uptau$ decay of $\dotwp$ to bound this by the energy, as the unweighted $\Lbar$ derivatives are not bounded by the energy flux. For the term $\callO(\dotwp^{(\ell)})LX^m\tilvarphi$, recall that in \eqref{eq:phienergyLEDboundtemp1}--\eqref{eq:T2phienergyLEDboundtemp1} we want to estimate the corresponding contribution by $\epsilon^2\sigma_1^{-2-2j}$ for $\RbfT^j\phi$, $j=1,2$ (in the more difficult case $3j+k\leq M-2$). The corresponding term we need to estimate in the multiplier identities is then the space-time integral of
%%%%%%%
%%%%%%%
\begin{align*}
\begin{split}
\callO(\dotwp^{(1+i)})(LX^m\RbfT^{j-i}\tilvarphi) (\Lbar X^k \RbfT^j\tilvarphi),\quad 0\leq i\leq j.
\end{split}
\end{align*}
Considering the extreme cases $i=j$ and $i=0$, and with the same notation as above this is bounded, using Lemma~\ref{lem:OmegaiTkphi} and the bootstrap assumptions \eqref{eq:a+trap}--\eqref{eq:Tjphienergyb2}, by (here the $\tilr^{n-1}$ part of the measure in the $L^2_y$ and $LE$ norms is already incorporated in $\tilvarphi$)
%%%%%%%
%%%%%%%
\begin{align*}
\begin{split}
\|\dotwp^{(1+j)}\|_{L^2_\uptau}\|\tilr^{\frac{1+\alpha}{2}}LX^m\tilvarphi\|_{L^\infty_\uptau L^2_y}\|X^m\RbfT^j\tilvarphi\|_{LE}\lesssim \epsilon \delta_\wp \sigma_1^{-\frac{1-\alpha}{2}} \|X^m\RbfT^j\tilvarphi\|_{LE}^2\lesssim \epsilon^3 \sigma_1^{-2j-2},
\end{split}
\end{align*}
when $i=j$, and
%%%%%%%
%%%%%%%
\begin{align*}
\begin{split}
\|\dotwp\|_{L^2_\uptau}\|\tilr^{\frac{1+\alpha}{2}}LX^m\RbfT^j\tilvarphi\|_{L^\infty_\uptau L^2_y}\|X^m\RbfT^j\tilvarphi\|_{LE}\lesssim \epsilon^2\sigma_1^{-2+\kappa}\sigma_1^{-j+\frac{1+\alpha}{2}}\|X^m\RbfT^j\tilvarphi\|_{LE}\lesssim \epsilon^3 \sigma_1^{-2j-2},
\end{split}
\end{align*}
when $i=0$.

Finally, we consider the case where some of the $X$ vectorfields are $\tilr L$. The only difference with what was already considered above is that now we have to deal with the contribution of $\tilcalP_1$ in \eqref{eq:errorhigher1}. For this, we prove the energy and LED estimate simultaneously, using the multiplier $\chi \beta L X^k\tilvarphi$, where $\chi$ is as above. The error terms when derivatives fall on $\chi$ can again be absorbed using the LED estimate for derivatives of size one. For simplicity of notation we consider first the case $\tilr L\tilvarphi$ (that is, when $k=1$ and $X=\tilr L$), and then the case $(\tilr L)^2\tilvarphi$ to demonstrate how to treat higher powers of $\tilr L$ inductively. One can of course replace $\tilvarphi$ by $X^j\tilvarphi$ where $X^j$ are a string of $\Omega$ and $T$ vectorfields. In the case $k=1$, a calculation using Lemma~\ref{eq:errorhigher1} gives (here $X=\tilr L$, and we are using the notation of Lemma~\ref{lem:higherordereqs1})
%%%%%%%
%%%%%%%
\begin{align}\label{eq:LEXtemp1}
\begin{split}
LX\tilvarphi(\tilcalP X\tilvarphi-\tilcalP_1\tilvarphi)&=-\frac{1}{2}\Lbar (LX\tilvarphi)^2+L(\Err_L)+\Omega (\Err_\Omega)\\
&\quad-\frac{1}{\tilr}\sum (\frac{1}{\tilr}\Omega X\tilvarphi)^2-\frac{(n-1)(n-3)}{4\tilr}(\frac{\tilvarphi}{\tilr})^2-\tilr (L^2\tilvarphi)^2-\frac{1}{\tilr}(L\Omega\tilvarphi)^2\\
&\quad+\callO(\tilr^{-1})(\tilr^{-1}\Omega\tilvarphi)^2+\callO(\tilr^{-1})(\tilr^{-1}\tilvarphi)^2+\callO(\dotwp)\Err_\wp,
\end{split}
\end{align}
where 
%%%%%%%
%%%%%%%
\begin{align*}
\begin{split}
|\Err_\wp|+|\Err_L|&\lesssim (LX\tilvarphi)^2+(L\tilvarphi)^2+(L\Omega\tilvarphi)^2+(\tilr^{-1}\Omega X\tilvarphi)^2\\
&\quad+(\tilr^{-1}\Omega\tilvarphi)^2+(\tilr^{-1}\Lbar X\tilvarphi)^2+(\tilr^{-1}\Lbar\tilvarphi)^2+(\tilr^{-1}X\tilvarphi)^2+(\tilr^{-1}\tilvarphi)^2.
\end{split}
\end{align*}
Multiplying \eqref{eq:LEXtemp1} by $\chi$, integrating, and adding a multiple of the LED estimate for size one derivatives, we get control of (the remaining error terms in \eqref{eq:errorhigher1} have better $\uptau$ or $\tilr$ decay and can be handled more easily)
%%%%%%%
%%%%%%%
\begin{align*}
\begin{split}
\int_{\Sigma_\tau}\chi (LX\tilvarphi)^2\ud\theta\ud r\quad \mand \int_{\sigma_1}^{\sigma_2}\int_{\Sigma_\tau}\chi (\tilr(L^2\tilvarphi)^2+\tilr^{-1}(\frac{\Omega}{\tilr} X\tilvarphi)^2+\tilr^{-1}(\frac{X\tilvarphi}{\tilr})^2)\ud\theta\ud r \ud\tau.
\end{split}
\end{align*}
Note that since we already have control of the $LE$ norm of $\phi$, the bulk term $\tilr (L^2\tilvarphi)^2$ gives control of the term $\tilr^{-1-\alpha}(LX\tilvarphi)^2$, but we will need this stronger estimate to treat the higher powers of $\tilr L$ inductively. To control the remaining terms in the energy and local energy norms we argue as follows. First, for the energy norm, note that
%%%%%%%
%%%%%%%
\begin{align*}
\begin{split}
\frac{1}{\tilr}\Omega X\tilvarphi = L\Omega\tilvarphi + [\Omega,L]\tilvarphi\quad \mand\quad\frac{1}{\tilr}T X \tilvarphi= LT\tilvarphi + \frac{T\tilr}{\tilr} L\tilvarphi + [T,L]\tilvarphi,
\end{split}
\end{align*}
and all of the terms on the right-hand sides are already controlled by the energies of $\tilvarphi$, $T\tilvarphi$, and $\Omega\tilvarphi$. Similarly, for the local energy norm,
%%%%%%%
%%%%%%%
\begin{align*}
\begin{split}
\tilr^{-1-\alpha}(\Lbar X\tilvarphi)^2\lesssim \tilr^{-1-\alpha}(\Lbar \tilr)^2(L\tilvarphi)^2+\tilr^{1-\alpha}(\Lbar L\tilvarphi).
\end{split}
\end{align*}
The first term is already controlled by the local energy norm of $\phi$, while for the second term we use the equation for $\tilvarphi$ to replace $\Lbar L\tilvarphi$ by terms which we have already estimated. Next, we consider the error terms when commuting $X^2$, $X=\tilr L$. The term in the multiplier argument that needs a different treatment is $\tilcalP_1\tilvarphi LX^2\tilvarphi$, where we no longer want to use the sign of the coefficient $c_{0,2}$ in \eqref{eq:errorhigher1}. These error terms can be estimated using the space-time control of $\tilr (L^2\tilvarphi)^2$ above. For instance, the term $L^2\tilvarphi$ in $\tilcalP_1$ contributes terms of the form
%%%%%%%
%%%%%%%
\begin{align*}
\begin{split}
\tilr^{2-j}L^2\tilvarphi L^{3-j}\tilvarphi,\qquad 0\leq j\leq 2,
\end{split}
\end{align*} 
all of which can be estimated in terms of $r(L^2\tilvarphi)^2$ and the $LE$ norm of $\tilvarphi$ after a few integration by parts. The other terms in $\tilcalP_1$ are treated similarly. We can now proceed inductively to prove energy and LED estimates for higher powers of $\tilr L$.
\end{proof}

%%%%%%%%%%%%%%
%%%%%%%%%%%%%%
\subsection{Proof of Proposition~\ref{prop:bootstrapphi1}}
%%%%%%%%%%%%%%
%%%%%%%%%%%%%%
%%%%%%%%%%%
%%%%%%%%%%%
We start by proving $\tau^{-2}$ decay for the energy at lower orders, and boundedness at higher orders.
%%%%%%%%%%%
%%%%%%%%%%%
\begin{lemma}\label{lem:energytau2decay}
Suppose the bootstrap assumptions \eqref{eq:a+trap}--\eqref{eq:Tjphienergyb2} hold. Then for $k\leq M-2$,
%%%%%%%
%%%%%%%
\begin{align}\label{eq:rpenergydecay1}
\begin{split}
E_k[\phi](\tau)\lesssim \epsilon^2\tau^{-2},
\end{split}
\end{align}
and
%%%%%%%
%%%%%%%
\begin{align}\label{eq:rpenergydecay2}
\begin{split}
\int_{\Sigma_\tau}\chi_{\geq \tilR}((L+\frac{n-1}{2})X^k\phi)^2 r\ud V\lesssim \epsilon \tau^{-1}.
\end{split}
\end{align}
Moreover, for any $k\leq M$,
%%%%%%%
%%%%%%%
\begin{align}
E_k[\phi](\tau)+\int_{\Sigma_\tau}\chi_{\geq \tilR}\big[r^2((L+\frac{n-1}{2})X^k\phi)^2+r(r^{-1}\Omega X^k\phi)^2+r(r^{-1}X^k\phi)^2\big]\ud V\lesssim \epsilon^2.\label{eq:rpenergyboundedness1}
\end{align}
\end{lemma}
%%%%%%%%%%%
%%%%%%%%%%%
\begin{proof}
The estimate for $E_k$ in \eqref{eq:rpenergyboundedness1} follows directly from \eqref{eq:phienergyLEDboundtemp1}. The estimate for the second term on the left-hand side of \eqref{eq:rpenergyboundedness1} follows from Lemma~\ref{lem:rpmult1}. Here the error terms in estimate \eqref{eq:rpmult1} are absorbed by adding a suitable multiple of the local energy bound in \eqref{eq:phienergyLEDboundtemp1}, and the contribution of $\tilf_k$ in \eqref{eq:rpmult1} is treated in the same way as in the proof of \eqref{eq:rpenergydecay1} below. We turn to the details for the proof of \eqref{eq:rpenergydecay1}. We introduce some auxiliary notation to avoid repeated long expressions in the proof:
%%%%%%%
%%%%%%%
\begin{align*}
\begin{split}
&\scE_k^p(\tau):=\int_{\Sigma_\tau}\chi_{\leq \tilR}|\partial \partial^k\phi|^2\ud V+\int_{\Sigma_\tau}\chi_{\geq \tilR}((L+\frac{n-1}{2})X^k\phi)^2 r^p\ud V,\\
&\scB_k^p(\tau):=\int_{\Sigma_\tau}\chi_{\leq \tilR}|\partial \partial^k\phi|^2\ud V+\int_{\Sigma_\tau}\chi_{\geq \tilR}\big[((L+\frac{n-1}{2})X^k\phi)^2 +(2-p)(r^{-1}\Omega X^k\phi)^2
\\
&\phantom{\scB_k^p(\tau):=\int_{\Sigma_\tau}\chi_{\leq \tilR}|\partial \partial^k\phi|^2\ud V+\int_{\Sigma_\tau}\chi_{\geq \tilR}\big[}\quad+(2-p)(r^{-1}X^k\phi)^2+r^{-p-\alpha}(TX^k\phi)^2\big]r^{p-1}\ud V.
\end{split}
\end{align*}
Adding a suitable multiple of the LED estimate \eqref{eq:phienergyLEDboundtemp1} at one higher order (to compensate for the degeneracy in the LE norm) to \eqref{eq:rpmult1} in Lemma~\ref{lem:rpmult1} with $\psi=\varphi$, for any $p\leq 2$ and $\sigma_1<\sigma_2$ gives, 
%%%%%%%
%%%%%%%
\begin{align}\label{eq:energytau2decaytemp0}
\begin{split}
\sum_{j\leq k}\scE_j^p(\sigma_2)+\sum_{j\leq k}\int_{\sigma_1}^{\sigma_2}\scB_j^p(\tau)\ud \tau\lesssim\sum_{j\leq k+1} \scE_{j}^p(\sigma_1)+\sum_{j\leq k}\Big|\int_{\sigma_1}^{\sigma_2}\int_{\Sigma_\tau}\chi_{\geq \tilR}\tilf_j L\tilfy_j \tilr^p\ud\theta\ud r\ud\tau\Big|.
\end{split}
\end{align}
Applying this identity with $p=2$, $k= M-1$, $\sigma_1=0$, and $\sigma_2=\sigma$ arbitrary, and using the boundedness of $\scE_{j+1}^2$ for $j\leq M-1$, we get
%%%%%%%
%%%%%%%
\begin{align}\label{eq:energytau2decaytemp1}
\begin{split}
\sum_{j\leq {M-1}}\scE_j^2(\sigma)+\sum_{j\leq M-1}\int_{0}^{\sigma}\scB_j^2(\tau)\ud \tau\lesssim \epsilon^2+\sum_{j\leq M-1}\Big|\int_{0}^{\sigma}\int_{\Sigma_\tau}\chi_{\geq \tilR}\tilf_j L\tilfy_j \tilr^2\ud\theta\ud r\ud\tau\Big|.
\end{split}
\end{align}
We claim that the contribution of the last term is bounded or can be absorbed on the left. Here and in what follows we carry out the details for the proof of the calculations for the two representative terms for $\tilf_k$ corresponding to the contributions of the source term $\calF_0$ and cubic term $\nabla^\mu\varphi\nabla^\nu\varphi\nabla_{\mu\nu}\varphi$. See \eqref{eq:abstractexteq1} and Lemmas~\ref{lem:sourceext1} and \ref{lem:purelycubic1}. For the contribution of $\calF_0$ we simply use  estimate \eqref{eq:calF0tauintegratedtemp1}, and absorb the $LE$ norm on the left. For the contribution of the cubic term, in view of Lemma~\ref{lem:purelycubic1} and the bootstrap assumptions \eqref{eq:wpb1}, \eqref{eq:phiptwiseb1}, \eqref{eq:dphiptwiseb2}, \eqref{eq:dphiptwiseb3}, these satisfy the same types of estimates as the error terms on the right-hand side of \eqref{eq:calPVF1}, with extra additional smallness, so their contribution can be bounded in the same way as in the proof of Lemma~\ref{lem:rpmult1}. Going back to \eqref{eq:energytau2decaytemp1}, we conclude that the left-hand side of this estimate is bounded by~$\epsilon^2$, and therefore, there is an increasing sequence of dyadic $\tiltau_m$ such that $\scB_j^2(\tiltau_n)\lesssim \tiltau_n^{-1}\epsilon^2$ for $j\leq M-1$. Since $\scE_j^1\lesssim \scB_j^2$, another application of \eqref{eq:energytau2decaytemp1}, but on $[\tiltau_{m-1},\tiltau_m]$ and with $p=1$ and $k=M-2$ gives
%%%%%%%
%%%%%%%
\begin{align}\label{eq:energytau2decaytemp2}
\begin{split}
\sum_{j\leq M-2}\scE_j^1(\tiltau_m)+\sum_{j\leq M-2}\int_{\tiltau_{m-1}}^{\tiltau_m}\scB_j^1(\tau)\ud \tau\lesssim \epsilon\tiltau_m^{-1}+\sum_{j\leq M-2}\Big|\int_{\tiltau_{m-1}}^{\tiltau_m}\int_{\Sigma_\tau}\chi_{\geq \tilR}\tilf_j L\tilfy_j \tilr\ud\theta\ud r\ud\tau\Big|.
\end{split}
\end{align}
Arguing as above, the contribution of the last term on the right can be absorbed or bounded by~$\epsilon^2\tiltau_m^{-1}$, and we can find a possibly different increasing dyadic sequence $\tau_m\simeq \tiltau_m$ such that $\scB_j^1(\tau_m)\lesssim \epsilon^2\tau_m^{-2}$ for $j\leq M-2$. Since $E_j\lesssim \scB_j^1$, estimate \eqref{eq:rpenergydecay1} follows from another application of the energy estimate \eqref{eq:phienergyLEDboundtemp1}. Finally for \eqref{eq:rpenergydecay2}, note that by \eqref{eq:energytau2decaytemp2} we already have this estimate for $\tau=\tau_m$. The estimate for all $\tau$ now follows from another application of \eqref{eq:energytau2decaytemp0} with $p=1$ and with arbitrary $\sigma_2\in(\tau_m,\tau_{m+1})$ and $\sigma_1=\tau_m$.
\end{proof}
%%%%%%%%%%%
%%%%%%%%%%%
To improve the pointwise decay assumptions \eqref{eq:wpb1}, \eqref{eq:phiptwiseb1}, \eqref{eq:dphiptwiseb2}, \eqref{eq:dphiptwiseb3} we need better decay for the energies of $\RbfT\phi$ and $\RbfT^2\phi$. This is the content of the next lemma.
%%%%%%%%%%%
%%%%%%%%%%%
\begin{lemma}\label{lem:energyhighertaudecay}
Suppose the bootstrap assumptions \eqref{eq:a+trap}--\eqref{eq:Tjphienergyb2} hold. Then for any $k\leq M-5$,
%%%%%%%
%%%%%%%
\begin{align}\label{eq:Tenergytau4decay}
\begin{split}
E_k[\RbfT\phi](\tau)\lesssim \epsilon^2\tau^{-4},
\end{split}
\end{align}
and for any $k\leq M-8$, (with constant that is independent of \eqref{eq:d2Tphienergyb1})
%%%%%%%
%%%%%%%
\begin{align}\label{eq:T2energytau5decay}
\begin{split}
E_k[\RbfT^2\phi](\tau)\lesssim \epsilon^2\tau^{-6}.
\end{split}
\end{align}
\end{lemma}
%%%%%%%%%%%
%%%%%%%%%%%
\begin{proof}
The main observation is that in view of Lemma~\ref{lem:higherordereqs1}, in particular using equation \eqref{eq:higher1}, we can estimate (note that in view of \eqref{eq:Tprecise1} the difference between $\RbfT$ and $T$ comes with factors of $\dotwp$ which give extra decay)
%%%%%%%
%%%%%%%
\begin{align*}
\begin{split}
\sum_{j\leq k}\scE^2_j[\RbfT\phi](\tau)\lesssim \epsilon^2\tau^{-2}+\sum_{j\leq k+1}E_j[\phi](\tau)\lesssim \epsilon^2\tau^{-2}.
\end{split}
\end{align*}
Here for the first inequality we have used the bootstrap assumptions \eqref{eq:wpb1}, \eqref{eq:phiptwiseb1}, \eqref{eq:dphiptwiseb2}, \eqref{eq:dphiptwiseb3} as well as \eqref{eq:calF0constanttautemp1} to estimate the left-hand side of \eqref{eq:higher1}, and for the second in equality we have used Lemma~\ref{lem:energytau2decay}. We can now repeat the proof of Lemma~\ref{lem:energytau2decay}, starting by applying \eqref{eq:energytau2decaytemp0} applied to $T\phi$ on an increasing dyadic sequence $\tau_m$. By the observations we just made, the right-hand side is now bounded by $\tau_m^{-2}$, so repeating the proof of Lemma~\ref{lem:energytau2decay} we obtain \eqref{eq:Tenergytau4decay}. Returning to \eqref{eq:higher1} and repeating this argument we obtain \eqref{eq:T2energytau5decay}.
\end{proof}
%%%%%%%%%%%
%%%%%%%%%%%
Lemma~\ref{lem:energyhighertaudecay} and elliptic estimates contained in the next lemma
allow us to obtain decay of higher derivative norms of $\phi$ for arbitrary derivatives. To state the lemma we recall from Remark~\ref{rem:nongeomglobal1}, part~(1), that in the global coordinates $(\uptau,\uprho,\upomega)$, $\calP$ admits the decomposition
%%%%%%%
%%%%%%%
\begin{align*}
\begin{split}
\calP=\calP_\uptau+\calP_\Ell,
\end{split}
\end{align*}
satisfying the properties stated there.
%%%%%%%%%%%
%%%%%%%%%%%
\begin{lemma}\label{lem:ellipticestimate1}
Suppose $\calP_\Ell\uppsi=g$ on $\Sigma_\tau$, and that the bootstrap assumptions~\eqref{eq:a+trap}--\eqref{eq:Tjphienergyb2} hold, then (here the sum is over the $L^2(\Sigma_\tau)$ inner products with the truncated eigenfunctions $Z_\mu,Z_1,\dots,Z_n$)
%%%%%%%
%%%%%%%
\begin{align}\label{eq:ellipticestimate1}
\begin{split}
\|\jap{r}^{-2}\uppsi\|_{L^2(\Sigma_\tau)}+\|\partial^2\uppsi\|_{L^2(\Sigma_\tau)}\lesssim \|g\|_{L^2(\Sigma_\tau)}+\sum|\angles{\uppsi}{Z_i}|+\tau^{-\frac{5}{2}+\kappa}E(\uppsi),
\end{split}
\end{align}
and for $s\in(2,\frac{5}{2})$,
%%%%%%%
%%%%%%%
\begin{align}\label{eq:ellipticestimate1.5}
\begin{split}
\|\jap{r}^{-s}\uppsi\|_{L^2(\Sigma_\tau)}\lesssim \|\partial_\Sigma g\|_{L^2(\Sigma_\tau)}^{s-2}\| g\|_{L^2(\Sigma_\tau)}^{3-s}+\sum|\angles{\uppsi}{Z_i}|+\tau^{-\frac{5}{2}+\kappa}\sum_{j\leq 1}E_j(\uppsi).
\end{split}
\end{align}
Moreover, for small $\tilkappa$,
%%%%%%%
%%%%%%%
\begin{align}\label{eq:ellipticestimate2}
\begin{split}
\|\partial_\Sigma^3\uppsi\|_{L^2(\Sigma_\tau)}\lesssim \|\partial_\Sigma g\|_{L^2(\Sigma_\tau)}+ \|\partial_\Sigma g\|_{L^2(\Sigma_\tau)}^{\frac{1}{2}-\tilkappa} \|g\|_{L^2(\Sigma_\tau)}^{\frac{1}{2}+\tilkappa}+\sum|\angles{\uppsi}{Z_i}|+\tau^{-\frac{5}{2}+\kappa}\sum_{j\leq 1}E_j(\uppsi).
\end{split}
\end{align}
\end{lemma}
%%%%%%%%%%%
%%%%%%%%%%%
\begin{proof}
Since all norms are on $\Sigma_\tau$ we drop $\Sigma_\tau$ from the notation. Recalling the decomposition $\calP_\Ell=\Delta_\barcalC+V+\calP_\Ell^{\pert}$, let $\uppsi_\far$ be a solution to 
%%%%%%%
%%%%%%%
\begin{align*}
\begin{split}
(\Delta_\barcalC+V_\far)\uppsi_\far=g-\calP_\Ell^\pert\uppsi,
\end{split}
\end{align*}
where $V_\far$ is a potential that vanishes inside a large compact set, and is equal to $V$ outsides a larger compact set. We further decompose $g-\calP_\Ell^\pert\uppsi$ as
%%%%%%%
%%%%%%%
\begin{align*}
\begin{split}
g-\calP_\Ell^\pert\uppsi=\tilg+o_{\wp,\Rone}(1)(\partial^2_\Sigma\uppsi+\jap{\uprho}^{-1}\partial_\Sigma\uppsi+\jap{\uprho}^{-2}\uppsi),
\end{split}
\end{align*}
where $\tilg=g+\callO(\dotwp)\partial_\Sigma\uppsi+\callO(\dotwp)\jap{\uprho}^{-2}\uppsi$.
We will prove estimates \eqref{eq:ellipticestimate1}, \eqref{eq:ellipticestimate1.5}, \eqref{eq:ellipticestimate2} with the last term on the right-hand sides removed, and with $g$ replaced by $\tilg$. The desired estimate then follows from the triangle inequality and the bootstrap assumptions. Treating  $V_\far$ perturbatively, and writing $\varepsilon$ for $o_{\wp,\Rone}(1)$, we have
%%%%%%%
%%%%%%%
\begin{align}\label{eq:lemellipticestimate1temp1}
\begin{split}
\|\jap{\uprho}^{-2}\uppsi_\far\|_{L^2}+\|\partial_\Sigma^2\uppsi_\far\|_{L^2}\lesssim \|\tilg\|_{L^2}+\varepsilon (\partial^2_\Sigma\uppsi+\jap{\uprho}^{-1}\partial_\Sigma\uppsi+\jap{\uprho}^{-2}\uppsi).
\end{split}
\end{align}
Note that $\uppsi_\near:=\uppsi-\uppsi_\far$ satisfies
%%%%%%%
%%%%%%%
\begin{align}\label{eq:lemellipticestimate1temp2}
\begin{split}
(\Delta_\barcalC+V)\uppsi_\near=V_\near \uppsi_\far,
\end{split}
\end{align}
where $V_\near=V_\far-V$. We further decompose $\uppsi_\near$ as (where the sum is over the truncated eigenfunctions $Z_i$ of $\Delta_\barcalC+V$, which we assumed are normalized in $L^2(\Sigma_\tau)$)
%%%%%%%
%%%%%%%
\begin{align*}
\begin{split}
&\uppsi_\near=\uppsi_\near^\perp+\sum\angles{\uppsi_\near}{Z_i}Z_i,\qquad \angles{\uppsi_\near^\perp}{Z_i}=0,\\
&(\Delta_\barcalC+V)\uppsi_\near^\perp=V_\near \uppsi_\far+\sum\angles{\uppsi_\near}{Z_i}(\Delta_\barcalC+V)Z_i.
\end{split}
\end{align*}
Since $\uppsi_\near^\perp$ is transversal to the eigenfunctions of $\Delta_\barcalC+V$,
%%%%%%%
%%%%%%%
\begin{align*}
\begin{split}
\|\jap{\uprho}^{-1}\uppsi_\near^\perp\|_{L^2}&\lesssim \|\partial_\Sigma\uppsi_\near^\perp\|_{L^2}\lesssim \|\jap{\uprho}V_\near \uppsi_\far\|_{L^2}+\sum|\angles{\uppsi_\near}{Z_i}|\\
&\lesssim \|\tilg\|_{L^2}+\sum|\angles{\uppsi}{Z_i}|+\varepsilon (\partial^2_\Sigma\uppsi+\jap{\uprho}^{-1}\partial_\Sigma\uppsi+\jap{\uprho}^{-2}\uppsi),
\end{split}
\end{align*}
where  we have used the fact that $V_\near$ and $Z_i$ are compactly supported. But using the decomposition $\uppsi_\near=\uppsi_\near^\perp+\sum\angles{\uppsi_\near}{Z_i}Z_i$ we can replace $\|\jap{\uprho}^{-1}\uppsi_\near^\perp\|_{L^2}$ on the left-hand side of the estimate above by $\|\jap{\uprho}^{-1}\uppsi_\near\|_{L^2}$. Using equation~\eqref{eq:lemellipticestimate1temp2} again,
%%%%%%%
%%%%%%%
\begin{align*}
\begin{split}
\|\partial^2_\Sigma \uppsi_\near\|_{L^2}&\lesssim \|\jap{\uprho}^{-2}\uppsi_\near\|_{L^2}+\|\jap{\uprho}^{-2}\uppsi_\far\|_{L^2}\\
&\lesssim \|\tilg\|_{L^2}+\sum|\angles{\uppsi}{Z_i}|+\varepsilon (\partial^2_\Sigma\uppsi+\jap{\uprho}^{-1}\partial_\Sigma\uppsi+\jap{\uprho}^{-2}\uppsi).
\end{split}
\end{align*}
Estimate \eqref{eq:ellipticestimate1} follows by combining the last two estimates with \eqref{eq:lemellipticestimate1temp1}  and observing that $|\jap{\uprho}^{-2}\uppsi|\lesssim |\jap{\uprho}^{-2}\uppsi_\near|+|\jap{\uprho}^{-2}\uppsi_\far|$. To prove \eqref{eq:ellipticestimate1.5} and \eqref{eq:ellipticestimate2} we use the global coordinates $(\uptau,\uprho,\upomega)$, with $(\uprho,\upomega)$ coordinates on $\Sigma_\tau$, to define the operator $\calP_\euc$ by smoothly modifying the coefficients of $\calP_\Ell$ such that
%%%%%%%
%%%%%%%
\begin{align*}
\begin{split}
\calP_\euc=\begin{cases}\Delta_{\euc}\quad &\uprho\leq R_\euc\\ \calP_\Ell\quad &\uprho\geq R_\euc+1\end{cases}.
\end{split}
\end{align*}
Here $R_\euc$ is a fixed large constant and $\Delta_{\euc}=\partial_\uprho^2+\frac{4}{\uprho}\partial_\uprho+\frac{1}{\uprho^2}\sDelta_{\bbS^{4}}$. Let $\chi\equiv \chi(\uprho)$ and $\tilchi\equiv(\uprho)$ be cutoff functions supported in the large $\uprho$ region such that $\tilchi\chi=\chi$, let $\upphi_\euc$ be the solution to
%%%%%%%
%%%%%%%
\begin{align*}
\begin{split}
\calP_\euc\upphi_\euc= \chi \tilg,
\end{split}
\end{align*}
and let $\uppsi_\euc=\tilchi\upphi_\euc$. The functions $\upphi_\euc$ and $\uppsi_\euc$ are defined in terms of the coordinates $(\uprho,\upomega)$, but since $\uppsi_\euc$ is supported in the large $\uprho$ region, we can view it as a function on $\Sigma_\tau$ as well. By the Euclidean theory and treating the difference between $\Delta_\euc$ and $\calP_\euc$ in $\{\uprho\geq R_\euc\}$ perturbativly (here fractional derivatives are defined on $\bbR^5$ using the coordinates $(\uprho,\upomega)$, $\partial_\Sigma$ denotes the coordinate derivatives $\partial_\uprho$ and $\uprho^{-1}\partial_\upomega$, and the volume form is also as in $\bbR^5$, which is comparable with the geometric volume form on $\Sigma_\tau$ for large $\uprho$),
%%%%%%%
%%%%%%%
\begin{align}\label{eq:lemellipticestimate1temp4}
\begin{split}
\sum_{j=0}^2\|\jap{\uprho}^{j-s}\partial_\Sigma^j\upphi_\euc\|_{L^2}+\|(-\Delta_{\euc})^{\frac{s}{2}}\upphi_\euc\|_{L^2}\lesssim \|\tilchi\partial_\Sigma \tilg\|_{L^2}^{s-2}\|\tilchi\tilg\|_{L^2}^{3-s}.
\end{split}
\end{align}
Now $\uppsi_\cat:=\uppsi-\uppsi_\euc$ satisfies
%%%%%%%
%%%%%%%
\begin{align}\label{eq:lemellipticestimate1temp5}
\begin{split}
\calP_\Ell\uppsi_\cat=(1-\chi)\tilg-[\calP_\Ell,\tilchi]\upphi_\euc-\tilchi(\calP_\Ell-\calP_\euc)\upphi_\euc.
\end{split}
\end{align}
We again decompose $\uppsi_\cat$ as $\uppsi_\cat=\uppsi_\cat^\perp+\angles{\uppsi_\cat}{Z_i}Z_i$. In view of the compact support of the right-hand side of \eqref{eq:lemellipticestimate1temp5} (note that the terms involving $\upphi_\euc$ are compactly supported in the large $\uprho$ region, so they can be viewed as functions on $\Sigma_\tau$), and by \eqref{eq:ellipticestimate1} and \eqref{eq:lemellipticestimate1temp4}, and arguing as we did for $\uppsi_\near$ above,
%%%%%%%
%%%%%%%
\begin{align*}
\begin{split}
\|\jap{\uprho}^{-1}\uppsi_\cat\|_{L^2}\lesssim \|\partial_\Sigma\uppsi_\cat\|_{L^2}\lesssim \|\partial_\Sigma \tilg\|_{L^2}^{s-2}\|\tilg\|_{L^2}^{3-s}+\sum|\angles{\uppsi}{Z_i}|,
\end{split}
\end{align*}
and \eqref{eq:ellipticestimate1.5} follows by observing that $|\jap{\uprho}^{-s}\uppsi|\lesssim |\jap{\uprho}^{-s}\uppsi_\euc|+|\jap{\uprho}^{-1}\uppsi_\cat|$. Note that the same argument in fact gives \eqref{eq:ellipticestimate1.5} with $\|\jap{\uprho}^{1-s}\partial_\Sigma \uppsi\|_{L^2}$ added on the left-hand side. 
Therefore, to prove \eqref{eq:ellipticestimate2}, in view of the equation $\Delta_\barcalC\uppsi=-V\uppsi+\tilg+o_{\wp,\Rone}(1)(\partial^2_\Sigma\uppsi+\jap{\uprho}^{-1}\partial_\Sigma\uppsi+\jap{\uprho}^{-2}\uppsi)$, we can use the already established estimates and elliptic estimates for $\Delta_\barcalC$, to get
%%%%%%%
%%%%%%%
\begin{align*}
\begin{split}
\|\partial_\Sigma^3\uppsi\|_{L^2}&\lesssim \|\partial_\Sigma \tilg\|_{L^2}+\|\partial_\Sigma( V\uppsi)\|_{L^2}\lesssim \|\partial_\Sigma \tilg\|_{L^2}+\|\jap{\uprho}^{-\frac{5}{2}+\tilkappa}\uppsi\|_{L^2}+\|\jap{\uprho}^{1-\frac{5}{2}+\tilkappa}\partial_\Sigma\uppsi\|_{L^2} \\
&\lesssim\|\partial_\Sigma \tilg\|_{L^2}+\|\partial_\Sigma \tilg\|_{L^2}^{\frac{1}{2}-\kappa}\|\tilg\|_{L^2}^{\frac{1}{2}+\kappa}+\sum|\angles{\uppsi}{Z_i}|.\qedhere
\end{split}
\end{align*}
\end{proof}
%%%%%%%%%%%
%%%%%%%%%%%
The $L^2(\Sigma_\tau)$ decay of arbitrary higher derivatives is now a corollary of the previous two lemmas.
%%%%%%%%%%%
%%%%%%%%%%%
\begin{corollary}\label{cor:higherdL2decay}
Suppose the bootstrap assumptions \eqref{eq:a+trap}--\eqref{eq:Tjphienergyb2} hold. Then for $k\leq M-3$,
%%%%%%%
%%%%%%%
\begin{align}\label{eq:d2highertaudecay1}
\begin{split}
\|\partial^2_\Sigma (\chi_{\leq \tilR}\partial^k\phi)\|_{L^2(\Sigma_\tau)}+\|\partial^2_\Sigma(\chi_{\geq \tilR}X^k\phi)\|_{L^2(\Sigma_\tau)}\lesssim \epsilon\tau^{-2},
\end{split}
\end{align}
and for $k\leq M-4$,
%%%%%%%
%%%%%%%
\begin{align}\label{eq:d2Thighertaudecay1}
\begin{split}
\|\partial^2_\Sigma (\chi_{\leq \tilR}\RbfT\partial^k\phi)\|_{L^2(\Sigma_\tau)}+\|\partial^2_\Sigma(\chi_{\geq \tilR}\RbfT X^k\phi)\|_{L^2(\Sigma_\tau)}\lesssim \epsilon\tau^{-3}.
\end{split}
\end{align} 
For $k\leq M-4$ and $s\geq \frac{5}{2}-\kappa$,
%%%%%%%
%%%%%%%
\begin{align}\label{eq:d3highertaudecay1}
\begin{split}
&\|\partial^3_\Sigma (\chi_{\leq \tilR}\partial^k\phi)\|_{L^2(\Sigma_\tau)}+\|\partial^3_\Sigma(\chi_{\geq \tilR}X^k\phi)\|_{L^2(\Sigma_\tau)}\\
&+\| \chi_{\leq \tilR}\partial^k\phi\|_{L^2(\Sigma_\tau)}+\|\jap{r}^{-s}(\chi_{\geq \tilR}X^k\phi)\|_{L^2(\Sigma_\tau)}\lesssim \epsilon\tau^{-\frac{5}{2}+\kappa}.
\end{split}
\end{align}
The implicit constants in \eqref{eq:d2highertaudecay1}, \eqref{eq:d2Thighertaudecay1}, \eqref{eq:d3highertaudecay1} are independent of $C_k$ in \eqref{eq:dphiL2b1} and \eqref{eq:d2Tphienergyb1}.
\end{corollary}
%%%%%%%%%%%
%%%%%%%%%%%
\begin{proof}
We start with the decomposition
%%%%%%%
%%%%%%%
\begin{align}\label{eq:higherdL2decaytemp1}
\begin{split}
\calP_\Ell\phi=\calP\phi+\callO(1)\partial \RbfT\phi+ \callO(r^{-1})\RbfT\phi.
\end{split}
\end{align}
It follows from Lemmas~\ref{lem:energyhighertaudecay} and~\ref{lem:ellipticestimate1}, as well as \eqref{eq:wpb1}, \eqref{eq:phiptwiseb1}, \eqref{eq:dphiptwiseb2}, \eqref{eq:dphiptwiseb3}, \eqref{eq:calF0constanttautemp1}, that
%%%%%%%
%%%%%%%
\begin{align}
\begin{split}
\|\partial^2_\Sigma\phi\|_{L^2(\Sigma_\tau)}\lesssim \epsilon \tau^{-2},
\end{split}
\end{align}
which, using Lemma~\ref{lem:ellipticestimate1} again, implies \eqref{eq:d2highertaudecay1} for $k=0$. Note that here the contribution of $|\angles{\phi}{Z_i}|$ is bounded in terms of $\bfOmega_i(\phi)$ and applying Lemma~\ref{lem:Omega1}. The case of higher $k$ is derived similarly using the $k$ times commuted equations. For \eqref{eq:d3highertaudecay1}, arguing as above we write
%%%%%%%
%%%%%%%
\begin{align*}
\begin{split}
\calP_\Ell \RbfT\phi=[\calP_\Ell,\RbfT]\phi+\RbfT\calP\phi+\callO(1)\partial \RbfT^2\phi+ \callO(r^{-1})\RbfT^2\phi+\callO(\dotwp)\partial\RbfT\phi+\callO(\dotwp r^{-1})\RbfT\phi,
\end{split}
\end{align*}
to get
%%%%%%%
%%%%%%%
\begin{align*}
\begin{split}
\|\partial^2 T\phi\|_{L^2(\Sigma_\tau)}\lesssim \tau^{-3}.
\end{split}
\end{align*}
Here we have again bounded $\angles{\RbfT\phi}{Z_i}$ in terms of $\bfOmega_i(\RbfT\phi)$, which is bounded by $o_{\wp,\Reigenfunctioncutoffscale}(1)\epsilon\uptau^{-3}$ by the same arguments as in Lemmas~\ref{eq:OmegaiTkphi} and~\ref{lem:Omega1}. Note that the factors $o_{\wp,\Reigenfunctioncutoffscale}(1)$ here and $\delta_\wp$ in \eqref{eq:TcalF0constanttautemp1} make the constants in this estimate independent of the bootstrap constants in \eqref{eq:d2Tphienergyb1}. By the same reasoning, and using equation \eqref{eq:higherdL2decaytemp1} and estimates \eqref{eq:calF0constanttautemp1} and \eqref{eq:ellipticestimate2} we get \eqref{eq:d3highertaudecay1} with a constant that is independent of \eqref{eq:dphiL2b1}. The estimate for higher $k$ is proved similarly.
\end{proof}
%%%%%%%%%%%
%%%%%%%%%%%
Corollary~\ref{cor:higherdL2decay} and the following standard Gagliardo-Nirenberg inequality (which we state without proof)  allow us to close the bootstrap assumptions \eqref{eq:phiptwiseb1} and \eqref{eq:dphiptwiseb1}.
%%%%%%%%%%%
%%%%%%%%%%%
\begin{lemma}\label{lem:GN}
For any function $\uppsi\in H^3(\Sigma_\tau)$, 
%%%%%%%
%%%%%%%
\begin{align*}
\begin{split}
\|\uppsi\|_{L^\infty(\Sigma_\tau)}\lesssim \|\partial_\Sigma^2\uppsi\|_{L^2(\Sigma_\tau)}^{\frac{1}{2}}\|\partial_\Sigma^3\uppsi\|_{L^2(\Sigma_\tau)}^{\frac{1}{2}}.
\end{split}
\end{align*}
\end{lemma}
%%%%%%%%%%%
%%%%%%%%%%%
%%%%%%%%%%%
%%%%%%%%%%%
We now have all the ingredients to prove Proposition~\ref{prop:bootstrapphi1}.
%%%%%%%%%%%
%%%%%%%%%%%
\begin{proof}[Proof of Proposition~\ref{prop:bootstrapphi1}]
Since the implicit constants in Corollary~\ref{cor:higherdL2decay} are independent of those in \eqref{eq:dphiL2b1} and \eqref{eq:d2Tphienergyb1}, estimates \eqref{eq:dphiL21} and \eqref{eq:d2Tphienergy1} follow if $C$ is chosen sufficiently large. Estimates \eqref{eq:phiptwise1} and \eqref{eq:dphiptwise1} then follow from \eqref{eq:dphiL21} and \eqref{eq:d2Tphienergy1} and Lemma~\ref{lem:GN}. Estimates \eqref{eq:Tjphienergy1} and \eqref{eq:Tjphienergy2} were also already proved in the proof of Lemmas~\ref{lem:energytau2decay} and~\ref{lem:energyhighertaudecay}. To prove \eqref{eq:dphiptwise2}, first note that for any function $\uppsi$ and for any~$r_1> \tilR\gg1$, by the fundamental theorem of calculus,
%%%%%%%
%%%%%%%
\begin{align*}
\begin{split}
\int_{\Sigma_\tau\cap \{r=r_1\}}(r^{\frac{3}{2}}\uppsi)^2\ud \theta&\lesssim \int_{\Sigma_\tau\cap \{r=\tilR\}}(r^{\frac{3}{2}}\uppsi)^2\ud \theta+\int_{\Sigma_\tau\cap\{\tilR\leq r\leq r_1\}}|r^{\frac{3}{2}}\uppsi||\partial_r(r^{\frac{3}{2}}\uppsi)|\ud\theta\ud r\\
&\lesssim E[\uppsi](\tau).
\end{split}
\end{align*}
Here to pass to the second line we have used the trace inequality for the integral on $\Sigma_\tau\cap \{r=\tilR\}$, and  \eqref{eq:VFbasis1} to express $\partial_r$ in terms of $L,T,\Omega$ for the integral on $\Sigma_\tau\cap\{\tilR\leq r\leq r_1\}$. Applying the Sobolev inequality on the (non-geometric) sphere $\Sigma_\tau\cap \{r=r_1\}$ to $\phi$, and using \eqref{eq:VFbasis1} to express angular derivatives in terms of $L,\Omega,T$, for $r\geq \tilR$ we get
%%%%%%%
%%%%%%%
\begin{align*}
\begin{split}
|r^{\frac{3}{2}}\phi(\tau)|^2\lesssim \sum_{j\leq 3}E_j[\phi].
\end{split}
\end{align*}
Estimate \eqref{eq:dphiptwise2} for $k=0$ now follows from \eqref{eq:rpenergydecay1}, and the estimate for higher $k$ is proved similarly. For \eqref{eq:dphiptwise3} we start with
%%%%%%%
%%%%%%%
\begin{align*}
\begin{split}
\int_{\Sigma_\tau\cap \{r=r_1\}}(\tilr^{2}\uppsi)^2\ud \theta&\lesssim \int_{\Sigma_\tau\cap \{r=\tilR\}}(\tilr^{2}\uppsi)^2\ud \theta+\int_{\Sigma_\tau\cap\{\tilR\leq r\leq r_1\}}|\tilr^{2}\uppsi||\partial_r(\tilr^{2}\uppsi)|\ud\theta\ud r\\
&\lesssim E[\uppsi](\tau)+\Big(\int_{\Sigma_\tau}\chi_{\geq \tilR}((L+\frac{n-1}{2\tilr})\uppsi)^2 r^2\ud V\Big)^{\frac{1}{2}}(E[\uppsi](\tau))^{\frac{1}{2}},
\end{split}
\end{align*}
where we have again used the trace inequality and \eqref{eq:VFbasis1} to pass to the last line. Estimate \eqref{eq:dphiptwise3} now follows by the Sobolev estimate on the sphere $\Sigma_\tau\cap \{r=r_1\}$ as above, as well as  \eqref{eq:rpenergydecay1} and ~\eqref{eq:rpenergyboundedness1}. 
\end{proof}
%%%%%%%%%%%
%%%%%%%%%%%

%%%%%%%%%%%
%%%%%%%%%%%
\bibliographystyle{plain}
\bibliography{researchbib}

\bigskip

\centerline{\scshape Jonas L\"uhrmann}
\smallskip
{\footnotesize
 \centerline{Department of Mathematics, Texas A\&M University}
 \centerline{Blocker 620B, College Station, TX 77843-3368, U.S.A.}
 \centerline{\email{luhrmann@tamu.edu}}
}

\medskip

\centerline{\scshape Sung-Jin Oh}
\smallskip
{\footnotesize
 \centerline{Department of Mathematics, UC Berkeley}
\centerline{Evans Hall 970, Berkeley, CA 94720-3840, U.S.A.}
\centerline{\email{sjoh@math.berkley.edu}}
} 

\medskip

\centerline{\scshape Sohrab Shahshahani}
\medskip
{\footnotesize
 \centerline{Department of Mathematics, University of Massachusetts, Amherst}
\centerline{710 N. Pleasant Street,
Amherst, MA 01003-9305, U.S.A.}
\centerline{\email{sshahshahani@umass.edu}}
}

\end{document}